\documentclass[12pt,twoside]{amsart}
\usepackage{amssymb,amscd,amsxtra,calc}
\usepackage{amsmath}
\usepackage{xypic}
\usepackage[graph]{xy}
\usepackage{supertabular}
\usepackage{longtable}

\usepackage[dvips]{color}

\title{Classification of log del Pezzo surfaces of index three}
\author{Kento Fujita} 
\author{Kazunori Yasutake}
\date{\today}
\address{Research Institute for Mathematical Sciences, 
Kyoto University, Kyoto 606-8502 Japan}
\email{fujita@kurims.kyoto-u.ac.jp}
\address{Organization for the Strategic Coordination of Research and 
Intellectual Properties, Meiji University, Kanagawa 214-8571 Japan}
\email{tz13008@meiji.ac.jp}

\newcommand{\pr}{\mathbb{P}}

\newcommand{\Z}{\mathbb{Z}}
\newcommand{\Q}{\mathbb{Q}}

\newcommand{\C}{\mathbb{C}}
\newcommand{\F}{\mathbb{F}}

\newcommand{\Supp}{\operatorname{Supp}}

\newcommand{\Pic}{\operatorname{Pic}}

\newcommand{\discrep}{\operatorname{discrep}}

\newcommand{\mult}{\operatorname{mult}}
\newcommand{\coeff}{\operatorname{coeff}}

\newcommand{\sI}{\mathcal{I}}
\newcommand{\sC}{\mathcal{C}}
\newcommand{\sO}{\mathcal{O}}

\newcommand{\sF}{\mathcal{F}}
\newcommand{\sB}{\mathcal{B}}

\newcommand{\dm}{\mathfrak{m}}

\newcommand{\gA}{\operatorname{A}}
\newcommand{\gD}{\operatorname{D}}

\newcommand{\bi}{\bfseries\itshape}

\newtheorem{thm}{Theorem}[section]
\newtheorem{lemma}[thm]{Lemma}
\newtheorem{proposition}[thm]{Proposition}
\newtheorem{corollary}[thm]{Corollary}
\newtheorem{claim}[thm]{Claim}

\theoremstyle{definition}
\newtheorem{definition}[thm]{Definition}
\newtheorem{proposition-definition}[thm]{Proposition-Definition}
\newtheorem{remark}[thm]{Remark}

\newtheorem*{ack}{Acknowledgments} 
\newtheorem*{nott}{Notation and terminology}

\begin{document}

\maketitle 

\begin{abstract}
A normal projective non-Gorenstein 
log-terminal surface $S$ is called a log del Pezzo surface 
of index three if the three-times of the anti-canonical divisor $-3K_S$ is an 
ample Cartier divisor. 
We classify all of the log del Pezzo surfaces of index three. 
The technique for the classification based on the argument of Nakayama.
\end{abstract}

\setcounter{tocdepth}{1}
\tableofcontents

\section{Introduction}\label{intro_section}

A normal projective surface $S$ is called a \emph{log del Pezzo surface} 
if $S$ is log-terminal and the anti-canonical divisor 
$-K_S$ is ample ($\Q$-Cartier divisor). 
Log del Pezzo surfaces constitute an interesting class of rational surfaces 
and naturally appear in the minimal model program (MMP, for short).   
An important invariant of a log del Pezzo surface S is the \emph{index}, 
which is defined to be the minimum of the 
positive integer $a$ such that $-aK_S$ is Cartier.   
Log del Pezzo surfaces with small index have been studied by many authors.
The classification of log del Pezzo surfaces with index one (that is,  with at most 
rational double points) is well-known 
(see \cite{brenton}, \cite{demazure}, \cite{HW}).

The next case, the classification of log del Pezzo surfaces of index two, 
was also studied by several authors. 
Alexeev and Nikulin specify all the deformation classes of log del Pezzo surfaces 
of index two over the complex number field $\C$ 
by using K3 surface theory \cite{AN1, AN2, AN3}.
Recently, Ohashi and Taki proceed this method and classify 
the deformation classes of log del Pezzo surfaces of index three 
under the condition $-3K_S\sim 2C$ 
where $C$ is a smooth curve 
which does not intersect the singularities. 
On the other hand, Nakayama introduce the geometric argument for the study of 
log del Pezzo surfaces of index two which is completely different to 
that of Alexeev-Nikulin, and he gave the complete list of isomorphic classes of 
log del Pezzo surfaces of index two in any characteristic \cite{N}. 
Nakayama's argument is useful in the study of 
log del Pezzo surfaces not only the case index is two but also 
the case index is arbitrary. In fact, by using Nakayama's idea, the first author 
classified some classes of log del Pezzo surfaces in \cite{F} that include 
the classes treated in the study of Ohashi and Taki.

In this paper, we extend a part of Nakayama's argument to work 
in arbitrary index. 
Moreover, we give the classification of log del Pezzo surfaces of index three 
by using this method. 
Our strategy to understand log del Pezzo surfaces is as follows. 
(Detail is given in Section \ref{dP_section}. See also \cite{dPvolume}.)
Let $S$ be a log del Pezzo surface of index $a>1$. 
Take the minimal resolution $\alpha\colon M\to S$ and set $E_M:=-aK_{M/S}$.
Then we know that $M$ is nonsingular rational and $E_M$ is nonzero effective.
We can recover $S$ from the pair $(M, E_M)$ by considering the 
morphism defined from a multiple of the divisor $L_M:=-aK_M-E_M$. 
Hence we can reduce the study of $S$ to the study of such $(M, E_M)$.
We remark that $K_M+L_M$ is nef and $(K_M+L_M\cdot L_M)>0$ hold 
(see Proposition \ref{dP-basic_prop}).
We call such pair $(M, E_M)$ an \emph{$a$-basic pair} 
(see Definition \ref{basic_dfn}). 

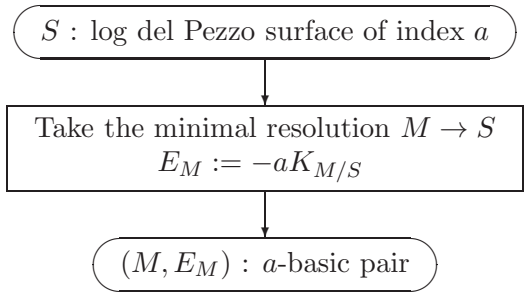
\begin{figure}[h]
\begin{center}
    \begin{picture}(100, 100)(0,0)
    \put(50, 90){\oval(190, 20)}
    \put(50, 90){\makebox(0, 0){{\small $S$ : log del Pezzo surface of index $a$}}}
    \put(50,80){\vector(0,-1){18}}
    \put(50, 46){\makebox(0, 0){\framebox{{\small\begin{minipage}[c]{65mm}
    \begin{center}
        \rule[0mm]{0mm}{3mm}
        Take the minimal resolution $M \rightarrow S$\\
        $E_{M}:=-aK_{M/S}$
    \end{center}
    \end{minipage}
    }}}}
    \put(50,30){\vector(0,-1){18}}
    \put(50, 2){\oval(130, 20)}
    \put(50, 2){\makebox(0, 0){{\small $(M,E_{M})$ : $a$-basic pair}}}
\end{picture}
\caption{Reduction to $a$-basic pairs}\label{fig1}
\end{center}
\end{figure}

From now on, let $(M, E_M)$ be an $a$-basic pair. Since $M$ is rational, 
we can get a birational morphism from $M$ to $\pr^2$ or a Hirzebruch surface 
$\F_n$ having a $(-n)$-curve. 
However, it is hard to analyze the morphism in general. 
To evade this problem, we ``decompose" the step contracting $(-1)$-curves into 
$((i+1)K+L)$-minimal model programs 
(\emph{$((i+1)K+L)$-MMPs}, for short) 
for $1\leqslant i\leqslant a-1$. 
More precisely, we give a sequence 
\[
M=M_0\xrightarrow{\pi_1}M_1\xrightarrow{\pi_2}\cdots\xrightarrow{\pi_b}M_b
\] 
for some integer $b$ such that $1\leqslant b\leqslant a-1$. 
The construction of each $\pi_i$ is done inductively as follows. 
We assume that $iK_{M_{i-1}}+L_{i-1}$ is nef and $E_{i-1}$ is nonzero effective, 
where $L_{i-1}$ (resp.\ $E_{i-1}$) is the strict transform of 
$L_M$ (resp.\ $E_M$) in $M_{i-1}$.
The morphism $\pi _i\colon M_{i-1}\rightarrow M_i$ is obtained by the composition 
of all the morphisms in the step of an $((i+1)K_{M_{i-1}}+L_{i-1})$-MMP. 
More precisely, in each step, we contract a $(-1)$-curve which intersects 
(the strict transform of) $(i+1)K_{M_{i-1}}+L_{i-1}$ negatively. 
We continue this process until we get a Mori fiber space or a 
minimal model with respect to $((i+1)K_{M_{i-1}}+L_{i-1})$-MMP. 
If this MMP induces a minimal model (with respect to 
the $((i+1)K_{M_{i-1}}+L_{i-1})$-MMP), then we go back to construct 
$\pi_{i+1}\colon M_i\to M_{i+1}$. If this MMP induces a Mori fiber space, then set 
$b:=i$ and stop the process.
We can show that $E_i$ is also nonzero effective for each $i$. 
We note that $1\leqslant b\leqslant a-1$ 
since $aK_{M_i}+L_i=-E_i$ cannot be nef for each $1\leqslant i\leqslant b$. 
The surface $M_b$ is either $\pr^2$ or $\F_n$. 
From the construction, we have $iK_{M_{i-1}}+L_{i-1}=\pi _i^*(iK_{M_{i}}+L_i)$ 
for each $i$. In particular, $-K_{M_{i-1}}$ is $\pi_i$-nef. 
Let $\Delta_i\subset M_i$ be a closed zero-dimensional subscheme such that 
the corresponding ideal sheaf $\sI_{\Delta_i}$ is defined as 
$\sI_{\Delta_i}:=(\pi_i) _*\sO_{M_{i-1}}(-K_{M_{i-1}/M_i})$.
The scheme $\Delta_i$ has a nice property (called the \emph{$(\nu1)$-condition} 
in Definition \ref{nu1_dfn}). For example, the morphism $\pi _i$ is recovered 
from $\Delta_i$ (see Definition \ref{elim_dfn} and Proposition \ref{elim_prop}).
The multiplet $(M_b, E_b; \Delta_1,\dots, \Delta_b)$ constructed as above 
is called an \emph{$a$-fundamental multiplet of length $b$}.
The classification of $a$-basic pairs reduce to the classification of 
$a$-fundamental multiplets. This is our strategy. 
In the case where $a=2$, this is nothing but Nakayama's argument.
(In Section \ref{dP_section}, we only consider the case $a=3$. However, 
the program we mentioned works for arbitrary $a$. See \cite{dPvolume} in detail.)
We summarize our strategy via flowcharts in Figures \ref{fig1} and \ref{fig2}.

\begin{figure}[ht]
\begin{center}
    \begin{picture}(200, 315)(0,0)
    \put(100, 300){\oval(130, 20)}
    \put(100, 300){\makebox(0, 0){{\small $(M, E_M)$ : $a$-basic pair}}}
    \put(100,290){\vector(0,-1){16}}
    \put(100, 258){\makebox(0, 0){\framebox{\begin{minipage}[]{50mm}
    \begin{center}
        \rule[0mm]{0mm}{3mm}
        {\small
        $M_0:=M$,
        $E_0:=E_M$\\
        $L_{0}:=-aK_{M_0}-E_{0}$,\,\,
        $i:=1$}
    \end{center}
    \end{minipage}
    }}}
    \put(100,242){\vector(0,-1){24}}
    \put(100, 176.3){\makebox(0, 0){\framebox{\begin{minipage}[]{60mm}
    \begin{center}
        {\small
        Run $((i+1)K_{M_{i-1}}+L_{i-1})$-MMP;\\
        $\pi_i\colon M_{i-1}\to M_i$ is the\\ 
        output of the MMP, where \\
        $\Delta_i\subset M_i$ : subscheme with
        $\sI_{\Delta_i}:=(\pi_i)_*\sO_{M_{i-1}}(-K_{M_{i-1}/M_i})$\\
        $E_i:=(\pi_i)_*E_{i-1}$,\, $L_i:=(\pi_i)_*L_{i-1}$}
    \end{center}
    \end{minipage}
    }}}
    \put(100,134.8){\vector(0,-1){16.8}}
    \put(100,118){\line(3,-1){82}}
    \put(100,118){\line(-3,-1){82}}
    \put(100,63){\line(3,1){82}}
    \put(100,63){\line(-3,1){82}}
    \put(42, 88){{\small Is $(i+1)K_{M_i}+L_i$ nef ?}}
    \put(100,63){\vector(0,-1){16}}
    \put(108, 52){{\small No}}
    \put(100, 40){\makebox(0, 0){{\small\framebox{$b:=i$}}}}
    \put(100,33){\vector(0,-1){16}}
    \put(100, 7){\oval(300, 20)}
    \put(100, 7){\makebox(0, 0){{\small $(M_b, E_b; \Delta_1,\dots, \Delta_b)$ : 
$a$-fundamental multiplet of length $b$}}}
    \put(182,90.5){\vector(1,0){30.5}}
    \put(187, 96){{\small Yes}}
    \put(237.9, 90){\makebox(0, 0){\framebox{{\small $i:=i+1$}}}}
    \put(238,97.2){\line(0,1){132.8}}
    \put(238,230){\vector(-1,0){138}}
\end{picture}
\caption{Reduction to $a$-fundamental multiplets}\label{fig2}
\end{center}
\end{figure}
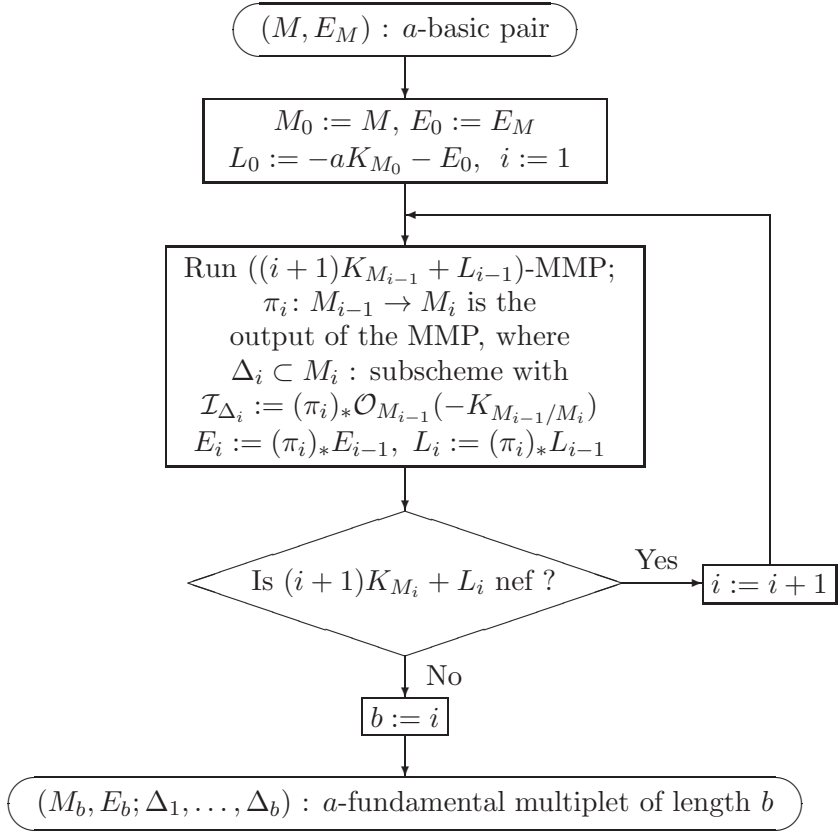

Our approach is useful from various viewpoints. 
For example, it is easy to handle $a$-fundamental multiplets 
$(M_b, E_b; \Delta_1,\dots, \Delta_b)$ since we can analyze 
each $\Delta_i$ deeply and
we can study the multiplets by somewhat numerical and combinatorial ways. 
Furthermore, the choice of the process $\pi_i\colon M_{i-1}\to M_i$ 
is less ambiguous. In fact, if $1\leqslant i\leqslant b-1$, then $\pi_i$ 
is uniquely determined since $iK_{M_{i-1}}+L_{i-1}$ is nef and big. 

Next we consider the case where $a=3$, which is the main subject treated 
in this paper. In this case we treat the following four objects:
\begin{itemize}
\item 
A log del Pezzo surface $S$ of index three.
\item 
A $3$-basic pair $(M, E_M)$ consisting of a kind of
nonsingular projective rational surface $M$ 
and an effective divisor $E_M$.
\item 
A \emph{median triplet} $(Z,E_Z;\Delta _Z)$, 
which is a kind of $3$-fundamental multiplet of length one, 
consisting of a rational surface $Z$ isomorphic to 
$\pr^2$ or $\F_n$, of an effective divisor $E_Z$, and of a zero-dimensional subscheme 
$\Delta _Z\subset Z$ with the $(\nu1)$-condition.      
\item 
A \emph{bottom tetrad} $(X,E_X;\Delta _Z, \Delta _X)$, 
which is a kind of $3$-funda-mental multiplet of length two, 
consisting of a rational surface $X$ 
isomorphic to $\pr^2$ or $\F_n$, of an effective divisor $E_X$, of a zero-dimensional 
subscheme $\Delta _X\subset X$, and of a zero-dimensional subscheme 
$\Delta _Z\subset Z$ with the $(\nu1)$-condition, where $Z\to X$ is the 
elimination of $\Delta_X$ (see Definition \ref{elim_dfn}). 
\end{itemize}
The classes of median triplets and bottom tetrads are introduced in order 
to get the list of log del Pezzo surfaces of index three without duplication. 
In Sections \ref{dP_section} and \ref{cubic_section}, 
we show that for any $3$-fundamental multiplet of 
length one (resp.\ of length two) we have a median triplet 
(resp.\ a bottom tetrad) such that the associated $3$-basic pairs are same. 
By virtue of this modifications we can obtain the list of 
log del Pezzo surfaces of index three without overlap. 

Now we summarize the contents of this paper. 
In Section \ref{prelim_section}, we review some basic properties of 
zero-dimensional schemes which satisfies the $(\nu1)$-condition and 
we give the list of (weighted) dual graphs associated with log-terminal 
singularities of index three. 
In Section \ref{dP_section}, we introduce the notions of 
$3$-basic pairs, $3$-(pseudo-)fundamental multiplets, 
(pseudo-)median triplets and bottom tetrads 
associated with log del Pezzo surfaces of index three. Moreover, 
we discuss relations among them.
Precisely, we show that the structure of log del Pezzo surfaces of index three 
is specified from the one of $3$-fundamental multiplets of length one and of 
length two through the $3$-basic pairs.
Furthermore, we see that the classification of $3$-fundamental multiplets 
of length one (resp.\ $3$-fundamental multiplets of length two)
can be reduced to that of median triplets (resp.\ bottom tetrads). 
In Section \ref{local_section}, we discuss some local properties of 
$3$-(pseudo-)fundamental multiplets which are used in latter sections.
More precisely, we determine the possibility of the structure of 
zero-dimensional subschemes $\Delta_Z$ and $\Delta_X$ with $(\nu 1)$-condition 
over a fixed point on some effective divisor.   
Thanks to the arguments in this section, we can pare down the candidates 
of zero-dimensional schemes of $3$-fundamental multiplets. 
Section \ref{cubic_section} is the most technical part in this paper. 
In this section, we treat $3$-fundamental multiplets 
of length two with trivial $2K_X+L_X$ which give the same log del Pezzo surface of 
index three. Thanks to the arguments in this section, it makes sense that 
the conditions of Definition \ref{bottom_dfn}. 
In Section \ref{ZC_section}, we classify median triplets.
There are exactly $77$ types of median triplets (see Theorem \ref{trip_thm}). 
From Section \ref{XCI_section} to Section \ref{XCIII_section}, 
we give the classification of bottom tetrads $(X, E_X; \Delta_Z, \Delta_X)$.
In Section \ref{XCI_section}, we classify bottom tetrads with big $2K_X+L_X$.
There are exactly $45$ types of such tetrads (see Theorem \ref{tetI_thm}).
In Section \ref{XCII_section}, we classify bottom tetrads with 
non-big and nontrivial $2K_X+L_X$.
There are exactly $115$ types of such tetrads (see Theorem \ref{tetII_thm}).
In Section \ref{XCIII_section}, we classify bottom tetrads such that 
$2K_X+L_X$ is trivial. 
There are exactly $63$ types of such tetrads (see Theorem \ref{tetIII_thm}). 
In Section \ref{struc_section}, we see some structure properties of log 
del Pezzo surfaces of index three. 
In Proposition \ref{kaburi_prop}, we show that the lists in Sections 
\ref{ZC_section}--\ref{XCIII_section} has no redundancy. 
In Proposition \ref{dual_graph_prop}, we tabulate the list of non-Gorenstein points for 
log del Pezzo surfaces of index three.

\begin{ack}
This work was started when the second author visited RIMS. 
The second author would like to express his gratitude to Professor Shigeru Mukai 
for hospitality. The authors thank Professor Noboru Nakayama for useful comments. 
The first author is partially supported by a JSPS Fellowship for Young Scientists. 
\end{ack}

\begin{nott}
We work in the category of algebraic (separated and of finite type) scheme over a 
fixed algebraically closed field $\Bbbk$ of arbitrary characteristic. 
A \emph{variety} means a reduced and irreducible algebraic scheme. A \emph{surface} 
means a two-dimensional variety. 
For a proper variety $X$, let $\rho(X)$ be the Picard number of $X$. 

For a normal variety $X$, we say that $D$ is a 
\emph{$\Q$-divisor} (resp.\ \emph{divisor} or \emph{$\Z$-divisor}) 
if $D$ is a finite sum $D=\sum a_iD_i$ where $D_i$ are prime divisors and 
$a_i\in\Q$ (resp.\ $a_i\in\Z$). 
For a $\Q$-divisor $D=\sum a_iD_i$, the value $a_i$ is denoted by $\coeff_{D_i}D$ and 
set $\coeff D:=\{a_i\}_i$. 
A normal variety $X$ is called \emph{log-terminal} if the canonical divisor $K_X$ is 
$\Q$-Cartier and the discrepancy $\discrep(X)$ of $X$ is bigger than $-1$ 
(see \cite[\S 2.3]{KoMo}).
For a proper birational morphism $f\colon Y\to X$ between normal varieties 
such that both $K_X$ and $K_Y$ are $\Q$-Cartier, we set 
\[
K_{Y/X}:=\sum_{E_0\subset Y \,\,f\text{-exceptional}}a(E_0, X)E_0, 
\]
where $a(E_0, X)$ is the discrepancy of $E_0$ with respects to $X$ 
(see \cite[\S 2.3]{KoMo}). (We note that 
if $aK_X$ and $aK_Y$ are Cartier for $a\in\Z_{>0}$, then 
$aK_{Y/X}$ is a $\Z$-divisor.)

For a nonsingular surface $S$ and a projective curve $C$ which is a closed subvariety 
of $S$, the curve $C$ is called a \emph{$(-n)$-curve} if $C$ is a nonsingular rational 
curve and $(C^2)=-n$. 
For a birational map $M\dashrightarrow S$ between normal surfaces 
and a curve $C\subset S$, the strict transform of $C$ on $M$ 
is denoted by $C^M$. 

For a zero-dimensional sucheme $\Delta$, let $|\Delta|$ be the support of $\Delta$.

Let $S$ be a nonsingular surface and let $E=\sum w_jD_j$ be an effective divisor on $S$ 
$(w_j>0)$. 
The \emph{weighted dual graph} of $E$ is defined as follows. A vertex corresponds 
to a component $D_j$. Let $v_j$ be the vertex corresponds to $D_j$. 
Assume that $D_i$ and $D_j$ satisfies that $|D_i\cap D_j|=\{P_1,\dots,P_m\}$
such that the local intersection number of $D_i$ and $D_j$ at $P_k$ is $s_k$. 
For any $1\leqslant h\leqslant m$, 
$v_i$ and $v_j$ are joined by a line with the numbered box \fbox{$s_h$} 
if $s_h\geqslant 2$. 
If $s_h=1$, then $v_i$ and $v_j$ are joined by a line with no box. 
Moreover, for each vertex $v$ corresponds to $D$, we define the \emph{weight} $w$
of $v$ as $w:=\coeff_DE$ and denoted by $v_{(w)}$. 
In the dual graphs of divisors, a vertex corresponding to $(-n)$-curve is expressed 
as \textcircled{\tiny $n$}. 
On the other hand, an arbitrary irreducible curve is expressed by the symbol 
{\large$\oslash$} when it is not necessary a $(-n)$-curve.

Let $\F_n\to\pr^1$ be a Hirzebruch surface $\pr_{\pr^1}(\sO\oplus\sO(n))$ of 
degree $n$ with the $\pr^1$-fibration. 
A section $\sigma\subset\F_n$ with $\sigma^2=-n$ 
is called a \emph{minimal section}. If $n>0$, then such $\sigma$ is unique.
A section $\sigma_{\infty}$ with $\sigma\cap\sigma_\infty=\emptyset$ is called a
\emph{section at infinity}. For a section at infinity $\sigma_\infty$, we have 
$\sigma_\infty\sim\sigma+nl$, where $l$ is a fiber of the fibration. 
For the projective plane $\pr^2$, we sometimes denote a line on $\pr^2$ by $l$. 

For two integers $c$ and $d$, we set $s(c, d):=\max\{0$, $c+d-1\}$.
\end{nott}

\section{Preliminaries}\label{prelim_section}

\subsection{Elimination of subschemes}\label{elim_subsec}

In this section, we recall the results in \cite{N}. 
Let $X$ be a nonsingular surface and $\Delta$ be a zero-dimensional subscheme 
of $X$. The ideal sheaf of $\Delta$ is denoted by $\sI_{\Delta}$.

\begin{definition}\label{nu1_dfn}
Let $P$ be a point of $\Delta$. 
\begin{enumerate}
\renewcommand{\theenumi}{\arabic{enumi}}
\renewcommand{\labelenumi}{(\theenumi)}
\item\label{nu1_dfn1}
Let 
$
\nu_P(\Delta):=\max\{\nu\in\Z_{>0} \, | \, \sI_{\Delta}\subset\dm_P^\nu\},
$
where $\dm_P$ is the maximal ideal sheaf in $\sO_X$ defining $P$.
If $\nu_P(\Delta)=1$ for any $P\in\Delta$, then we say that $\Delta$ \emph{satisfies 
the $(\nu1)$-condition}.
\item\label{nu1_dfn2}
The \emph{multiplicity} $\mult_P\Delta$ of $\Delta$ at $P$ is given by 
the length of the Artinian local ring $\sO_{\Delta, P}$.
\item\label{nu1_dfn3}
The \emph{degree} $\deg\Delta$ of $\Delta$ is 
given by $\sum_{P\in\Delta}\mult_P\Delta$.
\end{enumerate}
\end{definition}

\begin{definition}\label{mult_div}
For an effective divisor $D$ and a point $P$, we set 
$
\mult_PD:=\max\{\nu\in\Z_{>0} \, | \, \sO_X(-D)\subset\dm_P^\nu\}.
$
Let $\pi\colon Y\to X$ be the blowing up along $P$ and let $e$ be 
the exceptional curve. Then $\mult_PD$ is equal to $\coeff_{e}\pi^*D$. 
\end{definition}

\begin{definition}\label{elim_dfn}
Assume that $\Delta$ satisfies the $(\nu1)$-condition. 
Let $V\to X$ be the blowing up along $\Delta$. 
The \emph{elimination} of $\Delta$ is the birational projective morphism 
$\psi\colon Z\to X$ defined as the composition of the minimal resolution 
$Z\to V$ of $V$ and the morphism $V\to X$. 
For a divisor $E$ on $X$ and for a positive integer $s$, we set 
$E_Z^{\Delta,s}:=\psi^*E-sK_{Z/X}$. 
\end{definition}

\begin{proposition}[{\cite[Proposition 2.9]{N}}]\label{elim_prop}
\begin{enumerate}
\renewcommand{\theenumi}{\arabic{enumi}}
\renewcommand{\labelenumi}{(\theenumi)}
\item\label{elim_prop1}
We assume that the subscheme $\Delta$ satisfies the $(\nu1)$-condition 
and let $\psi\colon Z\to X$ 
be the elimination of $\Delta$. 
Then the anti-canonical divisor $-K_Z$ is $\psi$-nef. 
More precisely, for any $P\in\Delta$ with $\mult_P\Delta=k$, 
the set-theoretic inverse image $\psi^{-1}(P)$ is the 
straight chain $\sum_{j=1}^k\Gamma_{P,j}$ of 
nonsingular rational curves and the weighted dual graph of the divisor 
$K_{Z/X}$ around over $P$ is the following: 
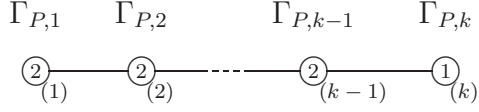
\begin{center}
    \begin{picture}(150, 50)(0, 45)
    \put(0, 60){\makebox(0, -3){\textcircled{\tiny $2$}}}
    \put(0, 75){\makebox(0, 0)[b]{$\Gamma_{P,1}$}}
    \put(7, 52){\makebox(0, 0){\tiny $(1)$}}
    \put(5, 60){\line(1, 0){30}}
    \put(40, 60){\makebox(0, -3){\textcircled{\tiny $2$}}}
    \put(47, 52){\makebox(0, 0){\tiny $(2)$}}
    \put(40, 75){\makebox(0, 0)[b]{$\Gamma_{P,2}$}}
    \put(45, 60){\line(1, 0){20}} 
    \put(67, 60){\line(1, 0){2}}
    \put(71, 60){\line(1, 0){2}}
    \put(75, 60){\line(1, 0){2}}    
    \put(79, 60){\line(1, 0){21}}
    \put(105, 60){\makebox(0, -3){\textcircled{\tiny $2$}}}
    \put(105, 75){\makebox(0, 0)[b]{$\Gamma_{P,k-1}$}}
    \put(120, 52){\makebox(0, 0){\tiny $(k-1)$}}
    \put(110, 60){\line(1, 0){40}}
    \put(155, 60){\makebox(0, -3){\textcircled{\tiny $1$}}}
    \put(155, 75){\makebox(0, 0)[b]{$\Gamma_{P,k}$}}
    \put(162, 52){\makebox(0, 0){\tiny $(k)$}}
    \end{picture}
\end{center}
\item\label{elim_prop2}
Conversely, for a non-isomorphic proper birational morphism $\psi\colon Z\to X$
between nonsingular surfaces such that $-K_Z$ is $\psi$-nef, 
the morphism $\psi$ 
is the elimination of $\Delta$ which satisfies the $(\nu1)$-condition 
defined by the ideal $\sI_\Delta:=\psi_*\sO_Z(-K_{Z/X})$.
\end{enumerate}
\end{proposition}

\begin{definition}\label{gamma_dfn}
Under the assumption of Proposition \ref{elim_prop} \eqref{elim_prop1}, we always 
denote the exceptional curves of $\psi$ over $P$ by 
$\Gamma_{P,1},\dots,\Gamma_{P,k}$. The order is determined as 
Proposition \ref{elim_prop} \eqref{elim_prop1}. We set $\Gamma_{P, 0}:=\emptyset$.
\end{definition}

\subsection{Curves in nonsingular surfaces}\label{singcurve_section}

\begin{lemma}\label{mult_seq_lem}
Let $\pi\colon M\to X$ be a birational morphism between nonsingular projective 
surfaces and let $C$ be a reduced and irreducible curve on $X$. Then 
$(C^2)-((C^M)^2)=(K_{M/X}\cdot C^M)+2p_a(C)-2p_a(C^M)$, where $p_a(\bullet)$ 
is the arithmetic genus. 
\end{lemma}

\begin{proof}
Follows from the genus formula. 
\end{proof}

\begin{proposition}[{\cite[Corollary 2.10]{F}}]\label{twocurve_prop}
Let $X$ be a nonsingular complete surface, $\Delta$ be a zero-dimensional 
subscheme of $X$ which satisfies the $(\nu1)$-condition, $\pi\colon M\to X$ 
be the elimination of $\Delta$ and $C_1$, $C_2$ be distinct nonsingular curves on $X$. 
We set $k:=\deg\Delta$ and $k_h:=\deg(\Delta\cap C_h)$ for $h=1$, $2$. 
Then $(C_1\cdot C_2)\geqslant k_1+k_2-k$ holds.
\end{proposition}

\subsection{Dual exceptional graphs}\label{graph_section}

In this section, we see the classification result of the weighted dual graphs of the 
exceptional divisors for non-Gorenstein log-terminal surface singularities of index three. 
If $\Bbbk=\C$, Ohashi-Taki completed the classification in \cite[\S 2]{OT}. 
We remark that the following list is same as the list in \cite[\S 2]{OT}.

\begin{proposition}
Let $P\in S$ be a non-Gorenstein log-terminal surface singularity and 
let $\alpha\colon M\to S$ be the minimal resolution of $P\in S$. 
Assume that $-3K_S$ is Cartier. 
Then the weighted dual graph of the effective $\Z$-divisor $-3K_{M/S}$ 
is one of the list in Table \ref{graph_table}. 
\end{proposition}

\begin{table}[h]
\begin{center}
\caption{The list of the weighted dual graphs of $-3K_{M/S}$.}\label{graph_table}
\begin{tabular}{|c|c|}\hline
Symbol & Graph \\ \hline \hline 
$\gA_1(1)$ & 
\begin{picture}(5,10)(0,0)
    \put(0, 2){\makebox(0, -3){\textcircled{\tiny $3$}}}
    \put(7, -6){\makebox(0, 0){\tiny $(1)$}}
\end{picture}
\\
 & 
\\ 
$\gA_1(2)$ & 
\begin{picture}(5,10)(0,0)
    \put(0, 2){\makebox(0, -3){\textcircled{\tiny $6$}}}
    \put(7, -6){\makebox(0, 0){\tiny $(2)$}}
\end{picture}
\\
 & 
\\ \hline
$\gA_2(1, 2)$ & 
\begin{picture}(60,10)(-15,0)
    \put(0, 2){\makebox(0, -3){\textcircled{\tiny $2$}}}
    \put(7, -6){\makebox(0, 0){\tiny $(1)$}}
    \put(5, 2){\line(1, 0){20}} 
    \put(30, 2){\makebox(0, -3){\textcircled{\tiny $5$}}}
    \put(37, -6){\makebox(0, 0){\tiny $(2)$}}    
\end{picture}
\\
 & 
\\ 
$\gA_2(2, 2)$ & 
\begin{picture}(60,10)(-15,0)
    \put(0, 2){\makebox(0, -3){\textcircled{\tiny $4$}}}
    \put(7, -6){\makebox(0, 0){\tiny $(2)$}}
    \put(5, 2){\line(1, 0){20}} 
    \put(30, 2){\makebox(0, -3){\textcircled{\tiny $4$}}}
    \put(37, -6){\makebox(0, 0){\tiny $(2)$}}    
\end{picture}
\\ & \\ \hline
$\gA_3(1, 1)$ & 
\begin{picture}(90,10)(-15,0)
    \put(0, 2){\makebox(0, -3){\textcircled{\tiny $2$}}}
    \put(7, -6){\makebox(0, 0){\tiny $(1)$}}
    \put(5, 2){\line(1, 0){20}} 
    \put(30, 2){\makebox(0, -3){\textcircled{\tiny $4$}}}
    \put(37, -6){\makebox(0, 0){\tiny $(2)$}}
    \put(35, 2){\line(1, 0){20}} 
    \put(60, 2){\makebox(0, -3){\textcircled{\tiny $2$}}}
    \put(67, -6){\makebox(0, 0){\tiny $(1)$}}
\end{picture}
\\
 & 
\\ 
$\gA_3(1, 2)$ & 
\begin{picture}(90,10)(-15,0)
    \put(0, 2){\makebox(0, -3){\textcircled{\tiny $2$}}}
    \put(7, -6){\makebox(0, 0){\tiny $(1)$}}
    \put(5, 2){\line(1, 0){20}} 
    \put(30, 2){\makebox(0, -3){\textcircled{\tiny $3$}}}
    \put(37, -6){\makebox(0, 0){\tiny $(2)$}}
    \put(35, 2){\line(1, 0){20}} 
    \put(60, 2){\makebox(0, -3){\textcircled{\tiny $4$}}}
    \put(67, -6){\makebox(0, 0){\tiny $(2)$}}
\end{picture}
\\
 & 
\\
$\gA_3(2, 2)$ & 
\begin{picture}(90,10)(-15,0)
    \put(0, 2){\makebox(0, -3){\textcircled{\tiny $4$}}}
    \put(7, -6){\makebox(0, 0){\tiny $(2)$}}
    \put(5, 2){\line(1, 0){20}} 
    \put(30, 2){\makebox(0, -3){\textcircled{\tiny $2$}}}
    \put(37, -6){\makebox(0, 0){\tiny $(2)$}}
    \put(35, 2){\line(1, 0){20}} 
    \put(60, 2){\makebox(0, -3){\textcircled{\tiny $4$}}}
    \put(67, -6){\makebox(0, 0){\tiny $(2)$}}
\end{picture}
\\ & \\ \hline
$\gA_t(1, 1)$ & 
\begin{picture}(205,10)(-15,0)
    \put(0, 2){\makebox(0, -3){\textcircled{\tiny $2$}}}
    \put(7, -6){\makebox(0, 0){\tiny $(1)$}}
    \put(5, 2){\line(1, 0){20}} 
    \put(30, 2){\makebox(0, -3){\textcircled{\tiny $3$}}}
    \put(37, -6){\makebox(0, 0){\tiny $(2)$}}
    \put(35, 2){\line(1, 0){20}} 
    \put(60, 2){\makebox(0, -3){\textcircled{\tiny $2$}}}
    \put(67, -6){\makebox(0, 0){\tiny $(2)$}}
    \put(65, 2){\line(1, 0){15}}
    \put(82, 2){\line(1, 0){2}}
    \put(86, 2){\line(1, 0){2}}
    \put(90, 2){\line(1, 0){2}}
    \put(94, 2){\line(1, 0){16}}
    \put(115, 2){\makebox(0, -3){\textcircled{\tiny $2$}}}
    \put(122, -6){\makebox(0, 0){\tiny $(2)$}}
    \put(120, 2){\line(1, 0){20}} 
    \put(145, 2){\makebox(0, -3){\textcircled{\tiny $3$}}}
    \put(152, -6){\makebox(0, 0){\tiny $(2)$}}
    \put(150, 2){\line(1, 0){20}} 
    \put(175, 2){\makebox(0, -3){\textcircled{\tiny $2$}}}
    \put(182, -6){\makebox(0, 0){\tiny $(1)$}}
\end{picture}
\\ & 
\\
$\gA_t(1, 2)$ & 
\begin{picture}(205,10)(-15,0)
    \put(0, 2){\makebox(0, -3){\textcircled{\tiny $2$}}}
    \put(7, -6){\makebox(0, 0){\tiny $(1)$}}
    \put(5, 2){\line(1, 0){20}} 
    \put(30, 2){\makebox(0, -3){\textcircled{\tiny $3$}}}
    \put(37, -6){\makebox(0, 0){\tiny $(2)$}}
    \put(35, 2){\line(1, 0){20}} 
    \put(60, 2){\makebox(0, -3){\textcircled{\tiny $2$}}}
    \put(67, -6){\makebox(0, 0){\tiny $(2)$}}
    \put(65, 2){\line(1, 0){15}}
    \put(82, 2){\line(1, 0){2}}
    \put(86, 2){\line(1, 0){2}}
    \put(90, 2){\line(1, 0){2}}
    \put(94, 2){\line(1, 0){16}}
    \put(115, 2){\makebox(0, -3){\textcircled{\tiny $2$}}}
    \put(122, -6){\makebox(0, 0){\tiny $(2)$}}
    \put(120, 2){\line(1, 0){20}} 
    \put(145, 2){\makebox(0, -3){\textcircled{\tiny $2$}}}
    \put(152, -6){\makebox(0, 0){\tiny $(2)$}}
    \put(150, 2){\line(1, 0){20}} 
    \put(175, 2){\makebox(0, -3){\textcircled{\tiny $4$}}}
    \put(182, -6){\makebox(0, 0){\tiny $(2)$}}
\end{picture}
\\
 & 
\\
$\gA_t(2, 2)$ & 
\begin{picture}(205,10)(-15,0)
    \put(0, 2){\makebox(0, -3){\textcircled{\tiny $4$}}}
    \put(7, -6){\makebox(0, 0){\tiny $(2)$}}
    \put(5, 2){\line(1, 0){20}} 
    \put(30, 2){\makebox(0, -3){\textcircled{\tiny $2$}}}
    \put(37, -6){\makebox(0, 0){\tiny $(2)$}}
    \put(35, 2){\line(1, 0){20}} 
    \put(60, 2){\makebox(0, -3){\textcircled{\tiny $2$}}}
    \put(67, -6){\makebox(0, 0){\tiny $(2)$}}
    \put(65, 2){\line(1, 0){15}}
    \put(82, 2){\line(1, 0){2}}
    \put(86, 2){\line(1, 0){2}}
    \put(90, 2){\line(1, 0){2}}
    \put(94, 2){\line(1, 0){16}}
    \put(115, 2){\makebox(0, -3){\textcircled{\tiny $2$}}}
    \put(122, -6){\makebox(0, 0){\tiny $(2)$}}
    \put(120, 2){\line(1, 0){20}} 
    \put(145, 2){\makebox(0, -3){\textcircled{\tiny $2$}}}
    \put(152, -6){\makebox(0, 0){\tiny $(2)$}}
    \put(150, 2){\line(1, 0){20}} 
    \put(175, 2){\makebox(0, -3){\textcircled{\tiny $4$}}}
    \put(182, -6){\makebox(0, 0){\tiny $(2)$}}
\end{picture}
\\ 
{\footnotesize $(t\geqslant 4)$}
& 
\\
\hline 
 & 
\begin{picture}(90,10)(-15,0)
    \put(0, 2){\makebox(0, -3){\textcircled{\tiny $2$}}}
    \put(7, -6){\makebox(0, 0){\tiny $(1)$}}
    \put(5, 2){\line(1, 0){20}} 
    \put(30, 2){\makebox(0, -3){\textcircled{\tiny $3$}}}
    \put(37, -6){\makebox(0, 0){\tiny $(2)$}}
    \put(35, 2){\line(1, 0){20}} 
    \put(60, 2){\makebox(0, -3){\textcircled{\tiny $2$}}}
    \put(67, -6){\makebox(0, 0){\tiny $(1)$}}
    \put(30, -3){\line(0, -1){11}}
    \put(30, -19){\makebox(0, -3){\textcircled{\tiny $2$}}}
    \put(37, -27){\makebox(0, 0){\tiny $(1)$}}
\end{picture}
\\
$\gD_4(1)$ & 
\\
 & 
\\
 & 
\begin{picture}(90,10)(-15,0)
    \put(0, 2){\makebox(0, -3){\textcircled{\tiny $4$}}}
    \put(7, -6){\makebox(0, 0){\tiny $(2)$}}
    \put(5, 2){\line(1, 0){20}} 
    \put(30, 2){\makebox(0, -3){\textcircled{\tiny $2$}}}
    \put(37, -6){\makebox(0, 0){\tiny $(2)$}}
    \put(35, 2){\line(1, 0){20}} 
    \put(60, 2){\makebox(0, -3){\textcircled{\tiny $2$}}}
    \put(67, -6){\makebox(0, 0){\tiny $(1)$}}
    \put(30, -3){\line(0, -1){11}}
    \put(30, -19){\makebox(0, -3){\textcircled{\tiny $2$}}}
    \put(37, -27){\makebox(0, 0){\tiny $(1)$}}
\end{picture}
\\
$\gD_4(2)$ & 
\\
 & 
\\
\hline
 & 
\begin{picture}(185,10)(-15,0)
    \put(0, 2){\makebox(0, -3){\textcircled{\tiny $2$}}}
    \put(7, -6){\makebox(0, 0){\tiny $(1)$}}
    \put(5, 2){\line(1, 0){20}} 
    \put(30, 2){\makebox(0, -3){\textcircled{\tiny $3$}}}
    \put(37, -6){\makebox(0, 0){\tiny $(2)$}}
    \put(35, 2){\line(1, 0){20}} 
    \put(60, 2){\makebox(0, -3){\textcircled{\tiny $2$}}}
    \put(67, -6){\makebox(0, 0){\tiny $(2)$}}
    \put(65, 2){\line(1, 0){15}}
    \put(82, 2){\line(1, 0){2}}
    \put(86, 2){\line(1, 0){2}}
    \put(90, 2){\line(1, 0){2}}
    \put(94, 2){\line(1, 0){16}}
    \put(115, 2){\makebox(0, -3){\textcircled{\tiny $2$}}}
    \put(122, -6){\makebox(0, 0){\tiny $(2)$}}
    \put(120, 2){\line(1, 0){20}} 
    \put(145, 2){\makebox(0, -3){\textcircled{\tiny $2$}}}
    \put(152, -6){\makebox(0, 0){\tiny $(1)$}}
    \put(115, -3){\line(0, -1){11}} 
    \put(115, -19){\makebox(0, -3){\textcircled{\tiny $2$}}}
    \put(122, -27){\makebox(0, 0){\tiny $(1)$}}
\end{picture}
\\
$\gD_t(1)$ & \\
 & 
\\
 & 
\begin{picture}(185,10)(-15,0)
    \put(0, 2){\makebox(0, -3){\textcircled{\tiny $4$}}}
    \put(7, -6){\makebox(0, 0){\tiny $(2)$}}
    \put(5, 2){\line(1, 0){20}} 
    \put(30, 2){\makebox(0, -3){\textcircled{\tiny $2$}}}
    \put(37, -6){\makebox(0, 0){\tiny $(2)$}}
    \put(35, 2){\line(1, 0){20}} 
    \put(60, 2){\makebox(0, -3){\textcircled{\tiny $2$}}}
    \put(67, -6){\makebox(0, 0){\tiny $(2)$}}
    \put(65, 2){\line(1, 0){15}}
    \put(82, 2){\line(1, 0){2}}
    \put(86, 2){\line(1, 0){2}}
    \put(90, 2){\line(1, 0){2}}
    \put(94, 2){\line(1, 0){16}}
    \put(115, 2){\makebox(0, -3){\textcircled{\tiny $2$}}}
    \put(122, -6){\makebox(0, 0){\tiny $(2)$}}
    \put(120, 2){\line(1, 0){20}} 
    \put(145, 2){\makebox(0, -3){\textcircled{\tiny $2$}}}
    \put(152, -6){\makebox(0, 0){\tiny $(1)$}}
    \put(115, -3){\line(0, -1){11}} 
    \put(115, -19){\makebox(0, -3){\textcircled{\tiny $2$}}}
    \put(122, -27){\makebox(0, 0){\tiny $(1)$}}
\end{picture}
\\
$\gD_t(2)$ & \\
{\footnotesize $(t\geqslant 5)$} & 
\\
\hline
\end{tabular}
\end{center}
(The dual graph of 
 $\gA_n(l, m)$ (resp.\ $\gD_n(m)$) is of type $\gA_n$ (resp.\ $\gD_n$).) 
\end{table}

\begin{proof}
By \cite[\S 4]{KoMo}, all of the exceptional curves are nonsingular rational curves 
and the weighted dual graph $\Gamma$ 
of $-3K_{M/S}$ is a tree and either a straight chain or having exactly one fork. 
Assume that $\Gamma$ is a straight chain. Then $\Gamma$ is of the form: 
\begin{center}
    \begin{picture}(150, 45)(0, 45)
    \put(0, 60){\makebox(0, -3){{$c_1$}}}
    \put(0, 60){\circle{20}}
    \put(0, 75){\makebox(0, 0)[b]{$E_1$}}
    \put(7, 52){\makebox(10, -10){\tiny $(w_1)$}}
    \put(10, 60){\line(1, 0){20}}
    \put(40, 60){\makebox(0, -3){{$c_2$}}}
    \put(40, 60){\circle{20}}
    \put(47, 52){\makebox(10, -10){\tiny $(w_2)$}}
    \put(40, 75){\makebox(0, 0)[b]{$E_2$}}
    \put(50, 60){\line(1, 0){15}} 
    \put(67, 60){\line(1, 0){2}}
    \put(71, 60){\line(1, 0){2}}
    \put(75, 60){\line(1, 0){2}}    
    \put(79, 60){\line(1, 0){16}}
    \put(105, 60){\makebox(0, -3){{\footnotesize $c_{t-1}$}}}
    \put(105, 60){\circle{20}}
    \put(105, 75){\makebox(0, 0)[b]{$E_{t-1}$}}
    \put(120, 52){\makebox(10, -10){\tiny $(w_{t-1})$}}
    \put(115, 60){\line(1, 0){30}}
    \put(155, 60){\makebox(0, -3){{$c_t$}}}
    \put(155, 60){\circle{20}}
    \put(155, 75){\makebox(0, 0)[b]{$E_t$}}
    \put(162, 52){\makebox(10, -10){\tiny $(w_t)$}}
    \end{picture}
\end{center}
We note that $c_i\geqslant 2$ and $w_i=1$ or $2$. 
We only consider the case $t\geqslant 4$. (The case $t\leqslant 3$ can be proven 
same way.) 
Since $(3K_{M/S}\cdot E_i)=3(K_M\cdot E_i)=3(c_i-2)$, we have 
\[
c_i=
\begin{cases}
(6-w_2)/(3-w_1) & \text{if }\,\,i=1, \\
(6-w_{i-1}-w_{i+1})/(3-w_i) & \text{if }\,\,2\leqslant i\leqslant t-1, \\
(6-w_{t-1})/(3-w_t) & \text{if }\,\,i=t. 
\end{cases}
\]
Suppose that $w_i=1$ for some $2\leqslant i\leqslant t-1$. We may assume that 
$w_j=2$ for any $2\leqslant j\leqslant i-1$ if $i\geqslant 3$. 
If $i\geqslant 3$, then $c_i=(6-w_{i-1}-w_{i+1})/2<2$, a contradiction. 
If $i=2$, then we have $c_2=2$. 
However, we see that $c_1=5/2$, a contradiction.
Thus $w_i=2$ for any $2\leqslant i\leqslant t-1$. 
Hence the form of $\Gamma$ is one of $\gA_t(1, 1)$, $\gA_t(1, 2)$ or $\gA_t(2, 2)$. 

Assume that $\Gamma$ has one fork. Then $\Gamma$ is of the form: 
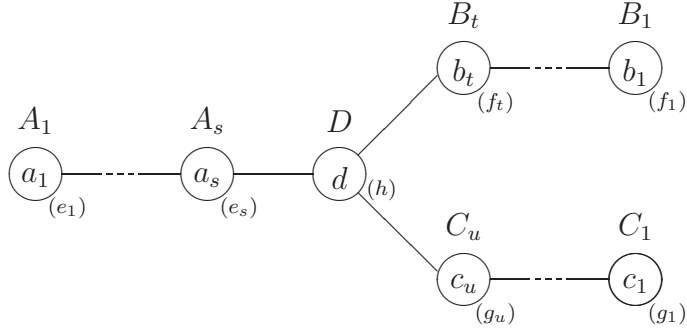
\begin{center}
    \begin{picture}(250, 130)(40, 100)
    \put(40, 160){\makebox(0, -3){{$a_1$}}}
    \put(40, 160){\circle{20}}
    \put(47, 152){\makebox(10, -10){\tiny $(e_1)$}}
    \put(40, 175){\makebox(0, 0)[b]{$A_1$}}
    \put(50, 160){\line(1, 0){15}} 
    \put(67, 160){\line(1, 0){2}}
    \put(71, 160){\line(1, 0){2}}
    \put(75, 160){\line(1, 0){2}}    
    \put(79, 160){\line(1, 0){16}}
    \put(105, 160){\makebox(0, -3){{$a_s$}}}
    \put(105, 160){\circle{20}}
    \put(105, 175){\makebox(0, 0)[b]{$A_s$}}
    \put(112, 152){\makebox(10, -10){\tiny $(e_s)$}}
    \put(115, 160){\line(1, 0){30}}
    \put(155, 160){\makebox(0, -3){{$d$}}}
    \put(155, 160){\circle{20}}
    \put(155, 175){\makebox(0, 10){$D$}}
    \put(166, 159){\makebox(10, -10){\tiny $(h)$}}
    \put(162, 167){\line(1, 1){30}}
    \put(202, 200){\makebox(0, -3){{$b_t$}}}
    \put(202, 200){\circle{20}}
    \put(202, 215){\makebox(0, 0)[b]{$B_t$}}
    \put(209, 192){\makebox(10, -10){\tiny $(f_t)$}}
    \put(212, 200){\line(1, 0){15}} 
    \put(229, 200){\line(1, 0){2}}
    \put(233, 200){\line(1, 0){2}}
    \put(237, 200){\line(1, 0){2}}    
    \put(241, 200){\line(1, 0){16}}
    \put(267, 200){\makebox(0, -3){{$b_1$}}}
    \put(267, 200){\circle{20}}
    \put(267, 215){\makebox(0, 0)[b]{$B_1$}}
    \put(274, 192){\makebox(10, -10){\tiny $(f_1)$}}
    \put(162, 153){\line(1, -1){30}}
    \put(202, 120){\makebox(0, -3){{$c_u$}}}
    \put(202, 120){\circle{20}}
    \put(202, 140){\makebox(0, 0){$C_u$}}
    \put(209, 112){\makebox(10, -10){\tiny $(g_u)$}}
    \put(212, 120){\line(1, 0){15}}
    \put(229, 120){\line(1, 0){2}}
    \put(233, 120){\line(1, 0){2}}
    \put(237, 120){\line(1, 0){2}}
    \put(241, 120){\line(1, 0){16}}
    \put(267, 120){\circle{20}}
    \put(267, 120){\makebox(0, -3){{$c_1$}}}
    \put(267, 120){\circle{20}}
    \put(267, 140){\makebox(0, 0){$C_1$}}
    \put(274, 112){\makebox(10, -10){\tiny $(g_1)$}}
    \end{picture}
\end{center}
Using the same argument, we have $d=(6-e_s-f_t-g_u)/(3-h)$. Thus $h=2$ and 
we can assume that $e_s=f_t=1$. Then we must have $s=t=1$ and $g_i=2$ for 
any $2\leqslant i\leqslant u$ by the same argument. Therefore the assertion holds. 
\end{proof}

\section{Log del Pezzo surfaces and related objects}\label{dP_section}

In this section, we define the notion of log del Pezzo surfaces, 
the notion of $3$-basic pairs, the notion of $3$-fundamental triplets, 
and the notion of bottom tetrads, 
and see the correspondence among them.

\subsection{Log del Pezzo surfaces}

\begin{definition}\label{dP_dfn}
\begin{enumerate}
\renewcommand{\theenumi}{\arabic{enumi}}
\renewcommand{\labelenumi}{(\theenumi)}
\item\label{dP_dfn1}
A normal projective surface $S$ is called a \emph{log del Pezzo surface} 
if $S$ is log-terminal and the anti-canonical divisor $-K_S$ is an ample $\Q$-Cartier 
divisor.
\item\label{dP_dfn2}
Fix $a\geqslant 2$. 
A log del Pezzo surface is called a \emph{log del Pezzo surface of index $a$}
if $-aK_S$ is Cartier and $-a'K_S$ is not Cartier for any positive integer $a'<a$. 
\end{enumerate}
\end{definition}

\begin{remark}\label{dP_rmk}
Any log del Pezzo surface is a rational surface by \cite[Proposition 3.6]{N}. 
In particular, the Picard group $\Pic(S)$ of $S$ is a finitely generated and 
torsion-free Abelian group. 
\end{remark}

\subsection{$a$-basic pairs}\label{basic_section}

We introduce the following notion which is a kind of 
modification of the notion of basic pairs in the sense of \cite[\S3]{N}.

\begin{definition}\label{basic_dfn}
Fix $a\geqslant 2$. A pair $(M, E_M)$ is called an \emph{$a$-basic pair} 
if the following conditions are satisfied: 
\begin{enumerate}
\renewcommand{\theenumi}{\arabic{enumi}}
\renewcommand{\labelenumi}{($\sC$\theenumi)}
\item\label{basic_dfn1}
$M$ is a nonsingular projective rational surface.
\item\label{basic_dfn2}
$E_M$ is a nonzero effective divisor on $M$ such that 
$\coeff E_M\subset\{1,\dots, a-1\}$ 
and $\Supp E_M$ is simple normal crossing.
\item\label{basic_dfn3}
A Cartier divisor $L_M\sim-aK_M-E_M$ (called the 
\emph{fundamental divisor} of $(M, E_M)$) satisfies that 
$K_M+L_M$ is nef and $(K_M+L_M\cdot L_M)>0$.
\item\label{basic_dfn4}
For any component $E_0\leqslant E_M$, we have $(L_M\cdot E_0)=0$.
\end{enumerate}
\end{definition}

Now we see the correspondence between log del Pezzo surfaces of index $a$ 
and $a$-basic pairs. The proof is essentially same as the proof in 
\cite[Proposition 3.7]{F}. 

\begin{proposition}\label{dP-basic_prop}
Fix $a\geqslant 2$. 
\begin{enumerate}
\renewcommand{\theenumi}{\arabic{enumi}}
\renewcommand{\labelenumi}{(\theenumi)}
\item\label{dP-basic_prop1}
Let $S$ be a non-Gorenstein log del Pezzo surface such that $-aK_S$ is Cartier. 
Let $\alpha\colon M\to S$ be the minimal resolution 
of $S$ and let $E_M:=-aK_{M/S}$. 
Then $(M, E_M)$ is an $a$-basic pair and the divisor $\alpha^*(-aK_S)$ is the 
fundamental divisor of $(M, E_M)$. 
\item\label{dP-basic_prop2}
Let $(M, E_M)$ be a $a$-basic pair and $L_M$ be 
the fundamental divisor of $(M, E_M)$.
Then there exists a projective and birational morphism $\alpha\colon M\to S$ 
such that $S$ is a non-Gorenstein log del Pezzo surface with $-aK_S$ Cartier and 
$L_M\sim\alpha^*(-aK_S)$ holds. 
Moreover, the morphism $\alpha$ is the minimal resolution of $S$. 
\end{enumerate}
In particular, there is a one-to-one correspondence between the set of 
isomorphism classes of log del Pezzo surfaces of index three and the set of 
isomorphism classes of $3$-basic pairs. 
\end{proposition}

\begin{proof}
The proof of \eqref{dP-basic_prop2} is essentially same as the proof in 
\cite[Proposition 3.7]{F}. We only prove \eqref{dP-basic_prop1}. 
The conditions $(\sC\ref{basic_dfn1})$, $(\sC\ref{basic_dfn2})$ and 
$(\sC\ref{basic_dfn4})$ follow immediately. 
We check the condition $(\sC\ref{basic_dfn3})$. 
Assume that $K_M+L_M$ is not nef. 
If there exists a $(-1)$-curve $\gamma$ on $M$ such that 
$(K_M+L_M\cdot \gamma)<0$, then $(L_M\cdot\gamma)=0$. However, this implies that 
$\gamma$ is $\alpha$-exceptional. This leads to a contradiction since $\alpha$ is 
the minimal resolution. 
Hence $M\simeq\pr^2$ or $\F_n$ by \cite[Theorem 2.1]{mori} and the fact $M$ is a 
nonsingular rational surface. Since $\alpha$ is not an isomorphism, 
$M\simeq\F_n$ and 
$S$ is isomorphic 
to the weighted projective plane $\pr(1, 1, n)$ for some $n\geqslant 2$. 
This implies that $E_M=(a(n-2)/n)\sigma$ and 
$K_M+L_M\sim(-2+a(n+2)/n)\sigma+(a-1)(n+2)l$. However, this leads to a contradiction 
since we assume that $K_M+L_M$ is not nef. Thus $K_M+L_M$ must be nef. 
If $(K_M+L_M\cdot L_M)=0$, then $-K_M$ is numerically equivalent to $L_M$ 
by the Hodge index theorem. In particular, $-K_M$ is nef and big. 
This implies that $S$ has at most du Val singularities. This leads to a contradiction. 
Thus $(K_M+L_M\cdot L_M)>0$.
\end{proof}

As a corollary of Proposition \ref{dP-basic_prop}, we have the following result. 
The proof is same as the proof in \cite[Lemma 3.8]{F}. We omit the proof. 

\begin{corollary}\label{dP-basic_cor}
Let $(M, E_M)$ be a $3$-basic pair and $L_M$ be the fundamental divisor. 
Then the following hold. 
\begin{enumerate}
\renewcommand{\theenumi}{\arabic{enumi}}
\renewcommand{\labelenumi}{(\theenumi)}
\item\label{dP-basic_cor1}
Any connected component of the weighted dual graph of $E_M$ is of the form in 
Table \ref{graph_table}. 
\item\label{dP-basic_cor2}
If a curve $C$ on $M$ satisfies that $C\cap E_M\neq\emptyset$ and 
$(L_M\cdot C)=0$, then $C\leqslant E_M$ holds. 
\item\label{dP-basic_cor1}
The anti-canonical divisor $-K_M$ is big and non-nef. 
In particular, $M$ is a Mori dream space $($for the definition, see \cite{testa}$)$. 
\end{enumerate}
\end{corollary}

\subsection{Median triplets}\label{triplet_section}

In order to classify $3$-basic pairs, we define the notion of 
median triplets which is a kind of modification of the notion of 
fundamental triplets in the sense of Nakayama \cite{N}. 
The correspondence between $3$-basic pairs and (pseudo-)median triplets 
will be given in Theorem \ref{corresp_thm}.

\begin{definition}\label{qfund_dfn}
A triplet $(Z, E_Z; \Delta_Z)$ is called a \emph{$3$-pseudo-fundamental multiplet 
of length one} if the following conditions are satisfied: 
\begin{enumerate}
\renewcommand{\theenumi}{\arabic{enumi}}
\renewcommand{\labelenumi}{($\sF$\theenumi)}
\item\label{fund_dfn1}
$Z$ is a nonsingular projective surface.
\item\label{fund_dfn2}
$\Delta_Z$ is a zero-dimensional subscheme of $Z$ 
which satisfies the $(\nu1)$-condition. 
\item\label{fund_dfn3}
$E_Z$ is a nonzero effective divisor on $Z$.
\item\label{fund_dfn4}
A divisor $L_Z\sim-3K_Z-E_Z$ $($called the \emph{fundamental 
divisor} of $(Z, E, \Delta_Z)$$)$ satisfies that
$(2K_Z+L_Z\cdot \gamma)\geqslant 0$ for any $(-1)$-curve $\gamma$ on $Z$.
\item\label{fund_dfn5}
Let $\phi\colon M\to Z$ be the elimination of $\Delta_Z$ and 
let $E_M:=(E_Z)_M^{\Delta_Z, 2}$. 
Then the pair $(M, E_M)$ is a $3$-basic pair 
$($called \emph{the associated $3$-basic pair}$)$. 
\end{enumerate}
Moreover, if $2K_Z+L_Z$ is not nef, then we call such triplet $(Z, E_Z; \Delta_Z)$ 
a \emph{$3$-fundamental multiplet of length one}.
\end{definition}

\begin{lemma}\label{fund_big_lem}
Let $(Z, E_Z; \Delta_Z)$ be a $3$-pseudo-fundamental multiplet of length one, 
$L_Z$ be the fundamental divisor of $(Z, E_Z; \Delta_Z)$ 
and $(M, E_M)$ be the associated $3$-basic pair. 
\begin{enumerate}
\renewcommand{\theenumi}{\arabic{enumi}}
\renewcommand{\labelenumi}{(\theenumi)}
\item\label{fund_big_lem1}
The divisor $L_M:=(L_Z)_M^{\Delta_Z, 1}$ is the fundamental divisor 
of the $3$-basic pair $(M, E_M)$. 
We have $L_Z$ is nef an big, $K_Z+L_Z$ is nef and $(K_Z+L_Z\cdot L_Z)>0$.
\item\label{fund_big_lem2}
If $2K_Z+L_Z$ is not nef, then $Z\simeq\pr^2$ or $\F_n$. Moreover, 
$(2K_Z+L_Z\cdot l)<0$ holds. 
\item\label{fund_big_lem3}
If $2K_Z+L_Z$ is nef, then $K_Z+L_Z$ is big. 
\item\label{fund_big_lem4}
We have $(L_Z\cdot E_Z)=2\deg\Delta_Z$. Moreover, for any nonsingular component 
$E_0\leqslant E_Z$, we have $(L_Z\cdot E_0)=\deg(\Delta_Z\cap E_0)$. 
\item\label{fund_big_lem5}
For any point $Q\in\Delta_Z$, $2\leqslant\mult_QE_Z\leqslant 4$ holds.
\end{enumerate}
\end{lemma}

\begin{proof}
\eqref{fund_big_lem1}
We know that 
$-3K_M-E_M=\phi^*(-3K_Z-E_Z)-K_{M/Z}\sim\phi^*L_Z-K_{M/Z}$, 
where $\phi$ is the elimination of $\Delta_Z$.
Since $K_M+L_M=\phi^*(K_Z+L_Z)$ and $L_M=\phi^*L_Z-K_{M/Z}$, the assertions hold.

\eqref{fund_big_lem2}
Since $(2K_Z+L_Z\cdot\gamma)\geqslant 0$ for any $(-1)$-curve, $Z\simeq\pr^2$ or 
$\F_n$, and $(2K_Z+L_Z\cdot l)<0$ hold
by \cite[Theorem 2.1]{mori}. 

\eqref{fund_big_lem3}
Follows from the equality $2(K_Z+L_Z)=(2K_Z+L_Z)+L_Z$. 

\eqref{fund_big_lem4}
Since $0=(L_M\cdot E_M)=(L_Z\cdot E_Z)+2(K_{M/Z}^2)$, 
we have $(L_Z\cdot E_Z)=2\deg\Delta_Z$. 
Similarly, for any nonsingular component $E_0\leqslant E_Z$, 
we have $0=(L_M\cdot E_0^M)=(L_Z\cdot E_0)-(K_{M/Z}\cdot E_0^M)$. 

\eqref{fund_big_lem5}
Follows from the equality $\coeff_{\Gamma_{Q, 1}}E_M=\mult_QE_Z-2$.
\end{proof}

\begin{definition}\label{fund_dfn}
Let $(Z, E_Z; \Delta_Z)$ be a $3$-pseudo-fundamental multiplet of length one. 
Such $(Z, E_Z; \Delta_Z)$ is called a 
\emph{pseudo-median triplet} either if $K_Z+L_Z$ is big or if $K_Z+L_Z$ is not big, 
$Z\simeq\F_n$ ($K_Z+L_Z$ is trivial with respects to $\F_n\to\pr^1$), 
and the following two conditions are satisfied: 
\begin{enumerate}
\setcounter{enumi}{5}
\renewcommand{\theenumi}{\arabic{enumi}}
\renewcommand{\labelenumi}{($\sF$\theenumi)}
\item\label{fund_dfn6}
$\Delta_Z\cap\sigma=\emptyset$ holds, where $\sigma\subset Z$ is a minimal section. 
In particular, if $n=0$, then $\Delta_Z=\emptyset$. 
\item\label{fund_dfn7}
Assume that $E_Z$ contains a section $D$ of $\F_n/\pr^1$, then 
$\sigma\leqslant E_Z$ and 
$\coeff_\sigma E_Z\geqslant\coeff_DE_Z$ holds. Moreover, if 
$\coeff_\sigma E_Z=\coeff_DE_Z$, then $n+(D^2)\geqslant\deg(\Delta_Z\cap D)$ holds. 
\end{enumerate}
If $2K_Z+L_Z$ is not nef in addition, then we call such a triplet 
$(Z, E_Z; \Delta_Z)$ a \emph{median triplet}.
\end{definition}

\subsection{Bottom tetrads}\label{tetrad_section}

In this section, we define the notion of bottom tetrads which is also a kind of modification 
of the notion of fundamental triplets in the sense of Nakayama \cite{N}. 
The correspondence between (special) pseudo-median triplets and bottom tetrads 
will be given in Theorems \ref{corresp_thm} and \ref{corresp_triv_thm}.

\begin{definition}\label{qbottom_dfn}
A tetrad $(X, E_X; \Delta_Z, \Delta_X)$ is called a 
\emph{$3$-fundamental multiplet of length two} 
if the following conditions are satisfied: 
\begin{enumerate}
\renewcommand{\theenumi}{\arabic{enumi}}
\renewcommand{\labelenumi}{($\sB$\theenumi)}
\item\label{bottom_dfn1}
$X$ is a nonsingular projective surface.
\item\label{bottom_dfn2}
$\Delta_X$ is a zero-dimensional subscheme of $X$ 
which satisfies the $(\nu1)$-condition. 
\item\label{bottom_dfn3}
$E_X$ is a nonzero effective divisor on $X$.
\item\label{bottom_dfn4}
A divisor $L_X\sim-3K_X-E_X$ $($called the \emph{fundamental 
divisor} of $(X, E_X; \Delta_Z, \Delta_X)$$)$ satisfies that $2K_X+L_X$ is nef and 
$(3K_X+L_X\cdot \gamma)\geqslant 0$ for any $(-1)$-curve $\gamma$ on $X$. 
\item\label{bottom_dfn5}
Let $\psi\colon Z\to X$ be the elimination of $\Delta_X$ and 
let $E_Z:=(E_X)_Z^{\Delta_X, 1}$. 
Then the triplet $(Z, E_Z; \Delta_Z)$ is a $3$-pseudo-fundamental multiplet 
of length one.
$($The triplet is in fact a pseudo-median triplet $($Lemma \ref{FuBo_lem}$)$. 
We call the triplet \emph{the associated pseudo-median triplet}.$)$
\end{enumerate}
\end{definition}

\begin{lemma}\label{FuBo_lem}
Let $(X, E_X; \Delta_Z, \Delta_X)$ be a $3$-fundamental multiplet of length two, 
let $L_X$ be the fundamental divisor and let $(Z, E_Z; \Delta_Z)$ be the 
associated pseudo-median triplet. 
\begin{enumerate}
\renewcommand{\theenumi}{\arabic{enumi}}
\renewcommand{\labelenumi}{(\theenumi)}
\item\label{FuBo_lem1}
$X$ is isomorphic to either $\pr^2$ or $\F_n$. Moreover, $(E\cdot l)>0$ holds. 
\item\label{FuBo_lem2}
$L_Z:=(L_X)_Z^{\Delta_X, 2}$ is the 
fundamental divisor of $(Z, E_Z; \Delta_Z)$, $L_Z$ is nef and big, and 
$K_Z+L_Z$ is big. 
\item\label{FuBo_lem3}
We have $(L_X\cdot E_X)=2(\deg\Delta_Z+\deg\Delta_X)$. 
Moreover, for any nonsingular component 
$E_0\leqslant E_X$, we have 
$(L_X\cdot E_0)=\deg(\Delta_Z\cap E_0^Z)+2\deg(\Delta_X\cap E_0)$. 
\item\label{FuBo_lem4}
For any point $P\in\Delta_X$, $1\leqslant\mult_PE_X\leqslant 3$ holds.
\item\label{FuBo_lem5}
We have $(K_X+L_X\cdot L_X)>2\deg\Delta_X$.
\end{enumerate}
\end{lemma}

\begin{proof}
\eqref{FuBo_lem1}
Since $E_X$ is nonzero effective, the divisor $3K_X+L_X$ is not nef. 
Then the assertion follows from \cite[Theorem 2.1]{mori}. 

\eqref{FuBo_lem2}
Follows from $L_Z\sim -3K_Z-(E_X)_Z^{\Delta_X, 1}$, 
$2K_Z+L_Z=\psi^*(2K_X+L_X)$ and 
Lemma \ref{fund_big_lem}, where $\psi$ is the elimination of $\Delta_X$. 

\eqref{FuBo_lem3}
We have $(L_X\cdot E_X)=(L_Z\cdot E_Z)+2(K_{Z/X}^2)$. 
Similarly, we have $(L_X\cdot E_0)=(L_Z\cdot E_0^Z)+2(K_{Z/X}\cdot E_0^Z)$. 
Thus the assertion holds by Lemma \ref{fund_big_lem}. 

\eqref{FuBo_lem4}
Follows from the equality $\coeff_{\Gamma_{P, 1}}E_Z=\mult_PE_X-1$.

\eqref{FuBo_lem5}
Follows from 
$(K_Z+L_Z\cdot L_Z)=(K_X+L_X\cdot L_X)-2\deg\Delta_X$.
\end{proof}

\begin{definition}\label{bottom_dfn}
Let $(X, E_X; \Delta_Z, \Delta_X)$ be a $3$-fundamental multiplet of length two 
and $L_X$ be a fundamental divisor. 
Such $(X, E_X; \Delta_Z, \Delta_X)$ is called a 
\emph{bottom tetrad} if one of the following holds: 
\begin{enumerate}
\renewcommand{\theenumi}{\Alph{enumi}}
\renewcommand{\labelenumi}{(\theenumi)}
\item\label{bot_dfn1}
$2K_X+L_X$ is big.
\item\label{bot_dfn2}
$2K_X+L_X$ is non-big and nontrivial, 
$X\simeq\F_n$ ($2K_X+L_X$ is trivial with respects to $\F_n\to\pr^1$) 
and the following conditions are satisfied: 
\begin{enumerate}
\setcounter{enumii}{5}
\renewcommand{\theenumii}{\arabic{enumii}}
\renewcommand{\labelenumii}{($\sB$\theenumii)}
\item\label{bottom_dfn6}
$\Delta_X\cap\sigma=\emptyset$ holds, where $\sigma\subset X$ is a minimal section. 
In particular, if $n=0$, then $\Delta_X=\emptyset$. 
\item\label{bottom_dfn7}
Assume that $\sigma\not\leqslant E_X$ or $n=0$, then any section 
$D\leqslant E_X$ of $\F_n/\pr^1$ satisfies that 
$(D^2)\geqslant\deg(\Delta_X\cap D)$. 
\item\label{bottom_dfn8}
Assume that $\sigma\leqslant E_X$ and $n\geqslant 1$. 
Then any section $D\leqslant E_X$ of $\F_n/\pr^1$ satisfies that 
$n+(D^2)\geqslant\deg(\Delta_X\cap D)$. 
\end{enumerate}
\item\label{bot_dfn3}
$2K_X+L_X$ is trivial. In this case, we require that either 
$X\simeq\pr^2$, or $\Delta_X=\emptyset$ and 
$X\simeq\pr^1\times\pr^1$, $\F_2$. Moreover, if $X\simeq\pr^2$, then 
the following conditions are satisfied: 
\begin{enumerate}
\setcounter{enumii}{8}
\renewcommand{\theenumii}{\arabic{enumii}}
\renewcommand{\labelenumii}{($\sB$\theenumii)}
\item\label{bottom_dfn9}
Assume that $E_X=C+l$, where $C$ is a nonsingular conic and $l$ is a line. 
Then $\Delta_X\cap C\cap l\neq\emptyset$. If we further assume that 
$|C\cap l|=\{P\}$ and $\deg(\Delta_X\setminus\{P\})\geqslant 4$, then 
$\Delta_Z\cap l\setminus\{P\}\neq\emptyset$. 
\item\label{bottom_dfn10}
Assume that $E_X=l_1+l_2+l_3$, where $l_1$, $l_2$, $l_3$ are distinct lines. 
Then $l_1\cap l_2\cap l_3=\emptyset$. Moreover, 
$\#|\Delta_X\cap((l_1\cap l_2)\cup(l_1\cap l_3)\cup(l_2\cap l_3))|\geqslant 2$. 
\item\label{bottom_dfn11}
Assume that $E_X=2l_1+l_2$, where $l_1$, $l_2$ are distinct lines. Set $P:=l_1\cap l_2$. 
Then the following conditions are satisfied: 
\begin{enumerate}
\renewcommand{\theenumiii}{\alph{enumiii}}
\renewcommand{\labelenumiii}{(\theenumiii)}
\item\label{bottom_dfn111}
$\#|\Delta_X\cap l_1\setminus\{P\}|\leqslant 1$. Moreover, if 
$\{P_1\}=|\Delta_X\cap l_1\setminus\{P\}|$, then $\mult_{P_1}\Delta_X\leqslant 2$ and 
$\mult_P\Delta_X=\mult_P(\Delta_X\cap l_2)$. 
\item\label{bottom_dfn112}
If $\deg\Delta_X=4$, then $\deg(\Delta_X\cap l_2)=3$. 
\item\label{bottom_dfn113}
If $\deg\Delta_X\geqslant 5$ and $\{P_1\}=|\Delta_X\cap l_1\setminus\{P\}|$, then 
either $\mult_{P_1}(\Delta_X\cap l_1)=2$ or $\deg(\Delta_X\cap l_1)=1$ holds. 
\end{enumerate}
\end{enumerate}
\end{enumerate}
\end{definition}

Now we see the correspondence among $3$-basic pairs, pseudo-median triplets 
and bottom tetrads. The relationship between pseudo-median triplets 
$(Z, E_Z; \Delta_Z)$ with $2K_Z+L_Z$ trivial and special bottom tetrads will be treaded 
in Section \ref{cubic_section}.

\begin{thm}\label{corresp_thm}
\begin{enumerate}
\renewcommand{\theenumi}{\arabic{enumi}}
\renewcommand{\labelenumi}{(\theenumi)}
\item\label{corresp_thm1}
Let $(M, E_M)$ be a $3$-basic pair and $L_M$ be the fundamental divisor. 
Then there exists a projective birational morphism $\phi\colon M\to Z$ onto 
a nonsingular surface and a zero-dimensional subscheme $\Delta_Z\subset Z$ 
satisfying the $(\nu1)$-condition such that the morphism $\phi$ is the elimination 
of $\Delta_Z$, the triplet $(Z, E_Z; \Delta_Z)$ is a pseudo-median triplet and 
the associated $3$-basic pair is equal to $(M, E_M)$, where 
$E_Z:=\phi_*E_M$. Moreover, the divisor $\phi_*L_M$ is the fundamental divisor of 
$(Z, E_Z; \Delta_Z)$. 
\item\label{corresp_thm2}
Let $(Z, E_Z; \Delta_Z)$ be a pseudo-median triplet such that $2K_Z+L_Z$ is nef 
and nontrivial, where $L_Z$ is the fundamental divisor. 
Then there exists a projective birational morphism $\psi\colon Z\to X$ onto 
a nonsingular surface and a zero-dimensional subscheme $\Delta_X\subset X$ 
satisfying the $(\nu1)$-condition such that the morphism $\psi$ is the elimination 
of $\Delta_X$, the tetrad $(X, E_X; \Delta_Z, \Delta_X)$ is a bottom tetrad and 
the associated pseudo-median triplet is equal to $(Z, E_Z; \Delta_Z)$, where 
$E_X:=\psi_*E_Z$. Moreover, the divisor $\psi_*L_Z$ is the fundamental divisor of 
$(X, E_X; \Delta_Z, \Delta_X)$. 
\end{enumerate}
\end{thm}

\begin{proof}
The idea of the proof based on the technique in \cite[Proposition 4.5]{N}. 
It is easy to get a $3$-pseudo-fundamental multiplet of length one from 
a $3$-basic pair (resp.\ to get a $3$-fundamental multiplet of length two 
from a pseudo-median triplet). 
Indeed, if there exists a $(-1)$-curve $\gamma$ 
such that $(2K_M+L_M\cdot\gamma)<0$ (resp.\ $(3K_Z+L_Z\cdot\gamma)<0$), then 
we contract the curve $\gamma$. We note that $(L_M\cdot\gamma)=1$ 
since $K_M+L_M$ is nef. 
(resp.\ $(L_Z\cdot\gamma)=2$ since $2K_Z+L_Z$ is nef). 
By continuing this process, we get a $3$-pseudo-fundamental multiplet of length one  
(resp.\ $3$-fundamental multiplet of length two). 

From now on, we assume that $K_M+L_M$ (resp.\ $2K_Z+L_Z$) is 
non-big and nontrivial. 
Then $Z$ (resp.\ $X$) is isomorphic to $\F_n$. 
We will replace the triplet (resp.\ the tetrad) if necessary. 
The condition ($\sF$\ref{fund_dfn6}) (resp.\ the condition ($\sB$6)) 
follows easily (see \cite[Proposition 4.5 Step 1]{N}).

\eqref{corresp_thm1}
We check the condition ($\sF$\ref{fund_dfn7}). 
Assume that $E_M$ contains a section of $M/\pr^1$.
We pick a section $D\leqslant E_M$ of $M/\pr^1$ 
such that the value $c:=\coeff_DE_M$ is 
largest among sections of $M/\pr^1$. Moreover, we replace $D$ such that 
the value $-n':=(D^2)$ is smallest among sections with $c=\coeff_DE_M$. 
Note that $n'\geqslant 2$ by Corollary \ref{dP-basic_cor}. 
By \cite[Lemma 4.4]{N}, there exists a morphism $\phi'\colon M\to Z'=\F_{n'}$ 
over $\pr^1$ 
such that $D$ is the total transform of the minimal section $\sigma'\subset\F_{n'}$. 
Then the triplet $(Z', \phi'_*E_M; \Delta_{Z'})$ satisfies the conditions 
($\sF$\ref{fund_dfn6}) and ($\sF$\ref{fund_dfn7}), where $\Delta_{Z'}$ corresponds to 
the morphism $\phi'$. 

\eqref{corresp_thm2}
We check the conditions ($\sB$7) and ($\sB$8). 
Assume that $E_Z$ contains a section of $Z/\pr^1$. 
If any section $D\leqslant E_Z$ satisfies that $(D^2)\geqslant 0$, then the condition 
($\sB$7) is satisfied. 
We assume that there exists a section $D\leqslant E_Z$ such that $(D^2)<0$. 
We replace $D\leqslant E_Z$ such that the value $-n':=(D^2)$ is smallest. 
By \cite[Lemma 4.4]{N}, there exists a morphism $\psi'\colon Z\to X'=\F_{n'}$ 
over $\pr^1$ 
such that $D$ is the total transform of the minimal section $\sigma'\subset\F_{n'}$. 
Then the tetrad $(X', \psi'_*E_Z; \Delta_Z, \Delta_{X'})$ satisfies the conditions 
($\sB$6) and ($\sB$8), 
where $\Delta_{X'}$ corresponds to 
the morphism $\psi'$. 
\end{proof}

\begin{proposition}\label{converse_prop}
\begin{enumerate}
\renewcommand{\theenumi}{\arabic{enumi}}
\renewcommand{\labelenumi}{(\theenumi)}
\item\label{converse_prop1}
Let $Z$ be a nonsingular projective rational surface, $E_Z$ be a nonzero effective 
divisor on $Z$, $L_Z$ be a divisor with $L_Z\sim -3K_Z-E_Z$, 
$\Delta_Z$ be a zero-dimensional closed subscheme of $Z$ which satisfies the 
$(\nu1)$-condition, $\phi\colon M\to Z$ be the elimination of $\Delta_Z$, 
$E_M:=(E_Z)_M^{\Delta_Z, 2}$ and $L_M:=(L_Z)_M^{\Delta_Z, 1}$. 
Assume that 
$K_Z+L_Z$ is nef and $(K_Z+L_Z\cdot L_Z)>0$, 
$\Supp E_M$ is simple normal crossing, 
$\coeff E_M\subset\{1, 2\}$ and $(L_M\cdot E_0)=0$ for any component 
$E_0\leqslant E_M$. Then the pair 
$(M, E_M)$ is a $3$-basic pair.
\item\label{converse_prop2}
Let $X$ be a nonsingular projective rational surface, $E_X$ be a nonzero effective 
divisor on $X$, $L_X$ be a divisor with $L_X\sim -3K_X-E_X$, 
$\Delta_X$ be a zero-dimensional closed subscheme of $X$ which satisfies the 
$(\nu1)$-condition, $\psi\colon Z\to X$ be the elimination of $\Delta_X$, 
$E_Z:=(E_X)_Z^{\Delta_X, 1}$, $L_Z:=(L_X)_Z^{\Delta_X, 2}$, 
$\Delta_Z$ be a zero-dimensional closed subscheme of $Z$ which satisfies the 
$(\nu1)$-condition, $\phi\colon M\to Z$ be the elimination of $\Delta_Z$, 
$E_M:=(E_Z)_M^{\Delta_Z, 2}$ and $L_M:=(L_Z)_M^{\Delta_Z, 1}$. 
Assume that 
$2K_X+L_X$ is nef, $(K_X+L_X\cdot L_X)>2\deg\Delta_X$, 
$\Supp E_M$ is simple normal crossing, 
$\coeff E_M\subset\{1, 2\}$ and $(L_M\cdot E_0)=0$ for any component 
$E_0\leqslant E_M$. Then the pair 
$(M, E_M)$ is a $3$-basic pair.
\end{enumerate}
\end{proposition}

\begin{proof}
\eqref{converse_prop1}
Since $K_M+L_M=\phi^*(K_Z+L_Z)$, the divisor $K_M+L_M$ is nef and 
$(K_M+L_M\cdot L_M)>0$. Thus the assertion holds. 

\eqref{converse_prop2}
We know that $(K_Z+L_Z\cdot L_Z)=(K_X+L_X\cdot L_X)-2\deg\Delta_X$. 
By \eqref{converse_prop1}, it is enough to show that $K_Z+L_Z$ is nef. 
Assume that there exists a curve $C\subset Z$ such that 
$(K_Z+L_Z\cdot C)<0$. Since $2K_Z+L_Z=\psi^*(2K_X+L_X)$ is nef, we have 
\[
0>(K_Z+L_Z\cdot C)=2(2K_Z+L_Z\cdot C)+(E_Z\cdot C)\geqslant(E_Z\cdot C). 
\]
Thus $C\leqslant E_Z$. In particular, $C^M\leqslant E_M$. 
However, we have 
\begin{eqnarray*}
0 & > & 2(K_Z+L_Z\cdot C)=(2K_Z+L_Z\cdot C)+(L_Z\cdot C)\\
 & \geqslant & (L_Z\cdot C)=(L_M+K_{M/Z}\cdot C^M)\geqslant(L_M\cdot C^M).
\end{eqnarray*}
This contradicts to the assumption. Thus $K_Z+L_Z$ is nef. 
\end{proof}

\section{Local properties}\label{local_section}

In this section, we analyze the local properties of pseudo-median triplets 
and bottom tetrads.

\subsection{Local properties of pseudo-median triplets}\label{ZS_section}

Let $(Z, E_Z; \Delta_Z)$ be a $3$-pseudo-fundamental multiplet of length one, 
$Q\in\Delta_Z$ 
be a point, $\phi\colon M\to Z$ be the elimination of $\Delta_Z$ and 
$(M, E_M)$ be the associated $3$-basic pair.

\begin{lemma}\label{ZS1}
Assume that $E_Z=sl$ around $Q$, where $Q\in l$ is nonsingular and $s\geqslant 0$. 
Then $s=2$ and $\Delta\subset l$ around $Q$. 
Moreover, $E_M=2l^M$ around over $Q$. 
\end{lemma}

\begin{proof}
Since $E_M=\phi^*E_Z-2K_{M/Z}$ is effective and does not contain a $(-1)$-curve, 
the assertion follows from \cite[Example 2.5]{F}. 
\end{proof}

\begin{lemma}\label{ZS2}
Assume that $E_Z=s_1l_1+s_2l_2$ around $Q$, where $Q\in l_i$ is nonsingular, 
$s_1\geqslant s_2\geqslant 1$, and $l_1$ and $l_2$ intersect transversally at $Q$. 
\begin{enumerate}
\renewcommand{\theenumi}{\arabic{enumi}}
\renewcommand{\labelenumi}{(\theenumi)}
\item\label{ZS21}
If $(s_1, s_2)=(1, 1)$, then $\mult_Q\Delta_Z=1$ and $E_M=l_1^M+l_2^M$ around over $Q$. 
The weighted dual graph of $E_M$ around over $Q$ is the following: 
\begin{center}
    \begin{picture}(40, 40)(0, 43)
    \put(0, 60){\makebox(0, 0){\large$\oslash$}}
    \put(3, 70){\makebox(0, 0)[b]{$l_1^M$}}
    \put(7, 52){\makebox(0, 0){\tiny $(1)$}}
    \put(21, 55){\makebox(0, 0)[b]{$\sqcup$}}
    \put(40, 60){\makebox(0, 0){\large$\oslash$}}
    \put(47, 52){\makebox(0, 0){\tiny $(1)$}}
    \put(43, 70){\makebox(0, 0)[b]{$l_2^M$}}
    \end{picture}
\end{center}
\item\label{ZS21}
If $(s_1, s_2)=(2, 1)$, then $\mult_Q\Delta_Z=\mult_Q(\Delta_Z\cap l_2)=2$, 
$\mult_Q(\Delta_Z\cap l_1)=1$ and $E_M=2l_1^M+\Gamma_{Q,1}+l_2^M$ around 
over $Q$. 
The weighted dual graph of $E_M$ around over $Q$ is the following: 
\begin{center}
    \begin{picture}(80, 43)(0, 43)
    \put(0, 60){\makebox(0, 0){\large$\oslash$}}
    \put(3, 70){\makebox(0, 0)[b]{$l_1^M$}}
    \put(7, 52){\makebox(0, 0){\tiny $(2)$}}
    \put(5, 60){\line(1, 0){30}}
    \put(40, 60){\makebox(0, -3){\textcircled{\tiny $2$}}}
    \put(47, 52){\makebox(0, 0){\tiny $(1)$}}
    \put(43, 70){\makebox(0, 0)[b]{$\Gamma_{Q, 1}$}}
    \put(61, 55){\makebox(0, 0)[b]{$\sqcup$}}
    \put(80, 60){\makebox(0, 0){\large$\oslash$}}
    \put(87, 52){\makebox(0, 0){\tiny $(1)$}}
    \put(83, 70){\makebox(0, 0)[b]{$l_2^M$}}
    \end{picture}
\end{center}
\item\label{ZS21}
If $(s_1, s_2)=(2, 2)$, we can assume that $\mult_Q(\Delta_Z\cap l_1)=1$. 
Let $k:=\mult_Q\Delta_Z$. Then $k=\mult_Q(\Delta_Z\cap l_2)+1$ and 
$E_M=2l_1^M+2\Gamma_{Q, 1}+\cdots+2\Gamma_{Q, k-1}+2l_2^M$ around over $Q$. 
The weighted dual graph of $E_M$ around over $Q$ is the following: 
\begin{center}
    \begin{picture}(150, 43)(0, 43)
    \put(0, 60){\makebox(0, 0){\large$\oslash$}}
    \put(3, 70){\makebox(0, 0)[b]{$l_1^M$}}
    \put(7, 52){\makebox(0, 0){\tiny $(2)$}}
    \put(5, 60){\line(1, 0){30}}
    \put(40, 60){\makebox(0, -3){\textcircled{\tiny $2$}}}
    \put(47, 52){\makebox(0, 0){\tiny $(2)$}}
    \put(43, 70){\makebox(0, 0)[b]{$\Gamma_{Q, 1}$}}
    \put(45, 60){\line(1, 0){20}}
    \put(67, 60){\line(1, 0){2}}
    \put(71, 60){\line(1, 0){2}}
    \put(75, 60){\line(1, 0){2}}
    \put(79, 60){\line(1, 0){21}}
    \put(105, 60){\makebox(0, -3){\textcircled{\tiny $2$}}}
    \put(112, 52){\makebox(0, 0){\tiny $(2)$}}
    \put(107, 70){\makebox(0, 0)[b]{$\Gamma_{Q, k-1}$}}
    \put(110, 60){\line(1, 0){30}}
    \put(145, 60){\makebox(0, 0){\large$\oslash$}}
    \put(152, 52){\makebox(0, 0){\tiny $(2)$}}
    \put(147, 70){\makebox(0, 0)[b]{$l_2^M$}}
    \end{picture}
\end{center}
\end{enumerate}
\end{lemma}

\begin{proof}
Follows immediately from \cite[Example 2.6]{F}. 
\end{proof}

\begin{lemma}\label{ZS3}
\begin{enumerate}
\renewcommand{\theenumi}{\arabic{enumi}}
\renewcommand{\labelenumi}{(\theenumi)}
\item\label{ZS31}
The divisor $E_Z$ is not of the form
$E_Z=l_1+l_2+l_3$ around $Q$, where $l_1$, $l_2$, $l_3$ are distinct and
$Q\in l_i$ is nonsingular for $1\leqslant i\leqslant 3$.
\item\label{ZS32}
Assume that $E_Z=2l_1+l_2+l_3$ around $Q$, where 
$l_1$, $l_2$, $l_3$ are distinct and $Q\in l_i$ is nonsingular for $1\leqslant i\leqslant 3$. 
Then either $l_2^M$ or $l_3^M$ is not a connected component of $E_M$.
\end{enumerate}
\end{lemma}

\begin{proof}
\eqref{ZS31}
Assume the contrary. 
Set $m_{ij}:=\mult_Q(l_i\cap l_j)$ for $1\leqslant i<j\leqslant 3$.
We can assume that $m_{12}\geqslant m_{13}\geqslant m_{23}\geqslant 1$.
Then $\mult_Q\Delta_Z\geqslant m_{23}$ and 
$\coeff_{\Gamma_{Q, m_{23}}}E_M=m_{23}$. 
Thus $m_{23}\leqslant 2$. Assume that $m_{23}=1$. Then 
$\coeff_{\Gamma_{Q, 1}}E_M=\coeff_{l_3^M}E_M=1$ and 
$\Gamma_{Q, 1}\cap l_3^M\neq\emptyset$. This contradicts to 
Corollary \ref{dP-basic_cor}. 
Thus $m_{23}=2$. Set $m:=\mult_Q(\Delta_Z\cap l_2)$. Then 
$\coeff_{\Gamma_{Q, m}}E_M=2$, and $\Gamma_{Q, m}$ intersects $l_2^M$. Moreover, 
$\Gamma_{Q, m}$ intersects $l_1^M$ or $\Gamma_{Q, m+1}$, and 
$\coeff_{\Gamma_{Q, m+1}}E_M\geqslant 1$ (if $m+1\leqslant\mult_Q\Delta_Z$). 
Thus the vertex of the dual graph of $E_M$ corresponds to $\Gamma_{Q, m}$ is a fork. 
On the other hand, $\Gamma_{Q, 2}$ intersects $l_3^M$ and $\Gamma_{Q, 1}$. 
Thus the vertex of the dual graph of $E_M$ corresponds to $\Gamma_{Q, 2}$ 
is also a fork. However, $\Gamma_{Q, 2}$ and $\Gamma_{Q, m}$ belong to a same 
connected component of $E_M$. This contradicts to Corollary \ref{dP-basic_cor}. 

\eqref{ZS32}
Assume the contrary. The morphism $\phi\colon M\to Z$ factors through the 
monoidal transform $Z_1\to Z$ at $Q$. Then $E_{Z_1}:=E_M^{Z_1}$ is equal to 
$2l_1^{Z_1}+2\Gamma_{Q,1}^{Z_1}+l_2^{Z_1}+l_3^{Z_1}$. 
If $\Gamma_{Q,1}^{Z_1}\cap l_2^{Z_1}\cap l_3^{Z_1}=\emptyset$, then 
either $\Gamma_{Q,1}\cap l_2^M\neq\emptyset$ or 
$\Gamma_{Q,1}\cap l_3^M\neq\emptyset$ holds, which leads to a contradiction. 
Thus we can take $Q_1\in\Gamma_{Q,1}^{Z_1}\cap l_2^{Z_1}\cap l_3^{Z_1}$ and 
the morphism $M\to Z_1$ factors through the monoidal transform $Z_2\to Z_1$ 
at $Q_1$. We note that $Q_1\not\in l_1^{Z_1}$ by Lemma \ref{fund_big_lem} 
\eqref{fund_big_lem5}. We must continue this process infinitely many times. 
This leads to a contradiction.
\end{proof}

\begin{lemma}\label{ZS4}
Assume that $E_Z=l_1+l_2$ around $Q$, where $Q\in l_i$ is nonsingular, 
$\{Q\}=|l_1\cap l_2|$, and $\mult_Q(l_1\cap l_2)=m\geqslant 2$. 
Then 
$\mult_Q\Delta_Z=\mult_Q(\Delta_Z\cap l_1)=\mult_Q(\Delta\cap l_2)=m$ holds. 
In other words, $\Delta_Z$ is equal to $l_1\cap l_2$ around $Q$. 
Moreover, 
$E_M=l_1^M+l_2^M$ and the weighted dual graph of $E_M$ around over $Q$ 
is the following:
\begin{center}
    \begin{picture}(40, 40)(0, 45)
    \put(0, 60){\makebox(0, 0){\large$\oslash$}}
    \put(3, 70){\makebox(0, 0)[b]{$l_1^M$}}
    \put(7, 52){\makebox(0, 0){\tiny $(1)$}}
    \put(21, 55){\makebox(0, 0)[b]{$\sqcup$}}
    \put(40, 60){\makebox(0, 0){\large$\oslash$}}
    \put(47, 52){\makebox(0, 0){\tiny $(1)$}}
    \put(43, 70){\makebox(0, 0)[b]{$l_2^M$}}
    \end{picture}
\end{center}
\end{lemma}

\begin{proof}
The morphism $\phi\colon M\to Z$ factors through the monoidal 
transform $\pi_1\colon Z_1\to Z$ at $Q$. Then $E_{Z_1}:=E_M^{Z_1}$ is equal to 
$l_1^{Z_1}+l_2^{Z_1}$ around over $Q$ 
such that $\{Q_1\}:=|l_1^{Z_1}\cap l_2^{Z_1}|$ and 
$\mult_{Q_1}(l_1^{Z_1}\cap l_2^{Z_1})=m-1$ hold. If $m-1\geqslant 2$, then 
$\phi_1\colon M\to Z_1$ factors through the monoidal 
transform $\pi_2\colon Z_2\to Z_1$ at $Q_1$. 
By repeating the same argument, we get the following sequence:
\[
M\xrightarrow{\phi_{m-1}}Z_{m-1}\xrightarrow{\pi_{m-1}}Z_{m-2}
\xrightarrow{\pi_{m-2}}\cdots\xrightarrow{\pi_1}Z. 
\]
If $\phi_{m-1}$ is an isomorphism around over $Q$, 
then the weighted dual graph of $E_M$ over $Q$ is the following:
\begin{center}
    \begin{picture}(40, 40)(0, 40)
    \put(0, 60){\makebox(0, 0){\large$\oslash$}}
    \put(3, 70){\makebox(0, 0)[b]{$l_1^M$}}
    \put(7, 52){\makebox(0, 0){\tiny $(1)$}}
    \put(5, 60){\line(1, 0){30}}
    \put(40, 60){\makebox(0, 0){\large$\oslash$}}
    \put(47, 52){\makebox(0, 0){\tiny $(1)$}}
    \put(43, 70){\makebox(0, 0)[b]{$l_2^M$}}
    \end{picture}
\end{center}
This contradicts to Corollary \ref{dP-basic_cor}. Indeed, two curves in $E_M$ such 
that both coefficients are equal to one 
does not meet together.
Thus $\phi_{m-1}$ around over $Q$ 
is equal to the monoidal transform at $Q_{m-1}$ by 
Lemmas \ref{ZS1} and \ref{ZS2}. 
\end{proof}

\begin{lemma}\label{ZS5}
Assume that $E_Z=2l_1+l_2$ around $Q$, where $Q\in l_i$ is nonsingular, 
$\{Q\}=|l_1\cap l_2|$, $\mult_Q(l_1\cap l_2)=2$.
\begin{enumerate}
\renewcommand{\theenumi}{\arabic{enumi}}
\renewcommand{\labelenumi}{(\theenumi)}
\item\label{ZS51}
Assume that $\mult_Q(\Delta_Z\cap l_2)\geqslant 3$. 
Then $\mult_Q\Delta_Z=\mult_Q(\Delta_Z\cap l_2)=4$, $\mult_Q(\Delta_Z\cap l_1)=2$ 
and $E_M=2l_1^M+\Gamma_{Q, 1}+2\Gamma_{Q, 2}+\Gamma_{Q, 3}+l_2^M$. 
The weighted dual graph of $E_M$ around over $Q$ is the following: 
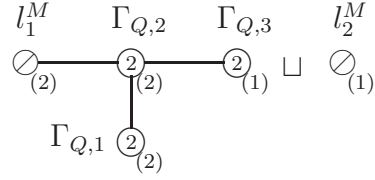
\begin{center}
    \begin{picture}(120, 65)(0, 20)
    \put(0, 60){\makebox(0, 0){\large$\oslash$}}
    \put(3, 70){\makebox(0, 0)[b]{$l_1^M$}}
    \put(7, 52){\makebox(0, 0){\tiny $(2)$}}
    \put(5, 60){\line(1, 0){30}}
    \put(40, 60){\makebox(0, -3){\textcircled{\tiny $2$}}}
    \put(47, 52){\makebox(0, 0){\tiny $(2)$}}
    \put(43, 70){\makebox(0, 0)[b]{$\Gamma_{Q, 2}$}}
    \put(45, 60){\line(1, 0){30}}
    \put(80, 60){\makebox(0, -3){\textcircled{\tiny $2$}}}
    \put(87, 52){\makebox(0, 0){\tiny $(1)$}}
    \put(83, 70){\makebox(0, 0)[b]{$\Gamma_{Q, 3}$}}
    \put(40, 55){\line(0, -1){20}}
    \put(40, 30){\makebox(0, -3){\textcircled{\tiny $2$}}}
    \put(47, 22){\makebox(0, 0){\tiny $(2)$}}
    \put(20, 25){\makebox(0, 0)[b]{$\Gamma_{Q, 1}$}}
    \put(101, 55){\makebox(0, 0)[b]{$\sqcup$}}
    \put(120, 60){\makebox(0, 0){\large$\oslash$}}
    \put(127, 52){\makebox(0, 0){\tiny $(1)$}}
    \put(123, 70){\makebox(0, 0)[b]{$l_2^M$}}
    \end{picture}
\end{center}
\item\label{ZS52}
Assume that $\mult_Q(\Delta_Z\cap l_2)=2$. 
Set $k:=\mult_Q\Delta_Z$.
Then $\mult_Q(\Delta_Z\cap l_1)=k-1$ and $E_M=2l_1^M+\Gamma_{Q, 1}
+2\Gamma_{Q, 2}+\cdots+2\Gamma_{Q, k-1}+l_2^M$. 
The weighted dual graph of $E_M$ around over $Q$ is the following: 
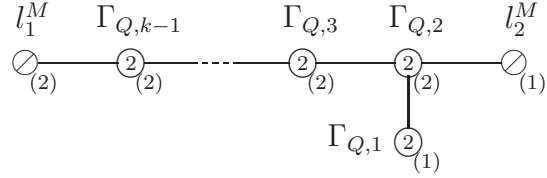
\begin{center}
    \begin{picture}(180, 65)(0, 20)
    \put(0, 60){\makebox(0, 0){\large$\oslash$}}
    \put(3, 70){\makebox(0, 0)[b]{$l_1^M$}}
    \put(7, 52){\makebox(0, 0){\tiny $(2)$}}
    \put(5, 60){\line(1, 0){30}}
    \put(40, 60){\makebox(0, -3){\textcircled{\tiny $2$}}}
    \put(47, 52){\makebox(0, 0){\tiny $(2)$}}
    \put(43, 70){\makebox(0, 0)[b]{$\Gamma_{Q, k-1}$}}
    \put(45, 60){\line(1, 0){20}}
    \put(67, 60){\line(1, 0){2}}
    \put(71, 60){\line(1, 0){2}}
    \put(75, 60){\line(1, 0){2}}
    \put(79, 60){\line(1, 0){21}}
    \put(105, 60){\makebox(0, -3){\textcircled{\tiny $2$}}}
    \put(112, 52){\makebox(0, 0){\tiny $(2)$}}
    \put(108, 70){\makebox(0, 0)[b]{$\Gamma_{Q, 3}$}}
    \put(110, 60){\line(1, 0){30}}
    \put(145, 60){\makebox(0, -3){\textcircled{\tiny $2$}}}
    \put(152, 52){\makebox(0, 0){\tiny $(2)$}}
    \put(148, 70){\makebox(0, 0)[b]{$\Gamma_{Q, 2}$}}
    \put(150, 60){\line(1, 0){30}}
    \put(145, 55){\line(0, -1){20}}
    \put(145, 30){\makebox(0, -3){\textcircled{\tiny $2$}}}
    \put(152, 22){\makebox(0, 0){\tiny $(1)$}}
    \put(125, 25){\makebox(0, 0)[b]{$\Gamma_{Q, 1}$}}
    \put(185, 60){\makebox(0, 0){\large$\oslash$}}
    \put(192, 52){\makebox(0, 0){\tiny $(1)$}}
    \put(188, 70){\makebox(0, 0)[b]{$l_2^M$}}
    \end{picture}
\end{center}
\end{enumerate}
\end{lemma}

\begin{proof}
The morphism $\phi\colon M\to Z$ factors through the monoidal 
transform $\pi_1\colon Z_1\to Z$ at $Q$. Then $E_{Z_1}:=E_M^{Z_1}$ 
around over $Q$ 
is equal to 
$2l_1^{Z_1}+l_2^{Z_1}+\Gamma_{Q, 1}^{Z_1}$ such that 
$Q_1:=l_1^{Z_1}\cap l_2^{Z_1}$ (meet transversally) and $Q_1\in\Gamma_{Q, 1}^{Z_1}$. 
Thus $\phi_1\colon M\to Z_1$ factors through the monoidal 
transform $\pi_2\colon Z_2\to Z_1$ at $Q_1$. Then $E_{Z_2}:=E_M^{Z_2}$ 
around over $Q$ is equal to 
$2l_1^{Z_2}+l_2^{Z_2}+\Gamma_{Q, 1}^{Z_2}+2\Gamma_{Q, 2}^{Z_2}$. 
For the case \eqref{ZS51}, the morphism $M\to Z_2$ is not isomorphic over 
$Q_{22}:=l_2^{Z_1}\cap\Gamma_{Q, 2}^{Z_2}$ 
since $\mult_Q(\Delta_Z\cap l_2)\geqslant 3$. 
Then we can apply Lemma \ref{ZS2} for the local property around $Q_{22}$; 
$\mult_Q\Delta_Z=\mult_Q(\Delta_Z\cap l_2)=4$ and 
$\mult_Q(\Delta_Z\cap l_1)=2$. 
For the case \eqref{ZS52}, if $\mult_Q(\Delta_Z\cap l_1)\geqslant 3$, 
then the morphism $M\to Z_2$ is not isomorphic over 
$Q_{21}:=l_1^{Z_2}\cap\Gamma_{Q, 2}^{Z_2}$. 
Then we can apply Lemma \ref{ZS2} for the local property around $Q_{21}$; 
we obtain that $\mult_Q(\Delta\cap l_1)=k-1$. 
The remaining parts follow easily.
\end{proof}

\subsection{Local properties of bottom tetrads}\label{XS_section}

Let $(X, E_X; \Delta_Z, \Delta_X)$ be a 
$3$-fundamental multiplet of length two, 
$P\in\Delta_X$ be a point, $\psi\colon Z\to X$ be the elimination of $\Delta_X$, 
$(Z, E_Z; \Delta_Z)$ be the associated pseudo-median triplet, 
$\phi\colon M\to Z$ be the elimination of $\Delta_Z$ and $(M, E_M)$ be the 
associated $3$-basic pair. 

\begin{lemma}\label{XS1}
Assume that $E_X=sl$ around $P$, where $P\in l$ is nonsingular and $s\geqslant 0$. 
Then $s=1$ or $2$ holds. If $s=1$, then $\Delta_X\subset l$ and 
$\Delta_Z=\emptyset$ around over $P$. In this case, $E_Z=l^Z$ and $E_M=l^M$ 
around over $P$. 
Assume that $s=2$. Set $k:=\mult_P\Delta_X$ and $j:=\mult_P(\Delta_X\cap l)$. 
Then one of the following holds: 
\begin{enumerate}
\renewcommand{\theenumi}{\arabic{enumi}}
\renewcommand{\labelenumi}{(\theenumi)}
\item\label{XS11}
$(k ,j)=(4, 2)$. In this case, $\Delta_Z=\emptyset$ and 
$E_Z(=E_M)=2l^Z+\Gamma_{P, 1}+2\Gamma_{P, 2}+\Gamma_{P, 3}$ around over $P$. 
The weighted dual graph of $E_Z$ around over $P$ is the following: 
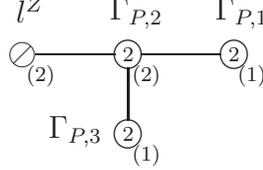
\begin{center}
    \begin{picture}(80, 60)(0, 20)
    \put(0, 60){\makebox(0, 0){\large$\oslash$}}
    \put(3, 72){\makebox(0, 0)[b]{$l^Z$}}
    \put(7, 52){\makebox(0, 0){\tiny $(2)$}}
    \put(5, 60){\line(1, 0){30}}
    \put(40, 60){\makebox(0, -3){\textcircled{\tiny $2$}}}
    \put(47, 52){\makebox(0, 0){\tiny $(2)$}}
    \put(43, 70){\makebox(0, 0)[b]{$\Gamma_{P, 2}$}}
    \put(45, 60){\line(1, 0){30}}
    \put(80, 60){\makebox(0, -3){\textcircled{\tiny $2$}}}
    \put(87, 52){\makebox(0, 0){\tiny $(1)$}}
    \put(83, 70){\makebox(0, 0)[b]{$\Gamma_{P, 1}$}}
    \put(40, 55){\line(0, -1){20}}
    \put(40, 30){\makebox(0, -3){\textcircled{\tiny $2$}}}
    \put(47, 22){\makebox(0, 0){\tiny $(1)$}}
    \put(20, 25){\makebox(0, 0)[b]{$\Gamma_{P, 3}$}}
    \end{picture}
\end{center}
\item\label{XS12}
$(k, j)=(2, 2)$. In this case, $E_Z=2l^Z+\Gamma_{P, 1}+2\Gamma_{P, 2}$, 
$|\Delta_Z|\subset\Gamma_{P, 2}$ around over $P$ 
and $\deg(\Delta_Z\cap\Gamma_{P, 2})=2$. 
The weighted dual graph of $E_Z$ around over $P$ is the following: 
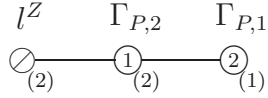
\begin{center}
    \begin{picture}(80, 43)(0, 40)
    \put(0, 60){\makebox(0, 0){\large$\oslash$}}
    \put(3, 72){\makebox(0, 0)[b]{$l^Z$}}
    \put(7, 52){\makebox(0, 0){\tiny $(2)$}}
    \put(5, 60){\line(1, 0){30}}
    \put(40, 60){\makebox(0, -3){\textcircled{\tiny $1$}}}
    \put(47, 52){\makebox(0, 0){\tiny $(2)$}}
    \put(43, 70){\makebox(0, 0)[b]{$\Gamma_{P, 2}$}}
    \put(45, 60){\line(1, 0){30}}
    \put(80, 60){\makebox(0, -3){\textcircled{\tiny $2$}}}
    \put(87, 52){\makebox(0, 0){\tiny $(1)$}}
    \put(83, 70){\makebox(0, 0)[b]{$\Gamma_{P, 1}$}}
    \end{picture}
\end{center}
\item\label{XS13}
$(k ,j)=(2, 1)$. In this case, $\Delta_Z=\emptyset$ and 
$E_Z(=E_M)=2l^Z+\Gamma_{P, 1}$ around over $P$. 
The weighted dual graph of $E_Z$ around over $P$ is the following: 
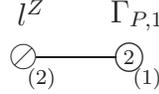
\begin{center}
    \begin{picture}(40, 43)(0, 40)
    \put(0, 60){\makebox(0, 0){\large$\oslash$}}
    \put(3, 72){\makebox(0, 0)[b]{$l^Z$}}
    \put(7, 52){\makebox(0, 0){\tiny $(2)$}}
    \put(5, 60){\line(1, 0){30}}
    \put(40, 60){\makebox(0, -3){\textcircled{\tiny $2$}}}
    \put(47, 52){\makebox(0, 0){\tiny $(1)$}}
    \put(43, 70){\makebox(0, 0)[b]{$\Gamma_{P, 1}$}}
    \end{picture}
\end{center}
\item\label{XS14}
$(k ,j)=(1, 1)$. In this case, $|\Delta_Z|=\{Q\}$ around over $P$, 
where $Q:=l^Z\cap\Gamma_{P, 1}$. Moreover, we have 
$\mult_Q\Delta_Z=\mult_Q(\Delta_Z\cap\Gamma_{P, 1})=2$ and 
$\mult_Q(\Delta_Z\cap l^Z)=1$ hold. The weighted dual graph of $E_Z$ 
around over $P$ is the following: 
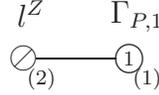
\begin{center}
    \begin{picture}(40, 43)(0, 40)
    \put(0, 60){\makebox(0, 0){\large$\oslash$}}
    \put(3, 72){\makebox(0, 0)[b]{$l^Z$}}
    \put(7, 52){\makebox(0, 0){\tiny $(2)$}}
    \put(5, 60){\line(1, 0){30}}
    \put(40, 60){\makebox(0, -3){\textcircled{\tiny $1$}}}
    \put(47, 52){\makebox(0, 0){\tiny $(1)$}}
    \put(43, 70){\makebox(0, 0)[b]{$\Gamma_{P, 1}$}}
    \end{picture}
\end{center}
The weighted dual graph of $E_M$ around over $P$ is the following: 
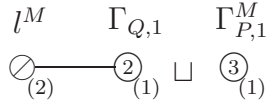
\begin{center}
    \begin{picture}(80, 43)(0, 43)
    \put(0, 60){\makebox(0, 0){\large$\oslash$}}
    \put(3, 73){\makebox(0, 0)[b]{$l^M$}}
    \put(7, 52){\makebox(0, 0){\tiny $(2)$}}
    \put(5, 60){\line(1, 0){30}}
    \put(40, 60){\makebox(0, -3){\textcircled{\tiny $2$}}}
    \put(47, 52){\makebox(0, 0){\tiny $(1)$}}
    \put(43, 71){\makebox(0, 0)[b]{$\Gamma_{Q, 1}$}}
    \put(61, 55){\makebox(0, 0)[b]{$\sqcup$}}
    \put(80, 60){\makebox(0, -3){\textcircled{\tiny $3$}}}
    \put(87, 52){\makebox(0, 0){\tiny $(1)$}}
    \put(83, 70){\makebox(0, 0)[b]{$\Gamma_{P, 1}^M$}}
    \end{picture}
\end{center}
\end{enumerate}
\end{lemma}

\begin{proof}
If $s=1$, then the assertion is trivial by \cite[Example 2.5]{F} and Lemma \ref{ZS1}. 
We assume that $s=2$. 
If $j\geqslant 3$, then $\coeff_{\Gamma_{P, 3}}E_Z=3$. This leads to a contradiction. 
Thus $j=1$ or $2$. 
If $j=1$ and $k\geqslant 3$, then $\coeff_{\Gamma_{P, 3}}E_Z=-1$, 
which is a contradiction. 
If $j=2$ and $k\geqslant 5$, then $\coeff_{\Gamma_{P, 5}}E_Z=-1$, 
which is a contradiction. 
If $(k ,j)=(3, 2)$ then $\Gamma_{P, 3}\cap\Delta_Z\neq\emptyset$ and 
$\Gamma_{P, 2}\cap\Delta_Z=\emptyset$. 
Indeed, $(L_Z\cdot\Gamma_{P, 3})=2$ and $(L_Z\cdot\Gamma_{P, 2})=0$ hold, 
where $L_Z$ is the fundamental divisor of $(Z, E_Z, \Delta_Z)$. 
However, we know that $\coeff_{\Gamma_{P, 3}}E_Z=1$ and the curve 
$\Gamma_{P, 2}$ is the only component of $E_Z$ which meets $\Gamma_{P, 3}$.
Thus $\Delta_Z\cap\Gamma_{P, 3}=\emptyset$, which is a contradiction. 
Therefore $(k ,j)=(4, 2)$, $(2, 2)$, $(2, 1)$ or $(1, 1)$. 
The remaining parts follow easily from Lemmas \ref{ZS1} and \ref{ZS2}. 
\end{proof}

\begin{lemma}\label{XS2}
Assume that $E_X=s_1l_1+s_2l_2$ around $P$, where $P\in l_i$ is nonsingular, 
$s_1\geqslant s_2\geqslant 1$, and $l_1$ and $l_2$ intersect transversally at $P$. 
Then $(s_1, s_2)=(1, 1)$ or $(2, 1)$. Moreover, we have the following:
\begin{enumerate}
\renewcommand{\theenumi}{\arabic{enumi}}
\renewcommand{\labelenumi}{(\theenumi)}
\item\label{XS21}
Assume that $(s_1, s_2)=(1, 1)$. Then $\mult_P\Delta_X=1$. Set 
$Q_i:=l_i^Z\cap \Gamma_{P, 1}$. Then $|\Delta_Z|=\{Q_1, Q_2\}$ around over $P$ and 
$\mult_{Q_i}\Delta_Z=1$. In this case, $E_Z=l_1^Z+\Gamma_{P, 1}+l_2^Z$ and 
$E_M=l_1^M+\Gamma_{P, 1}^M+l_2^M$ around over $P$. 
The weighted dual graph of $E_Z$ around over $P$ is the following: 
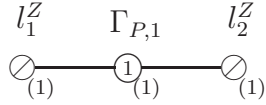
\begin{center}
    \begin{picture}(80, 47)(0, 40)
    \put(0, 60){\makebox(0, 0){\large$\oslash$}}
    \put(3, 72){\makebox(0, 0)[b]{$l_1^Z$}}
    \put(7, 52){\makebox(0, 0){\tiny $(1)$}}
    \put(5, 60){\line(1, 0){30}}
    \put(40, 60){\makebox(0, -3){\textcircled{\tiny $1$}}}
    \put(47, 52){\makebox(0, 0){\tiny $(1)$}}
    \put(43, 70){\makebox(0, 0)[b]{$\Gamma_{P, 1}$}}
    \put(45, 60){\line(1, 0){30}}
    \put(80, 60){\makebox(0, 0){\large$\oslash$}}
    \put(87, 52){\makebox(0, 0){\tiny $(1)$}}
    \put(83, 72){\makebox(0, 0)[b]{$l_2^Z$}}
    \end{picture}
\end{center}
The weighted dual graph of $E_M$ around over $P$ is the following: 
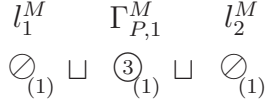
\begin{center}
    \begin{picture}(80, 47)(0, 40)
    \put(0, 60){\makebox(0, 0){\large$\oslash$}}
    \put(3, 72){\makebox(0, 0)[b]{$l_1^M$}}
    \put(7, 52){\makebox(0, 0){\tiny $(1)$}}
    \put(21, 55){\makebox(0, 0)[b]{$\sqcup$}}
    \put(40, 60){\makebox(0, -3){\textcircled{\tiny $3$}}}
    \put(47, 52){\makebox(0, 0){\tiny $(1)$}}
    \put(43, 70){\makebox(0, 0)[b]{$\Gamma_{P, 1}^M$}}
    \put(61, 55){\makebox(0, 0)[b]{$\sqcup$}}
    \put(80, 60){\makebox(0, 0){\large$\oslash$}}
    \put(87, 52){\makebox(0, 0){\tiny $(1)$}}
    \put(83, 72){\makebox(0, 0)[b]{$l_2^M$}}
    \end{picture}
\end{center}
\item\label{XS22}
Assume that $(s_1, s_2)=(2, 1)$. Then $\mult_P(\Delta_X\cap l_1)=1$. Set 
$k:=\mult_P\Delta_X$ and $j:=\mult_P(\Delta_X\cap l_2)$. 
Then one of the following holds: 
\begin{enumerate}
\renewcommand{\theenumii}{\alph{enumii}}
\renewcommand{\labelenumii}{(\theenumii)}
\item\label{XS221}
$k=j\geqslant 1$ holds. In this case, $|\Delta_Z|\subset\Gamma_{P, k}$, 
$\deg(\Delta_Z\cap\Gamma_{P, k})=2$ and 
$E_Z=2l_1^Z+2\Gamma_{P, 1}+\cdots+2\Gamma_{P, k}+l_2^Z$ 
around over $P$. 
The weighted dual graph of $E_Z$ around over $P$ is the following: 
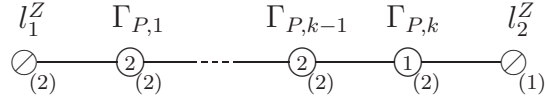
\begin{center}
    \begin{picture}(190, 43)(0, 43)
    \put(0, 60){\makebox(0, 0){\large$\oslash$}}
    \put(3, 70){\makebox(0, 0)[b]{$l_1^Z$}}
    \put(7, 52){\makebox(0, 0){\tiny $(2)$}}
    \put(5, 60){\line(1, 0){30}}
    \put(40, 60){\makebox(0, -3){\textcircled{\tiny $2$}}}
    \put(47, 52){\makebox(0, 0){\tiny $(2)$}}
    \put(43, 70){\makebox(0, 0)[b]{$\Gamma_{P, 1}$}}
    \put(45, 60){\line(1, 0){20}}
    \put(67, 60){\line(1, 0){2}}
    \put(71, 60){\line(1, 0){2}}
    \put(75, 60){\line(1, 0){2}}
    \put(79, 60){\line(1, 0){21}}
    \put(105, 60){\makebox(0, -3){\textcircled{\tiny $2$}}}
    \put(112, 52){\makebox(0, 0){\tiny $(2)$}}
    \put(107, 70){\makebox(0, 0)[b]{$\Gamma_{P, k-1}$}}
    \put(110, 60){\line(1, 0){30}}
    \put(145, 60){\makebox(0, -3){\textcircled{\tiny $1$}}}
    \put(152, 52){\makebox(0, 0){\tiny $(2)$}}
    \put(147, 70){\makebox(0, 0)[b]{$\Gamma_{P, k}$}}
    \put(150, 60){\line(1, 0){30}}
    \put(185, 60){\makebox(0, 0){\large$\oslash$}}
    \put(192, 52){\makebox(0, 0){\tiny $(1)$}}
    \put(188, 70){\makebox(0, 0)[b]{$l_2^Z$}}
    \end{picture}
\end{center}
\item\label{XS222}
$k=j+2\geqslant 3$ holds. In this case, $\Delta_Z=\emptyset$ and 
$E_Z(=E_M)=2l_1^Z+2\Gamma_{P, 1}+\cdots+2\Gamma_{P, k-2}
+\Gamma_{P, k-1}+l_2^Z$ 
around over $P$. 
The weighted dual graph of $E_Z$ around over $P$ is the following: 
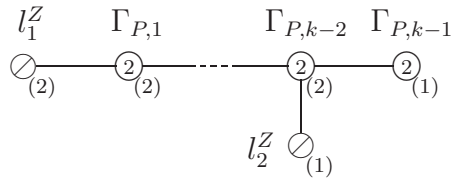
\begin{center}
    \begin{picture}(150, 80)(0, 10)
    \put(0, 60){\makebox(0, 0){\large$\oslash$}}
    \put(3, 70){\makebox(0, 0)[b]{$l_1^Z$}}
    \put(7, 52){\makebox(0, 0){\tiny $(2)$}}
    \put(5, 60){\line(1, 0){30}}
    \put(40, 60){\makebox(0, -3){\textcircled{\tiny $2$}}}
    \put(47, 52){\makebox(0, 0){\tiny $(2)$}}
    \put(43, 70){\makebox(0, 0)[b]{$\Gamma_{P, 1}$}}
    \put(45, 60){\line(1, 0){20}}
    \put(67, 60){\line(1, 0){2}}
    \put(71, 60){\line(1, 0){2}}
    \put(75, 60){\line(1, 0){2}}
    \put(79, 60){\line(1, 0){21}}
    \put(105, 60){\makebox(0, -3){\textcircled{\tiny $2$}}}
    \put(112, 52){\makebox(0, 0){\tiny $(2)$}}
    \put(107, 70){\makebox(0, 0)[b]{$\Gamma_{P, k-2}$}}
    \put(110, 60){\line(1, 0){30}}
    \put(145, 60){\makebox(0, -3){\textcircled{\tiny $2$}}}
    \put(152, 52){\makebox(0, 0){\tiny $(1)$}}
    \put(147, 70){\makebox(0, 0)[b]{$\Gamma_{P, k-1}$}}
    \put(105, 55){\line(0, -1){20}}
    \put(105, 30){\makebox(0, 0){\large$\oslash$}}
    \put(112, 22){\makebox(0, 0){\tiny $(1)$}}
    \put(90, 22){\makebox(0, 0)[b]{$l_2^Z$}}
    \end{picture}
\end{center}
\end{enumerate}
\end{enumerate}
\end{lemma}

\begin{proof}
If $(s_1, s_2)=(2, 2)$, then $\coeff_{\Gamma_{P, 1}}E_Z=3$, a contradiction. Thus 
$(s_1, s_2)=(1, 1)$ or $(2, 1)$. 

\eqref{XS21}
Assume that $(s_1, s_2)=(1, 1)$.
Set $k:=\mult_P\Delta_X$.
 If $k\geqslant 2$, then 
$\Gamma_{P, 1}\cap\Delta_Z=\emptyset$, $\coeff_{\Gamma_{P, 1}}E_Z=1$, 
$\coeff_{l_1^Z}E_Z=1$ and the curve $l_1^Z$ is the unique component of $E_Z$ which 
meets $\Gamma_{P, 1}$. This contradicts to Corollary \ref{dP-basic_cor}. 
Thus $k=1$. Then $\deg(\Gamma_{P, 1}\cap\Delta_Z)=2$. 
By Lemma \ref{ZS2}, we have $\Delta_Z=\{Q_1, Q_2\}$ and $\mult_{Q_i}E_Z=1$ 
around over $P$. 

\eqref{XS22}
Assume that $(s_1, s_2)=(2, 1)$. If $\mult_P(\Delta_X\cap l_1)\geqslant 2$, then 
$\coeff_{\Gamma_{P, 2}}E_Z=3$. This leads to a contradiction. 
Thus $\mult_P(\Delta_X\cap l_1)=1$. 
If $k\geqslant j+3$, then $\coeff_{\Gamma_{P, j+3}}E_Z=-1$, a contradiction. 
If $k=j+1$, then $\coeff_{\Gamma_{P, k}}E_Z=1$, $\deg(\Delta_Z\cap\Gamma_{P, k})=2$
and $\deg(\Delta_Z\cap\Gamma_{P, k-1})=0$. Note that the curve 
$\Gamma_{P, k-1}$ is the unique component of $E_Z$ which meets $\Gamma_{P, k}$. 
Thus $\Delta_X\cap\Gamma_{P, k}=\emptyset$, a contradiction. 
Thus either $k=j$ or $j+2$ holds. 
The remaining assertions follow from Lemmas \ref{ZS1} and \ref{ZS2}. 
\end{proof}

\begin{lemma}\label{XS3}
Assume that $E_X=l_1+l_2+l_3$ around $P$, where $P\in l_i$ is nonsingular, and 
$l_i$ and $l_j$ intersect transversally at $P$ for any $1\leqslant i<j\leqslant 3$.
Then we can assume that $\mult_P(\Delta_X\cap l_2)=\mult_P(\Delta_X\cap l_3)=1$. 
Set $k:=\mult_P\Delta_X$ and $j:=\mult_P(\Delta_X\cap l_1)$. 
Then $k=j$, $|\Delta_Z|\subset\Gamma_{P, k}$, 
$\deg(\Delta_Z\cap\Gamma_{P, k})=2$ and 
$E_Z=l_2^Z+l_3^Z+2\Gamma_{P, 1}+\cdots+2\Gamma_{P, k}+l_1^Z$ around over $P$. 
The weighted dual graph of $E_Z$ around over $P$ 
is the following $($if $k=1$, then 
$\Gamma_{P, 1}$ is a $(-1)$-curve and meets $l_1^Z$, $l_2^Z$ and $l_3^Z$$)$: 
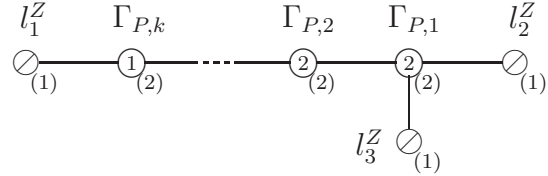
\begin{center}
    \begin{picture}(190, 70)(0, 20)
    \put(0, 60){\makebox(0, 0){\large$\oslash$}}
    \put(3, 70){\makebox(0, 0)[b]{$l_1^Z$}}
    \put(7, 52){\makebox(0, 0){\tiny $(1)$}}
    \put(5, 60){\line(1, 0){30}}
    \put(40, 60){\makebox(0, -3){\textcircled{\tiny $1$}}}
    \put(47, 52){\makebox(0, 0){\tiny $(2)$}}
    \put(43, 70){\makebox(0, 0)[b]{$\Gamma_{P, k}$}}
    \put(45, 60){\line(1, 0){20}}
    \put(67, 60){\line(1, 0){2}}
    \put(71, 60){\line(1, 0){2}}
    \put(75, 60){\line(1, 0){2}}
    \put(79, 60){\line(1, 0){21}}
    \put(105, 60){\makebox(0, -3){\textcircled{\tiny $2$}}}
    \put(112, 52){\makebox(0, 0){\tiny $(2)$}}
    \put(107, 70){\makebox(0, 0)[b]{$\Gamma_{P, 2}$}}
    \put(110, 60){\line(1, 0){30}}
    \put(145, 60){\makebox(0, -3){\textcircled{\tiny $2$}}}
    \put(152, 52){\makebox(0, 0){\tiny $(2)$}}
    \put(147, 70){\makebox(0, 0)[b]{$\Gamma_{P, 1}$}}
    \put(150, 60){\line(1, 0){30}}
    \put(185, 60){\makebox(0, 0){\large$\oslash$}}
    \put(192, 52){\makebox(0, 0){\tiny $(1)$}}
    \put(188, 70){\makebox(0, 0)[b]{$l_2^Z$}}
    \put(145, 55){\line(0, -1){20}}
    \put(145, 30){\makebox(0, 0){\large$\oslash$}}
    \put(152, 22){\makebox(0, 0){\tiny $(1)$}}
    \put(130, 22){\makebox(0, 0)[b]{$l_3^Z$}}
    \end{picture}
\end{center}
\end{lemma}

\begin{proof}
Assume that $k\geqslant j+3$. 
Then $\coeff_{\Gamma_{P, k}}E_Z\leqslant -1$, which is a contradiction. 
Assume that $k=j+1$. Then $\coeff_{\Gamma_{P, k}}E_Z=1$, 
$\deg(\Delta_Z\cap\Gamma_{P, k})=2$, $\Delta_Z\cap\Gamma_{P, k-1}=\emptyset$, 
and the curve $\Gamma_{P, k-1}$ is the unique component of $E_Z$ which meets 
$\Gamma_{P, k}$. This leads to a contradiction. 
Assume that $k=j+2$. Then $\Delta_Z=\emptyset$ around over $P$
and the weighted dual graph of $E_Z(=E_M)$ around over $P$ is the following: 
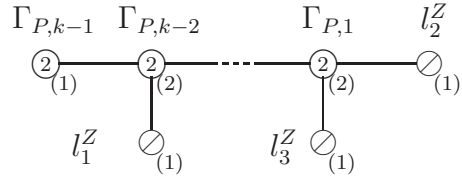
\begin{center}
    \begin{picture}(150, 60)(0, 20)
    \put(0, 60){\makebox(0, -3){\textcircled{\tiny $2$}}}
    \put(3, 70){\makebox(0, 0)[b]{$\Gamma_{P, k-1}$}}
    \put(7, 52){\makebox(0, 0){\tiny $(1)$}}
    \put(5, 60){\line(1, 0){30}}
    \put(40, 60){\makebox(0, -3){\textcircled{\tiny $2$}}}
    \put(47, 52){\makebox(0, 0){\tiny $(2)$}}
    \put(43, 70){\makebox(0, 0)[b]{$\Gamma_{P, k-2}$}}
    \put(45, 60){\line(1, 0){20}}
    \put(67, 60){\line(1, 0){2}}
    \put(71, 60){\line(1, 0){2}}
    \put(75, 60){\line(1, 0){2}}
    \put(79, 60){\line(1, 0){21}}
    \put(105, 60){\makebox(0, -3){\textcircled{\tiny $2$}}}
    \put(112, 52){\makebox(0, 0){\tiny $(2)$}}
    \put(107, 70){\makebox(0, 0)[b]{$\Gamma_{P, 1}$}}
    \put(110, 60){\line(1, 0){30}}
    \put(145, 60){\makebox(0, 0){\large$\oslash$}}
    \put(152, 52){\makebox(0, 0){\tiny $(1)$}}
    \put(147, 70){\makebox(0, 0)[b]{$l_2^Z$}}
    \put(105, 55){\line(0, -1){20}}
    \put(105, 30){\makebox(0, 0){\large$\oslash$}}
    \put(112, 22){\makebox(0, 0){\tiny $(1)$}}
    \put(90, 22){\makebox(0, 0)[b]{$l_3^Z$}}
    \put(40, 55){\line(0, -1){20}}
    \put(40, 30){\makebox(0, 0){\large$\oslash$}}
    \put(47, 22){\makebox(0, 0){\tiny $(1)$}}
    \put(15, 22){\makebox(0, 0)[b]{$l_1^Z$}}
    \end{picture}
\end{center}
This leads a contradiction to Corollary \ref{dP-basic_cor}. 
The remaining assertions follow from Lemmas \ref{ZS1} and \ref{ZS2}. 
\end{proof}

\begin{lemma}\label{XS4}
Assume that $E_X=l_1+l_2$ around $P$, where $P\in l_i$ is nonsingular, 
$\{P\}=|l_1\cap l_2|$, and $\mult_P(l_1\cap l_2)=2$. 
Set $k:=\mult_P\Delta_X$, $j_i:=\mult_P(\Delta_X\cap l_i)$ and assume that 
$j_1\geqslant j_2$. Then $k=j_1$, $j_2=2$, $|\Delta_Z|\subset\Gamma_{P, k}$,
$\deg(\Delta_Z\cap\Gamma_{P, k})=2$ and 
$E_Z=l_2^Z+\Gamma_{P, 1}+2\Gamma_{P, 2}+\cdots+2\Gamma_{P, k}+l_1^Z$ 
around over $P$. 
The weighted dual graph of $E_Z$ around over $P$ is the following: 
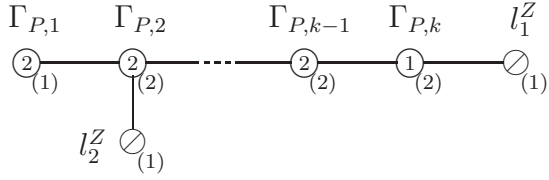
\begin{center}
    \begin{picture}(190, 70)(0, 20)
    \put(0, 60){\makebox(0, -3){\textcircled{\tiny $2$}}}
    \put(3, 70){\makebox(0, 0)[b]{$\Gamma_{P, 1}$}}
    \put(7, 52){\makebox(0, 0){\tiny $(1)$}}
    \put(5, 60){\line(1, 0){30}}
    \put(40, 60){\makebox(0, -3){\textcircled{\tiny $2$}}}
    \put(47, 52){\makebox(0, 0){\tiny $(2)$}}
    \put(43, 70){\makebox(0, 0)[b]{$\Gamma_{P, 2}$}}
    \put(45, 60){\line(1, 0){20}}
    \put(67, 60){\line(1, 0){2}}
    \put(71, 60){\line(1, 0){2}}
    \put(75, 60){\line(1, 0){2}}
    \put(79, 60){\line(1, 0){21}}
    \put(105, 60){\makebox(0, -3){\textcircled{\tiny $2$}}}
    \put(112, 52){\makebox(0, 0){\tiny $(2)$}}
    \put(107, 70){\makebox(0, 0)[b]{$\Gamma_{P, k-1}$}}
    \put(110, 60){\line(1, 0){30}}
    \put(145, 60){\makebox(0, -3){\textcircled{\tiny $1$}}}
    \put(152, 52){\makebox(0, 0){\tiny $(2)$}}
    \put(147, 70){\makebox(0, 0)[b]{$\Gamma_{P, k}$}}
    \put(150, 60){\line(1, 0){30}}
    \put(185, 60){\makebox(0, 0){\large$\oslash$}}
    \put(192, 52){\makebox(0, 0){\tiny $(1)$}}
    \put(188, 70){\makebox(0, 0)[b]{$l_1^Z$}}
    \put(40, 55){\line(0, -1){20}}
    \put(40, 30){\makebox(0, 0){\large$\oslash$}}
    \put(47, 22){\makebox(0, 0){\tiny $(1)$}}
    \put(25, 22){\makebox(0, 0)[b]{$l_2^Z$}}
    \end{picture}
\end{center}
\end{lemma}

\begin{proof}
The morphism $\psi\colon Z\to X$ factors though the monoidal transform 
$\pi\colon X_1\to X$ at $P$. Set $E_{X_1}:=E_Z^{X_1}$. Then 
$E_{X_1}=l_1^{X_1}+l_2^{X_1}+\Gamma_{P, 1}^{X_1}$ around over $P$. 
We note that any two curves 
intersect transversally at $P_1:=l_1^{X_1}\cap l_2^{X_1}$. 
If $\psi_1\colon Z\to X_1$ is isomorphic around $P_1$, then contradicts to 
Lemma \ref{ZS3}. Thus $\psi_1$ factors through the monoidal transform at $P_1$. 
Then we can apply the argument of Lemma \ref{XS3} and we can get the assertion. 
\end{proof}

\begin{lemma}\label{XS5}
Assume that $E_X=C$ around $P$, where $C$ is defined by $x^2=y^3$ such that 
$\{x$, $y\}$ is the regular parameter system of $P$. 
Then $\mult_P\Delta_X=1$, $|\Delta_Z|=\{Q\}$, 
$\mult_Q\Delta_Z=\mult_Q(\Delta_Z\cap C^Z)=\mult_Q(\Delta_Z\cap\Gamma_{P, 1})=2$
around over $P$, 
where $Q:=C^Z\cap\Gamma_{P, 1}$. The weighted dual graph of $E_M$ 
around over $P$ is the following: 
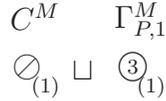
\begin{center}
    \begin{picture}(40, 45)(0, 40)
    \put(0, 60){\makebox(0, 0){\large$\oslash$}}
    \put(3, 74){\makebox(0, 0)[b]{$C^M$}}
    \put(7, 52){\makebox(0, 0){\tiny $(1)$}}
    \put(21, 55){\makebox(0, 0)[b]{$\sqcup$}}
    \put(40, 60){\makebox(0, -3){\textcircled{\tiny $3$}}}
    \put(47, 52){\makebox(0, 0){\tiny $(1)$}}
    \put(43, 70){\makebox(0, 0)[b]{$\Gamma_{P, 1}^M$}}
    \end{picture}
\end{center}
\end{lemma}

\begin{proof}
The morphism $\psi\colon Z\to X$ factors though the monoidal transform 
$\pi\colon X_1\to X$ at $P$. Set $E_{X_1}:=E_Z^{X_1}$. Then 
$E_{X_1}=C^{X_1}+\Gamma_{P, 1}^{X_1}$ and both components are nonsingular 
around over $P$. Moreover, $\{Q\}:=|C^{X_1}\cap\Gamma_{P, 1}^{X_1}|$ satisfies 
that $\mult_Q(C^{X_1}\cap\Gamma_{P, 1}^{X_1})=2$. 
If $Z\to X_1$ is not isomorphism around $Q$, then $-K_Z$ is not 
$\psi$-nef by Lemma \ref{XS4}, which leads to a contradiction. 
Thus $Z\to X_1$ is an isomorphism around $Q$. The remaining 
assertions follows from Lemma \ref{ZS4}. 
\end{proof}

\section{Special bottom tetrads}\label{cubic_section}

In this section, we consider the relationship between pseudo-median triplets 
$(Z, E_Z; \Delta_Z)$ ($L_Z$: the fundamental divisor) with $2K_Z+L_Z$ trivial and bottom 
tetrads $(X, E_X; \Delta_Z, \Delta_X)$ ($L_X$: the fundamental divisor) 
with $2K_X+L_X$ trivial. 
Since $-K_Z$ is nef and big, there exists a birational morphism $Z\to X=\pr^2$ 
unless $Z=\pr^1\times\pr^1$ or $\F_2$ by \cite[Corollary 3.6]{HW}. 
Moreover, for any birational morphism $Z\to X=\pr^2$, there exists a zero-dimensional 
subscheme $\Delta_X\subset X$ which satisfies the $(\nu1)$-condition and the 
morphism $Z\to X$ is the elimination of $\Delta_X$. 
By this way, we obtain a $3$-fundamental multiplet $(X, E_X; \Delta_Z, \Delta_X)$ 
of length two. 
The following lemmas show that we can replace the tetrad with a ``suitable " one. 

\begin{lemma}\label{reduction1_lem}
Let $(X=\pr^2, E_X; \Delta_Z, \Delta_X)$ be a 
$3$-fundamental multiplet of length two with $E_X=2l_1+l_2$, 
where $l_1$, $l_2$ are distinct lines. Set $P:=l_1\cap l_2$. 
Assume that one of the following holds: 
\begin{enumerate}
\renewcommand{\theenumi}{\arabic{enumi}}
\renewcommand{\labelenumi}{(\theenumi)}
\item\label{reduction1_lem1}
There exists a point $P_1\in\Delta_X\cap l_1\setminus\{P\}$ such that 
one of the following holds: 
\begin{enumerate}
\renewcommand{\theenumii}{\alph{enumii}}
\renewcommand{\labelenumii}{(\theenumii)}
\item\label{reduction1_lem11}
$\mult_{P_1}\Delta_X>2$.
\item\label{reduction1_lem12}
$\deg\Delta_X\geqslant 5$, $\mult_{P_1}(\Delta_X\cap l_1)=1$ and 
$\deg(\Delta_X\cap l_1)\geqslant 2$. 
\end{enumerate}
\item\label{reduction1_lem2}
$\#|\Delta_X\cap l_1\setminus\{P\}|\geqslant 2$. 
\item\label{reduction1_lem3}
$\#|\Delta_X\cap l_1\setminus\{P\}|=1$ and 
$\mult_P\Delta_X>\mult_P(\Delta_X\cap l_2)$. 
\item\label{reduction1_lem4}
$\deg\Delta_X=4$ and $\deg(\Delta_X\cap l_2)=2$. 
\end{enumerate}
Then there exists a $3$-fundamental multiplet 
$(X'=\pr^2, E_{X'}; \Delta_Z, \Delta_{X'})$ of length two 
such that both $(X, E_X; \Delta_Z, \Delta_X)$ and 
$(X', E_{X'}; \Delta_Z, \Delta_{X'})$ induces the same pseudo-median triplet, 
and either holds: 
\begin{enumerate}
\renewcommand{\theenumi}{\roman{enumi}}
\renewcommand{\labelenumi}{(\theenumi)}
\item\label{reduction1_lem01}
$E_{X'}$ is reduced, or 
\item\label{reduction1_lem02}
$E_{X'}=2l'_1+l'_2$ such that $l'_1$, $l'_2$ are distinct lines and none of the conditions 
\eqref{reduction1_lem1}, \eqref{reduction1_lem2}, \eqref{reduction1_lem3}, 
\eqref{reduction1_lem4} hold. 
\end{enumerate}
\end{lemma}

\begin{proof}
Set $d_i^X:=\deg(\Delta_X\cap l_i)$, $d_i^Z:=\deg(\Delta_Z\cap l_i^Z)$ for $i=1$, $2$, 
and $b:=\mult_P\Delta_X$. 
Note that $2d_i^X+d_i^Z=6$ and 
$d_i^X+d_i^Z=1-((l_i^M)^2)$.
Thus $(d_1^X, d_1^Z)=(3, 0)$, $(2, 2)$, $(1, 4)$, $(0, 6)$, and 
$(d_2^X, d_2^Z)=(3, 0)$, $(2, 2)$. 
By Lemma \ref{XS2}, $\mult_P(\Delta_X\cap l_1)=1$ if $b\geqslant 1$. 
Let $(Z, E_Z; \Delta_Z)$ be the associated pseudo-median triplet and $L_Z$ be 
the fundamental divisor. We note that $E_Z\sim -K_Z$. 

\emph{Step 1:}
Assume that \eqref{reduction1_lem11}, 
\eqref{reduction1_lem2} or \eqref{reduction1_lem3}. 
We will show that we can replace with 
another tetrad such that the condition \eqref{reduction1_lem01} holds. 

\eqref{reduction1_lem11}
By Lemma \ref{ZS1}, $(\mult_{P_1}\Delta_X, \mult_{P_1}(\Delta_X\cap l_1))=(4, 2)$. 
Let $X_1\to X$ be the elimination of $\Delta_X$ around $P_1$. 
Then $\rho(X_1)=5$, $Z\to X$ factors through $X_1\to X$ and 
$E_Z^{X_1}=l_2^{X_1}+2l_1^{X_1}+2\Gamma_{P_1, 2}^{X_1}+\Gamma_{P_1, 1}^{X_1}+
\Gamma_{P_1, 3}^{X_1}$. Since $(l_1^{X_1})^2=-1$, $(\Gamma_{P_1, 2}^{X_1})^2=-2$ 
and $\rho(X_1)=5$, there exists a birational morphism 
$\psi'\colon Z\to X_1\to X'=\pr^2$ such that $\psi'_*(l_1^Z+\Gamma_{P_1, 2})=0$. 
Thus $E_{X'}:=\psi'_*E_Z=\psi'_*(l_2^Z+\Gamma_{P_1, 1}+\Gamma_{P_1, 3})$ is reduced. 

\eqref{reduction1_lem2}
Set $\{P_1,\dots,P_j\}=|\Delta_Z\cap l_1\setminus\{P\}|$ $(j\geqslant 2)$. 
Assume that $j\geqslant 3$. Then $(d_1^X, d_1^Z)=(3, 0)$, $j=3$ and $P\not\in\Delta_X$. 
Moreover, $(\mult_{P_i}\Delta_X, \mult_{P_i}(\Delta_X\cap l_i))=(2, 1)$ 
for any $1\leqslant i\leqslant 3$. This implies that $l_1^Z$ intersects with 
$\Gamma_{P_1, 1}^M$, $\Gamma_{P_2, 1}^M$, $\Gamma_{P_3, 1}^M$ and $l_2^M$, 
which leads to a contradiction. Thus $j=2$. 
Assume that $P\in\Delta_X$. Then $(d_1^X, d_1^Z)=(3, 0)$ and 
$(\mult_{P_i}\Delta_X, \mult_{P_i}(\Delta_X\cap l_i))=(2, 1)$ for $i=1$, $2$. 
Let $X_1\to X$ be the elimination of $\Delta_X$ around $P_1$, $P_2$. 
Then $\rho(X_1)=5$, $Z\to X$ factors through $X_1\to X$ and 
$E_Z^{X_1}=l_2^{X_1}+2l_1^{X_1}+\Gamma_{P_1, 1}^{X_1}+\Gamma_{P_2, 1}^{X_1}$. 
Since $(l_1^{X_1})^2=-1$, there exists a birational morphism 
$\psi'\colon Z\to X_1\to X'=\pr^2$ such that $\psi'_*l_1^Z=0$. 
Thus $E_{X'}:=\psi'_*E_Z=\psi'_*(l_2^Z+\Gamma_{P_1, 1}+\Gamma_{P_2, 1})$ is reduced. 
Assume that $P\not\in\Delta_X$. 
Let $X_1\to X$ be the composition of the elimination of $\Delta_X$ around $l_2$ 
and the monoidal transform at $P_1$, $P_2$. 
Then $\rho(X_1)\geqslant 4$, $Z\to X$ factors through $X_1\to X$ and 
$E_Z^{X_1}=l_2^{X_1}+2l_1^{X_1}+\Gamma_{P_1, 1}^{X_1}+\Gamma_{P_2, 1}^{X_1}$. 
Since $(l_1^{X_1})^2=-1$, there exists a birational morphism 
$\psi'\colon Z\to X_1\to X'=\pr^2$ such that $\psi'_*l_1^Z=0$. 
Thus $E_{X'}:=\psi'_*E_Z=\psi'_*(l_2^Z+\Gamma_{P_1, 1}+\Gamma_{P_2, 1})$ is reduced. 

\eqref{reduction1_lem3}
Set $\{P_1\}=|\Delta_Z\cap l_1\setminus\{P\}|$. 
By Lemma \ref{XS2}, $b=\mult_P(\Delta_X\cap l_2)+2\geqslant 3$. 
Let $X_1\to X$ be the composition of the elimination of $\Delta_X$ around $P$ 
and the monoidal transform at $P_1$. 
Then $\rho(X_1)=b+2$, $Z\to X$ factors through $X_1\to X$ and 
$E_Z^{X_1}=l_2^{X_1}+2l_1^{X_1}+2\Gamma_{P, 1}^{X_1}+\dots+
2\Gamma_{P, b-2}^{X_1}+\Gamma_{P, b-1}^{X_1}+\Gamma_{P_1, 1}^{X_1}$. 
Since $(l_1^{X_1})^2=-1$ and $(\Gamma_{P, i}^{X_1})^2=-2$ 
for $1\leqslant i\leqslant b-2$, there exists a birational morphism 
$\psi'\colon Z\to X_1\to X'=\pr^2$ such that 
$\psi'_*(l_1^Z+\Gamma_{P, 1}+\dots+\Gamma_{P, b-2})=0$. 
Thus $E_{X'}:=\psi'_*E_Z=\psi'_*(l_2^Z+\Gamma_{P_1, 1}+\Gamma_{P, b-1})$ is reduced. 

\emph{Step 2:}
We assume the case \eqref{reduction1_lem12}. 
We can assume that $\{P_1\}=|\Delta_X\cap l_1\setminus\{P\}|$, 
$d_1^X=2$ and $b=\mult_P(\Delta_X\cap l_2)$. 
Assume that $\mult_{P_1}\Delta_X=1$. 
Set $Q_1:=l_1^Z\cap\Gamma_{P_1, 1}$. Since $\mult_{Q_1}\Delta_Z=2$ and 
$\mult_{Q_1}(\Delta_Z\cap l_1^Z)=1$, we have 
$\deg\Delta_Z\geqslant\deg(\Delta_Z\cap l_1^Z)+\deg(\Delta_Z\cap\Gamma_{P, b})
+(2-1)=5$. However, $\deg\Delta_X+\deg\Delta_Z=(L_X\cdot E_X)/2=9$. This leads 
to a contradiction. Thus $\mult_{P_1}\Delta_X=2$, $((l_1^Z)^2)=-1$ and 
$((\Gamma_{P_1, 1})^2)=-2$. 
There exists a birational morphism $\chi\colon Z\to X_0$ such that 
$\rho(Z)-\rho(X_0)=2$ and $\chi(l_1^Z\cup\Gamma_{P_1,1})=\{R\}$. 
Moreover, there exists a birational morphism $\tau\colon X_0\to X'=\pr^2$. 
Set $\psi':=\tau\circ\chi$. Since $E_{X'}:=\psi'_*E_Z=\psi'_*(l_2^Z+2(\Gamma_{P, 1}
+\dots+\Gamma_{P, b}))$, unless $E_{X'}$ is reduced, we can write that 
$E_{X'}=2l_1'+l'_2$ with $l'_1$, $l'_2$ distinct lines, where $l'_1=\psi'_*\Gamma_{P, 1}$ 
and $l'_2=\psi'_*l_2^Z$. Indeed, $\psi'_*\Gamma_{P, 1}\neq 0$ since 
$((\chi_*\Gamma_{P, 1})^2)\geqslant 0$. 
Let $P'_1$ be the image of $R$. 
Since $E_Z\sim -K_Z$, $\tau$ is an isomorphism around $R$. 
Thus $\mult_{P'_1}\Delta_{X'}=\mult_{P'_1}(\Delta_{X'}\cap l'_1)=2$, where $\Delta_{X'}$ 
corresponds to the morphism $\psi'$. 
Moreover, $\deg\Delta_{X'}=\deg\Delta_X\geqslant 5$. 
Therefore, by combining with the argument in Step 1, we can get another tetrad 
which satisfies that none of the conditions \eqref{reduction1_lem1}, 
\eqref{reduction1_lem2}, \eqref{reduction1_lem3}, \eqref{reduction1_lem4} 
are satisfied and $\deg\Delta_{X'}\geqslant 5$. 

We assume the case \eqref{reduction1_lem4}. 
We can assume that $b=\mult_P(\Delta_X\cap l_2)$. 
If $\Delta_X\cap l_1\setminus\{P\}=\emptyset$, then $\Delta_X\subset l_2$. This 
implies that $\deg\Delta_X=2$, which leads to a contradiction. 
Thus we can assume that $\{P_1\}=|\Delta_X\cap l_1\setminus\{P\}|$ 
and $\mult_{P_1}\Delta_X=2$. 
Then we can write that $E_Z=\Gamma_{P_1, 1}+2D+l_2^Z$, where $D$ is an effective 
divisor on $Z$. Moreover, $\rho(Z)\geqslant 5$. 
There exists a birational morphism 
$\psi'\colon Z\to X_1\to X'=\pr^2$ such that 
$\psi'_*l_2^Z=0$. 
Unless $E_{X'}:=\psi'_*E_Z=\psi'_*(\Gamma_{P_1, 1}+2D)$ is not reduced, 
we can write that 
$E_{X'}=2l_1'+l'_2$ with $l'_1$, $l'_2$ distinct lines, where $l'_2=\psi'_*\Gamma_{P_1, 1}$. 
Note that $\deg(\Delta_{X'}\cap l'_2)=3$. By combining with the previous 
arguments, we can get another tetrad satisfying the conditions 
\eqref{reduction1_lem01} and \eqref{reduction1_lem02}. 
\end{proof}

\begin{lemma}\label{reduction2_lem}
Let $(X=\pr^2, E_X; \Delta_Z, \Delta_X)$ be a $3$-fundamental multiplet of length two 
with $E_X=l_1+l_2+l_3$, 
where $l_1$, $l_2$, $l_3$ are distinct lines. Assume that one of the following holds: 
\begin{enumerate}
\renewcommand{\theenumi}{\arabic{enumi}}
\renewcommand{\labelenumi}{(\theenumi)}
\item\label{reduction2_lem1}
$l_1\cap l_2\cap l_3\neq\emptyset$.
\item\label{reduction2_lem2}
$l_1\cap l_2\cap l_3=\emptyset$ and 
$\#|\Delta_X\cap((l_1\cap l_2)\cup(l_1\cap l_3)\cup(l_2\cap l_3))|\leqslant 1$. 
\end{enumerate}
Then there exists a $3$-fundamental multiplet 
$(X'=\pr^2, E_{X'}; \Delta_Z, \Delta_{X'})$ of length two 
such that both $(X, E_X; \Delta_Z, \Delta_X)$ and 
$(X', E_{X'}; \Delta_Z, \Delta_{X'})$ induces the same pseudo-median triplet, 
$E_{X'}$ is reduced and the number of the component of $E_{X'}$ is less than three.  
\end{lemma}

\begin{proof}
Set $d_i^X:=\deg(\Delta_X\cap l_i)$, $d_i^Z:=\deg(\Delta_Z\cap l_i^Z)$ for 
$1\leqslant i\leqslant 3$. Then, we have $(d_i^X, d_i^Z)=(2, 2)$ or $(3, 0)$. 

Assume the case \eqref{reduction2_lem1}. Set $P:=l_1\cap l_2\cap l_3$. 
By Lemma \ref{ZS3}, $P\in\Delta_X$. 
If $\deg(\Delta_X\cap l_i\setminus\{P\})=1$ for all $1\leqslant i\leqslant 3$, then 
$\mult_P\Delta_X=1$ and $(d_i^X, d_i^Z)=(2, 2)$ for all $1\leqslant i\leqslant 3$ 
by Lemma \ref{XS3}. However, this implies that 
$\deg(\Delta_Z\cap\Gamma_{P, 1})\geqslant 3$. This leads to a contradiction. 
Thus we can assume that $\deg(\Delta_X\cap l_1\setminus\{P\})=2$. 
Let $X_1\to X$ be the elimination of $\Delta_X\setminus\{P\}$. 
Then $\rho(X_1)\geqslant 5$, $Z\to X$ factors through $X_1\to X$ and 
$E_Z^{X_1}=l_1^{X_1}+l_2^{X_1}+l_3^{X_1}$. Since $((l_1^{X_1})^2)=-1$, 
there exists a birational morphism 
$\psi'\colon Z\to X_1\to X'=\pr^2$ such that $\psi'_*l_1^Z=0$. 
Thus $E_{X'}:=\psi'_*E_Z=\psi'_*(l_2^Z+l_3^Z)$. 

Assume the case \eqref{reduction2_lem2}. Set $P_{ij}:=l_i\cap l_j$
for $1\leqslant i<j\leqslant 3$. We can assume that $P_{12}$, $P_{13}\not\in\Delta_X$. 
By Lemmas \ref{ZS2} and \ref{XS2}, $d_i^X=2$
for any $1\leqslant i\leqslant 3$. 
Let $X_1\to X$ be the elimination of $\Delta_X\setminus\{P_{23}\}$. 
Then $\rho(X_1)\geqslant 5$, $Z\to X$ factors through $X_1\to X$ and 
$E_Z^{X_1}=l_1^{X_1}+l_2^{X_1}+l_3^{X_1}$. Since $((l_1^{X_1})^2)=-1$, 
there exists a birational morphism 
$\psi'\colon Z\to X_1\to X'=\pr^2$ such that $\psi'_*l_1^Z=0$. 
Thus $E_{X'}:=\psi'_*E_Z=\psi'_*(l_2^Z+l_3^Z)$. 
\end{proof}

\begin{lemma}\label{reduction3_lem}
Let $(X=\pr^2, E_X; \Delta_Z, \Delta_X)$ be a $3$-fundamental multiplet of length two  
with $E_X=C+l$, 
where $C$ is a nonsingular conic and $l$ is a line. 
Assume that one of the following holds: 
\begin{enumerate}
\renewcommand{\theenumi}{\arabic{enumi}}
\renewcommand{\labelenumi}{(\theenumi)}
\item\label{reduction3_lem1}
$\Delta_X\cap C\cap l=\emptyset$.
\item\label{reduction3_lem2}
$|C\cap l|=\{P\}$, $\deg(\Delta_X\setminus\{P\})\geqslant 4$ 
and $\Delta_X\cap l\setminus\{P\}=\emptyset$.
\end{enumerate}
Then there exists a $3$-fundamental multiplet 
$(X'=\pr^2, E_{X'}; \Delta_Z, \Delta_{X'})$ of length two 
such that both $(X, E_X; \Delta_Z, \Delta_X)$ and 
$(X', E_{X'}; \Delta_Z, \Delta_{X'})$ induces the same $3$-fundamental triplet, 
$E_{X'}$ is the union of a nonsingular conic and a line and neither the conditions 
\eqref{reduction3_lem1} nor \eqref{reduction3_lem2} holds 
unless $E_{X'}$ is reduced and irreducible. 
\end{lemma}

\begin{proof}
Assume the case \eqref{reduction3_lem1}. 
Then $E_Z=C^Z+l^Z$. By Lemmas \ref{ZS2} and \ref{XS4}, 
$\deg(\Delta_Z\cap C^Z)=\deg(\Delta_Z\cap l^Z)=2$. Thus $((l^Z)^2)=-1$ and 
$\rho(Z)=8$. 
Then there exists a birational morphism 
$\psi'\colon Z\to X'=\pr^2$ such that $E_{X'}:=\psi'_*E_Z=\psi'_*C^Z$ is reduced 
and irreducible. 

Assume the case \eqref{reduction3_lem2}. 
We can assume that $P\in\Delta_X$. By the assumption, 
$\deg(\Delta_X\cap C\setminus\{P\})\geqslant 4$. 
There exists a line $l_0\subset X$ such that $P\not\in l_0$ and 
$\deg(\Delta_X\cap l_0)=2$ since $\Delta_X\setminus\{P\}\subset C$. 
Let $X_1\to X$ be the elimination of $\Delta_X\setminus\{P\}$. 
Then $\rho(X_1)\geqslant 5$, $Z\to X$ factors through $X_1\to X$ and 
$E_Z^{X_1}=C^{X_1}+l^{X_1}$. We note that there exists a $(-1)$-curve 
$\Gamma$ on $X_1$ over $X$ such that $C^{X_1}\cap\Gamma\neq\emptyset$ 
and $l_0^{X_1}\cap\Gamma=\emptyset$ since 
$\deg(\Delta_X\cap C\setminus\{P\})\geqslant 4$. 
There exists a birational morphism 
$\psi'\colon Z\to X_1\to X'=\pr^2$ such that 
the strict transforms of $l_0$ and $\Gamma$ map $\psi'$ to points. 
In this case, $E_{X'}=\psi'_*(C^Z+l^Z)$. We can assume that $E_{X'}=C'+l'$, where 
$C'$ is a nonsingular conic and $l'$ is a line. 
By construction, $|C'\cap l'|=\{P'\}$, 
$\Delta_{X'}\cap C'\setminus\{P'\}\neq\emptyset$ and 
$\Delta_{X'}\cap l'\setminus\{P'\}\neq\emptyset$. 
Thus the assertion holds. 
\end{proof}

As an immediate consequence of Lemmas \ref{reduction1_lem}, 
\ref{reduction2_lem} and \ref{reduction3_lem}, we have the following theorem. 

\begin{thm}\label{corresp_triv_thm}
Let $(Z, E_Z; \Delta_Z)$ be a pseudo-median triplet such that $2K_Z+L_Z$ is trivial, 
where $L_Z$ is the fundamental divisor. 
Then there exists a projective birational morphism $\psi\colon Z\to X$ onto 
a nonsingular surface and a zero-dimensional subscheme $\Delta_X\subset X$ 
satisfying the $(\nu1)$-condition such that the morphism $\psi$ is the elimination 
of $\Delta_X$, the tetrad $(X, E_X; \Delta_Z, \Delta_X)$ is a bottom tetrad and 
the associated pseudo-median triplet is equal to $(Z, E_Z; \Delta_Z)$, where 
$E_X:=\psi_*E_Z$. Moreover, the divisor $\psi_*L_Z$ is the fundamental divisor of 
$(X, E_X; \Delta_Z, \Delta_X)$. 
\end{thm}

\section{Classification of median triplets}\label{ZC_section}

We classify median triplets $(Z, E_Z; \Delta_Z)$. 

\begin{thm}\label{trip_thm}
The median triplets $(Z, E_Z; \Delta_Z)$ are classified by the types defined as follows: 

\smallskip

The case $Z=\pr^2:$

\begin{description}
\item[{$[$4$]_0$}]
$E_Z=2C$ $(C$: nonsingular conic$)$, $\deg\Delta_Z=10$ and 
$\Delta_Z\subset C$. 

\smallskip

\item[{$[$4$]_2$({\bi c,d})} $((c,d)=(0,0), (1,1),\dots,(5,1))$]
$E_Z=2l_1+2l_2$ $(l_1$, $l_2$: distinct lines$)$, $\deg\Delta_Z=10$, 
$\deg(\Delta_Z\cap l_1)=\deg(\Delta_Z\cap l_2)=5$, $\mult_Q(\Delta_Z\cap l_1)=c$, 
$\mult_Q(\Delta_Z\cap l_2)=d$ and $\mult_Q\Delta_Z=c+d$, 
where $Q=l_1\cap l_2$. 

\smallskip

\item[{$[$5$]_K$}]
$E_Z=2C+l$ $(C$: nonsingular conic, $l$: line$)$, $|C\cap l|=\{Q\}$, 
$\deg\Delta_Z=10$, $\deg(\Delta_Z\cap C)=8$, 
$\mult_Q\Delta_Z=\mult_Q(\Delta_Z\cap l)=4$ and 
$\mult_Q(\Delta_Z\cap C)=2$.  

\smallskip

\item[{$[$5$]_A$}]
$E_Z=2C+l$ $(C$: nonsingular conic, $l$: line$)$, $|C\cap l|=\{Q_1, Q_2\}$, 
$\deg\Delta_Z=10$, $\deg(\Delta_Z\cap C)=8$ and 
$\mult_{Q_i}\Delta_Z=\mult_{Q_i}(\Delta_Z\cap l)=2$ for $i=1$, $2$. 

\smallskip

\item[{$[$5$]_3$({\bi c,d})} $((c,d)=(0,0), (1,1), (2,1), (3,1))$]
$E_Z=2l_1+2l_2+l_3$ $(l_1$, $l_2$, $l_3$: distinct lines$)$, 
$l_1\cap l_2\cap l_3=\emptyset$, $\deg\Delta_Z=10$, 
$\deg(\Delta_Z\cap l_i)=4$, $\mult_{Q_{i3}}\Delta_Z=\mult_{Q_{i3}}(\Delta_Z\cap l_3)=2$ 
for $i=1$, $2$, $\mult_{Q_{12}}(\Delta_Z\cap l_1)=c$, 
$\mult_{Q_{12}}(\Delta_Z\cap l_2)=d$ and $\mult_{Q_{12}}\Delta_Z=c+d$, 
where $Q_{ij}=l_i\cap l_j$ for $1\leqslant i<j\leqslant 3$. 

\smallskip

\item[{$[$5$]_4$}]
$E_Z=2l_1+l_2+l_3+l_4$ $(l_1$,\dots, $l_4$: distinct lines$)$, 
$Q_{ij}$ are distinct for $1\leqslant i<j\leqslant 4$, 
$\deg\Delta_Z=10$, $\deg(\Delta_Z\cap l_1)=4$, 
$\mult_{Q_{ij}}\Delta_Z=1$ for $2\leqslant i<j\leqslant 4$ and 
$\mult_{Q_{1j}}\Delta_Z=\mult_{Q_{1j}}(\Delta_Z\cap l_j)=2$ 
for $2\leqslant j\leqslant 4$, 
where $Q_{ij}=l_i\cap l_j$ for $1\leqslant i<j\leqslant 4$. 

\smallskip

\item[{$[$5$]_5$}]
$E_Z=l_1+l_2+l_3+l_4+l_5$ $(l_1,\dots,l_5$: distinct lines$)$, 
$Q_{ij}$ are distinct for $1\leqslant i<j\leqslant 5$, 
$\deg\Delta_Z=10$ and $\mult_{Q_{ij}}\Delta_Z=1$ for $1\leqslant i<j\leqslant 5$, 
where $Q_{ij}=l_i\cap l_j$ for $1\leqslant i<j\leqslant 5$. 
\end{description}

\smallskip

The case $Z=\pr^1\times\pr^1:$

\begin{description}
\item[{$[$0;3,3$]_D$}]
$E_Z=2C+\sigma+l$ $(C\sim\sigma+l$ nonsingular$)$, 
$C\cap\sigma\cap l=\emptyset$, 
$\deg\Delta_Z=9$, $\deg(\Delta_Z\cap C)=6$, 
$\mult_Q\Delta_Z=1$, 
$\mult_{Q_\sigma}\Delta_Z=\mult_{Q_\sigma}(\Delta_Z\cap\sigma)=2$ and 
$\mult_{Q_l}\Delta_Z=\mult_{Q_l}(\Delta_Z\cap l)=2$, 
where $Q=\sigma\cap l$, $Q_\sigma=C\cap l$ and $Q_l=C\cap l$. 

\smallskip

\item[{$[$0;3,3$]_{22}$({\bi c,d})} $((c,d)=(0,0), (1,1), (2,1))$]
$E_Z=2\sigma_1+\sigma_2+2l_1+l_2$
$(\sigma_1$, $\sigma_2$: distinct minimal sections, $l_1$, $l_2$: distinct fibers$)$, 
$\deg\Delta_Z=9$, $\deg(\Delta_Z\cap\sigma_1)=\deg(\Delta_Z\cap l_1)=3$,
$\mult_{Q_{11}}(\Delta_Z\cap\sigma_1)=c$, 
$\mult_{Q_{11}}(\Delta_Z\cap l_1)=d$, $\mult_{Q_{11}}\Delta_Z=c+d$, 
$\mult_{Q_{12}}\Delta_Z=\mult_{Q_{12}}(\Delta_Z\cap l_2)=2$, 
$\mult_{Q_{21}}\Delta_Z=\mult_{Q_{21}}(\Delta_Z\cap\sigma_2)=2$ and 
$\mult_{Q_{22}}\Delta_Z=1$, 
where $Q_{ij}=\sigma_i\cap l_j$ for $1\leqslant i, j\leqslant 2$.

\smallskip

\item[{$[$0;3,3$]_{23}$}]
$E_Z=2\sigma_1+\sigma_2+l_1+l_2+l_3$ 
$(\sigma_1$, $\sigma_2$: distinct minimal sections, $l_1$, $l_2$, $l_3$: distinct 
fibers$)$, $\deg\Delta_Z=9$, $\mult_{Q_{1j}}\Delta_Z=\mult_{Q_{1j}}(\Delta_Z\cap l_j)=2$ 
and $\mult_{Q_{2j}}\Delta_Z=1$ for $1\leqslant j\leqslant 3$, where 
$Q_{ij}=\sigma_i\cap l_j$ for $1\leqslant i\leqslant 2$ and $1\leqslant j\leqslant 3$. 

\smallskip

\item[{$[$0;3,3$]_{33}$}]
$E_Z=\sigma_1+\sigma_2+\sigma_3+l_1+l_2+l_3$ 
$(\sigma_1$, $\sigma_2$, $\sigma_3$: distinct minimal sections, 
$l_1$, $l_2$, $l_3$: distinct fibers$)$, 
$\deg\Delta_Z=9$ and $\mult_{Q_{ij}}\Delta_Z=1$, 
where $Q_{ij}=\sigma_i\cap l_j$ for $1\leqslant i, j\leqslant 3$. 
\end{description}

\smallskip

The case $Z=\F_1:$

\begin{description}
\item[{$[$1;3,4$]_0$}]
$E_Z=2C+\sigma$ $(C\sim\sigma+2l$ nonsingular$)$, 
$\deg\Delta_Z=9$, $\deg(\Delta_Z\cap C)=8$ and 
$\mult_Q\Delta_Z=\mult_Q(\Delta_Z\cap\sigma)=2$, where $Q=C\cap\sigma$. 

\smallskip

\item[{$[$1;3,4$]_1$({\bi c,d})} $((c,d)=(0,0), (1,1),\dots,(5,1), (1,2))$]
$E_Z=2\sigma_\infty+\sigma+2l$, 
$\deg\Delta_Z=9$, $\deg(\Delta_Z\cap\sigma)=\mult_Q\Delta_Z=2$, 
$\deg(\Delta_Z\cap l)=3$, $\deg(\Delta_Z\cap\sigma_\infty)=5$, 
$\mult_{Q_\infty}(\Delta_Z\cap\sigma_\infty)=c$, 
$\mult_{Q_\infty}(\Delta_Z\cap l)=d$ and 
$\mult_{Q_\infty}\Delta_Z=c+d$, where $Q=\sigma\cap l$ and 
$Q_\infty=\sigma_\infty\cap l$. 

\smallskip

\item[{$[$1;3,4$]_2$}]
$E_Z=2\sigma_\infty+\sigma+l_1+l_2$ $(l_1$, $l_2$: distinct fibers$)$, 
$\deg\Delta_Z=9$, $\deg(\Delta_Z\cap\sigma_\infty)=5$, 
$\mult_{Q_1}\Delta_Z=\mult_{Q_2}\Delta_Z=1$, 
$\mult_{Q_{\infty 1}}\Delta_Z=\mult_{Q_{\infty 1}}(\Delta_Z\cap l_1)
=\mult_{Q_{\infty 2}}\Delta_Z=\mult_{Q_{\infty 2}}(\Delta_Z\cap l_2)=2$, 
where $Q_i=\sigma\cap l_i$ and $Q_{\infty i}=\sigma_\infty\cap l_i$ 
for $i=1$, $2$. 

\smallskip

\item[{$[$1;4,4$]$}]
$E_Z=2C$ $(C\sim 2\sigma+2l$ nonsingular$)$, 
$\deg\Delta_Z=10$ and $\Delta_Z\subset C$. 

\smallskip

\item[{$[$1;4,5$]_K$({\bi c})} $(3\leqslant c\leqslant 9)$]
$E_Z=2C+l$ $(C\sim 2\sigma+2l$ nonsingular, $C\cap l =\{Q\})$, 
$\deg\Delta_Z=9$, $\deg(\Delta_Z\cap C)=8$, $\mult_Q\Delta_Z=c$, 
$\mult_Q(\Delta_Z\cap C)=c-1$ and $\mult_Q(\Delta_Z\cap l)=2$. 

\smallskip

\item[{$[$1;4,5$]_A$}]
$E_Z=2C+l$ $(C\sim 2\sigma+2l$ nonsingular, $C\cap l =\{Q_1, Q_2\})$, 
$\deg\Delta_Z=9$, $\deg(\Delta_Z\cap C)=8$, 
$|\Delta_Z|\cap l=\{Q_1\}$ and $\mult_{Q_1}\Delta_Z=\mult_{Q_1}(\Delta_Z\cap l)=2$. 
\end{description}

\smallskip

The case $Z=\F_2:$

\begin{description}
\item[{$[$2;3,5$]_1$}]
$E_Z=2\sigma_\infty+\sigma+l$, 
$\deg\Delta_Z=9$, $\deg(\Delta_Z\cap\sigma_\infty)=7$, 
$\mult_Q\Delta_Z=1$ and $\mult_{Q_\infty}\Delta_Z=\mult_{Q_\infty}(\Delta_Z\cap l)=2$, 
where $Q=\sigma\cap l$ and $Q_\infty=\sigma_\infty\cap l$. 

\smallskip

\item[{$[$2;3,6$]_0$}]
$E_Z=2C+\sigma$ $(C\sim\sigma+3l$ nonsingular$)$, 
$\deg\Delta_Z=9$, $\deg(\Delta_Z\cap C)=9$ and $\Delta_Z\cap\sigma=\emptyset$. 

\smallskip

\item[{$[$2;3,6$]_1$({\bi c,d})} $((c, d)=(0, 0), (1, 1),\dots,(6, 1), (2, 1), (3, 1))$]
$E_Z=2\sigma_\infty+\sigma+2l$, 
$\deg\Delta_Z=9$, $\Delta_Z\cap\sigma=\emptyset$, 
$\deg(\Delta_Z\cap\sigma_\infty)=6$, $\deg(\Delta_Z\cap l)=3$, 
$\mult_Q(\Delta_Z\cap\sigma_\infty)=c$, $\mult_Q(\Delta_Z\cap l)=d$ 
and $\mult_Q\Delta_Z=c+d$, where $Q=\sigma_\infty\cap l$. 
\end{description}

\smallskip

The case $Z=\F_3:$

\begin{description}
\item[{$[$3;3,6$]$}]
$E_Z=2\sigma_\infty+\sigma$, $\deg\Delta_Z=9$ and $\Delta_Z\subset\sigma_\infty$. 

\smallskip

\item[{$[$3;4,9$]_A$}]
$E_Z=2C+2\sigma+l$ $(C\sim\sigma+4l$ nonsingular, 
$\sigma\cap C\cap l=\emptyset)$, $\deg\Delta_Z=9$, 
$\Delta_Z\cap\sigma=\emptyset$, 
$\deg(\Delta_Z\cap C)=8$ and $\mult_Q\Delta_Z=\mult_Q(\Delta_Z\cap l)=2$, 
where $Q=C\cap l$. 

\smallskip

\item[{$[$3;4,9$]_B$}]
$E_Z=\sigma_{\infty, 1}+\sigma_{\infty, 2}+\sigma_{\infty, 3}+\sigma$ 
$(\sigma_{\infty, 1}$, $\sigma_{\infty, 2}$, $\sigma_{\infty, 3}$: distinct 
sections at infinity$)$, 
$\sigma_{\infty, 1}\cap\sigma_{\infty, 2}\cap\sigma_{\infty, 3}=\emptyset$, 
$\deg\Delta_Z=9$ such that $\Delta_Z$ is the disjoint union of 
$\sigma_{\infty, 1}\cap\sigma_{\infty, 2}$, $\sigma_{\infty, 1}\cap\sigma_{\infty, 3}$ 
and $\sigma_{\infty, 2}\cap\sigma_{\infty, 3}$. 

\smallskip

\item[{$[$3;4,9$]_C$({\bi c,d})} $((c,d)=(0,0), (1,1),\dots,(5,1), (1,2))$]
$E_Z=2\sigma_\infty+2\sigma+2l_1+l_2$ $(l_1$, $l_2$: distinct fibers$)$, 
$\deg\Delta_Z=9$, $\Delta_Z\cap\sigma=\emptyset$, 
$\deg(\Delta_Z\cap\sigma_\infty)=6$, $\deg(\Delta_Z\cap l_1)=2$, 
$\mult_{Q_1}(\Delta_Z\cap\sigma_\infty)=c$, $\mult_{Q_1}(\Delta_Z\cap l_1)=d$, 
$\mult_{Q_1}\Delta_Z=c+d$ and $\mult_{Q_2}\Delta_Z=\mult_{Q_2}(\Delta_Z\cap l_2)=2$, 
where $Q_i=\sigma_\infty\cap l_i$ for $i=1$, $2$. 
 
\smallskip

\item[{$[$3;4,9$]_D$}]
$E_Z=2\sigma_\infty+2\sigma+l_1+l_2+l_3$ $(l_1$, $l_2$, $l_3$: distinct fibers$)$, 
$\deg\Delta_Z=9$, $\deg(\Delta_Z\cap\sigma_\infty)=6$ and 
$\mult_{Q_i}\Delta_Z=\mult_{Q_i}(\Delta_Z\cap l_i)=2$ for $1\leqslant i\leqslant 3$, 
where $Q_i=\sigma_\infty\cap l_i$ for $1\leqslant i\leqslant 3$. 

\smallskip

\item[{$[$3;4,9$]_E$}]
$E_Z=\sigma_{\infty, 1}+\sigma_{\infty, 2}+2\sigma+2l_1+l_2$ 
$(\sigma_{\infty, 1}$, $\sigma_{\infty, 2}$: distinct sections at infinity, 
$l_1$, $l_2$: distinct fibers$)$, 
$\sigma_{\infty, 1}\cap\sigma_{\infty, 2}\cap(l_1\cup l_2)=\emptyset$, 
$\deg\Delta_Z=9$, 
$\mult_{Q_{i1}}\Delta_Z=\mult_{Q_{i1}}(\Delta_Z\cap\sigma_{\infty, i})=2$, 
$\mult_{Q_{i2}}\Delta_Z=1$ for $i=1$, $2$, and 
$\Delta_Z\setminus\{Q_{11}, Q_{12}, Q_{21}, Q_{22}\}
=\sigma_{\infty, 1}\cap\sigma_{\infty, 2}$, where 
$Q_{ij}=\sigma_{\infty, i}\cap l_j$ for $1\leqslant i, j\leqslant 2$. 

\smallskip

\item[{$[$3;4,9$]_F$}]
$E_Z=\sigma_{\infty, 1}+\sigma_{\infty, 2}+2\sigma+l_1+l_2+l_3$ 
$(\sigma_{\infty, 1}$, $\sigma_{\infty, 2}$: distinct sections at infinity, 
$l_1$, $l_2$, $l_3$: distinct fibers$)$, 
$\sigma_{\infty, 1}\cap\sigma_{\infty, 2}\cap(l_1\cup l_2\cup l_3)=\emptyset$, 
$\deg\Delta_Z=9$, 
$\mult_{Q_{ij}}\Delta_Z=1$ for $1\leqslant i\leqslant 2$, $1\leqslant j\leqslant 3$, and 
$\Delta_Z\setminus\{Q_{ij}\}_{ij}=\sigma_{\infty, 1}\cap\sigma_{\infty, 2}$, 
where $Q_{ij}=\sigma_{\infty, i}\cap l_j$ for 
$1\leqslant i\leqslant 2$, $1\leqslant j\leqslant 3$. 
\end{description}

\smallskip

The case $Z=\F_4:$

\begin{description}
\item[{$[$4;4,10$]_0$}]
$E_Z=2C+2\sigma$ $(C\sim\sigma+5l$ nonsingular$)$, 
$\deg\Delta_Z=10$, $\Delta_Z\cap\sigma=\emptyset$ and $\Delta_Z\subset C$. 

\smallskip

\item[{$[$4;4,10$]_1$({\bi c,d})} $((c, d)=(0, 0), (1, 1),\dots,(1, 8), (2, 1))$]
$E_Z=2\sigma+2\sigma_\infty+2l$, $\deg\Delta_Z=10$, 
$\Delta_Z\cap\sigma=\emptyset$, $\deg(\Delta_Z\cap\sigma_\infty)=8$, 
$\deg(\Delta_Z\cap l)=2$, $\mult_Q(\Delta_Z\cap\sigma_\infty)=c$, 
$\mult_Q(\Delta_Z\cap l)=d$ and $\mult_Q\Delta_Z=c+d$, where 
$Q=\sigma_\infty\cap l$. 

\smallskip

\item[{$[$4;4,10$]_2$}]
$E_Z=2\sigma_\infty+2\sigma+l_1+l_2$ $(l_1$, $l_2$: distinct fibers$)$, 
$\deg\Delta_Z=10$, $\deg(\Delta_Z\cap\sigma_\infty)=8$ and 
$\mult_{Q_i}\Delta_Z=\mult_{Q_i}(\Delta_Z\cap l_i)=2$, 
where $Q_i=\sigma_\infty\cap l_i$ 
for $i=1$, $2$. 
\end{description}

\smallskip

The case $Z=\F_5:$

\begin{description}
\item[{$[$5;4,11$]_1$}]
$E_Z=2\sigma_\infty+2\sigma+l$, $\deg\Delta_Z=11$, 
$\deg(\Delta_Z\cap\sigma_\infty)=10$ and 
$\mult_Q\Delta_Z=\mult_Q(\Delta_Z\cap l)=2$, where $Q=\sigma_\infty\cap l$. 
\end{description}

\smallskip

The case $Z=\F_6:$

\begin{description}
\item[{$[$6;4,12$]_0$}]
$E_Z=2\sigma_\infty+2\sigma$, $\deg\Delta_Z=12$ and $\Delta_Z\subset\sigma_\infty$. 
\end{description}
\end{thm}

We start to prove Theorem \ref{trip_thm}. 
Any of the triplet in Theorem \ref{trip_thm} is a median triplet 
by Proposition \ref{converse_prop}. 
We see the converse. 
Let $(Z, E_Z; \Delta_Z)$ be a 
median triplet, $L_Z$ be the fundamental divisor, $\phi\colon M\to Z$ be 
the elimination of $\Delta_Z$, $E_M:=(E_Z)_M^{\Delta_Z, 2}$ and $k_Z:=\deg\Delta_Z$. 
By Lemma \ref{fund_big_lem}, 
$Z=\pr^2$ or $\F_n$ and $(2K_Z+L_Z\cdot l)<0$.

\subsection{The case $Z=\pr^2$}\label{ZP_section}

We consider the case $Z=\pr^2$. 
Set $L_Z\sim hl$ and $E_Z\sim el$. Then $e=9-h$ and $4\leqslant h\leqslant 5$ hold 
since $E_Z\sim -3K_Z-L_Z$, $(K_Z+L_Z\cdot L_Z)>0$ and $2K_Z+L_Z$ is not nef. 
Thus $(h, e)=(5, 4)$ or $(4, 5)$. Moreover, $k_Z=(L_Z\cdot E_Z)/2=10$.

\begin{claim}\label{P2_claim}
Any component $C\leqslant E_Z$ is either a nonsingular conic or a line. 
Moreover, $\coeff_CE_Z=2$ holds unless $C$ is a line and $h=4$.
\end{claim}

\begin{proof}
Set $m:=\deg C$. By Lemma \ref{mult_seq_lem}, 
$m^2-((C^M)^2)=(L_Z\cdot C)+2p_a(C)=m^2+(h-3)m+2$. 
Thus $-2-((C^M)^2)=(h-3)m$. Hence $((C^M)^2)\leqslant -4$ 
(this implies that $\coeff_CE_Z=2$) unless $(h, m)=(4, 1)$. 
Therefore $m\leqslant 2$ since $e\leqslant 5$.
\end{proof}

\subsubsection{The case $(h, e)=(5, 4)$.}\label{P2_54_section}

By Claim \ref{P2_claim}, 
we have either $E_Z=2C$ for a nonsingular conic $C$, or $E_Z=2l_1+2l_2$ for 
distinct lines $l_1$, $l_2$.

\textbf{The case $E_Z=2C$:}
In this case, $\deg(\Delta_Z\cap C)=10$. Thus $\Delta_Z\subset C$. 
This triplet is nothing but the type \textbf{$[$4$]_0$}.

\textbf{The case $E_Z=2l_1+2l_2$:}
We know that $\deg(\Delta_Z\cap l_i)=5$ for $i=1$, $2$.
Set $Q:=l_1\cap l_2$, $c:=\mult_Q(\Delta_Z\cap l_1)$ and 
$d:=\mult_Q(\Delta_Z\cap l_2)$. 
We may assume that $c\geqslant d$. By Lemma \ref{ZS2}, 
$\mult_Q\Delta_Z=c+d$. 
This triplet is nothing but the type \textbf{$[$4$]_2$({\bi c,d})}.

\subsubsection{The case $(h, e)=(4, 5)$.}\label{P2_45_section}

By Claim \ref{P2_claim}, any component of $E_Z$ is either a nonsingular conic or a line. 

{\noindent\textbf{The case $E_Z=2C+l$:}}

We consider the case $E_Z$ contains a nonsingular conic $C$. Then 
$E_Z=2C+l$, where $l$ is a line. 
We know that $\deg(\Delta_Z\cap C)=8$ and $\deg(\Delta_Z\cap l)=4$. 
We assume that $C$ contacts $l$ at one point $Q$. 
Note that $\mult_Q(\Delta_Z\cap l)=\deg(\Delta_Z\cap l)=4$. 
By Lemma \ref{ZS5}, we have 
$\mult_Q\Delta_Z=4$ and $\mult_Q(\Delta_Z\cap C)=2$. 
This triplet is nothing but the type \textbf{$[$5$]_K$}. 
We assume that $C$ and $l$ meet two points $Q_1$ and $Q_2$. 
By Lemma \ref{ZS2}, we have 
$\mult_{Q_i}\Delta_Z=\mult_{Q_i}(\Delta_Z\cap l)=2$ for $i=1$, $2$. 
This triplet is nothing but the type \textbf{$[$5$]_A$}. 

{\noindent\textbf{The case $E_Z=2l_1+2l_2+l_3$:}}

We consider the case 
$E_Z=2l_1+2l_2+l_3$, where $l_1$, $l_2$, $l_3$ are distinct lines. Set 
$Q_{ij}:=l_i\cap l_j$ for $1\leqslant i<j\leqslant 3$, 
$c:=\mult_{Q_{12}}(\Delta_Z\cap l_1)$ and $d:=\mult_{Q_{12}}(\Delta_Z\cap l_2)$. 
We may assume that $c\geqslant d$.
By Lemma \ref{fund_big_lem}, $Q_{ij}$ are distinct points. 
By Lemma \ref{ZS2}, $\mult_{Q_{i3}}\Delta_Z=\mult_{Q_{i3}}(\Delta_Z\cap l_3)=2$ 
for $i=1$, $2$. Moreover, $\mult_{Q_{12}}\Delta_Z=c+d$. 
This triplet is nothing but the type \textbf{$[$5$]_3$({\bi c,d})}. 

{\noindent\textbf{The case $E_Z=2l_1+l_2+l_3+l_4$:}}

We assume that $E_Z=2l_1+l_2+l_3+l_4$, where $l_1,\dots, l_4$ are distinct lines. Set 
$Q_{ij}:=l_i\cap l_j$ for $1\leqslant i<j\leqslant 4$. 
By Lemmas \ref{fund_big_lem} and \ref{ZS3}, 
$Q_{ij}$ are distinct for $1\leqslant i<j\leqslant 4$. 
By Lemma \ref{ZS2}, 
$\mult_{Q_{1j}}\Delta_Z=\mult_{Q_{1j}}(\Delta_Z\cap l_j)=2$ for $2\leqslant j\leqslant 4$ 
and $\mult_{Q_{ij}}\Delta_Z=1$ for $2\leqslant i<j\leqslant 4$. 
This triplet is nothing but the type \textbf{$[$5$]_4$}. 

{\noindent\textbf{The case $E_Z=l_1+l_2+l_3+l_4+l_5$:}}

We assume that $E_Z=l_1+\cdots+l_5$, where $l_1,\dots, l_5$ are distinct lines. Set 
$Q_{ij}:=l_i\cap l_j$ for $1\leqslant i<j\leqslant 5$. 
Assume that $Q_{12}=Q_{13}$. By Lemmas \ref{fund_big_lem} and \ref{ZS3}, 
we can assume that $Q_{12}=Q_{14}$ and $Q_{12}\neq Q_{15}$. 
Then we can assume that $\mult_{Q_{12}}(\Delta_Z\cap l_1)\leqslant 1$. 
Since $|\Delta_Z|\cap l_1\subset\{Q_{12}, Q_{15}\}$ and 
$\mult_{Q_{15}}(\Delta_Z\cap l_1)\leqslant 1$, 
we have $\deg(\Delta_Z\cap l_1)\leqslant 2$. 
This leads to a contradiction. 
Therefore $Q_{ij}$ are distinct for $1\leqslant i<j\leqslant 5$. 
We know that $\#\{Q_{ij}\}_{ij}=10$, $\deg\Delta_Z=10$, $|\Delta_Z|\subset\{Q_{ij}\}_{ij}$ 
and $\mult_{Q_{ij}}\Delta_Z\leqslant 1$. Thus 
$\mult_{Q_{ij}}\Delta_Z=1$ for $1\leqslant i<j\leqslant 5$. 
This triplet is nothing but the type \textbf{$[$5$]_5$}. 

\subsection{The case $Z=\F_n$ with $K_Z+L_Z$ big}\label{ZF_nsection}

We consider the case $Z=\F_n$ such that $K_Z+L_Z$ is big. 
Set $L_Z\sim h_0\sigma+hl$, $E_Z\sim e_0\sigma+el$ and $k_Z:=\deg\Delta_Z$. 
Then $e_0=6-h_0$ and $e=3(n+2)-h$ hold 
since $E_Z\sim -3K_Z-L_Z$ and $K_Z\sim-2\sigma-(n+2)l$.

\begin{claim}\label{Z_num_claim} 
We have $h_0=3$ $($hence $e_0=3$$)$, $k_Z=9$ and 
$\max\{2n+2, 3n\}\leqslant h\leqslant 2n+6$. In particular, $n\leqslant 6$. 
Furthermore, we have $3\leqslant h\leqslant 6$ if $n=0$, 
and $5\leqslant h\leqslant 8$ if $n=1$. 
\end{claim}

\begin{proof}
Since $K_Z+L_Z$ is nef and big and $(2K_Z+L_Z\cdot l)<0$, we have 
$h_0=3$ and $h\geqslant 2n+2$. 
Since $L_Z$ is nef, we have $h\geqslant 3n$.
Moreover, if $n=0$ then $h\geqslant 3$ since $K_Z+L_Z$ is big; 
if $n=1$ then $h\geqslant 5$ since $(2K_Z+L_Z\cdot\sigma)\geqslant 0$. 
We know that $E_Z\not\geqslant 3\sigma$. Thus $e=3(n+2)-h\geqslant n$. 
Finally, we have $k_Z=(L_Z\cdot E_Z)/2=9$. 
\end{proof}

\begin{claim}\label{Zb_sing_claim}
\begin{enumerate}
\renewcommand{\theenumi}{\arabic{enumi}}
\renewcommand{\labelenumi}{(\theenumi)}
\item\label{Zb_sing_claim1}
We have $(n, h)=(0, 3)$, 
$(1, 5)$, $(2, 6)$, $(2, 7)$, $(3, 9)$.
\item\label{Zb_sing_claim2}
Any irreducible 
component $C\leqslant E_Z$ apart from $\sigma$, $l$ is a section of $\F_n/\pr^1$ 
and $\coeff_CE_Z=2$. Moreover, either holds: 
\begin{enumerate}
\renewcommand{\theenumii}{\roman{enumii}}
\renewcommand{\labelenumii}{(\theenumii)}
\item\label{Zb_sing_claim21}
$C=\sigma_\infty$ with $n\geqslant 1$ and $(n, h)=(1, 5)$, $(2, 6)$, $(2, 7)$, $(3, 9)$. 
\item\label{Zb_sing_claim22}
$C\sim\sigma+(n+1)l$ and $(n, h)=(0, 3)$, $(1, 5)$, $(2, 6)$. 
\end{enumerate}
\end{enumerate}
\end{claim}

\begin{proof}
Assume that there exists an irreducible component $C\leqslant E_Z$ apart from 
$\sigma$, $l$. (If $n\geqslant 1$, then such a component always exists since 
$3\sigma\not\leqslant E_Z$.) Set $C\sim m\sigma+(nm+u)l$ with 
$1\leqslant m\leqslant 3$ and $u\geqslant 0$. 
If $n=0$, then we assume further that $u\geqslant 1$. Furthermore, 
if $(n, h)=(0, 3)$, then we can further assume that $u\geqslant m$. 
By Lemma \ref{mult_seq_lem}, 
$nm^2+2um-((C^M)^2)=(C^2)-((C^M)^2)=(L_Z\cdot C)+2p_a(C)=
nm^2+(2u+h-n-2)m+u+2$. Thus $-((C^M)^2)=(h-n-2)m+u+2\geqslant 4$. 
This implies that $\coeff_CE_Z=2$. Thus $m=1$ 
(i.e., $C$ is a section) since $2C\leqslant E_Z$.
We have $\deg(\Delta_Z\cap C)=(L_Z\cdot C)=h+3u$. 
Since $\deg(\Delta_Z\cap C)\leqslant k_Z=9$, we have $u\leqslant 3-h/3(\leqslant 2)$. 
In particular, $n\leqslant 3$ since $h\leqslant 9$. 
Since $\sigma+(n+6-h-2u)l\sim E_Z-2C\geqslant 0$, we have $n+6-h-2u\geqslant 0$. 
If $u=2$ then $n+2\geqslant h$, a contradiction. 
If $u=1$, then  $(n, h)=(0, 3)$, $(0, 4)$, $(1, 5)$ or $(2, 6)$. 
If $u=0$, then $(n, h)=(1, 5)$, $(1, 6)$, $(1, 7)$, $(2, 6)$, $(2, 7)$, $(2, 8)$ or $(3, 9)$. 

We assume that $n=0$. If $\sigma\leqslant E_Z$, then $((\sigma^M)^2)=-h$ 
since $\deg(\Delta_Z\cap\sigma)=h$. Thus $\coeff_\sigma E_Z=2$ unless $h=3$. 
From the above claim, we must have $h=3$ if $n=0$. 

We assume that $(n, h)=(1, 6)$, $(1, 7)$ or $(2, 8)$. By the above claim, 
$E_Z=\sigma+2\sigma_\infty$ if $(n, h)=(1, 7)$ or $(2, 8)$; 
$E_Z=\sigma+2\sigma_\infty+l$ if $(n, h)=(1, 6)$. 
However, by Lemmas \ref{ZS1} and \ref{ZS2}, $\deg(\Delta_Z\cap\sigma)\leqslant 1$. 
This contradicts to the fact 
$\deg(\Delta_Z\cap\sigma)=(L_Z\cdot\sigma)=h-3n\geqslant 2$. 
Therefore $(n, h)=(0, 3)$, $(1, 5)$, $(2, 6)$, $(2, 7)$ or $(3, 9)$. 
\end{proof}

\subsubsection{The case $(n, h)=(0, 3)$.}\label{Fb_0_3_section}

In this case, we know that $E_Z\sim3\sigma+3l$. 
Assume that there exists an irreducible component $C\leqslant E_Z$ such that 
$C\sim\sigma+l$. 
Then $E_Z=2C+\sigma+l$. Set $Q:=\sigma\cap l$. Assume that $Q\in C$. 
We can assume that $\mult_Q(\Delta_Z\cap l)=1$. However, by Lemmas \ref{ZS1} 
and \ref{ZS2}, $3=\deg(\Delta_Z\cap l)=\mult_Q(\Delta_Z\cap l)$. This is a 
contradiction. Thus $C\cap\sigma\cap l=\emptyset$. 
Set $Q_\sigma:=C\cap\sigma$ and $Q_l:=C\cap l$. Then 
$\mult_Q\Delta_Z=1$, $\mult_{Q_\sigma}\Delta_Z=\mult_{Q_\sigma}(\Delta_Z\cap\sigma)
=2$ and 
$\mult_{Q_l}\Delta_Z=\mult_{Q_l}(\Delta_Z\cap l)=2$ by Lemma \ref{ZS2}. 
This is nothing but the type \textbf{$[$0;3,3$]_D$}.
Assume that any irreducible component of $E_Z$ is either $\sigma$ or $l$. 
We consider the case $E_Z=2\sigma_1+\sigma_2+2l_1+l_2$
$(\sigma_1$, $\sigma_2$: distinct minimal sections, $l_1$, $l_2$: distinct fibers$)$. 
Set $c:=\mult_{Q_{11}}(\Delta_Z\cap\sigma_1)$ and 
$d:=\mult_{Q_{11}}(\Delta_Z\cap l_1)$, where $Q_{11}:=\sigma_1\cap l_1$. 
Then $\mult_{Q_{11}}\Delta_Z=c+d$. 
We may assume that $c\geqslant d$. This induces 
the type \textbf{$[$0;3,3$]_{22}$({\bi c,d})}.
If $E_Z=2\sigma_1+\sigma_2+l_1+l_2+l_3$
$(\sigma_1$, $\sigma_2$: distinct minimal sections, 
$l_1$, $l_2$, $l_3$: distinct fibers$)$, 
then this induces the type \textbf{$[$0;3,3$]_{23}$}.
If $E_Z=\sigma_1+\sigma_2+\sigma_3+l_1+l_2+l_3$
$(\sigma_1$, $\sigma_2$, $\sigma_3$: distinct minimal sections, 
$l_1$, $l_2$, $l_3$: distinct fibers$)$, 
then this induces the type \textbf{$[$0;3,3$]_{33}$}.

\subsubsection{The case $(n, h)=(1, 5)$.}\label{Fb_1_5_section}

In this case, we know that $E_Z\sim3\sigma+4l$. 
Assume that there exists an irreducible component $C\leqslant E_Z$ with 
$C\sim\sigma+2l$. 
Then $E_Z=2C+\sigma$ and $\deg(\Delta_Z\cap C)=8$. Set $Q:=C\cap\sigma$. 
By Lemma \ref{ZS2}, $\mult_Q\Delta_Z=\mult_Q(\Delta_Z\cap\sigma)=2$. 
This is nothing but the type \textbf{$[$1;3,4$]_0$}.
Assume that $E_Z=2\sigma_\infty+\sigma+2l$. Set $Q_\infty:=\sigma_\infty\cap l$, 
$c:=\mult_{Q_\infty}(\Delta_Z\cap\sigma_\infty)$ and 
$d:=\mult_{Q_\infty}(\Delta_Z\cap l)$. Then $\mult_{Q_\infty}\Delta_Z=c+d$. 
This induces the type \textbf{$[$1;3,4$]_1$({\bi c,d})}.
Assume that $E_Z=2\sigma_\infty+\sigma+l_1+l_2$
($l_1$, $l_2$: distinct fibers). 
This induces the type \textbf{$[$1;3,4$]_2$}.

\subsubsection{The case $(n, h)=(2, 6)$.}\label{Fb_2_6_section}

In this case, we know that $E_Z\sim3\sigma+6l$. 
Assume that there exists an irreducible component $C\leqslant E_Z$ such that 
$C\sim\sigma+3l$. 
Then $E_Z=2C+\sigma$ and $\deg(\Delta_Z\cap C)=8$. Set $Q:=C\cap\sigma$. 
By Lemma \ref{ZS2}, $\mult_Q\Delta_Z=\mult_Q(\Delta_Z\cap\sigma)=2$. 
This is nothing but the type \textbf{$[$2;3,6$]_0$}.
Assume that $E_Z=2\sigma_\infty+\sigma+l_1+l_2$
($l_1$, $l_2$: distinct fibers). 
Since $\Delta_Z\cap\sigma=\emptyset$, we have $|\Delta_Z|\cap l_1\subset\{Q_1\}$, 
where $Q_1:=\sigma_\infty\cap l_1$. By Lemma \ref{ZS2}, 
$\mult_{Q_1}(\Delta_Z\cap l_1)\leqslant 2$. However, $\deg(\Delta_Z\cap l_1)=3$, 
which leads to the contradiction. 
Assume that $E_Z=2\sigma_\infty+\sigma+2l$. 
Set $Q:=\sigma_\infty\cap l$, $c:=\mult_Q(\Delta_Z\cap\sigma_\infty)$ 
and $d:=\mult_Q(\Delta_Z\cap l)$. 
Then $\mult_Q\Delta_Z=c+d$. This induces the type \textbf{$[$2;3,6$]_1$({\bi c,d})}.

\subsubsection{The case $(n, h)=(2, 7)$.}\label{Fb_2_7_section}

In this case, we know that $E_Z=2\sigma_\infty+\sigma+l$ 
by Claim \ref{Zb_sing_claim}. 
This case induces the type \textbf{$[$2;3,5$]_1$}.

\subsubsection{The case $(n, h)=(3, 9)$.}\label{Fb_3_9_section}

In this case, we know that $E_Z=2\sigma_\infty+\sigma$ by Claim \ref{Zb_sing_claim}. 
This case induces the type \textbf{$[$3;3,6$]$}.

\subsection{The case $Z=\F_n$ with $K_Z+L_Z$ non-big}\label{ZF_n2section}

We consider the case $Z=\F_n$ such that $K_Z+L_Z$ is not big. 
Set $L_Z\sim h_0\sigma+hl$, $E_Z\sim e_0\sigma+el$ and $k_Z:=\deg\Delta_Z$. 
Then $e_0=6-h_0$ and $e=3(n+2)-h$ hold 
since $E_Z\sim -3K_Z-L_Z$. We remark that $n\geqslant 1$ 
by the condition $(\sF$\ref{fund_dfn6}$)$. 

\begin{claim}\label{Zs_num_claim}
We have $h_0=2$ $($hence $e_0=4$$)$ and 
$\max\{n+3, 2n\}\leqslant h\leqslant n+6$. 
$($In particular, $n\leqslant 6$.$)$ 
Moreover, $k_Z=h-n+6$. 
\end{claim}

\begin{proof}
Since $K_Z+L_Z$ is nef, nontrivial, non-big and $(2K_Z+L_Z\cdot l)<0$, 
$h_0=2$ and $h\geqslant n+3$ hold.
Since $L_Z$ is nef, we have $h\geqslant 2n$.
We know that $E_Z\not\geqslant 3\sigma$. Thus $e=3(n+2)-h\geqslant 2n$.  
Finally, we have $k_Z=(L_Z\cdot E_Z)/2=h-n+6$. 
\end{proof}

\begin{claim}\label{Zs_sing_claim}
\begin{enumerate}
\renewcommand{\theenumi}{\arabic{enumi}}
\renewcommand{\labelenumi}{(\theenumi)}
\item\label{Zs_sing_claim1}
The pair $(n, h)$ is one of $(1, 4)$, 
$(1, 5)$, $(3, 6)$, $(4, 8)$, $(5, 10)$ or $(6, 12)$.
\item\label{Zs_sing_claim2}

\begin{enumerate}
\renewcommand{\theenumii}{\roman{enumii}}
\renewcommand{\labelenumii}{(\theenumii)}
\item\label{Zs_sing_claim21}
If $n=1$, then there exists a nonsingular curve $C$ with $C\sim 2\sigma+2l$ 
such that $2C\leqslant E_Z$. 
\item\label{Zs_sing_claim22}
If $n\geqslant 3$, then any irreducible component $C\leqslant E_Z$ apart from 
$\sigma$, $l$ is a section of $\F_n/\pr^1$ and either $C\sim\sigma+nl$ or 
$C\sim\sigma+(n+1)l$ holds. Furthermore, if $n\geqslant 4$, then such $C$ 
satisfies that $\coeff_CE_Z=2$. 
\end{enumerate}
\end{enumerate}
\end{claim}

\begin{proof}
Since $3\sigma\not\leqslant E_Z$, there exists an irreducible component 
$C\leqslant E_Z$ apart from $\sigma$, $l$. 
Set $C\sim m\sigma+(nm+u)l$ with $m\geqslant 1$, $u\geqslant 0$. 
Assume that $m\geqslant 2$. By Lemma \ref{mult_seq_lem}, 
$nm^2+2um-((C^M)^2)=(C^2)-((C^M)^2)=(L_Z\cdot C)+2p_a(C)=nm^2+(2u+h-n-2)m+2$. 
Thus $-((C^M)^2)=(h-n-2)m+2\geqslant 4$. This implies that $\coeff_CE_Z=2$.
Since $2C\leqslant E_Z$, $m=2$ and 
$3n+6-h\geqslant 2(2n+u)$. Hence $(n, h, u)=(1, 4, 0)$ or $(1, 5, 0)$. 

Assume that $m=1$, that is, $C$ is a section. By the condition 
$(\sF$\ref{fund_dfn7}$)$, $\sigma\leqslant E_Z$. 
By the condition $(\sF$\ref{fund_dfn6}$)$, $\Delta_Z\cap\sigma=\emptyset$. 
Thus $h=2n$. In particular, $n\geqslant 3$. 
We know that $\deg(\Delta_Z\cap C)=2n+2u\leqslant k_Z=n+6$. Thus $u=0$ or $1$. 
Moreover, $((C^M)^2)=(C^2)-\deg(\Delta_Z\cap C)=-n$. 
Thus if $n\geqslant 4$, then $\coeff_CE_Z=2$. 
\end{proof}

\subsubsection{The case $(n, h)=(1, 4)$.}\label{Fs_1_4_section}

In this case, we know that $E_Z=2C+l$ ($C\sim 2\sigma+2l$ nonsingular), 
$\deg(\Delta_Z\cap C)=8$ and $k_Z=9$. 
Assume that $|C\cap l|=\{Q\}$. 
Then $\mult_Q(\Delta_Z\cap l)=\deg(\Delta_Z\cap l)=2$. 
Set $c:=\mult_Q\Delta_Z$. By Lemma \ref{ZS5}, we have 
$\mult_Q(\Delta_Z\cap C)=c-1$. 
This is nothing but the type \textbf{$[$1;4,5$]_K$({\bi c})}.
Assume that $|C\cap l|=\{Q_1, Q_2\}$. 
By Lemma \ref{ZS2} and the fact $\deg(\Delta_Z\cap l)=2$, we can assume that 
$|\Delta_Z|\cap l=\{Q_1\}$ and $\mult_{Q_1}\Delta_Z=\mult_{Q_1}(\Delta_Z\cap l)=2$. 
This is nothing but the type \textbf{$[$1;4,5$]_A$}.

\subsubsection{The case $(n, h)=(1, 5)$.}\label{Fs_1_5_section}

In this case, we know that $E_Z=2C$ ($C\sim 2\sigma+2l$ nonsingular), 
$\deg(\Delta_Z\cap C)=10$ and $k_Z=10$. 
This is nothing but the type \textbf{$[$1;4,4$]$}.

\subsubsection{The case $(n, h)=(3, 6)$.}\label{Fs_3_6_section}

In this case, we know that $E_Z\sim4\sigma+9l$ and $k_Z=9$.
 
Assume that there exists an irreducible component $C\leqslant E_Z$ with 
$C\sim\sigma+4l$. Then $\deg(\Delta_Z\cap C)=8$. 
Since $3\sigma\not\leqslant E_Z$, there exists an irreducible component 
$C'\leqslant E_Z-C$ such that $C'$ is a section apart from $\sigma$. 
Assume that $C\neq C'$. We can write $C'\sim\sigma+(3+u)l$ with $u=0$ or $1$ and 
$\deg(\Delta_Z\cap C')=6+2u$. Thus $\deg(\Delta_Z\cap C\cap C')\geqslant 5+2u$ 
by Proposition \ref{twocurve_prop}. However, $(C\cdot C')=4+u$. 
This leads to a contradiction. 
Thus $\coeff_CE_Z=2$. By the condition $(\sF$\ref{fund_dfn7}$)$, we have 
$E_Z=2C+2\sigma+l$. Since $\Delta_Z\cap\sigma=\emptyset$, 
$C\cap\sigma\cap l=\emptyset$. 
This case induces the type \textbf{$[$3;4,9$]_A$}.

From now on, we can assume that any component of $E_Z$ is one of $\sigma_\infty$, 
$\sigma$ or $l$. 
Assume that $\coeff_\sigma E_Z=1$. By the condition $(\sF$\ref{fund_dfn7}$)$, 
$E_Z=\sigma_{\infty, 1}+\sigma_{\infty, 2}+\sigma_{\infty, 3}+\sigma$, where 
$\sigma_{\infty, i}$ are distinct sections at infinity. 
By Lemma \ref{ZS3}, 
$\sigma_{\infty, 1}\cap\sigma_{\infty, 2}\cap\sigma_{\infty, 3}=\emptyset$. 
Moreover, by Lemma \ref{ZS4}, for any $Q\in\sigma_{\infty, i}\cap\sigma_{\infty, j}$, 
$\Delta_Z$ is equal to $\sigma_{\infty, i}\cap\sigma_{\infty, j}$ around $Q$. 
This case is nothing but the type \textbf{$[$3;4,9$]_B$}.

Assume that $\coeff_\sigma E_Z=2$ and $2\sigma_\infty\leqslant E_Z$. 
Consider the case $E_Z=2\sigma_\infty+2\sigma+2l_1+l_2$ 
($l_1$, $l_2$: distinct fibers). 
Set $Q_1:=\sigma_\infty\cap l_1$, $c:=\mult_{Q_1}(\Delta_Z\cap\sigma_\infty)$ 
and $d:=\mult_{Q_1}(\Delta_Z\cap l_1)$. 
This case induces the type \textbf{$[$3;4,9$]_C$({\bi c,d})}.
Consider the case $E_Z=2\sigma_\infty+2\sigma+l_1+l_2+l_3$ 
($l_1$, $l_2$, $l_3$: distinct fibers). 
This case induces the type \textbf{$[$3;4,9$]_D$}.

Assume that $\coeff_\sigma E_Z=2$ and any other section $C$ satisfies that 
$\coeff_CE_Z\leqslant 1$. 
Consider the case $E_Z=\sigma_{\infty, 1}+\sigma_{\infty, 2}+2\sigma+2l_1+l_2$ 
($\sigma_{\infty, 1}$, $\sigma_{\infty, 2}$: distinct sections at infinity, 
$l_1$, $l_2$: distinct fibers). 
We know that $((\sigma_{\infty, i})^2)=-3$. Thus $\sigma_{\infty, i}^M$ is a connected 
component of $E_M$ for $i=1$, $2$. By Lemma \ref{ZS3}, 
$\sigma_{\infty, 1}\cap\sigma_{\infty, 2}\cap(l_1\cup l_2)=\emptyset$. 
By Lemmas \ref{ZS2} and \ref{ZS4}, this case induces the type \textbf{$[$3;4,9$]_E$}.
Consider the case $E_Z=\sigma_{\infty, 1}+\sigma_{\infty, 2}+2\sigma+l_1+l_2+l_3$ 
($\sigma_{\infty, 1}$, $\sigma_{\infty, 2}$: distinct sections at infinity, 
$l_1$, $l_2$, $l_3$: distinct fibers). 
By Lemma \ref{ZS3}, 
$\sigma_{\infty, 1}\cap\sigma_{\infty, 2}\cap(l_1\cup l_2\cup l_3)=\emptyset$. 
By Lemmas \ref{ZS2} and \ref{ZS4}, this case induces the type \textbf{$[$3;4,9$]_F$}.

\subsubsection{The case $(n, h)=(4, 8)$.}\label{Fs_4_8_section}

In this case, we know that $E_Z\sim4\sigma+10l$ and $k_Z=10$. 
Assume that there exists an irreducible component $C\leqslant E_Z$ with 
$C\sim\sigma+5l$. Then $E_Z=2C+2\sigma$ and $\deg(\Delta_Z\cap C)=10$. 
This case is nothing but the type \textbf{$[$4;4,10$]_0$}.
Assume that $\sigma_\infty\leqslant E_Z$. Then $2\sigma_\infty+2\sigma\leqslant E_Z$. 
Consider the case $E_Z=2\sigma_\infty+2\sigma+2l$. 
Set $Q:=\sigma_\infty\cap l$, $c:=\mult_Q(\Delta_Z\cap\sigma_\infty)$ 
and $d:=\mult_Q(\Delta_Z\cap l)$. 
This case induces the type \textbf{$[$4;4,10$]_1$({\bi c,d})}.
Consider the case $E_Z=2\sigma_\infty+2\sigma+l_1+l_2$ ($l_1$, $l_2$: distinct fibers). 
This case induces the type \textbf{$[$4;4,10$]_2$}.

\subsubsection{The case $(n, h)=(5, 10)$.}\label{Fs_5_10_section}

In this case, we know that $E_Z=2\sigma_\infty+2\sigma+l$ and $k_Z=11$. 
This case induces the type \textbf{$[$5;4,11$]_1$}.

\subsubsection{The case $(n, h)=(6, 12)$.}\label{Fs_6_12_section}

In this case, we know that $E_Z=2\sigma_\infty+2\sigma$ and $k_Z=12$. 
Since $\deg(\Delta_Z\cap C)=12$, this case is nothing but the type 
\textbf{$[$6;4,12$]_0$}.

As a consequence,  we have completed the proof of Theorem \ref{trip_thm}.

\section{Classification of bottom tetrads, I}\label{XCI_section}

We classify bottom tetrads $(X, E_X; \Delta_Z, \Delta_X)$ with big 
$2K_X+L_X$. 

\begin{thm}\label{tetI_thm}
The bottom tetrads $(X, E_X; \Delta_Z, \Delta_X)$ with 
big $2K_X+L_X$ 
are classified by the types defined as follows $($We assume that any of them 
satisfies that both $\Delta_X$ and $\Delta_Z$ satisfy the $(\nu1)$-condition.$)$: 

\smallskip

The case $X=\pr^2$ and $E_X=l$ $(l$ is a line$):$

\begin{description}
\item[{$[$1$]_0$}]
$\Delta_X\subset l$ with $\deg\Delta_X=4$ and $\Delta_Z=\emptyset$.
\end{description}

\smallskip

The case $X=\pr^2$ and $E_X=C$ $(C$ is a nonsingular conic$):$

\begin{description}
\item[{$[$2$]_0$}]
$\Delta_X\subset C$ with $\deg\Delta_X=7$ and $\Delta_Z=\emptyset$.
\end{description}

\smallskip

The case $E_X=2l$ $(l$ is a line$):$

\begin{description}
\item[{$[$2$]_{1A}$}]
$|\Delta_X|=\{P_1, P_2, P_3\}$ such that 
$(\mult_{P_i}\Delta_X, \mult_{P_i}(\Delta_X\cap l))=(2, 1)$ for any $i=1$, $2$, $3$.
$|\Delta_Z|=\{Q\}$ with 
$Q\in l^Z\setminus(\Gamma_{P_1, 1}\cup\Gamma_{P_2, 1}\cup\Gamma_{P_3, 1})$ 
such that $\mult_Q\Delta_Z=1$. 

\smallskip

\item[{$[$2$]_{1B}$}]
$|\Delta_X|=\{P_1, P_2, P_3\}$ such that 
$(\mult_{P_i}\Delta_X, \mult_{P_i}(\Delta_X\cap l))=(2, 1)$ for $i=1$, $2$ and 
$(\mult_{P_3}\Delta_X, \mult_{P_3}(\Delta_X\cap l))=(1, 1)$.
$|\Delta_Z|=\{Q\}$ with $Q=l^Z\cap \Gamma_{P_3, 1}$, 
$\Delta_Z\subset\Gamma_{P_3, 1}$ and $\deg\Delta_Z=2$.

\smallskip

\item[{$[$2$]_{1C}$}]
$|\Delta_X|=\{P_1, P_2\}$ such that 
$(\mult_{P_1}\Delta_X, \mult_{P_1}(\Delta_X\cap l))=(4, 2)$ and 
$(\mult_{P_2}\Delta_X, \mult_{P_2}(\Delta_X\cap l))=(2, 1)$.
$|\Delta_Z|=\{Q\}$ with 
$Q\in l^Z\setminus(\Gamma_{P_1, 2}\cup\Gamma_{P_2, 1})$ 
such that $\mult_Q\Delta_Z=1$. 

\smallskip

\item[{$[$2$]_{1D}$}]
$|\Delta_X|=\{P_1, P_2\}$ such that 
$(\mult_{P_1}\Delta_X, \mult_{P_1}(\Delta_X\cap l))=(4, 2)$ and 
$(\mult_{P_2}\Delta_X, \mult_{P_2}(\Delta_X\cap l))=(1, 1)$.
$|\Delta_Z|=\{Q\}$ with $Q=l^Z\cap \Gamma_{P_2, 1}$, 
$\Delta_Z\subset\Gamma_{P_2, 1}$ and $\deg\Delta_Z=2$.

\smallskip

\item[{$[$2$]_{1E}$({\bi c,d})} $((c,d)=(0,0)$, $(1,1)$ {\it or} $(1,2))$]
$|\Delta_X|=\{P_1, P_2\}$ such that 
$(\mult_{P_1}\Delta_X$, $\mult_{P_1}(\Delta_X\cap l))=(2, 2)$ and 
$(\mult_{P_2}\Delta_X$, $\mult_{P_2}(\Delta_X\cap l))=(2, 1)$.
$\deg\Delta_Z=3$, $\deg(\Delta_Z\cap\Gamma_{P_1, 2})=2$, $\deg(\Delta_Z\cap l^Z)=1$ 
such that $\Delta_Z\cap(\Gamma_{P_1, 1}\cup\Gamma_{P_2, 1})=\emptyset$, 
$\mult_Q(\Delta_Z\cap l^Z)=c$, $\mult_Q(\Delta_Z\cap\Gamma_{P_1, 2})=d$ 
and $\mult_Q\Delta_Z=c+d$, where $Q=l^Z\cap\Gamma_{P_1, 2}$.

\smallskip

\item[{$[$2$]_{1F}$}]
$|\Delta_X|=\{P_1, P_2\}$ such that 
$(\mult_{P_1}\Delta_X, \mult_{P_1}(\Delta_X\cap l))=(2, 2)$ and 
$(\mult_{P_2}\Delta_X, \mult_{P_2}(\Delta_X\cap l))=(1, 1)$.
$\deg\Delta_Z=4$, $\mult_Q\Delta_Z=\mult_Q(\Delta\cap\Gamma_{P_2, 1})=2$, 
$\deg(\Delta_Z\cap\Gamma_{P_1, 2})=2$ and 
$|\Delta_Z|\cap(l^Z\cup \Gamma_{P_1, 1})=\emptyset$, 
where $Q:=l^Z\cap\Gamma_{P_2, 1}$. 

\smallskip

\item[{$[$2$]_{1G}$}]
$|\Delta_X|=\{P_1, P_2\}$ such that 
$(\mult_{P_i}\Delta_X, \mult_{P_i}(\Delta_X\cap l))=(1, 1)$ for any $i=1$, $2$.
$\deg\Delta_Z=5$, $\deg(\Delta_Z\cap l^Z)=3$ and 
$\mult_{Q_i}\Delta_Z=\mult_{Q_i}(\Delta_Z\cap\Gamma_{P_i, 1})=2$ hold, 
where $Q_i:=l^Z\cap\Gamma_{P_i, 1}$. 

\smallskip

\item[{$[$2$]_{1H}$}]
$|\Delta_X|=\{P_1, P_2\}$ such that 
$(\mult_{P_1}\Delta_X, \mult_{P_1}(\Delta_X\cap l))=(2, 1)$ and 
$(\mult_{P_2}\Delta_X, \mult_{P_2}(\Delta_X\cap l))=(1, 1)$.
$\deg\Delta_Z=4$, $\deg(\Delta_Z\cap l^Z)=3$, 
$\mult_Q\Delta_Z=\mult_Q(\Delta_Z\cap\Gamma_{P_2, 1})=2$ and 
$\Delta_Z\cap\Gamma_{P_1, 1}=\emptyset$ hold, where 
$Q:=l^Z\cap\Gamma_{P_2, 1}$. 

\smallskip

\item[{$[$2$]_{1I}$}]
$|\Delta_X|=\{P_1, P_2\}$ such that 
$(\mult_{P_i}\Delta_X, \mult_{P_i}(\Delta_X\cap l))=(2, 1)$ for any $i=1$, $2$.
$\deg\Delta_Z=3$, $\Delta_Z\subset l^Z$ and 
$\Delta_Z\cap(\Gamma_{P_1, 1}\cup\Gamma_{P_2, 1})=\emptyset$. 

\smallskip

\item[{$[$2$]_{1J}$({\bi c,d})} $((c, d)=(0, 0), (1, 1), (2, 1), (3, 1), (1, 2))$]
$|\Delta_X|=\{P\}$ such that 
$(\mult_P\Delta_X, \mult_P(\Delta_X\cap l))=(2, 2)$.
$\deg\Delta_Z=5$, $\deg(\Delta_Z\cap l^Z)=3$, $\deg(\Delta_Z\cap\Gamma_{P, 2})=2$, 
$\Delta_Z\cap\Gamma_{P, 1}=\emptyset$, 
$c=\mult_Q(\Delta_Z\cap l^Z)$, 
$d=\mult_Q(\Delta_Z\cap\Gamma_{P, 2})$ and $\mult_Q\Delta_Z=c+d$, 
where $Q=l^Z\cap\Gamma_{P, 2}$.

\smallskip

\item[{$[$2$]_{1K}$}]
$|\Delta_X|=\{P\}$ such that 
$(\mult_P\Delta_X, \mult_P(\Delta_X\cap l))=(4, 2)$.
$\deg\Delta_Z=3$, $\Delta_Z\subset l^Z$ and $\Delta_Z\cap\Gamma_{P, 2}=\emptyset$. 

\smallskip

\item[{$[$2$]_{1L}$}]
$|\Delta_X|=\{P\}$ such that 
$(\mult_P\Delta_X, \mult_P(\Delta_X\cap l))=(2, 1)$.
$\deg\Delta_Z=5$, $\Delta_Z\subset l^Z$ and $\Delta_Z\cap\Gamma_{P, 1}=\emptyset$. 

\smallskip

\item[{$[$2$]_{1M}$}]
$|\Delta_X|=\{P\}$ such that 
$\mult_P\Delta_X=1$.
$\deg\Delta_Z=6$, $\deg(\Delta_Z\cap l^Z)=5$ 
and $\mult_Q\Delta_Z=\mult_Q(\Delta_Z\cap\Gamma_{P, 1})=2$, where 
$Q=l^Z\cap\Gamma_{P, 1}$. 

\smallskip

\item[{$[$2$]_{1N}$}]
$\Delta_X=\emptyset$, $\deg\Delta_Z=7$ and $\Delta_Z\subset l^Z$. 
\end{description}

\smallskip

The case $X=\pr^2$ and $E_X=l_1+l_2$ $(l_i$ are distinct lines. Set $P:=l_1\cap l_2$.$):$

\begin{description}
\item[{$[$2$]_{2A}$}]
$\deg\Delta_X=5$, $\deg(\Delta_X\cap l_i)=3$ and $\mult_P\Delta_X=1$. 
$|\Delta_Z|=\{Q_1, Q_2\}$ such that $\mult_{Q_i}\Delta_Z=1$, where 
$Q_i=l_i^Z\cap\Gamma_{P, 1}$. 

\smallskip

\item[{$[$2$]_{2B}$}]
$\deg\Delta_X=6$, $\deg(\Delta_X\cap l_i)=3$ and $P\not\in\Delta_X$. 
$|\Delta_Z|=\{Q\}$ such that $\mult_Q\Delta_Z=1$, where 
$Q=l_1^Z\cap l_2^Z$.

\end{description}

\smallskip

The case $X=\pr^1\times\pr^1:$

\begin{description}
\item[{$[$0;1,0$]$}]
$E_X=\sigma$, $\deg\Delta_X=3$, 
$\Delta_X\subset\sigma$ and $\Delta_Z=\emptyset$. 

\smallskip

\item[{$[$0;1,1$]_0$}]
$E_X=C$ such that $C$ is nonsingular, $C\in|\sigma+l|$, $\deg\Delta_X=5$, 
$\Delta_X\subset C$ and $\Delta_Z=\emptyset$. 

\smallskip

\item[{$[$0;1,1$]_1\langle$0$\rangle$}]
$E_X=\sigma+l$, $\deg\Delta_X=4$, 
$\deg(\Delta_X\cap\sigma)=\deg(\Delta_X\cap l)=2$, $P\not\in\Delta_X$, 
$\deg\Delta_Z=1$ and $Q\in\Delta_Z$, where $P=\sigma\cap l$ and 
$Q=\sigma^Z\cap l^Z$. 

\smallskip

\item[{$[$0;1,1$]_1\langle$1$\rangle$}]
$E_X=\sigma+l$, $\deg\Delta_X=3$, 
$\deg(\Delta_X\cap\sigma)=\deg(\Delta_X\cap l)=2$, $\mult_P\Delta_X=1$, 
$\deg\Delta_Z=2$ and $Q_\sigma$, $Q_l\in\Delta_Z$, where $P=\sigma\cap l$, 
$Q_\sigma=\sigma^Z\cap\Gamma_{P, 1}$ and $Q_l=l^Z\cap\Gamma_{P, 1}$. 
\end{description}

\smallskip

The case $X=\F_1:$

\begin{description}
\item[{$[$1;1,0$]$}]
$E_X=\sigma$, $\deg\Delta_X=2$, $\Delta_X\subset\sigma$ and 
$\Delta_Z=\emptyset$. 

\smallskip

\item[{$[$1;1,1$]_0$}]
$E_X=\sigma_\infty$, $\deg\Delta_X=4$, $\Delta_X\subset\sigma_\infty$ and 
$\Delta_Z=\emptyset$. 

\smallskip

\item[{$[$1;1,1$]_1\langle$0$\rangle$}]
$E_X=\sigma+l$, $\deg\Delta_X=3$, $P\not\in\Delta_X$, $\deg(\Delta_X\cap\sigma)=1$, 
$\deg(\Delta_X\cap l)=2$, $\deg\Delta_Z=1$ and $Q\in\Delta_Z$, 
where $P=\sigma\cap l$ and $Q=\sigma^Z\cap l^Z$. 

\smallskip

\item[{$[$1;1,1$]_1\langle$1$\rangle$}]
$E_X=\sigma+l$, $\deg\Delta_X=2$, $\mult_P\Delta_X=1$, 
$\deg(\Delta_X\cap l)=2$, $\deg\Delta_Z=2$ and $Q_\sigma$, $Q_l\in\Delta_Z$, 
where $P=\sigma\cap l$, $Q_\sigma=\sigma^Z\cap\Gamma_{P, 1}$ 
and $Q_l=l^Z\cap\Gamma_{P, 1}$. 
\end{description}

\smallskip

The case $X=\F_2:$

\begin{description}
\item[{$[$2;1,0$]$}]
$E_X=\sigma$, $\deg\Delta_X=1$, $\Delta_X\subset\sigma$ and $\Delta_Z=\emptyset$.  

\smallskip

\item[{$[$2;1,1$]$}]
$E_X=\sigma+l$, $\deg\Delta_X=2$, $\Delta_X\subset l\setminus\sigma$, 
$\deg\Delta_Z=1$ and 
$Q\in\Delta_Z$, where $Q=\sigma^Z\cap l^Z$. 

\smallskip

\item[{$[$2;1,2$]_0$}]
$E_X=\sigma_\infty$, $\deg\Delta_X=5$, $\Delta_X\subset\sigma_\infty$ and 
$\Delta_Z=\emptyset$. 

\smallskip

\item[{$[$2;1,2$]_{1A}$}]
$E_X=\sigma+2l$, $\deg\Delta_X=4$, $|\Delta_X|=\{P_1$, $P_2\}$ such that 
$P_i\not\in\sigma$ and $(\mult_{P_i}\Delta_X, \mult_{P_i}(\Delta_X\cap l))=(2, 1)$ 
for $i=1$, $2$. $\deg\Delta_Z=1$ and 
$\Delta_Z\subset l^Z\setminus(\sigma^Z\cup\Gamma_{P_1, 1}\cup\Gamma_{P_2, 1})$. 

\smallskip

\item[{$[$2;1,2$]_{1B}$}]
$E_X=\sigma+2l$, $\deg\Delta_X=3$, $|\Delta_X|=\{P_1$, $P_2\}$ such that 
$P_1$, $P_2\in l\setminus\sigma$, 
$(\mult_{P_1}\Delta_X, \mult_{P_1}(\Delta_X\cap l))=(2, 1)$ 
and $\mult_{P_2}\Delta_X=1$. $\deg\Delta_Z=2$ and 
$\mult_Q\Delta_Z=\mult_Q(\Delta_Z\cap\Gamma_{P_2, 1})=2$, 
where $Q=l^Z\cap\Gamma_{P_2, 1}$. 

\smallskip

\item[{$[$2;1,2$]_{1C}$}]
$E_X=\sigma+2l$, $\deg\Delta_X=4$, $|\Delta_X|=\{P\}$ such that 
$P\not\in\sigma$ and $(\mult_P\Delta_X, \mult_P(\Delta_X\cap l))=(4, 2)$. 
$\deg\Delta_Z=1$ and 
$\Delta_Z\subset l^Z\setminus(\sigma^Z\cup\Gamma_{P, 2})$. 

\smallskip

\item[{$[$2;1,2$]_{1D}$({\bi c,d})} $((c, d)=(0, 0), (1, 1), (1, 2))$]
$E_X=\sigma+2l$, $|\Delta_X|=\{P\}$, 
$\deg\Delta_X=2$, $\Delta_X\subset l\setminus\sigma$, 
$\deg\Delta_Z=3$, $\deg(\Delta_Z\cap l^Z)=1$, 
$\deg(\Delta_Z\cap\Gamma_{P, 2})=2$, 
$\Delta_Z\cap(\sigma^Z\cup\Gamma_{P, 1})=\emptyset$, 
$\mult_Q(\Delta_Z\cap l^Z)=c$, $\mult_Q(\Delta_Z\cap\Gamma_{P, 2})=d$, 
$\mult_Q\Delta_Z=c+d$, where $Q=l^Z\cap\Gamma_{P, 2}$. 

\smallskip

\item[{$[$2;1,2$]_{1E}$}]
$E_X=\sigma+2l$, $\deg\Delta_X=2$, $|\Delta_X|=\{P\}$ such that 
$P\not\in\sigma$ and $(\mult_P\Delta_X, \mult_P(\Delta_X\cap l))=(2, 1)$. 
$\deg\Delta_Z=3$ and 
$\Delta_Z\subset l^Z\setminus(\sigma^Z\cup\Gamma_{P, 1})$. 

\smallskip

\item[{$[$2;1,2$]_{1F}$}]
$E_X=\sigma+2l$, $\deg\Delta_X=1$, $|\Delta_X|=\{P\}$ such that 
$P\in l\setminus\sigma$ and $\mult_P\Delta_X=1$. 
$\deg\Delta_Z=4$, 
$\mult_Q\Delta_Z=\mult_Q(\Delta_Z\cap\Gamma_{P, 1})=2$ and 
$\Delta_Z\setminus\{Q\}\subset l^Z\setminus\sigma^Z$, where 
$Q=l^Z\cap\Gamma_{P, 1}$. 

\smallskip

\item[{$[$2;1,2$]_{1G}$}]
$E_X=\sigma+2l$, $\Delta_X=\emptyset$, $\deg\Delta_Z=5$ and 
$\Delta_Z\subset l^Z\setminus\sigma^Z$. 
\end{description}

\smallskip

The case $X=\F_3:$

\begin{description}
\item[{$[$3;1,0$]_0$}]
$E_X=\sigma$, $\Delta_X=\emptyset$ and $\Delta_Z=\emptyset$.  
\end{description}
\end{thm}

We start to prove Theorem \ref{tetI_thm}. Any tetrad in Theorem \ref{tetI_thm} 
is a bottom tetrad by Proposition \ref{converse_prop}. 
We see the converse. 

\subsection{The case $X=\pr^2$}\label{pr2_section}

Let $(X=\pr^2, E_X; \Delta_Z, \Delta_X)$ be a 
bottom tetrad, $L_X$ be the fundamental divisor, 
$\psi\colon Z\to X$ be the elimination of $\Delta_X$, 
$\phi\colon M\to Z$ be the elimination of $\Delta_Z$, 
$E_Z:=(E_X)_Z^{\Delta_X, 1}$ and 
$E_M:=(E_Z)_M^{\Delta_Z, 2}$.
Set $L_X\sim hl$, $E_X\sim el$, 
$k_X:=\deg\Delta_X$ and $k_Z:=\deg\Delta_Z$. 
Then $e=9-h$, $h\geqslant 6$ and $k_X+k_Z=he/2$ hold. 
Thus $(h, e, k_X+k_Z)=(6, 3, 9)$, $(7, 2, 7)$ or $(8, 1, 4)$. 
Moreover, if $h=6$ then $k_X\leqslant 8$ holds since $(K_X+L_X\cdot L_X)>2k_X$.

\begin{claim}\label{tetI_curve_claim}
Pick any nonsingular component $C\leqslant E_X$. 
\begin{enumerate}
\renewcommand{\theenumi}{\arabic{enumi}}
\renewcommand{\labelenumi}{(\theenumi)}
\item\label{tetI_curve_claim1}
If $C$ is a conic, then 
$(h, ((C^M)^2), \deg(\Delta_X\cap C), \deg(\Delta_Z\cap C^Z))=(6, -2, 6, 0)$, 
$(6, -3, 5, 2)$ or $(7, -3, 7, 0)$.
\item\label{tetI_curve_claim2}
If $C$ is a line, then 
$(h, ((C^M)^2), \deg(\Delta_X\cap C), \deg(\Delta_Z\cap C^Z))=(6, -2, 3, 0)$, 
$(6, -3, 2, 2)$, $(6, -4, 1, 4)$, $(6, -5, 0, 6)$, $(7, -3, 3, 1)$,
$(7, -4, 2, 3)$, $(7, -5, 1, 5)$, $(7, -6, 0, 7)$ or $(8, -3, 4, 0)$.
\end{enumerate}
\end{claim}

\begin{proof}
Set $m:=\deg C$ ($m=1$ or $2$). We note that if $m=2$ then $h\leqslant 7$. 
We also note that if $m=1$ and $h=8$ then $((C^M)^2)=-2$ or $-3$ by 
Corollary \ref{dP-basic_cor}. 
We have 
$hm=2\deg(\Delta_X\cap C)+\deg(\Delta_Z\cap C^Z)$ and 
$((C^M)^2)=m^2-\deg(\Delta_X\cap C)-\deg(\Delta_Z\cap C^Z)$. 
Thus the assertion holds.
\end{proof}

If $2K_X+L_X$ is big, then $h=7$ or $8$.
We consider the case $E_X=l$, i.e., $h=8$. 
Then $k_X=\deg(\Delta_X\cap l)=4$ 
and $k_Z=0$. This is nothing but the type \textbf{$[$1$]_0$}. 
Now we consider the case $E_X\sim 2l$, i.e., $h=7$. 

\subsubsection{The case $E_X=C$ $(C:$ nonsingular conic$)$}

In this case, we have $k_X=\deg(\Delta_X\cap C)=7$ 
and $k_Z=0$. This is nothing but the type \textbf{$[$2$]_0$}.

\subsubsection{The case $E_X=2l$ $(l:$ line$)$}

Set $d_X:=\deg(\Delta_X\cap l)$ and $d_Z:=\deg(\Delta_Z\cap l^Z)$. 
By Claim \ref{tetI_curve_claim}, we have $(d_X, d_Z, ((l^M)^2))=(3, 1, -3)$, 
$(2, 3, -4)$, $(1, 5, -5)$ or $(0, 7, -6)$. 

{\noindent\textbf{The case $(d_X, d_Z)=(3, 1)$:}}

By Lemma \ref{XS1}, one of the following holds: 

\begin{enumerate}
\renewcommand{\theenumi}{\Alph{enumi}}
\renewcommand{\labelenumi}{(\theenumi)}
\item\label{2_1_A}
$|\Delta_X|=\{P_1, P_2, P_3\}$ such that 
$(\mult_{P_i}\Delta_X, \mult_{P_i}(\Delta_X\cap l))=(2, 1)$ for any $i=1$, $2$, $3$.
\item\label{2_1_B}
$|\Delta_X|=\{P_1, P_2, P_3\}$ such that 
$(\mult_{P_i}\Delta_X, \mult_{P_i}(\Delta_X\cap l))=(2, 1)$ for $i=1$, $2$ and 
$(\mult_{P_3}\Delta_X, \mult_{P_3}(\Delta_X\cap l))=(1, 1)$.
\item\label{2_1_C}
$|\Delta_X|=\{P_1, P_2\}$ such that 
$(\mult_{P_1}\Delta_X, \mult_{P_1}(\Delta_X\cap l))=(4, 2)$ and 
$(\mult_{P_2}\Delta_X, \mult_{P_2}(\Delta_X\cap l))=(2, 1)$.
\item\label{2_1_D}
$|\Delta_X|=\{P_1, P_2\}$ such that 
$(\mult_{P_1}\Delta_X, \mult_{P_1}(\Delta_X\cap l))=(4, 2)$ and 
$(\mult_{P_2}\Delta_X, \mult_{P_2}(\Delta_X\cap l))=(1, 1)$.
\item\label{2_1_E}
$|\Delta_X|=\{P_1, P_2\}$ such that 
$(\mult_{P_1}\Delta_X, \mult_{P_1}(\Delta_X\cap l))=(2, 2)$ and 
$(\mult_{P_2}\Delta_X, \mult_{P_2}(\Delta_X\cap l))=(2, 1)$.
\item\label{2_1_F}
$|\Delta_X|=\{P_1, P_2\}$ such that 
$(\mult_{P_1}\Delta_X, \mult_{P_1}(\Delta_X\cap l))=(2, 2)$ and 
$(\mult_{P_2}\Delta_X, \mult_{P_2}(\Delta_X\cap l))=(1, 1)$.
\end{enumerate}
Indeed, if there exist two points $P_1$, $P_2\in\Delta_X$ such that 
$\mult_{P_i}\Delta_X=1$ for $i=1$, $2$, then $\deg\Delta_Z\geqslant 2$. 
This is a contradiction.

We consider the case \eqref{2_1_A}. Then $k_Z=1$ and 
$\Delta_Z\cap\Gamma_{P_i, 1}=\emptyset$ for $i=1$, $2$, $3$. 
This is nothing but the type \textbf{$[$2$]_{1A}$}.

We consider the case \eqref{2_1_B}. Then $k_Z=2$. Moreover, 
$\mult_Q\Delta_Z=\mult_Q(\Delta_Z\cap\Gamma_{P_3, 1})=2$ and 
$\mult_Q(\Delta_Z\cap l^Z)=1$, where $Q:=l^Z\cap\Gamma_{P_3, 1}$. 
This is nothing but the type \textbf{$[$2$]_{1B}$}.

We consider the case \eqref{2_1_C}. Then $k_Z=1$ and 
$\Delta_Z\subset l^Z\setminus(\Gamma_{P_1,2}\cup\Gamma_{P_2, 1})$. 
This is nothing but the type \textbf{$[$2$]_{1C}$}.

We consider the case \eqref{2_1_D}. Then $k_Z=2$. Moreover, 
$\mult_Q\Delta_Z=\mult_Q(\Delta_Z\cap\Gamma_{P_2, 1})=2$, 
where $Q:=l^Z\cap\Gamma_{P_2, 1}$. 
This is nothing but the type \textbf{$[$2$]_{1D}$}.

We consider the case \eqref{2_1_E}. Then $k_Z=3$, 
$\deg(\Delta_Z\cap\Gamma_{P_1, 2})=2$ and $\deg(\Delta_Z\cap l^Z)=1$. 
Set $Q:=l^Z\cap\Gamma_{P_1, 2}$, $c:=\mult_Q(\Delta_Z\cap l^Z)$ and 
$d:=\mult_Q(\Delta_Z\cap\Gamma_{P_1, 2}$). Then $\mult_Q\Delta_Z=c+d$. 
This is nothing but the type \textbf{$[$2$]_{1E}$({\bi c,d})}.

We consider the case \eqref{2_1_F}. Then $k_Z=4$. 
Moreover, $\deg\Delta_Z=4$, 
$\mult_Q\Delta_Z=\mult_Q(\Delta_Z\cap\Gamma_{P_2, 1})=2$ and  
$\deg(\Delta_Z\cap\Gamma_{P_1, 2})=2$ hold, 
where $Q:=l^Z\cap\Gamma_{P_2, 1}$. 
This is nothing but the type \textbf{$[$2$]_{1F}$}.

{\noindent\textbf{The case $(d_X, d_Z)=(2, 3)$:}}

By Lemma \ref{XS1}, one of the following holds: 

\begin{enumerate}
\setcounter{enumi}{6}
\renewcommand{\theenumi}{\Alph{enumi}}
\renewcommand{\labelenumi}{(\theenumi)}
\item\label{2_1_G}
$|\Delta_X|=\{P_1, P_2\}$ such that 
$(\mult_{P_i}\Delta_X, \mult_{P_i}(\Delta_X\cap l))=(1, 1)$ for any $i=1$, $2$.
\item\label{2_1_H}
$|\Delta_X|=\{P_1, P_2\}$ such that 
$(\mult_{P_1}\Delta_X, \mult_{P_1}(\Delta_X\cap l))=(2, 1)$ and 
$(\mult_{P_2}\Delta_X, \mult_{P_2}(\Delta_X\cap l))=(1, 1)$.
\item\label{2_1_I}
$|\Delta_X|=\{P_1, P_2\}$ such that 
$(\mult_{P_i}\Delta_X, \mult_{P_i}(\Delta_X\cap l))=(2, 1)$ for any $i=1$, $2$.
\item\label{2_1_J}
$|\Delta_X|=\{P\}$ such that 
$(\mult_P\Delta_X, \mult_P(\Delta_X\cap l))=(2, 2)$.
\item\label{2_1_K}
$|\Delta_X|=\{P\}$ such that 
$(\mult_P\Delta_X, \mult_P(\Delta_X\cap l))=(4, 2)$.
\end{enumerate}

We consider the case \eqref{2_1_G}. Then $k_Z=5$. 
Set $Q_i:=l^Z\cap\Gamma_{P_i, 1}$. Then 
$\mult_{Q_i}\Delta_Z=\mult_{Q_i}(\Delta_Z\cap\Gamma_{P_i, 1})=2$ and 
$\mult_{Q_i}(\Delta_Z\cap l^Z)=1$ hold. Moreover, there exists a point 
$Q\in l^Z\setminus\{Q_1, Q_2\}$ such that $\mult_Q\Delta_Z=1$ since 
$d_Z=3$. This is nothing but the type \textbf{$[$2$]_{1G}$}.

We consider the case \eqref{2_1_H}. Then $k_Z=4$. 
Set $Q:=l^Z\cap\Gamma_{P_2, 1}$. Then 
$\mult_Q\Delta_Z=\mult_Q(\Delta_Z\cap\Gamma_{P_2, 1})=2$ and 
$\mult_Q(\Delta_Z\cap l^Z)=1$ hold. Moreover, 
$\Delta_Z\cap\Gamma_{P_1, 1}=\emptyset$. 
This is nothing but the type \textbf{$[$2$]_{1H}$}.

We consider the case \eqref{2_1_I}. Then $k_Z=d_Z=3$. 
Moreover, $\Delta_Z\cap(\Gamma_{P_1, 1}\cup\Gamma_{P_2, 1})=\emptyset$. 
This is nothing but the type \textbf{$[$2$]_{1I}$}.

We consider the case \eqref{2_1_J}. Then $k_Z=5$. 
Set $Q:=l^Z\cap\Gamma_{P, 2}$, $c:=\mult_Q(\Delta_Z\cap l^Z)$ and 
$d:=\mult_Q(\Delta_Z\cap\Gamma_{P, 2})$. 
Then $\mult_Q\Delta_Z=c+d$. Moreover, $(c, d)=(0, 0)$, $(1, 1)$, $(2, 1)$, $(3, 1)$ or 
$(1, 2)$ since $k_Z=3$ and $\deg(\Delta_Z\cap\Gamma_{P, 2})=2$. 
This is nothing but the type \textbf{$[$2$]_{1J}$({\bi c,d})}.

We consider the case \eqref{2_1_K}. Then $k_Z=d_Z=3$, 
$\Delta_Z\cap\Gamma_{P, 2}=\emptyset$. 
This is nothing but the type \textbf{$[$2$]_{1K}$}.

{\noindent\textbf{The case $(d_X, d_Z)=(1, 5)$:}}

By Lemma \ref{XS1}, one of the following holds: 

\begin{enumerate}
\setcounter{enumi}{11}
\renewcommand{\theenumi}{\Alph{enumi}}
\renewcommand{\labelenumi}{(\theenumi)}
\item\label{2_1_L}
$|\Delta_X|=\{P\}$ such that 
$(\mult_P\Delta_X, \mult_P(\Delta_X\cap l))=(2, 1)$.
\item\label{2_1_M}
$|\Delta_X|=\{P\}$ such that 
$(\mult_P\Delta_X, \mult_P(\Delta_X\cap l))=(1, 1)$.
\end{enumerate}

We consider the case \eqref{2_1_L}. Then $k_Z=d_Z=5$, 
$\Delta_Z\cap\Gamma_{P, 1}=\emptyset$. 
This is nothing but the type \textbf{$[$2$]_{1L}$}.

We consider the case \eqref{2_1_M}. Then $k_Z=6$.
Set $Q:=l^Z\cap\Gamma_{P, 1}$. Then 
$\mult_Q\Delta_Z=\mult_Q(\Delta_Z\cap\Gamma_{P, 1})=2$ and 
$\mult_Q(\Delta_Z\cap l^Z)=1$. 
This is nothing but the type \textbf{$[$2$]_{1M}$}.

{\noindent\textbf{The case $(d_X, d_Z)=(0, 7)$:}}

In this case, $\Delta_X=\emptyset$, $\Delta_Z\subset l^Z$. 
This is nothing but the type \textbf{$[$2$]_{1N}$}.

\subsubsection{The case $E_X=l_1+l_2$ $(l_i:$ distinct lines$)$}

Set $P:=l_1\cap l_2$. By Claim \ref{tetI_curve_claim}, $((l_1^M)^2)=((l_2^M)^2)=-3$. 
Thus $(\deg(\Delta_X\cap l_i), \deg(\Delta_Z\cap l_i^Z))=(3, 1)$. 
Assume that $P\in\Delta_X$. Then $\mult_P\Delta_X=1$ by Lemma \ref{XS2}. 
This case induces the type \textbf{$[$2$]_{2A}$}. 
Assume that $P\not\in\Delta_X$. Then $\mult_Q\Delta_Z=1$ by Lemma \ref{ZS2}, 
where $Q:=l_1^Z\cap l_2^Z$. 
This case induces the type \textbf{$[$2$]_{2B}$}.

\subsection{The case $X=\F_n$}

Let $(X=\F_n, E_X; \Delta_Z, \Delta_X)$ be a 
bottom tetrad such that $2K_X+L_X$ is big, where $L_X$ is the fundamental divisor, 
$\psi\colon Z\to X$ be the elimination of $\Delta_X$, 
$\phi\colon M\to Z$ be the elimination of $\Delta_Z$, 
$E_Z:=(E_X)_Z^{\Delta_X, 1}$ and 
$E_M:=(E_Z)_M^{\Delta_Z, 2}$.
Set $L_X\sim h_0\sigma+hl$, $E_X\sim e_0\sigma+el$, 
$k_X:=\deg\Delta_X$ and $k_Z:=\deg\Delta_Z$. 
Then $e_0=6-h_0$ and $e=3(n+2)-h$. 
Since $2K_X+L_X$ is nef and big, we have $h_0=5$. Thus $e_0=1$. 
We know that $k_X+k_Z=(L_X\cdot E_X)/2=5n-2h+15$.

\begin{claim}\label{tetII_claim}
We have $(n , h, k_X+k_Z)=(0, 5, 5)$, $(0, 6, 3)$, $(1, 8, 4)$, $(1, 9, 2)$, $(2, 10, 5)$, 
$(2, 11, 3)$, $(2, 12, 1)$ or $(3, 15, 0)$.
\end{claim}

\begin{proof}
We have $\max\{5n, 3n+4\}\leqslant h\leqslant 3n+6$ 
since $L_X$ and $2K_X+L_X$ are nef and big and $E_X$ is effective. 
In particular, $n\leqslant 3$. 
Moreover, if $n=0$, then $h\geqslant 5$. 
If $n=1$, then $h\geqslant 8$ since $(E_X\cdot\sigma)\leqslant 0$. 
\end{proof}

\subsubsection{The case $(n, h)=(0, 5)$}\label{tetII05_section}

In this case, $E_X\sim\sigma+l$. 
Assume that $E_X=C$, where $C$ is nonsingular. 
Then $\Delta_Z=\emptyset$ and $\Delta_X\subset C$. 
This is nothing but the type \textbf{$[$0;1,1$]_0$}. 
Assume that $E_X=\sigma+l$. Set $P:=\sigma\cap l$. 
Then $2\deg(\Delta_X\cap\sigma)+\deg(\Delta_Z\cap\sigma^Z)=5$ and 
$2\deg(\Delta_X\cap l)+\deg(\Delta_Z\cap l^Z)=5$. 
By Lemmas \ref{ZS2} and \ref{XS2}, 
if $P\not\in\Delta_X$ then this induces the type 
\textbf{$[$0;1,1$]_1\langle$0$\rangle$}; if $P\in\Delta_X$ then this induces the type 
\textbf{$[$0;1,1$]_1\langle$1$\rangle$}.

\subsubsection{The case $(n, h)=(0, 6)$}\label{tetII06_section}

In this case, $E_X=\sigma$. Thus $\Delta_Z=\emptyset$ and 
$\Delta_X\subset\sigma$. This is nothing but the type \textbf{$[$0;1,0$]$}.

\subsubsection{The case $(n, h)=(1, 8)$}\label{tetII18_section}

In this case, $E_X\sim\sigma+l$. Assume that $E_X=\sigma_\infty$. 
Then $\Delta_Z=\emptyset$ and $\Delta_X\subset\sigma_\infty$. 
This is nothing but the type \textbf{$[$1;1,1$]_0$}. 
Assume that $E_X=\sigma+l$. Set $P:=\sigma\cap l$. 
Then $2\deg(\Delta_X\cap\sigma)+\deg(\Delta_Z\cap\sigma^Z)=3$ and 
$2\deg(\Delta_X\cap l)+\deg(\Delta_Z\cap l^Z)=5$. 
By Lemmas \ref{ZS2} and \ref{XS2}, 
we can show that if $P\not\in\Delta_X$ then this induces the type 
\textbf{$[$1;1,1$]_1\langle$0$\rangle$}; if $P\in\Delta_X$ then this induces the type 
\textbf{$[$1;1,1$]_1\langle$1$\rangle$}.

\subsubsection{The case $(n, h)=(1, 9)$}\label{tetII19_section}

In this case, $E_X=\sigma$. Thus $\Delta_Z=\emptyset$ and 
$\Delta_X\subset\sigma$. This is nothing but the type \textbf{$[$1;1,0$]$}.

\subsubsection{The case $(n, h)=(2, 10)$}\label{tetII210_section}

In this case, $E_X\sim\sigma+2l$.

{\noindent\textbf{The case $E_X=\sigma_\infty$:}}

Then $\Delta_Z=\emptyset$ and $\Delta_X\subset\sigma_\infty$. 
This is nothing but the type \textbf{$[$2;1,2$]_0$}.

{\noindent\textbf{The case $E_X=\sigma+l_1+l_2$} ($l_1$, $l_2$ are distinct):}

In this case, $\Delta_X\cap\sigma=\emptyset$ and $\Delta_Z\cap\sigma^Z=\emptyset$. 
Thus $\sigma^M$, $l_1^M\leqslant E_M$ meet together. This contradicts to 
Corollary \ref{dP-basic_cor}.

{\noindent\textbf{The case $E_X=\sigma+2l$:}}

In this case, $\Delta_X\cap\sigma=\emptyset$ and $\Delta_Z\cap\sigma^Z=\emptyset$. 
Set $d_X:=\deg(\Delta_X\cap l)$ and $d_Z:=\deg(\Delta_Z\cap l^Z)$. 
Since $2d_X+d_Z=5$, we have $(d_X, d_Z)=(2, 1)$, $(1, 3)$ or $(0, 5)$. 

We consider the case $(d_X, d_Z)=(2, 1)$. One of the following holds: 

\begin{enumerate}
\renewcommand{\theenumi}{\Alph{enumi}}
\renewcommand{\labelenumi}{(\theenumi)}
\item\label{212_1A}
$|\Delta_X|\cap l=\{P_1, P_2\}$ such that 
$(\mult_{P_i}\Delta_X, \mult_{P_i}(\Delta_X\cap l))=(2, 1)$ for $i=1$, $2$. 
\item\label{212_1B}
$|\Delta_X|\cap l=\{P_1, P_2\}$ such that 
$(\mult_{P_1}\Delta_X, \mult_{P_1}(\Delta_X\cap l))=(2, 1)$ and 
$(\mult_{P_2}\Delta_X, \mult_{P_2}(\Delta_X\cap l))=(1, 1)$. 
\item\label{212_1C}
$|\Delta_X|\cap l=\{P\}$ such that 
$(\mult_P\Delta_X, \mult_P(\Delta_X\cap l))=(4, 2)$. 
\item\label{212_1D}
$|\Delta_X|\cap l=\{P\}$ such that 
$(\mult_P\Delta_X, \mult_P(\Delta_X\cap l))=(2, 2)$. 
\end{enumerate}
We can show that the case ({\tt X})
(${\tt X}\in\{$A, B, C$\}$) corresponds to the 
type \textbf{$[$2;1,2$]_{1{\tt X}}$}. 
We consider the case \eqref{212_1D}. Set $Q:=l^Z\cap\Gamma_{P, 2}$, 
$c:=\mult_Q(\Delta_Z\cap l^Z)$ and 
$d:=\mult_Q(\Delta_Z\cap\Gamma_{P, 2})$. 
Then we can show that this case corresponds to the 
type \textbf{$[$2;1,2$]_{1D}$}.

We consider the case $(d_X, d_Z)=(1, 3)$. One of the following holds: 

\begin{enumerate}
\setcounter{enumi}{4}
\renewcommand{\theenumi}{\Alph{enumi}}
\renewcommand{\labelenumi}{(\theenumi)}
\item\label{212_1E}
$|\Delta_X|\cap l=\{P\}$ such that 
$(\mult_P\Delta_X, \mult_P(\Delta_X\cap l))=(2, 1)$. 
\item\label{212_1F}
$|\Delta_X|\cap l=\{P\}$ such that 
$(\mult_P\Delta_X, \mult_P(\Delta_X\cap l))=(1, 1)$. 
\end{enumerate}
The case ({\tt X})
(${\tt X}\in\{$E, F$\}$) corresponds to the 
type \textbf{$[$2;1,2$]_{1{\tt X}}$}. 

We consider the case $(d_X, d_Z)=(0, 5)$. Then $\Delta_X=\emptyset$ and 
$\Delta_Z\subset l^Z$. This is nothing but the type \textbf{$[$2;1,2$]_{1G}$}.

\subsubsection{The case $(n, h)=(2, 11)$}\label{tetII211_section}

In this case, $E_X=\sigma+l$. Then 
$\Delta_X\cap\sigma=\emptyset$ and $\deg(\Delta_Z\cap\sigma^Z)=1$. 
Thus $\deg\Delta_Z=1$, $|\Delta_Z|=\{Q\}$, where $Q:=\sigma^Z\cap l^Z$. 
Moreover, we have $\deg(\Delta_X\cap l)=2$. 
This is nothing but the type \textbf{$[$2;1,1$]$}.

\subsubsection{The case $(n, h)=(2, 12)$}\label{tetII212_section}

In this case, $E_X=\sigma$. Thus $\Delta_Z=\emptyset$ and 
$\Delta_X\subset\sigma$. This is nothing but the type \textbf{$[$2;1,0$]$}.

\subsubsection{The case $(n, h)=(3, 15)$}\label{tetII315_section}

In this case, $E_X=\sigma$, $\Delta_X=\emptyset$ and $\Delta_Z=\emptyset$. 
This is nothing but the type \textbf{$[$3;1,0$]$}.

As a consequence,  we have completed the proof of Theorem \ref{tetI_thm}.

\section{Classification of bottom tetrads, II}\label{XCII_section}

We classify bottom tetrads $(X, E_X; \Delta_Z, \Delta_X)$ such that 
$X=\F_n$, $2K_X+L_X$ is non-big and nontrivial. 

\begin{thm}\label{tetII_thm}
The bottom tetrads $(X, E_X; \Delta_Z, \Delta_X)$ such that $X=\F_n$ 
and non-big, nontrivial $2K_X+L_X$ 
are classified by the types defined as follows $($We assume that any of them 
satisfies that both $\Delta_X$ and $\Delta_Z$ satisfy the $(\nu1)$-condition.$)$:

\smallskip

The case $X=\pr^1\times\pr^1:$

\begin{description}
\item[{$[$0;2,0$]$}]
$E_X=2\sigma$, $\Delta_X=\emptyset$, 
$\deg\Delta_Z=6$ and $\Delta_Z\subset\sigma^Z$. 
\end{description}

\smallskip

The case $X=\F_1:$

\begin{description}
\item[{$[$1;2,0$]$}]
$E_X=2\sigma$, $\Delta_X=\emptyset$, 
$\deg\Delta_Z=5$ 
and 
$\Delta_Z\subset\sigma^Z$. 

\smallskip

\item[{$[$1;2,1$]_{1A}$}]
$E_X=2\sigma+l$, $\deg\Delta_X=2$, 
$\Delta_X\subset l\setminus\sigma$, 
$\deg\Delta_Z=4$ 
and 
$\Delta_Z\subset\sigma^Z\setminus l^Z$. 

\smallskip

\item[{$[$1;2,1$]_{1B}$}]
$E_X=2\sigma+l$, $\deg\Delta_X=1$, 
$\Delta_X\subset l\setminus\sigma$, 
$\deg\Delta_Z=5$, 
$\mult_Q\Delta_Z=\mult_Q(\Delta_Z\cap l^Z)=2$
and 
$\Delta_Z\setminus\{Q\}\subset\sigma^Z$, 
where $Q=\sigma^Z\cap l^Z$. 

\smallskip

\item[{$[$1;2,2$]_U$}]
$E_X=C$ with $C:$ nonsingular and $C\sim 2\sigma+2l$, $\deg\Delta_X=7$, 
$\Delta_X\subset C$ and $\Delta_Z=\emptyset$. 

\smallskip

\item[{$[$1;2,2$]_{0A}$}]
$E_X=2\sigma_\infty$, $\deg\Delta_X=2$, 
$|\Delta_X|=\{P\}$, $\mult_P(\Delta_X\cap\sigma_\infty)=1$, 
$\deg\Delta_Z=5$ and $\Delta_Z\subset\sigma_\infty^Z\setminus\Gamma_{P, 1}$. 

\smallskip

\item[{$[$1;2,2$]_{0B}$}]
$E_X=2\sigma_\infty$, $\deg\Delta_X=1$, 
$|\Delta_X|=\{P\}$ with $P\in\sigma_\infty$, 
$\deg\Delta_Z=6$, $\mult_Q\Delta_Z=\mult_Q(\Delta_Z\cap\Gamma_{P, 1})=2$, 
$\Delta_Z\setminus\{Q\}\subset\sigma_\infty^Z$, where 
$Q=\sigma_\infty\cap\Gamma_{P, 1}$. 

\smallskip

\item[{$[$1;2,2$]_{0C}$}]
$E_X=2\sigma_\infty$, $\Delta_X=\emptyset$, 
$\deg\Delta_Z=7$ and $\Delta_Z\subset\sigma_\infty^Z$. 

\smallskip

\item[{$[$1;2,2$]_{1A}$}]
$E_X=2\sigma+2l$, $\deg\Delta_X=4$, 
$\Delta_X\cap\sigma=\emptyset$, 
$|\Delta_X|=\{P_1, P_2\}$, 
$(\mult_{P_i}\Delta_X, \mult_{P_i}(\Delta_X\cap l_i))=(2, 1)$ for $i=1$, $2$, 
$\deg\Delta_Z=3$ 
and 
$\Delta_Z\subset\sigma^Z\setminus l^Z$. 

\smallskip

\item[{$[$1;2,2$]_{1B}$}]
$E_X=2\sigma+2l$, $\deg\Delta_X=4$, 
$\Delta_X\cap\sigma=\emptyset$, 
$|\Delta_X|=\{P\}$, 
$(\mult_P\Delta_X, \mult_P(\Delta_X\cap l))=(4, 2)$, 
$\deg\Delta_Z=3$, 
$\Delta_Z\subset\sigma^Z\setminus l^Z$. 

\smallskip

\item[{$[$1;2,2$]_{1C}$}]
$E_X=2\sigma+2l$, $\deg\Delta_X=2$, 
$\Delta_X\cap\sigma=\emptyset$, 
$|\Delta_X|=\{P\}$, 
$\Delta_X\subset l$, 
$\deg\Delta_Z=5$, 
$\deg(\Delta_Z\cap\sigma^Z)=3$, $\deg(\Delta_Z\cap\Gamma_{P, 2})=2$ 
and 
$\Delta_Z\subset(\sigma^Z\cup\Gamma_{P, 2})\setminus(l^Z\cup\Gamma_{P, 1})$. 

\smallskip

\item[{$[$1;2,2$]_{1D}$({\bi c,d})} $((c, d)=(0, 0)$, $(1, 1)$, $(2, 1)$, $(3, 1)$, $(1, 2))$]
$E_X=2\sigma+2l$, $\deg\Delta_X=2$, 
$|\Delta_X|=\{P\}$, 
$(\mult_P\Delta_X, \mult_P(\Delta_X\cap l))=(2, 1)$, 
$\Delta_X\cap\sigma=\emptyset$, 
$\deg\Delta_Z=5$, 
$\Delta_Z\cap\Gamma_{P, 1}=\emptyset$, 
$\deg(\Delta_Z\cap\sigma^Z)=3$, $\deg(\Delta_Z\cap l^Z)=2$, 
$\mult_Q(\Delta_Z\cap\sigma^Z)=c$, $\mult_Q(\Delta_Z\cap l^Z)=d$
and $\mult_Q\Delta_Z=c+d$, where $Q=\sigma^Z\cap l^Z$. 

\smallskip

\item[{$[$1;2,2$]_{1E}$({\bi c,d})} $((c, d)=(0, 0)$, $(1, 1)$, $(2, 1)$, $(3, 1))$]
$E_X=2\sigma+2l$, $\deg\Delta_X=1$, 
$|\Delta_X|=\{P\}$, 
$P\in l\setminus\sigma$, 
$\deg\Delta_Z=6$, 
$\deg(\Delta_Z\cap\sigma^Z)=3$, $\deg(\Delta_Z\cap l^Z)=2$, 
$\mult_{Q_1}\Delta_Z=\mult_{Q_1}(\Delta_Z\cap\Gamma_{P, 1})=2$, 
$\mult_{Q_0}(\Delta_Z\cap\sigma^Z)=c$, $\mult_{Q_0}(\Delta_Z\cap l^Z)=d$
and $\mult_{Q_0}\Delta_Z=c+d$, where 
$Q_0=\sigma^Z\cap l^Z$ and $Q_1=l^Z\cap\Gamma_{P, 1}$. 

\smallskip

\item[{$[$1;2,2$]_{1F}$({\bi c,d})} $((c, d)=(0, 0)$, $(1, 1),\dots,(3, 1)$, $(1, 2),\dots,(1, 4))$]
$E_X=2\sigma+2l$, $\Delta_X=\emptyset$, 
$\deg\Delta_Z=7$, 
$\deg(\Delta_Z\cap\sigma^Z)=3$, $\deg(\Delta_Z\cap l^Z)=4$, 
$\mult_Q(\Delta_Z\cap\sigma^Z)=c$, $\mult_Q(\Delta_Z\cap l^Z)=d$
and $\mult_Q\Delta_Z=c+d$, where $Q=\sigma^Z\cap l^Z$. 

\smallskip

\item[{$[$1;2,2$]_{2A}$}]
$E_X=2\sigma+l_1+l_2$ $(l_1$, $l_2:$ distinct fibers$)$, $\deg\Delta_X=4$, 
$\Delta_X\cap\sigma=\emptyset$, 
$\deg(\Delta_X\cap l_1)=\deg(\Delta_X\cap l_2)=2$, 
$\deg\Delta_Z=3$ and 
$\Delta_Z\subset\sigma^Z\setminus(l_1^Z\cup l_2^Z)$. 

\smallskip

\item[{$[$1;2,2$]_{2B}$}]
$E_X=2\sigma+l_1+l_2$ $(l_1$, $l_2:$ distinct fibers$)$, $\deg\Delta_X=3$, 
$\Delta_X\cap\sigma=\emptyset$, 
$\deg(\Delta_X\cap l_1)=1$, $\deg(\Delta_X\cap l_2)=2$, 
$\deg\Delta_Z=4$, 
$\mult_Q\Delta_Z=\mult_Q(\Delta_Z\cap l_1^Z)=2$
and 
$\Delta_Z\setminus\{Q\}\subset\sigma^Z\setminus l_2^Z$, 
where $Q=\sigma^Z\cap l_1^Z$. 

\smallskip

\item[{$[$1;2,2$]_{2C}$}]
$E_X=2\sigma+l_1+l_2$ $(l_1$, $l_2:$ distinct fibers$)$, $\deg\Delta_X=2$, 
$\Delta_X\cap\sigma=\emptyset$, 
$\deg(\Delta_X\cap l_1)=\deg(\Delta_X\cap l_2)=1$, 
$\deg\Delta_Z=5$, 
$\mult_{Q_i}\Delta_Z=\mult_{Q_i}(\Delta_Z\cap l_i^Z)=2$ for $i=1$, $2$
and 
$\Delta_Z\setminus\{Q_1, Q_2\}\subset\sigma^Z$, 
where $Q_i=\sigma^Z\cap l_i^Z$. 
\end{description}

\smallskip

The case $X=\F_2:$

\begin{description}
\item[{$[$2;2,0$]$}]
$E_X=2\sigma$, $\Delta_X=\emptyset$, 
$\deg\Delta_Z=4$ 
and 
$\Delta_Z\subset\sigma^Z$. 

\smallskip

\item[{$[$2;2,1$]_{1A}$}]
$E_X=2\sigma+l$, $\deg\Delta_X=2$, 
$\Delta_X\subset l\setminus\sigma$, 
$\deg\Delta_Z=3$ 
and 
$\Delta_Z\subset\sigma^Z\setminus l^Z$. 

\smallskip

\item[{$[$2;2,1$]_{1B}$}]
$E_X=2\sigma+l$, $\deg\Delta_X=1$, 
$\Delta_X\subset l\setminus\sigma$, 
$\deg\Delta_Z=4$, 
$\mult_Q\Delta_Z=\mult_Q(\Delta_Z\cap l^Z)=2$
and 
$\Delta_Z\setminus\{Q\}\subset\sigma^Z$, 
where $Q=\sigma^Z\cap l^Z$. 

\smallskip

\item[{$[$2;2,2$]_{1A}$}]
$E_X=2\sigma+2l$, $\deg\Delta_X=4$, 
$\Delta_X\cap\sigma=\emptyset$, 
$|\Delta_X|=\{P_1, P_2\}$, 
$(\mult_{P_i}\Delta_X, \mult_{P_i}(\Delta_X\cap l))=(2, 1)$ for $i=1$, $2$, 
$\deg\Delta_Z=2$ 
and 
$\Delta_Z\subset\sigma^Z\setminus l^Z$. 

\smallskip

\item[{$[$2;2,2$]_{1B}$}]
$E_X=2\sigma+2l$, $\deg\Delta_X=4$, 
$\Delta_X\cap\sigma=\emptyset$, 
$|\Delta_X|=\{P\}$, 
$(\mult_P\Delta_X, \mult_P(\Delta_X\cap l))=(4, 2)$, 
$\deg\Delta_Z=2$,
$\Delta_Z\subset\sigma^Z\setminus l^Z$. 

\smallskip

\item[{$[$2;2,2$]_{1C}$}]
$E_X=2\sigma+2l$, $\deg\Delta_X=2$, 
$\Delta_X\cap\sigma=\emptyset$, 
$|\Delta_X|=\{P\}$, 
$\Delta_X\subset l$, 
$\deg\Delta_Z=4$, 
$\deg(\Delta_Z\cap\sigma^Z)=2$, $\deg(\Delta_Z\cap\Gamma_{P, 2})=2$ 
and 
$\Delta_Z\subset(\sigma^Z\cup\Gamma_{P, 2})\setminus(l^Z\cup\Gamma_{P, 1})$. 

\smallskip

\item[{$[$2;2,2$]_{1D}$({\bi c,d})} $((c, d)=(0, 0)$, $(1, 1)$, $(2, 1)$, $(1, 2))$]
$E_X=2\sigma+2l$, $\deg\Delta_X=2$, 
$|\Delta_X|=\{P\}$, 
$(\mult_{P}\Delta_X, \mult_{P}(\Delta_X\cap l_1))=(2, 1)$, 
$\Delta_X\cap\sigma=\emptyset$, 
$\deg\Delta_Z=4$, 
$\Delta_Z\cap\Gamma_{P, 1}=\emptyset$, 
$\deg(\Delta_Z\cap\sigma^Z)=\deg(\Delta_Z\cap l^Z)=2$, 
$\mult_Q(\Delta_Z\cap\sigma^Z)=c$, $\mult_Q(\Delta_Z\cap l^Z)=d$
and $\mult_Q\Delta_Z=c+d$, where $Q=\sigma^Z\cap l^Z$. 

\smallskip

\item[{$[$2;2,2$]_{1E}$({\bi c,d})} $((c, d)=(0, 0)$, $(1, 1)$, \,{\it or}\, $(2, 1))$]
$E_X=2\sigma+2l$, 
$|\Delta_X|=\{P\}$, $\deg\Delta_X=1$, 
$P\in l\setminus\sigma$, 
$\deg\Delta_Z=5$, 
$\deg(\Delta_Z\cap\sigma^Z)=\deg(\Delta_Z\cap l^Z)=2$, 
$\mult_{Q_1}\Delta_Z=\mult_{Q_1}(\Delta_Z\cap\Gamma_{P, 1})=2$, 
$\mult_{Q_0}(\Delta_Z\cap\sigma^Z)=c$, $\mult_{Q_0}(\Delta_Z\cap l^Z)=d$
and $\mult_{Q_0}\Delta_Z=c+d$, where $Q_0=\sigma^Z\cap l^Z$ and 
$Q_1=l^Z\cap\Gamma_{P, 1}$. 

\smallskip

\item[{$[$2;2,2$]_{1F}$({\bi c,d})} $((c, d)=(0, 0)$, $(1, 1),\dots,(1, 4)$, 
\,\it{or}\, $(2, 1))$]
$E_X=2\sigma+2l$, $\Delta_X=\emptyset$, 
$\deg\Delta_Z=6$, 
$\deg(\Delta_Z\cap\sigma^Z)=2$, $\deg(\Delta_Z\cap l^Z)=4$, 
$\mult_Q(\Delta_Z\cap\sigma^Z)=c$, $\mult_Q(\Delta_Z\cap l^Z)=d$
and $\mult_Q\Delta_Z=c+d$, where $Q=\sigma^Z\cap l^Z$. 

\smallskip

\item[{$[$2;2,2$]_{2A}$}]
$E_X=2\sigma+l_1+l_2$ $(l_1$, $l_2:$ distinct fibers$)$, $\deg\Delta_X=4$, 
$\Delta_X\cap\sigma=\emptyset$, 
$\deg(\Delta_X\cap l_1)=\deg(\Delta_X\cap l_2)=2$, 
$\deg\Delta_Z=2$ and 
$\Delta_Z\subset\sigma^Z\setminus(l_1^Z\cup l_2^Z)$. 

\smallskip

\item[{$[$2;2,2$]_{2B}$}]
$E_X=2\sigma+l_1+l_2$ $(l_1$, $l_2:$ distinct fibers$)$, $\deg\Delta_X=3$, 
$\Delta_X\cap\sigma=\emptyset$, 
$\deg(\Delta_X\cap l_1)=1$, $\deg(\Delta_X\cap l_2)=2$, 
$\deg\Delta_Z=3$, 
$\mult_Q\Delta_Z=\mult_Q(\Delta_Z\cap l_1^Z)=2$
and 
$\Delta_Z\setminus\{Q\}\subset\sigma^Z\setminus l_2^Z$, 
where $Q=\sigma^Z\cap l_1^Z$. 

\smallskip

\item[{$[$2;2,2$]_{2C}$}]
$E_X=2\sigma+l_1+l_2$ $(l_1$, $l_2:$ distinct fibers$)$, $\deg\Delta_X=2$, 
$\Delta_X\cap\sigma=\emptyset$, 
$\deg(\Delta_X\cap l_1)=\deg(\Delta_X\cap l_2)=1$, 
$\deg\Delta_Z=4$ and 
$\mult_{Q_i}\Delta_Z=\mult_{Q_i}(\Delta_Z\cap l_i^Z)=2$ for $i=1$, $2$, 
where $Q_i=\sigma^Z\cap l_i^Z$. 

\smallskip

\item[{$[$2;2,3$]_V$}]
$E_X=\sigma+C$ with $C:$ nonsingular, $C\sim\sigma+3l$, 
$\deg\Delta_X=6$, $\Delta_X\subset C\setminus\sigma$, 
$\deg\Delta_Z=1$ 
and 
$Q\in\Delta_Z$, 
where $Q=\sigma^Z\cap C^Z$. 

\smallskip

\item[{$[$2;2,3$]_H\langle$0$\rangle$}]
$E_X=\sigma+\sigma_\infty+l$, 
$\deg\Delta_X=5$, $P\not\in\Delta_X$, $\Delta_X\cap\sigma=\emptyset$, 
$\deg(\Delta_X\cap\sigma_\infty)=4$ and $\deg(\Delta_X\cap l)=1$, 
where $P=\sigma_\infty\cap l$.
$\deg\Delta_Z=2$ 
and 
$|\Delta_Z|=\{Q, Q_\infty\}$, 
where $Q=\sigma^Z\cap l^Z$ and $Q_\infty=\sigma_\infty^Z\cap l^Z$. 

\smallskip

\item[{$[$2;2,3$]_H\langle$1$\rangle$}]
$E_X=\sigma+\sigma_\infty+l$, 
$\deg\Delta_X=4$, $\mult_P\Delta_X=1$, $\Delta_X\subset\sigma_\infty$, 
where $P=\sigma_\infty\cap l$.
$\deg\Delta_Z=3$ 
and 
$Q_1$, $Q_2$, $Q_3\in\Delta_Z$, 
where $Q_1=\sigma^Z\cap l^Z$, $Q_2=\sigma_\infty^Z\cap\Gamma_{P,1}$ and 
$Q_3=l^Z\cap\Gamma_{P,1}$. 

\smallskip

\item[{$[$2;2,3$]_{2A1}$}]
$E_X=2\sigma+2l_1+l_2$ $(l_1$, $l_2:$ distinct fibers$)$, $\deg\Delta_X=6$, 
$\Delta_X\cap\sigma=\emptyset$, 
$|\Delta_X|\cap l_1=\{P_1, P_2\}$, 
$(\mult_{P_i}\Delta_X, \mult_{P_i}(\Delta_X\cap l_i))=(2, 1)$ for $i=1$, $2$, 
$\deg(\Delta_X\cap l_2)=2$, 
$\deg\Delta_Z=1$ and 
$\Delta_Z\subset\sigma^Z\setminus(l_1^Z\cup l_2^Z)$. 

\smallskip

\item[{$[$2;2,3$]_{2A2}$}]
$E_X=2\sigma+2l_1+l_2$ $(l_1$, $l_2:$ distinct fibers$)$, $\deg\Delta_X=5$, 
$\Delta_X\cap\sigma=\emptyset$, 
$|\Delta_X|\cap l_1=\{P_1, P_2\}$, 
$(\mult_{P_i}\Delta_X, \mult_{P_i}(\Delta_X\cap l_i))=(2, 1)$ for $i=1$, $2$, 
$\deg(\Delta_X\cap l_2)=1$, 
$\deg\Delta_Z=2$ and 
$\mult_Q\Delta_Z=\mult_Q(\Delta_Z\cap l_2^Z)=2$, 
where $Q=\sigma^Z\cap l_2^Z$. 

\smallskip

\item[{$[$2;2,3$]_{2B1}$}]
$E_X=2\sigma+2l_1+l_2$ $(l_1$, $l_2:$ distinct fibers$)$, $\deg\Delta_X=6$, 
$\Delta_X\cap\sigma=\emptyset$, 
$|\Delta_X|\cap l_1=\{P\}$, 
$(\mult_P\Delta_X, \mult_P(\Delta_X\cap l_1))=(4, 2)$, 
$\deg(\Delta_X\cap l_2)=2$, 
$\deg\Delta_Z=1$ and 
$\Delta_Z\subset\sigma^Z\setminus(l_1^Z\cup l_2^Z)$. 

\smallskip

\item[{$[$2;2,3$]_{2B2}$}]
$E_X=2\sigma+2l_1+l_2$ $(l_1$, $l_2:$ distinct fibers$)$, $\deg\Delta_X=5$, 
$\Delta_X\cap\sigma=\emptyset$, 
$|\Delta_X|\cap l_1=\{P\}$, 
$(\mult_P\Delta_X, \mult_P(\Delta_X\cap l_1))=(4, 2)$, 
$\deg(\Delta_X\cap l_2)=1$, 
$\deg\Delta_Z=2$ and 
$\mult_Q\Delta_Z=\mult_Q(\Delta_Z\cap l_2^Z)=2$, 
where $Q=\sigma^Z\cap l_2^Z$. 

\smallskip

\item[{$[$2;2,3$]_{2C1}$}]
$E_X=2\sigma+2l_1+l_2$ $(l_1$, $l_2:$ distinct fibers$)$, $\deg\Delta_X=4$, 
$\Delta_X\cap\sigma=\emptyset$, 
$|\Delta_X|\cap l_1=\{P\}$, 
$(\mult_P\Delta_X, \mult_P(\Delta_X\cap l_1))=(2, 2)$, 
$\deg(\Delta_X\cap l_2)=2$, 
$\deg\Delta_Z=3$, 
$\deg(\Delta_Z\cap\sigma^Z)=1$, $\deg(\Delta_Z\cap\Gamma_{P, 2})=2$ 
and 
$\Delta_Z\cap(l_1^Z\cup l_2^Z\cup\Gamma_{P, 1})=\emptyset$. 

\smallskip

\item[{$[$2;2,3$]_{2C2}$}]
$E_X=2\sigma+2l_1+l_2$ $(l_1$, $l_2:$ distinct fibers$)$, $\deg\Delta_X=3$, 
$\Delta_X\cap\sigma=\emptyset$, 
$|\Delta_X|\cap l_1=\{P\}$, 
$(\mult_{P}\Delta_X, \mult_{P}(\Delta_X\cap l_1))=(2, 2)$, 
$\deg(\Delta_X\cap l_2)=1$, 
$\deg\Delta_Z=4$, 
$\mult_Q\Delta_Z=\mult_Q(\Delta_Z\cap l_2^Z)=2$, 
$\deg(\Delta_Z\cap\Gamma_{P, 2})=2$ and 
$\Delta_Z\cap(l_1^Z\cup\Gamma_{P, 1})=\emptyset$, 
where $Q=\sigma^Z\cap l_2^Z$. 

\smallskip

\item[{$[$2;2,3$]_{2D1}$({\bi c,d})} $((c, d)=(0, 0)$, $(1, 1)$, $(1, 2))$]
$E_X=2\sigma+2l_1+l_2$ $(l_1$, $l_2:$ distinct fibers$)$, $\deg\Delta_X=4$, 
$\Delta_X\cap\sigma=\emptyset$, 
$|\Delta_X|\cap l_1=\{P\}$, 
$(\mult_P\Delta_X, \mult_P(\Delta_X\cap l_1))=(2, 1)$, 
$\deg(\Delta_X\cap l_2)=2$, 
$\deg\Delta_Z=3$, 
$\deg(\Delta_Z\cap\sigma^Z)=1$, $\deg(\Delta_Z\cap l_1^Z)=2$, 
$\mult_Q(\Delta_Z\cap\sigma^Z)=c$, 
$\mult_Q(\Delta_Z\cap l_1^Z)=d$, 
$\mult_Q\Delta_Z=c+d$, 
and 
$\Delta_Z\cap(l_2^Z\cup\Gamma_{P, 1})=\emptyset$, 
where $Q=\sigma^Z\cap l_1^Z$. 

\smallskip

\item[{$[$2;2,3$]_{2D2}$}]
$E_X=2\sigma+2l_1+l_2$ $(l_1$, $l_2:$ distinct fibers$)$, $\deg\Delta_X=3$, 
$\Delta_X\cap\sigma=\emptyset$, 
$|\Delta_X|\cap l_1=\{P\}$, 
$(\mult_P\Delta_X, \mult_P(\Delta_X\cap l_1))=(2, 1)$, 
$\deg(\Delta_X\cap l_2)=1$, 
$\deg\Delta_Z=4$, 
$\mult_Q\Delta_Z=\mult_Q(\Delta_Z\cap l_2^Z)=2$, 
$\deg(\Delta_Z\cap l_1^Z)=2$ and 
$\Delta_Z\cap l_1^Z\cap(\sigma^Z\cup\Gamma_{P, 1})=\emptyset$, 
where $Q=\sigma^Z\cap l_2^Z$. 

\smallskip

\item[{$[$2;2,3$]_{2E1}$({\bi c,d})} $((c, d)=(0, 0)$, $(1, 1))$]
$E_X=2\sigma+2l_1+l_2$ $(l_1$, $l_2:$ distinct fibers$)$, $\deg\Delta_X=3$, 
$\Delta_X\cap\sigma=\emptyset$, 
$|\Delta_X|\cap l_1=\{P\}$, 
$\deg(\Delta_X\cap l_2)=2$, 
$\deg\Delta_Z=4$, 
$\deg(\Delta_Z\cap\sigma^Z)=1$, $\deg(\Delta_Z\cap l_1^Z)=2$, 
$\Delta_Z\cap l_2^Z=\emptyset$, 
$\mult_{Q_1}\Delta_Z=\mult_{Q_1}(\Delta_Z\cap\Gamma_{P, 1})=2$, 
$\mult_{Q_2}(\Delta_Z\cap\sigma^Z)=c$, 
$\mult_{Q_2}(\Delta_Z\cap l_1^Z)=d$ and 
$\mult_{Q_2}\Delta_Z=c+d$, 
where $Q_1=l_1^Z\cap\Gamma_{P, 1}$ and 
$Q_2=\sigma^Z\cap l_1^Z$. 

\smallskip

\item[{$[$2;2,3$]_{2E2}$}]
$E_X=2\sigma+2l_1+l_2$ $(l_1$, $l_2:$ distinct fibers$)$, $\deg\Delta_X=2$, 
$\Delta_X\cap\sigma=\emptyset$, 
$|\Delta_X|\cap l_1=\{P\}$, 
$\deg(\Delta_X\cap l_2)=1$, 
$\deg\Delta_Z=5$, 
$\mult_{Q_1}\Delta_Z=\mult_{Q_1}(\Delta_Z\cap\Gamma_{P, 1})=2$, 
$\mult_{Q_2}\Delta_Z=\mult_{Q_2}(\Delta_Z\cap l_2^Z)=2$, 
$\deg(\Delta_Z\cap l_1^Z)=2$ and 
$\Delta_Z\cap\sigma^Z\cap l_1^Z=\emptyset$, 
where $Q_1=l_1^Z\cap\Gamma_{P, 1}$ and
$Q_2=\sigma^Z\cap l_2^Z$. 

\smallskip

\item[{$[$2;2,3$]_{2F1}$({\bi c,d})} $((c, d)=(0, 0)$, $(1, 1),\dots,(1, 4))$]
$E_X=2\sigma+2l_1+l_2$ $(l_1$, $l_2:$ distinct fibers$)$, $\deg\Delta_X=2$, 
$\Delta_X\subset l_2\setminus\sigma$, 
$\deg\Delta_Z=5$, 
$\deg(\Delta_Z\cap\sigma^Z)=1$, $\deg(\Delta_Z\cap l_1^Z)=4$, 
$\Delta_Z\cap l_2^Z=\emptyset$, 
$\mult_Q(\Delta_Z\cap\sigma^Z)=c$, 
$\mult_Q(\Delta_Z\cap l_1^Z)=d$ and 
$\mult_Q\Delta_Z=c+d$, 
where $Q=\sigma^Z\cap l_1^Z$. 

\smallskip

\item[{$[$2;2,3$]_{2F2}$}]
$E_X=2\sigma+2l_1+l_2$ $(l_1$, $l_2:$ distinct fibers$)$, $\deg\Delta_X=1$, 
$\Delta_X\subset l_2\setminus\sigma$, 
$\deg\Delta_Z=6$, 
$\mult_Q\Delta_Z=\mult_Q(\Delta_Z\cap l_2^Z)=2$, 
$\deg(\Delta_Z\cap l_1^Z)=4$ and 
$\Delta_Z\cap\sigma^Z\cap l_1^Z=\emptyset$, 
where $Q=\sigma^Z\cap l_2^Z$. 

\smallskip

\item[{$[$2;2,3$]_{3A}$}]
$E_X=2\sigma+l_1+l_2+l_3$ $(l_1$, $l_2$, $l_3:$ distinct fibers$)$, $\deg\Delta_X=6$, 
$\Delta_X\cap\sigma=\emptyset$, 
$\deg(\Delta_X\cap l_i)=2$ for $i=1$, $2$, $3$, 
$\deg\Delta_Z=1$ and 
$\Delta_Z\subset\sigma^Z\setminus(l_1^Z\cup l_2^Z\cup l_3^Z)$. 

\smallskip

\item[{$[$2;2,3$]_{3B}$}]
$E_X=2\sigma+l_1+l_2+l_3$ $(l_1$, $l_2$, $l_3:$ distinct fibers$)$, $\deg\Delta_X=5$, 
$\Delta_X\cap\sigma=\emptyset$, 
$\deg(\Delta_X\cap l_i)=2$ for $i=1$, $2$, 
$\deg(\Delta_X\cap l_3)=1$, 
$\deg\Delta_Z=2$ and 
$\mult_Q\Delta_Z=\mult_Q(\Delta_Z\cap l_3^Z)=2$, 
where 
$Q=\sigma^Z\cap l_3^Z$. 
\end{description}

\smallskip

The case $X=\F_3:$

\begin{description}
\item[{$[$3;2,0$]$}]
$E_X=2\sigma$, $\Delta_X=\emptyset$, 
$\deg\Delta_Z=3$ 
and 
$\Delta_Z\subset\sigma^Z$. 

\smallskip

\item[{$[$3;2,1$]_{1A}$}]
$E_X=2\sigma+l$, $\deg\Delta_X=2$, 
$\Delta_X\subset l\setminus\sigma$, 
$\deg\Delta_Z=2$ 
and 
$\Delta_Z\subset\sigma^Z\setminus l^Z$. 

\smallskip

\item[{$[$3;2,1$]_{1B}$}]
$E_X=2\sigma+l$, $\deg\Delta_X=1$, 
$\Delta_X\subset l\setminus\sigma$, 
$\deg\Delta_Z=3$, 
$\mult_Q\Delta_Z=\mult_Q(\Delta_Z\cap l^Z)=2$
and 
$\Delta_Z\setminus\{Q\}\subset\sigma^Z$, 
where $Q=\sigma^Z\cap l^Z$. 

\smallskip

\item[{$[$3;2,2$]_{1A}$}]
$E_X=2\sigma+2l$, $\deg\Delta_X=4$, 
$\Delta_X\cap\sigma=\emptyset$, 
$|\Delta_X|=\{P_1, P_2\}$, 
$(\mult_{P_i}\Delta_X, \mult_{P_i}(\Delta_X\cap l))=(2, 1)$ for $i=1$, $2$, 
$\deg\Delta_Z=1$ 
and 
$\Delta_Z\subset\sigma^Z\setminus l^Z$. 

\smallskip

\item[{$[$3;2,2$]_{1B}$}]
$E_X=2\sigma+2l$, $\deg\Delta_X=4$, 
$\Delta_X\cap\sigma=\emptyset$, 
$|\Delta_X|=\{P\}$, 
$(\mult_P\Delta_X, \mult_P(\Delta_X\cap l))=(4, 2)$, 
$\deg\Delta_Z=1$ 
and 
$\Delta_Z\subset\sigma^Z\setminus l^Z$. 

\smallskip

\item[{$[$3;2,2$]_{1C}$}]
$E_X=2\sigma+2l$, $\deg\Delta_X=2$, 
$\Delta_X\cap\sigma=\emptyset$, 
$|\Delta_X|=\{P\}$, 
$\Delta_X\subset l$, 
$\deg\Delta_Z=3$, 
$\deg(\Delta_Z\cap\sigma^Z)=1$, $\deg(\Delta_Z\cap\Gamma_{P, 2})=2$ 
and 
$\Delta_Z\subset(\sigma^Z\cup\Gamma_{P, 2})\setminus(l^Z\cup\Gamma_{P, 1})$. 

\smallskip

\item[{$[$3;2,2$]_{1D}$({\bi c,d})} $((c, d)=(0, 0)$, $(1, 1)$, \,{\it or}\, $(1, 2))$]
$E_X=2\sigma+2l$, 
$|\Delta_X|=\{P\}$, $\deg\Delta_X=2$, 
$(\mult_P\Delta_X, \mult_P(\Delta_X\cap l))=(2, 1)$, 
$\Delta_X\cap\sigma=\emptyset$, 
$\deg\Delta_Z=3$, 
$\Delta_Z\cap\Gamma_{P, 1}=\emptyset$, 
$\deg(\Delta_Z\cap\sigma^Z)=1$, $\deg(\Delta_Z\cap l^Z)=2$, 
$\mult_Q(\Delta_Z\cap\sigma^Z)=c$, $\mult_Q(\Delta_Z\cap l^Z)=d$
and $\mult_Q\Delta_Z=c+d$, where $Q=\sigma^Z\cap l^Z$. 

\smallskip

\item[{$[$3;2,2$]_{1E}$({\bi c,d})} $((c, d)=(0, 0)$, $(1, 1))$]
$E_X=2\sigma+2l$, 
$|\Delta_X|=\{P\}$, $\deg\Delta_X=1$, 
$P\in l\setminus\sigma$, 
$\deg\Delta_Z=4$, 
$\deg(\Delta_Z\cap\sigma^Z)=1$, $\deg(\Delta_Z\cap l^Z)=2$, 
$\mult_{Q_1}\Delta_Z=\mult_{Q_1}(\Delta_Z\cap\Gamma_{P, 1})=2$, 
$\mult_{Q_2}(\Delta_Z\cap\sigma^Z)=c$, $\mult_{Q_2}(\Delta_Z\cap l^Z)=d$
and $\mult_{Q_2}\Delta_Z=c+d$, where $Q_1=l^Z\cap\Gamma_{P_1, 1}$ 
and $Q_2=\sigma^Z\cap l^Z$. 

\smallskip

\item[{$[$3;2,2$]_{1F}$({\bi c,d})} $((c, d)=(0, 0)$, $(1, 1),\dots,(1, 4))$]
$E_X=2\sigma+2l$, $\Delta_X=\emptyset$, 
$\deg\Delta_Z=5$, 
$\deg(\Delta_Z\cap\sigma^Z)=1$, $\deg(\Delta_Z\cap l^Z)=4$, 
$\mult_Q(\Delta_Z\cap\sigma^Z)=c$, $\mult_Q(\Delta_Z\cap l^Z)=d$
and $\mult_Q\Delta_Z=c+d$, where $Q=\sigma^Z\cap l^Z$. 

\smallskip

\item[{$[$3;2,2$]_{2A}$}]
$E_X=2\sigma+l_1+l_2$ $(l_1$, $l_2:$ distinct fibers$)$, $\deg\Delta_X=4$, 
$\Delta_X\cap\sigma=\emptyset$, 
$\deg(\Delta_X\cap l_1)=\deg(\Delta_X\cap l_2)=2$, 
$\deg\Delta_Z=1$ and 
$\Delta_Z\subset\sigma^Z\setminus(l_1^Z\cup l_2^Z)$. 

\smallskip

\item[{$[$3;2,2$]_{2B}$}]
$E_X=2\sigma+l_1+l_2$ $(l_1$, $l_2:$ distinct fibers$)$, $\deg\Delta_X=3$, 
$\Delta_X\cap\sigma=\emptyset$, 
$\deg(\Delta_X\cap l_1)=1$, $\deg(\Delta_X\cap l_2)=2$, 
$\deg\Delta_Z=2$ and 
$\mult_Q\Delta_Z=\mult_Q(\Delta_Z\cap l_1^Z)=2$, 
where $Q=\sigma^Z\cap l_1^Z$. 

\smallskip

\item[{$[$3;2,3$]_0$}]
$E_X=\sigma+\sigma_\infty$, 
$\deg\Delta_X=6$, $\Delta_X\subset\sigma_\infty$ and $\Delta_Z=\emptyset$. 

\smallskip

\item[{$[$3;2,3$]_{2A}$}]
$E_X=2\sigma+2l_1+l_2$ $(l_1$, $l_2:$ distinct fibers$)$, $\deg\Delta_X=6$, 
$\Delta_X\cap\sigma=\emptyset$, 
$|\Delta_X|\cap l_1=\{P_1, P_2\}$, 
$(\mult_{P_i}\Delta_X, \mult_{P_i}(\Delta_X\cap l_1))=(2, 1)$ for $i=1$, $2$, 
$\deg(\Delta_X\cap l_2)=2$ and 
$\Delta_Z=\emptyset$. 

\smallskip

\item[{$[$3;2,3$]_{2B}$}]
$E_X=2\sigma+2l_1+l_2$ $(l_1$, $l_2:$ distinct fibers$)$, $\deg\Delta_X=6$, 
$\Delta_X\cap\sigma=\emptyset$, 
$|\Delta_X|\cap l_1=\{P\}$, 
$(\mult_P\Delta_X, \mult_P(\Delta_X\cap l_1))=(4, 2)$, 
$\deg(\Delta_X\cap l_2)=2$ and 
$\Delta_Z=\emptyset$. 

\smallskip

\item[{$[$3;2,3$]_{2C}$}]
$E_X=2\sigma+2l_1+l_2$ $(l_1$, $l_2:$ distinct fibers$)$, $\deg\Delta_X=4$, 
$\Delta_X\cap\sigma=\emptyset$, 
$|\Delta_X|\cap l_1=\{P\}$, 
$(\mult_P\Delta_X, \mult_P(\Delta_X\cap l_1))=(2, 2)$, 
$\deg(\Delta_X\cap l_2)=2$, 
$\deg\Delta_Z=2$ and 
$\Delta_Z\subset\Gamma_{P, 2}\setminus(l_1^Z\cup\Gamma_{P, 1})$. 

\smallskip

\item[{$[$3;2,3$]_{2D}$}]
$E_X=2\sigma+2l_1+l_2$ $(l_1$, $l_2:$ distinct fibers$)$, $\deg\Delta_X=4$, 
$\Delta_X\cap\sigma=\emptyset$, 
$|\Delta_X|\cap l_1=\{P\}$, 
$(\mult_P\Delta_X, \mult_P(\Delta_X\cap l_1))=(2, 1)$, 
$\deg(\Delta_X\cap l_2)=2$, 
$\deg\Delta_Z=2$ and 
$\Delta_Z\subset l_1^Z\setminus(\sigma^Z\cup\Gamma_{P, 1})$. 

\smallskip

\item[{$[$3;2,3$]_{2E}$}]
$E_X=2\sigma+2l_1+l_2$ $(l_1$, $l_2:$ distinct fibers$)$, $\deg\Delta_X=3$, 
$\Delta_X\cap\sigma=\emptyset$, 
$|\Delta_X|\cap l_1=\{P\}$, 
$\deg(\Delta_X\cap l_2)=2$, 
$\deg\Delta_Z=3$, 
$\mult_Q\Delta_Z=\mult_Q(\Delta_Z\cap\Gamma_{P, 1})=2$ 
and $\Delta_Z\setminus\{Q\}\subset l_1^Z\setminus\sigma^Z$, 
where $Q=l_1^Z\cap\Gamma_{P, 1}$. 

\smallskip

\item[{$[$3;2,3$]_{2F}$}]
$E_X=2\sigma+2l_1+l_2$ $(l_1$, $l_2:$ distinct fibers$)$, $\deg\Delta_X=2$, 
$\Delta_X\subset l_2\setminus\sigma$, 
$\deg\Delta_Z=4$ and $\Delta_Z\subset l_1^Z\setminus\sigma^Z$. 

\smallskip

\item[{$[$3;2,3$]_{3}$}]
$E_X=2\sigma+l_1+l_2+l_3$ $(l_1$, $l_2$, $l_3:$ distinct fibers$)$, $\deg\Delta_X=6$, 
$\Delta_X\cap\sigma=\emptyset$, 
$\deg(\Delta_X\cap l_i)=2$ for $i=1$, $2$, $3$ and 
$\Delta_Z=\emptyset$. 
\end{description}

\smallskip

The case $X=\F_4:$

\begin{description}
\item[{$[$4;2,0$]$}]
$E_X=2\sigma$, $\Delta_X=\emptyset$, 
$\deg\Delta_Z=2$ 
and 
$\Delta_Z\subset\sigma^Z$. 

\smallskip

\item[{$[$4;2,1$]_{1A}$}]
$E_X=2\sigma+l$, $\deg\Delta_X=2$, 
$\Delta_X\subset l\setminus\sigma$, 
$\deg\Delta_Z=1$ 
and 
$\Delta_Z\subset\sigma^Z\setminus l^Z$. 

\smallskip

\item[{$[$4;2,1$]_{1B}$}]
$E_X=2\sigma+l$, $\deg\Delta_X=1$, 
$\Delta_X\subset l\setminus\sigma$, 
$\deg\Delta_Z=2$ and 
$\mult_Q\Delta_Z=\mult_Q(\Delta_Z\cap l^Z)=2$, 
where $Q=\sigma^Z\cap l^Z$. 

\smallskip

\item[{$[$4;2,2$]_{1A}$}]
$E_X=2\sigma+2l$, $\deg\Delta_X=4$, 
$\Delta_X\cap\sigma=\emptyset$, 
$|\Delta_X|=\{P_1, P_2\}$, 
$(\mult_{P_i}\Delta_X, \mult_{P_i}(\Delta_X\cap l))=(2, 1)$ for $i=1$, $2$ 
and 
$\Delta_Z=\emptyset$. 

\smallskip

\item[{$[$4;2,2$]_{1B}$}]
$E_X=2\sigma+2l$, $\deg\Delta_X=4$, 
$\Delta_X\cap\sigma=\emptyset$, 
$|\Delta_X|=\{P\}$, 
$(\mult_P\Delta_X, \mult_P(\Delta_X\cap l))=(4, 2)$ and 
$\Delta_Z=\emptyset$. 

\smallskip

\item[{$[$4;2,2$]_{1C}$}]
$E_X=2\sigma+2l$, $\deg\Delta_X=2$, 
$\Delta_X\cap\sigma=\emptyset$, 
$|\Delta_X|=\{P\}$, 
$\Delta_X\subset l$, 
$\deg\Delta_Z=2$ 
and 
$\Delta_Z\subset\Gamma_{P, 2}\setminus(l^Z\cup\Gamma_{P, 1})$. 

\smallskip

\item[{$[$4;2,2$]_{1D}$}]
$E_X=2\sigma+2l$, 
$|\Delta_X|=\{P\}$, $\deg\Delta_X=2$, 
$\Delta_X\cap\sigma=\emptyset$, 
$(\mult_P\Delta_X, \mult_P(\Delta_X\cap l))=(2, 1)$, 
$\deg\Delta_Z=2$ and $\Delta_Z\subset l^Z\setminus(\sigma^Z\cup\Gamma_{P, 1})$. 

\smallskip

\item[{$[$4;2,2$]_{1E}$}]
$E_X=2\sigma+2l$, 
$|\Delta_X|=\{P\}$, $\deg\Delta_X=1$, 
$P\in l\setminus\sigma$, 
$\deg\Delta_Z=3$, 
$\mult_Q\Delta_Z=\mult_Q(\Delta_Z\cap\Gamma_{P, 1})=2$ and 
$\Delta_Z\setminus\{Q\}\subset l^Z\setminus\sigma^Z$, 
where $Q=l^Z\cap\Gamma_{P, 1}$. 

\smallskip

\item[{$[$4;2,2$]_{1F}$}]
$E_X=2\sigma+2l$, $\Delta_X=\emptyset$, 
$\deg\Delta_Z=4$ and $\Delta_Z\subset l^Z\setminus\sigma^Z$. 

\smallskip

\item[{$[$4;2,2$]_2$}]
$E_X=2\sigma+l_1+l_2$ $(l_1$, $l_2:$ distinct fibers$)$, $\deg\Delta_X=4$, 
$\Delta_X\cap\sigma=\emptyset$, 
$\deg(\Delta_X\cap l_1)=\deg(\Delta_X\cap l_2)=2$ and 
$\Delta_Z=\emptyset$. 
\end{description}

\smallskip

The case $X=\F_5:$

\begin{description}
\item[{$[$5;2,0$]$}]
$E_X=2\sigma$, $\Delta_X=\emptyset$, 
$\deg\Delta_Z=1$ 
and 
$\Delta_Z\subset\sigma^Z$. 

\smallskip

\item[{$[$5;2,1$]_1$}]
$E_X=2\sigma+l$, $\deg\Delta_X=2$, 
$\Delta_X\subset l\setminus\sigma$ and 
$\Delta_Z=\emptyset$. 
\end{description}

\smallskip

The case $X=\F_6:$

\begin{description}
\item[{$[$6;2,0$]$}]
$E_X=2\sigma$, $\Delta_X=\emptyset$ and 
$\deg\Delta_Z=\emptyset$. 
\end{description}
\end{thm}

We start to prove Theorem \ref{tetII_thm}. Any tetrad in Theorem \ref{tetII_thm} 
is a bottom tetrad by Proposition \ref{converse_prop}. 
We see the converse. 
Let $(X=\F_n, E_X; \Delta_Z, \Delta_X)$ be a 
bottom tetrad such that $2K_X+L_X$ is non-big and nontrivial, 
where $L_X$ is the fundamental divisor, 
$\psi\colon Z\to X$ be the elimination of $\Delta_X$, 
$\phi\colon M\to Z$ be the elimination of $\Delta_Z$, 
$E_Z:=(E_X)_Z^{\Delta_X, 1}$ and 
$E_M:=(E_Z)_M^{\Delta_Z, 2}$.
Set $L_X\sim h_0\sigma+hl$, $E_X\sim e_0\sigma+el$, 
$k_X:=\deg\Delta_X$ and $k_Z:=\deg\Delta_Z$. 
Then $e_0=6-h_0$ and $e=3(n+2)-h$. 
Since $2K_X+L_X$ is nef and non-big, we can assume that $h_0=4$. Thus $e_0=2$. 
We know that $k_X+k_Z=(L_X\cdot E_X)/2=2n-h+12$.

\begin{claim}\label{tetIII_claim}
We have $(n , h)=(0, 5)$, $(0, 6)$, $(1, 7)$, $(1, 8)$, $(1, 9)$, 
$(2, 9)$, $(2, 10)$, $(2, 11)$, $(2, 12)$, $(3, 12)$, $(3, 13)$, $(3, 14)$, $(3, 15)$, $(4, 16)$, 
$(4, 17)$, $(4, 18)$, $(5, 20)$, $(5, 21)$ or $(6, 24)$.
\end{claim}

\begin{proof}
Since $2K_X+L_X\sim(h-2n-4)l$ is nef and nontrivial, we have $h\geqslant 2n+5$. 
Moreover, $4n\leqslant h\leqslant 3n+6$ holds 
since $L_X$ is nef and $E_X$ is effective. In particular, $n\leqslant 6$. 
\end{proof}

We consider the case that $E_X$ contains an irreducible 
component $C$ such that $C$ is 
neither $\sigma$ nor $l$. Then one of the following holds: 

\begin{enumerate}
\renewcommand{\theenumi}{\arabic{enumi}}
\renewcommand{\labelenumi}{(\theenumi)}
\item\label{III_inf01}
$(n, h)=(0, 5)$ and $C\sim\sigma+l$.
\item\label{III_inf02}
$(n, h)=(0, 5)$ and $C\sim 2\sigma+l$.
\item\label{III_inf03}
$(n, h)=(1, 7)$ and $C=\sigma_\infty$.
\item\label{III_inf04}
$(n, h)=(1, 7)$ and $C\sim\sigma+2l$.
\item\label{III_inf05}
$(n, h)=(1, 7)$ and $C\sim 2\sigma+2l$.
\item\label{III_inf06}
$(n, h)=(1, 8)$ and $C=\sigma_\infty$.
\item\label{III_inf07}
$(n, h)=(2, 9)$ and $C=\sigma_\infty$.
\item\label{III_inf08}
$(n, h)=(2, 9)$ and $C\sim\sigma+3l$.
\item\label{III_inf09}
$(n, h)=(2, 10)$ and $C=\sigma_\infty$.
\item\label{III_inf10}
$(n, h)=(3, 12)$ and $C=\sigma_\infty$.
\end{enumerate}

We consider the case \eqref{III_inf01}. Then $E_X=\sigma+C$ and 
$\Delta_X=\emptyset$. Thus $k_Z\leqslant 1$. This leads to a contradiction. 
We consider the case \eqref{III_inf02}. Then $E_X=C$ and 
$\Delta_X=\emptyset$. Thus $k_Z=0$. This leads to a contradiction. 
We consider the case \eqref{III_inf03}. If $\coeff_{\sigma_\infty}E_X=1$, then 
$\deg(\Delta_Z\cap\sigma_\infty^Z)\leqslant 1$ by Lemma \ref{XS2}. 
Since $2\deg(\Delta_X\cap\sigma_\infty)+\deg(\Delta_Z\cap\sigma_\infty^Z)=7$, 
we have 
$\deg(\Delta_X\cap\sigma_\infty)=3$. 
This contradicts to the conditions $(\sB 7)$ 
and $(\sB 8)$. 
Thus $E_X=2\sigma_\infty$. By $(\sB 7)$ 
and $(\sB 8)$, we have 
$\deg(\Delta_X\cap\sigma_\infty)\leqslant 1$. 
Assume that $\deg(\Delta_X\cap\sigma_\infty)=1$. Then 
$\deg(\Delta_Z\cap\sigma_\infty^Z)=5$, $|\Delta_X|=\{P\}$, and either 
$(\mult_P\Delta_X, \mult_P(\Delta_X\cap\sigma_\infty))=(2, 1)$ or $(1, 1)$. 
We can show that these cases correspond to the types 
\textbf{$[$1;2,2$]_{0A}$} and \textbf{$[$1;2,2$]_{0B}$} respectively. 
Assume that $\deg(\Delta_X\cap\sigma_\infty)=0$. Then $\Delta_X=\emptyset$, 
$\deg(\Delta_Z\cap\sigma_\infty^Z)=7$ and $\Delta_Z\subset\sigma_\infty^Z$. 
This is nothing but the type \textbf{$[$1;2,2$]_{0C}$}. 
We consider the case \eqref{III_inf04}. Then $E_X=\sigma+C$ and 
$\deg(\Delta_Z\cap C^Z)\leqslant 2$ by Lemmas \ref{XS2} and \ref{XS4}. 
Since $2\deg(\Delta_X\cap C)+\deg(\Delta_Z\cap C^Z)=11$, 
we have 
$\deg(\Delta_X\cap\sigma_\infty)=5$. 
This contradicts to the conditions $(\sB 7)$ and 
$(\sB 8)$. 
We consider the case \eqref{III_inf05}. Then $E_X=C$, $C$ is nonsingular, 
$\Delta_X\subset C$ and $\Delta_Z=\emptyset$. 
This is nothing but the type \textbf{$[$1;2,2$]_U$}. 
We consider the case \eqref{III_inf06}. Then $E_X=\sigma+\sigma_\infty$, 
$\Delta_X\cap\sigma=\emptyset$ and $\Delta_Z\cap\sigma^Z=\emptyset$, 
which leads to a contradiction. 
Indeed, $E_M$ does not contain any $(-1)$-curve. 
We consider the case \eqref{III_inf07}. Then $E_X=\sigma+\sigma_\infty+l$ and 
$2\deg(\Delta_X\cap\sigma_\infty)+\deg(\Delta_Z\cap\sigma_\infty^Z)=9$. 
By Lemma \ref{XS2}, we have $\deg(\Delta_Z\cap\sigma_\infty^Z)=1$ 
and $\deg(\Delta_Z\cap\sigma^Z)=1$. 
Set $P:=\sigma_\infty\cap l$. 
If $P\not\in\Delta_X$, then the case corresponds to the type 
\textbf{$[$2;2,3$]_H\langle$0$\rangle$}. 
If $P\in\Delta_X$, then the case corresponds to the type 
\textbf{$[$2;2,3$]_H\langle$1$\rangle$}. 
We consider the case \eqref{III_inf08}. Then $E_X=\sigma+C$ and 
$2\deg(\Delta_X\cap C)+\deg(\Delta_Z\cap C^Z)=13$. 
By Lemma \ref{XS2}, we have $\deg(\Delta_Z\cap C^Z)=1$. 
This corresponds to the type \textbf{$[$2;2,3$]_V$}. 
We consider the case \eqref{III_inf09}. Then $E_X=\sigma+\sigma_\infty$, 
$\Delta_X\cap\sigma=\emptyset$ and $\Delta_Z\cap\sigma^Z=\emptyset$, 
which leads to a contradiction. 
Indeed, any irreducible connected component of $E_M$ is not a $(-2)$-curve 
by Corollary \ref{dP-basic_cor}. 
We consider the case \eqref{III_inf10}. Then $E_X=\sigma+\sigma_\infty$, 
$\Delta_X\subset\sigma_\infty$ and $\Delta_Z=\emptyset$. 
This is nothing but the type \textbf{$[$3;2,3$]_0$}.

From now on, we can assume that $E_X=2\sigma+\sum_{i=1}^jc_il_i$, where 
$l_i$ are distinct fibers and $c_i>0$ with $\sum_{i=1}^jc_i=e$. 
Indeed, if $(n, h)=(0, 5)$ and $E_Z=\sigma+\sigma'+l$, or 
$(n, h)=(0, 6)$ and $E_Z=\sigma+\sigma'$ ($\sigma$, $\sigma'$ are distinct minimal 
sections), then $\Delta_X=\emptyset$ and $k_Z\leqslant 2$. This leads to 
a contradiction. 
Set $d_i^X:=\deg(\Delta_X\cap l_i)$ and $d_i^Z:=\deg(\Delta_Z\cap l_i^Z)$. 
We know that $2d_i^X+d_i^Z=4$. Thus $(d_i^X, d_i^Z)=(2, 0)$, $(1, 2)$ or $(0, 4)$. 

Assume the case $c_i=2$ for some $i$. 
Then one of the following holds: 

\begin{enumerate}
\renewcommand{\theenumi}{\Alph{enumi}}
\renewcommand{\labelenumi}{(\theenumi)}
\item\label{III_A}
$(d_i^X, d_i^Z)=(2, 0)$, $|\Delta_X|=\{P_1, P_2\}$ and 
$(\mult_{P_t}\Delta_X, \mult_{P_t}(\Delta_X\cap l_i))=(2, 1)$ for $t=1$, $2$. 
\item\label{III_B}
$(d_i^X, d_i^Z)=(2, 0)$, $|\Delta_X|=\{P\}$ and 
$(\mult_P\Delta_X, \mult_P(\Delta_X\cap l_i))=(4, 2)$. 
\item\label{III_C}
$(d_i^X, d_i^Z)=(2, 0)$, $|\Delta_X|=\{P\}$ and 
$(\mult_P\Delta_X, \mult_P(\Delta_X\cap l_i))=(2, 2)$. 
\item\label{III_D}
$(d_i^X, d_i^Z)=(1, 2)$, $|\Delta_X|=\{P\}$ and 
$(\mult_P\Delta_X, \mult_P(\Delta_X\cap l_i))=(2, 1)$. 
\item\label{III_E}
$(d_i^X, d_i^Z)=(1, 2)$, $|\Delta_X|=\{P\}$ and 
$(\mult_P\Delta_X, \mult_P(\Delta_X\cap l_i))=(1, 1)$. 
\item\label{III_F}
$(d_i^X, d_i^Z)=(0, 4)$. 
\end{enumerate}

Assume the case $c_i=1$ for some $i$. 
By Lemma \ref{ZS2}, one of the following holds: 

\begin{enumerate}
\renewcommand{\theenumi}{\arabic{enumi}}
\renewcommand{\labelenumi}{(\theenumi)}
\item\label{III_1}
$(d_i^X, d_i^Z)=(2, 0)$.
\item\label{III_2}
$(d_i^X, d_i^Z)=(1, 2)$. 
\end{enumerate}

We note that $\mult_Q\Delta_Z=\mult_Q(\Delta_Z\cap l_i^Z)=2$ and 
$\mult_Q(\Delta_Z\cap\sigma^Z)=1$ for the case \eqref{III_2} since 
$\Delta_X\cap\sigma=\emptyset$, where $Q:=\sigma^Z\cap l_i^Z$.

\subsection{The case $(n, h)=(0, 5)$}\label{tetIII05_section}

In this case, $k_X=0$, $j=1$ and $c_1=1$, which leads to a contradiction; 
neither the case \eqref{III_1} nor \eqref{III_2} occurs.

\subsection{The case $(n, h)=(0, 6)$}\label{tetIII06_section}

In this case, $k_X=0$, $k_Z=6$, $E_X=2\sigma$ and $\deg(\Delta_Z\cap\sigma^Z)=6$. 
This case is nothing but the type \textbf{$[$0;2,0$]$}.

\subsection{The case $(n, h)=(1, 7)$}\label{tetIII17_section}

Assume that $j=1$. Then $c_1=2$. We can show that the case $({\tt X})$ 
(${\tt X}\in\{$\ref{III_A}$,\dots,$\ref{III_F}$\}$) corresponds to the type 
\textbf{$[$1;2,2$]_{1{\tt X}}$}. More precisely, the case $({\tt X})$ 
(${\tt X}\in\{$\ref{III_D}, \ref{III_E}, \ref{III_F}$\}$) with 
$c:=\mult_Q(\Delta_Z\cap\sigma^Z)$ and $d:=\mult_Q(\Delta_Z\cap l_1^Z)$
corresponds to the type \textbf{$[$1;2,2$]_{1{\tt X}}$({\bi c,d})}, where 
$Q:=\sigma^Z\cap l_1^Z$. 

Assume that $j=2$. Then $c_1=c_2=1$. If both $l_1$ and $l_2$ satisfy 
the condition \eqref{III_1}, then this corresponds to the type \textbf{$[$1;2,2$]_{2A}$}. 
If $l_1$ satisfies the condition \eqref{III_1}
and $l_2$ satisfies the condition \eqref{III_2}, 
then this corresponds to the type \textbf{$[$1;2,2$]_{2B}$}. 
If both $l_1$ and $l_2$ satisfy 
the condition \eqref{III_2}, then this corresponds to the type \textbf{$[$1;2,2$]_{2C}$}.

\subsection{The case $(n, h)=(1, 8)$}\label{tetIII18_section}

In this case, $j=1$ and $c_1=1$. 
If $l_1$ satisfies the condition \eqref{III_1}, 
then this corresponds to the type \textbf{$[$1;2,1$]_{1A}$}. 
If $l_1$ satisfies the condition \eqref{III_2}, 
then this corresponds to the type \textbf{$[$1;2,1$]_{1B}$}.

\subsection{The case $(n, h)=(1, 9)$}\label{tetIII19_section}

In this case, $k_X=0$, $k_Z=5$, $E_X=2\sigma$ and $\deg(\Delta_Z\cap\sigma^Z)=5$. 
This case is nothing but the type \textbf{$[$1;2,0$]$}.

\subsection{The case $(n, h)=(2, 9)$}\label{tetIII29_section}

Assume that $j=2$. Then we can assume that $c_1=2$ and $c_2=1$. 
We can show that the case $({\tt X})$, $({\tt y})$
(${\tt X}\in\{$\ref{III_A}$,\dots,$\ref{III_F}$\}$, ${\tt y}\in\{$\ref{III_1}, \ref{III_2}$\}$) 
corresponds to the type 
\textbf{$[$2;2,3$]_{2{\tt Xy}}$}. More precisely, the case $({\tt X})$, \eqref{III_1} 
(${\tt X}\in\{$\ref{III_D}, \ref{III_E}, \ref{III_F}$\}$) with 
$c:=\mult_Q(\Delta_Z\cap\sigma^Z)$ and $d:=\mult_Q(\Delta_Z\cap l_1^Z)$
corresponds to the type \textbf{$[$2;2,3$]_{2{\tt X}1}$({\bi c,d})}, where 
$Q:=\sigma^Z\cap l_1^Z$. 

Assume that $j=3$. Then $c_1=c_2=c_3=1$. 
Since $\deg(\Delta_Z\cap\sigma^Z)=1$, we can assume that either 
$(d_1^Z, d_2^Z, d_3^Z)=(0, 0, 0)$ or $(0, 0, 2)$ holds. 
The case $(d_1^Z, d_2^Z, d_3^Z)=(0, 0, 0)$ corresponds to the 
type \textbf{$[$2;2,3$]_{3A}$} and the case $(d_1^Z, d_2^Z, d_3^Z)=(0, 0, 2)$ 
corresponds to the type \textbf{$[$2;2,3$]_{3B}$}.

\subsection{The case $(n, h)=(2, 10)$}\label{tetIII210_section}

Assume that $j=1$. Then $c_1=2$. We can show that the case $({\tt X})$ 
(${\tt X}\in\{$\ref{III_A}$,\dots,$\ref{III_F}$\}$) corresponds to the type 
\textbf{$[$2;2,2$]_{1{\tt X}}$}. More precisely, the case $({\tt X})$ 
(${\tt X}\in\{$\ref{III_D}, \ref{III_E}, \ref{III_F}$\}$) with 
$c:=\mult_Q(\Delta_Z\cap\sigma^Z)$ and $d:=\mult_Q(\Delta_Z\cap l_1^Z)$
corresponds to the type \textbf{$[$2;2,2$]_{1{\tt X}}$({\bi c,d})}, where 
$Q:=\sigma^Z\cap l_1^Z$. 

Assume that $j=2$. Then $c_1=c_2=1$. 
We can assume that one of 
$(d_1^Z, d_2^Z)=(0, 0)$, $(2, 0)$ or $(2, 2)$ holds. 
We can show that the case $(d_1^Z, d_2^Z)=(0, 0)$ corresponds to the 
type \textbf{$[$2;2,2$]_{2A}$}, 
the case $(d_1^Z, d_2^Z)=(2, 0)$ corresponds to the 
type \textbf{$[$2;2,2$]_{2B}$}, 
and the case $(d_1^Z, d_2^Z)=(2, 2)$ corresponds to the 
type \textbf{$[$2;2,2$]_{2C}$}.

\subsection{The case $(n, h)=(2, 11)$}\label{tetIII211_section}

In this case, $j=1$ and $c_1=1$. 
If $l_1$ satisfies the condition \eqref{III_1}, 
then this corresponds to the type \textbf{$[$2;2,1$]_{1A}$}. 
If $l_1$ satisfies the condition \eqref{III_2}, 
then this corresponds to the type \textbf{$[$2;2,1$]_{1B}$}.

\subsection{The case $(n, h)=(2, 12)$}\label{tetIII212_section}

In this case, $k_X=0$, $k_Z=4$, $E_X=2\sigma$ and $\deg(\Delta_Z\cap\sigma^Z)=4$. 
This case is nothing but the type \textbf{$[$2;2,0$]$}.

\subsection{The case $(n, h)=(3, 12)$}\label{tetIII312_section}

In this case, we have $\Delta_Z\cap\sigma^Z=\emptyset$. 
Assume that $j=2$. Then we can assume that $c_1=2$ and $c_2=1$. 
We know that the curve $l_2$ satisfies the condition \eqref{III_1}. 
We can show that the case $({\tt X})$ 
(${\tt X}\in\{$\ref{III_A}$,\dots,$\ref{III_F}$\}$) 
corresponds to the type 
\textbf{$[$3;2,3$]_{2{\tt X}}$}. 

Assume that $j=3$. Then $c_1=c_2=c_3=1$ and $(d_1^Z, d_2^Z, d_3^Z)=(0, 0, 0)$ hold. 
This corresponds to the type \textbf{$[$3;2,3$]_3$}.

\subsection{The case $(n, h)=(3, 13)$}\label{tetIII313_section}

Assume that $j=1$. Then $c_1=2$. We can show that the case $({\tt X})$ 
(${\tt X}\in\{$\ref{III_A}$,\dots,$\ref{III_F}$\}$) corresponds to the type 
\textbf{$[$3;2,2$]_{1{\tt X}}$}. More precisely, the case $({\tt X})$ 
(${\tt X}\in\{$\ref{III_D}, \ref{III_E}, \ref{III_F}$\}$) with 
$c:=\mult_Q(\Delta_Z\cap\sigma^Z)$ and $d:=\mult_Q(\Delta_Z\cap l_1^Z)$
corresponds to the type \textbf{$[$3;2,2$]_{1{\tt X}}$({\bi c,d})}, where 
$Q:=\sigma^Z\cap l_1^Z$. 

Assume that $j=2$. Then $c_1=c_2=1$. 
Since $\deg(\Delta_Z\cap\sigma^Z)=1$, we can assume that either 
$(d_1^Z, d_2^Z)=(0, 0)$ or $(2, 0)$ holds. 
We can show that the case $(d_1^Z, d_2^Z)=(0, 0)$ corresponds to the 
type \textbf{$[$3;2,2$]_{2A}$} 
and the case $(d_1^Z, d_2^Z)=(2, 0)$ corresponds to the 
type \textbf{$[$3;2,2$]_{2B}$}.

\subsection{The case $(n, h)=(3, 14)$}\label{tetIII314_section}

In this case, $j=1$ and $c_1=1$. 
If $l_1$ satisfies the condition \eqref{III_1}, 
then this corresponds to the type \textbf{$[$3;2,1$]_{1A}$}. 
If $l_1$ satisfies the condition \eqref{III_2}, 
then this corresponds to the type \textbf{$[$3;2,1$]_{1B}$}.

\subsection{The case $(n, h)=(3, 15)$}\label{tetIII315_section}

In this case, $k_X=0$, $k_Z=3$, $E_X=2\sigma$ and $\deg(\Delta_Z\cap\sigma^Z)=3$. 
This case is nothing but the type \textbf{$[$3;2,0$]$}.

\subsection{The case $(n, h)=(4, 16)$}\label{tetIII416_section}

In this case, we have $\Delta_Z\cap\sigma^Z=\emptyset$. 
Assume that $j=1$. Then $c_1=2$. 
We can show that the case $({\tt X})$ 
(${\tt X}\in\{$\ref{III_A}$,\dots,$\ref{III_F}$\}$) 
corresponds to the type 
\textbf{$[$4;2,2$]_{1{\tt X}}$}. 

Assume that $j=2$. Then $c_1=c_2=1$ and $(d_1^Z, d_2^Z)=(0, 0)$ hold. 
This corresponds to the type \textbf{$[$4;2,2$]_2$}.

\subsection{The case $(n, h)=(4, 17)$}\label{tetIII417_section}

In this case, $j=1$ and $c_1=1$. 
If $l_1$ satisfies the condition \eqref{III_1}, 
then this corresponds to the type \textbf{$[$4;2,1$]_{1A}$}. 
If $l_1$ satisfies the condition \eqref{III_2}, 
then this corresponds to the type \textbf{$[$4;2,1$]_{1B}$}.

\subsection{The case $(n, h)=(4, 18)$}\label{tetIII418_section}

In this case, $k_X=0$, $k_Z=2$, $E_X=2\sigma$ and $\deg(\Delta_Z\cap\sigma^Z)=2$. 
This case is nothing but the type \textbf{$[$4;2,0$]$}.

\subsection{The case $(n, h)=(5, 20)$}\label{tetIII520_section}

We note that $\Delta_Z\cap\sigma^Z=\emptyset$. 
In this case, $j=1$, $c_1=1$ and 
the curve $l_1$ satisfies the condition \eqref{III_1}. 
This corresponds to the type \textbf{$[$5;2,1$]_1$}.

\subsection{The case $(n, h)=(5, 21)$}\label{tetIII521_section}

In this case, $k_X=0$, $k_Z=1$, $E_X=2\sigma$ and $\deg(\Delta_Z\cap\sigma^Z)=1$. 
This case is nothing but the type \textbf{$[$5;2,0$]$}.

\subsection{The case $(n, h)=(6, 24)$}\label{tetIII624_section}

In this case, $k_X=k_Z=0$ and $E_X=2\sigma$. 
This case is nothing but the type \textbf{$[$6;2,0$]$}.

As a consequence, we have completed the proof of Theorem \ref{tetII_thm}.

\section{Classification of bottom tetrads, III}\label{XCIII_section}

We classify bottom tetrads $(X, E_X; \Delta_Z, \Delta_X)$ with trivial 
$2K_X+L_X$.

\begin{thm}\label{tetIII_thm}
The bottom tetrads $(X, E_X; \Delta_Z, \Delta_X)$ with trivial $2K_X+L_X$
are classified by the types defined as follows $($We assume that any of them 
satisfies that $\Delta_Z$ satisfies the $(\nu1)$-condition.$)$: 

\smallskip

The case $X=\pr^2$ and 
$E_X=C$ $(C$ is an irreducible nodal cubic curve. Let $P$ be the singular 
point of $C$.$):$

\begin{description}
\item[{$[$3$]_{NA}$}]
$\Delta_X\subset C\setminus\{P\}$ and $\deg\Delta_X=8$. 
$\Delta_Z=\{Q\}$ and $\deg\Delta_Z=1$, where $Q$ is the singular point of $C^Z$. 

\smallskip

\item[{$[$3$]_{NB}$}]
$\deg\Delta_X=7$, $\mult_P\Delta_X=1$ and $\Delta_X\setminus\{P\}\subset C$. 
$|\Delta_Z|=\{Q_1, Q_2\}$ and $\mult_{Q_i}\Delta_Z=1$, where 
$\{Q_1, Q_2\}=C^Z\cap\Gamma_{P, 1}$. 
\end{description}

\smallskip

The case $X=\pr^2$ and 
$E_X=C$ $(C$ is an irreducible cuspidal cubic curve. Let $P$ be the singular 
point of $C$.$):$

\begin{description}
\item[{$[$3$]_{CA}$}]
$\Delta_X\subset C\setminus\{P\}$ and $\deg\Delta_X=8$. 
$\Delta_Z=\{Q\}$ and $\deg\Delta_Z=1$, where $Q$ is the singular point of $C^Z$. 

\smallskip

\item[{$[$3$]_{CB}$}]
$\deg\Delta_X=7$, $\mult_P\Delta_X=1$ and $\Delta_X\setminus\{P\}\subset C$. 
$|\Delta_Z|=\{Q\}$, 
$\mult_Q\Delta_Z=\mult_Q(\Delta_Z\cap C^Z)=\mult_Q(\Delta_Z\cap\Gamma_{P, 1})=2$, 
where $\{Q\}=C^Z\cap\Gamma_{P, 1}$. 
\end{description}

\smallskip

The case $X=\pr^2$ and 
$E_X=C+l$ $(C$ is a nonsingular conic and $l$ is a line. $C$ and $l$ 
meet two points $P_1$, $P_2$.$):$

\begin{description}
\item[{$[$3$]_{AA}$}]
$\deg\Delta_X=5$, $\deg(\Delta_X\cap C)=5$, $\deg(\Delta_X\cap l)=2$ and 
$\mult_{P_i}\Delta_X=1$ for $i=1$, $2$. 
$\deg\Delta_Z=4$ and $|\Delta_Z|=\{Q_{1C}$, $Q_{1l}$, $Q_{2C}$, $Q_{2l}\}$, where 
$Q_{iC}:=C^Z\cap\Gamma_{P_i, 1}$ and $Q_{il}:=l^Z\cap\Gamma_{P_i, 1}$. 

\smallskip

\item[{$[$3$]_{AB}$}]
$\deg\Delta_X=6$, $\deg(\Delta_X\cap C)=5$, $\deg(\Delta_X\cap l)=2$,
$P_2\not\in\Delta_X$ and $\mult_{P_1}\Delta_X=1$. 
$\deg\Delta_Z=3$ and $|\Delta_Z|=\{Q_2$, $Q_{1C}$, $Q_{1l}\}$, where 
$Q_2:=C^Z\cap l^Z$, 
$Q_{1C}:=C^Z\cap\Gamma_{P_1, 1}$ and $Q_{1l}:=l^Z\cap\Gamma_{P_1, 1}$. 
\end{description}

\smallskip

The case $X=\pr^2$ and 
$E_X=C+l$ $(C$ is a nonsingular conic and $l$ is a line. $C$ and $l$ 
contacts with each other at one point $P$.$):$

\begin{description}
\item[{$[$3$]_{KA}$}]
$\deg\Delta_X=5$, $\deg(\Delta_X\cap C)=5$, $\deg(\Delta_X\cap l)=2$ and 
$\mult_P\Delta_X=\mult_P(\Delta_X\cap C)=\mult_P(\Delta_X\cap l)=2$.
$\deg\Delta_Z=4$, 
$\mult_{Q_C}\Delta_Z=\mult_{Q_C}(\Delta_Z\cap C^Z)=2$ and 
$\mult_{Q_l}\Delta_Z=\mult_{Q_l}(\Delta_Z\cap l^Z)=2$, 
where $Q_C=C^Z\cap\Gamma_{P, 2}$ and $Q_l=l^Z\cap\Gamma_{P, 2}$. 

\smallskip

\item[{$[$3$]_{KB}\langle${\bi b}$\rangle$} $(2\leqslant b\leqslant 6)$]
$\deg\Delta_X=7$, $\deg(\Delta_X\cap C)=6$, $\deg(\Delta_X\cap l)=3$, 
$b=\mult_P\Delta_X=\mult_P(\Delta_X\cap C)$ and 
$\mult_P(\Delta_X\cap l)=2$. 
$\deg\Delta_Z=2$, $\Delta_Z\subset\Gamma_{P, b}$ and 
$\Delta_Z\cap(C^Z\cup l^Z\cup\Gamma_{P, b-1})=\emptyset$. 

\smallskip

\item[{$[$3$]_{KC}\langle${\bi b}$\rangle$} $(2\leqslant b\leqslant 5)$]
$\deg\Delta_X=6$, $\deg(\Delta_X\cap C)=5$, $\deg(\Delta_X\cap l)=3$, 
$b=\mult_P\Delta_X=\mult_P(\Delta_X\cap C)$ and 
$\mult_P(\Delta_X\cap l)=2$. 
$\deg\Delta_Z=3$, $\mult_Q\Delta_Z=\mult_Q(\Delta_Z\cap C^Z)=2$ and 
$\Delta_Z\setminus\{Q\}\subset\Gamma_{P, b}
\setminus(l^Z\cup\Gamma_{P, b-1})$, 
where $Q=C^Z\cap\Gamma_{P, b}$. 
\end{description}

\smallskip

The case $X=\pr^2$ and $E_X=2l_1+l_2$ $(l_i$ are distinct lines. Set $P:=l_1\cap l_2$.$):$

\begin{description}
\item[{$[$3$]_{2A}\langle${\bi b}$\rangle$} $(1\leqslant b\leqslant 3)$]
$\deg\Delta_X=5$, $b=\mult_P\Delta_X=\mult_P(\Delta_X\cap l_2)$, 
$\mult_P(\Delta_X\cap l_1)=1$, $|\Delta_X|\cap l_1=\{P, P_1\}$ with 
$(\mult_{P_1}\Delta_X$, $\mult_{P_1}(\Delta_X\cap l_1))=(2, 2)$ and 
$\deg(\Delta_X\cap l_2)=3$. 
$\deg\Delta_Z=4$, $\Delta_Z\subset\Gamma_{P, b}\cup\Gamma_{P_1, 2}$, 
$\Delta_Z\cap
(l_1^Z\cup l_2^Z\cup\Gamma_{P_1, 1}\cup\Gamma_{P, b-1})=\emptyset$ 
and $\deg(\Delta_Z\cap\Gamma_{P, b})=\deg(\Delta_Z\cap\Gamma_{P_1, 2})=2$.

\smallskip

\item[{$[$3$]_{2B1}\langle$1$\rangle$({\bi c,d})} $((c, d)=(0, 0), (1, 1), (1, 2))$]
$\deg\Delta_X=4$, 
$|\Delta_X|\cap l_1=\{P, P_1\}$ with 
$\mult_{P_1}\Delta_X=1$, $\mult_P\Delta_X=1$ and 
$\deg(\Delta_X\cap l_2)=3$. 
$\deg\Delta_Z=5$, $\deg(\Delta_Z\cap l_1^Z)=\deg(\Delta_Z\cap\Gamma_{P, 1})=2$, 
$\mult_{Q_2}\Delta_Z=\mult_{Q_2}(\Delta_Z\cap\Gamma_{P_1, 1})=2$,  
$\Delta_Z\cap l_2^Z=\emptyset$, 
$\mult_{Q_1}(\Delta_Z\cap l_1^Z)=c$, $\mult_{Q_1}(\Delta_Z\cap\Gamma_{P, 1})=d$ 
and $\mult_{Q_1}\Delta_Z=c+d$, 
where $Q_1=l_1^Z\cap\Gamma_{P,1}$ and $Q_2=l_1^Z\cap\Gamma_{P_1, 1}$.

\smallskip

\item[{$[$3$]_{2B1}\langle${\bi b}$\rangle$} $(2\leqslant b\leqslant 3)$]
$\deg\Delta_X=4$, $b=\mult_P\Delta_X=\mult_P(\Delta_X\cap l_2)$, 
$|\Delta_X|\cap l_1=\{P, P_1\}$ with 
$\mult_{P_1}\Delta_X=1$ and 
$\deg(\Delta_X\cap l_2)=3$. 
$\deg\Delta_Z=5$, $\deg(\Delta_Z\cap l_1^Z)=\deg(\Delta_Z\cap\Gamma_{P, b})=2$, 
$\mult_Q\Delta_Z=\mult_Q(\Delta_Z\cap\Gamma_{P_1, 1})=2$ and 
$\Delta_Z\cap(l_2^Z\cup\Gamma_{P, 1}\cup\Gamma_{P, b-1})=\emptyset$, 
where $Q=l_1^Z\cap\Gamma_{P_1, 1}$.

\smallskip

\item[{$[$3$]_{2B2}\langle$1$\rangle$({\bi c,d})} $((c, d)=(0, 0), (1, 1))$]
$\deg\Delta_X=3$, $\mult_P\Delta_X=1$, 
$|\Delta_X|\cap l_1=\{P, P_1\}$ with 
$\mult_{P_1}\Delta_X=1$ and 
$\deg(\Delta_X\cap l_2)=2$. $\deg\Delta_Z=6$, 
$\deg(\Delta_Z\cap l_1^Z)=\deg(\Delta_Z\cap\Gamma_{P, 1})=2$, 
$\mult_{Q_1}\Delta_Z=\mult_{Q_1}(\Delta_Z\cap l_2^Z)=2$, 
$\mult_{Q_2}\Delta_Z=\mult_{Q_2}(\Delta_Z\cap\Gamma_{P_1, 1})=2$, 
$\mult_{Q_3}(\Delta_Z\cap l_1^Z)=c$, $\mult_{Q_3}(\Delta_Z\cap\Gamma_{P, 1})=d$ 
and $\mult_{Q_3}\Delta_Z=c+d$, 
where $Q_1=l_2^Z\cap\Gamma_{P, 1}$, $Q_2=l_1^Z\cap\Gamma_{P_1, 1}$
and $Q_3=l_1^Z\cap\Gamma_{P, 1}$.

\smallskip

\item[{$[$3$]_{2B2}\langle$2$\rangle$}]
$\deg\Delta_X=3$, $\mult_P\Delta_X=\mult_P(\Delta_X\cap l_2)=2$, 
$|\Delta_X|\cap l_1=\{P, P_1\}$ with 
$\mult_{P_1}\Delta_X=1$.
$\deg\Delta_Z=6$, $\mult_Q\Delta_Z=\mult_Q(\Delta_Z\cap l_2^Z)=2$, 
$\deg(\Delta_Z\cap l_1^Z)=\deg(\Delta_Z\cap\Gamma_{P, 2})=2$, 
$\mult_{Q_1}\Delta_Z=\mult_{Q_1}(\Delta_Z\cap\Gamma_{P_1, 1})=2$ and 
$\Delta_Z\cap\Gamma_{P, 1}=\emptyset$, 
where $Q=l_2^Z\cap\Gamma_{P, 2}$ and $Q_1=l_1^Z\cap\Gamma_{P_1, 1}$.

\smallskip

\item[{$[$3$]_{2C1}\langle$1$\rangle$({\bi c,d})} $((c, d)=(0, 0), (1, 1),\dots,(4, 1), (1, 2))$]
$\deg\Delta_X$$=$$3$, $\mult_P\Delta_X=1$ and 
$\Delta_X\subset l_2$. 
$\deg\Delta_Z=6$, $\deg(\Delta_Z\cap l_1^Z)=4$, 
$\deg(\Delta_Z\cap\Gamma_{P, 1})=2$, 
$\Delta_Z\cap l_2^Z=\emptyset$, 
$\mult_Q(\Delta_Z\cap l_1^Z)=c$, $\mult_Q(\Delta_Z\cap\Gamma_{P, 1})=d$, 
$\mult_Q\Delta_Z=c+d$, 
where $Q=l_1^Z\cap\Gamma_{P,1}$.

\smallskip

\item[{$[$3$]_{2C1}\langle${\bi b}$\rangle$} $(2\leqslant b\leqslant 3)$]
$\deg\Delta_X=3$, $b=\mult_P\Delta_X=\mult_P(\Delta_X\cap l_2)$, 
$\mult_P(\Delta_X\cap l_1)=1$ and 
$\Delta_X\subset l_2$. 
$\deg\Delta_Z=6$, $\deg(\Delta_Z\cap l_1^Z)=4$, 
$\deg(\Delta_Z\cap\Gamma_{P, b})=2$ and 
$\Delta_Z\cap(\Gamma_{P, b-1}\cup l_2^Z)=\emptyset$.

\smallskip

\item[{$[$3$]_{2C2}\langle$1$\rangle$({\bi c,d})} $((c, d)=(0, 0)$, $(1, 1),\dots,(4, 1))$]
$\deg\Delta_X=2$ and $\Delta_X\subset l_2$, $\mult_P\Delta_X=1$.
$\deg\Delta_Z=7$, $\deg(\Delta_Z\cap l_1^Z)=4$, 
$\deg(\Delta_Z\cap\Gamma_{P, 1})=2$, 
$\mult_{Q_1}(\Delta_Z\cap l_1^Z)=c$, $\mult_{Q_1}(\Delta_Z\cap\Gamma_{P, 1})=d$, 
$\mult_{Q_1}\Delta_Z=c+d$ and 
$\mult_{Q_2}\Delta_Z=\mult_{Q_2}(\Delta_Z\cap l_2^Z)=2$,  
where $Q_1=l_1^Z\cap\Gamma_{P,1}$ and $Q_2=l_2^Z\cap\Gamma_{P, 1}$. 

\smallskip

\item[{$[$3$]_{2C2}\langle$2$\rangle$}]
$\deg\Delta_X=2$ and $\mult_P\Delta_X=\mult_P(\Delta_X\cap l_2)=2$. 
$\deg\Delta_Z=7$, $\deg(\Delta_Z\cap l_1^Z)=4$, 
$\deg(\Delta_Z\cap\Gamma_{P, 2})=2$, 
$\Delta_Z\cap\Gamma_{P, 1}=\emptyset$ and 
$\mult_Q\Delta_Z=\mult_Q(\Delta_Z\cap l_2^Z)=2$, 
where $Q=l_2^Z\cap\Gamma_{P, 2}$. 

\smallskip

\item[{$[$3$]_{2C3}\langle${\bi b}$\rangle$} $(3\leqslant b\leqslant 5)$]
$\deg\Delta_X=5$, $b=\mult_P\Delta_X=\mult_P(\Delta_X\cap l_2)+2$, 
$\mult_P(\Delta_X\cap l_1)=1$ and 
$\deg(\Delta_X\cap l_2)=3$. 
$\deg\Delta_Z=4$, $\Delta_Z\subset l_1^Z$ 
and $\Delta_Z\cap\Gamma_{P,1}=\emptyset$. 

\smallskip

\item[{$[$3$]_{2D}$({\bi c,d})} $((c, d)=(0, 0), (1, 1), (1, 2), (2, 1))$]
$\deg\Delta_X=5$, $P\not\in\Delta_X$, 
$|\Delta_X|\cap l_1=\{P_1\}$, 
$(\mult_{P_1}\Delta_X$, $\mult_{P_1}(\Delta_X\cap l_1))=(2, 2)$, 
$\deg(\Delta_X\cap l_2)=3$. $\deg\Delta_Z=4$, 
$\mult_Q(\Delta_Z\cap l_1^Z)=c$, $\mult_Q(\Delta_Z\cap\Gamma_{P_1, 2})=d$, 
$\mult_Q\Delta_Z=c+d$, 
$\deg(\Delta_Z\cap l_1^Z)=\deg(\Delta_Z\cap\Gamma_{P_1, 2})=2$,
$\Delta_Z\cap(l_2^Z\cup\Gamma_{P_1, 1})=\emptyset$, 
where $Q=l_1^Z\cap\Gamma_{P_1, 2}$. 

\smallskip

\item[{$[$3$]_{2E}$}]
$\deg\Delta_X=5$, $P\not\in\Delta_X$, 
$|\Delta_X|\cap l_1=\{P_1\}$ with 
$(\mult_{P_1}\Delta_X$, $\mult_{P_1}(\Delta_X\cap l_1))=(2, 1)$ and 
$\deg(\Delta_X\cap l_2)=3$. $\deg\Delta_Z=4$, 
$\Delta_Z\subset l_1^Z$ and $\Delta_Z\cap(l_2^Z\cup\Gamma_{P_1, 1})=\emptyset$. 

\smallskip

\item[{$[$3$]_{2F1}$}]
$\deg\Delta_X=3$, $P\not\in\Delta_X$, 
$|\Delta_X|\cap l_1=\{P_1\}$ with 
$\mult_{P_1}\Delta_X=1$ and 
$\deg(\Delta_X\cap l_2)=2$. $\deg\Delta_Z=6$, 
$\mult_{Q_1}\Delta_Z=\mult_{Q_1}(\Delta_Z\cap l_2^Z)=2$, $\deg(\Delta_Z\cap l_1^Z)=2$, 
$\mult_{Q_2}\Delta_Z=\mult_{Q_2}(\Delta_Z\cap\Gamma_{P_1, 1})=2$, 
$\deg(\Delta_Z\cap l_1^Z)=4$, where $Q_1=l_1^Z\cap l_2^Z$ and 
$Q_2=l_1^Z\cap\Gamma_{P_1, 1}$. 

\smallskip

\item[{$[$3$]_{2F2}$}]
$\deg\Delta_X=4$, $P\not\in\Delta_X$, 
$|\Delta_X|\cap l_1=\{P_1\}$, 
$\mult_{P_1}\Delta_X=1$, 
$\deg(\Delta_X\cap l_2)=3$. $\deg\Delta_Z=5$, 
$\mult_Q\Delta_Z=\mult_Q(\Delta_Z\cap\Gamma_{P_1, 1})=2$, 
$\deg(\Delta_Z\cap l_1^Z)=4$, $\Delta_Z\cap l_2^Z=\emptyset$, 
where $Q=l_1^Z\cap\Gamma_{P_1, 1}$. 

\smallskip

\item[{$[$3$]_{2G1}$}]
$\deg\Delta_X=2$, $P\not\in\Delta_X$, 
$\Delta_X\subset l_2$. $\deg\Delta_Z=7$, 
$\mult_Q\Delta_Z=\mult_Q(\Delta_Z\cap l_2^Z)=2$ and 
$\deg(\Delta_Z\cap l_1^Z)=6$, where $Q=l_1^Z\cap l_2^Z$. 

\smallskip

\item[{$[$3$]_{2G2}$}]
$\deg\Delta_X=3$, $P\not\in\Delta_X$, 
$\Delta_X\subset l_2$. $\deg\Delta_Z=6$, 
$\Delta_Z\subset l_1^Z$ and $\Delta_Z\cap l_2^Z=\emptyset$. 
\end{description}

\smallskip

The case $X=\pr^2$ and $E_X=l_1+l_2+l_3$ $(l_i$ are distinct lines and 
$l_1\cap l_2\cap l_3=\emptyset$. Set $P_{ij}:=l_i\cap l_{j} (1\leqslant i<j\leqslant 3)$.$):$

\begin{description}
\item[{$[$3$]_{3A}$}]
$\deg\Delta_X=4$, $\mult_{P_{12}}\Delta_X=\mult_{P_{13}}\Delta_X=1$, 
$P_{23}\not\in\Delta_X$ and 
$\deg(\Delta_X\cap l_i)=2$ for $i=1$, $2$, $3$. 
$\deg\Delta_Z=5$ and $|\Delta_Z|=\{Q_{12}$, $Q_{21}$, $Q_{13}$, $Q_{31}$, $Q_{23}\}$, 
where 
$Q_{12}=l_1^Z\cap\Gamma_{P_{12}, 1}$, $Q_{21}=l_2^Z\cap\Gamma_{P_{12}, 1}$, 
$Q_{13}=l_1^Z\cap\Gamma_{P_{13}, 1}$, $Q_{31}=l_3^Z\cap\Gamma_{P_{13}, 1}$, 
and $Q_{23}=l_2^Z\cap l_3^Z$. 

\smallskip

\item[{$[$3$]_{3B}$}]
$\deg\Delta_X=3$ and $|\Delta_X|=\{P_{12}$, $P_{13}$, $P_{23}\}$.
$\deg\Delta_Z=6$ and $|\Delta_Z|=\{Q_{12}$, $Q_{21}$, $Q_{13}$, $Q_{31}$, $Q_{23}$, 
$Q_{32}\}$, where 
$Q_{12}=l_1^Z\cap\Gamma_{P_{12}, 1}$, $Q_{21}=l_2^Z\cap\Gamma_{P_{12}, 1}$, 
$Q_{13}=l_1^Z\cap\Gamma_{P_{13}, 1}$, $Q_{31}=l_3^Z\cap\Gamma_{P_{13}, 1}$, 
$Q_{23}=l_2^Z\cap\Gamma_{P_{23}, 1}$ and $Q_{32}=l_3^Z\cap\Gamma_{P_{23}, 1}$. 
\end{description}

\smallskip

The case $X=\pr^1\times\pr^1:$

\begin{description}
\item[{$[$0;2,2$]_0$}]
$E_X=2C$ such that $C:$ nonsingular, $C\sim\sigma+l$, 
$\Delta_X=\emptyset$, $\deg\Delta_Z=8$ and $\Delta_Z\subset C^Z$. 

\smallskip

\item[{$[$0;2,2$]_1$({\bi c,d})} $((c, d)=(0, 0), (1, 1),\dots, (4, 1))$]
$E_X=2\sigma+2l$ and 
$\Delta_X=\emptyset$. $\deg\Delta_Z=8$, 
$\deg(\Delta_Z\cap\sigma^Z)=\deg(\Delta_Z\cap l^Z)=4$, 
$\mult_Q(\Delta_Z\cap\sigma^Z)=c$, 
$\mult_Q(\Delta_Z\cap l^Z)=d$ and $\mult_Q\Delta_Z=c+d$, where 
$Q=\sigma^Z\cap l^Z$. 
\end{description}

\smallskip

The case $X=\F_2:$

\begin{description}
\item[{$[$2;2,4$]_0$}]
$E_X=2\sigma_\infty$, $\Delta_X=\emptyset$, $\deg\Delta_Z=8$ and 
$\Delta_Z\subset\sigma_\infty^Z$. 

\smallskip

\item[{$[$2;2,4$]_1$}]
$E_X=2\sigma+2l_1+2l_2$ $(l_1$, $l_2:$ distinct fibers$)$, 
$\Delta_X=\emptyset$, $\deg\Delta_Z=8$, 
$\deg(\Delta_Z\cap l_1^Z)=\deg(\Delta_Z\cap l_2^Z)=4$ 
and $\Delta_Z\cap\sigma^Z=\emptyset$. 
\end{description}
\end{thm}

We start to prove Theorem \ref{tetIII_thm}. Any tetrad in Theorem \ref{tetIII_thm} 
is a bottom tetrad by Proposition \ref{converse_prop}. 
We see the converse.

\subsection{The case $X=\pr^2$}

We consider the case $X=\pr^2$ and $E_X\sim 3l$. 
Set $\psi\colon Z\to X$, $\phi\colon M\to Z$, $E_Z$, $E_X$, $k_Z$ and $k_X$ as 
in the beginning of Section \ref{pr2_section}.
We note that $k_X\leqslant 8$ holds.

\subsubsection{The case $E_X=C$ $(C:$ irreducible singular cubic$)$}

Let $P$ be the singular point of $C$. 
We note that $\mult_PC=2$. 
By Lemmas \ref{XS2} and \ref{XS5}, $C^M$ is a connected component of $E_M$. 
Thus $((C^M)^2)=-3$. 
Assume that $P\not\in\Delta_X$. Then $E_Z=C^Z$ and $C^Z$ has a unique singular 
point $Q$ (the point over $P$). Thus $k_Z=1$ and $|\Delta_Z|=\{Q\}$. 
Since $((C^M)^2)=-3$, $k_X=8$ and $\Delta_X\subset C\setminus\{P\}$. 
This case is nothing but the type \textbf{$[$3$]_{NA}$} (if $C$ nodal) or the 
type \textbf{$[$3$]_{CA}$} (if $C$ cuspidal). 
Assume that $P\in\Delta_X$. By Lemmas \ref{XS2} and \ref{XS5}, 
$\mult_P\Delta_X=1$, $E_Z=C^Z+\Gamma_{P, 1}$ and $C^Z$ is nonsingular. 
If $C$ is nodal, then $|C^Z\cap\Gamma_{P, 1}|=\{Q_1, Q_2\}$. 
Thus $k_Z=2$ and $|\Delta_Z|=\{Q_1, Q_2\}$ by Lemma \ref{ZS2}. 
Since $((C^M)^2)=-3$, $\deg(\Delta_X\setminus\{P\})=6$ and 
$\Delta_X\setminus\{P\}\subset C$. 
This case is nothing but the type \textbf{$[$3$]_{NB}$}. 
If $C$ is cuspidal, then $|C^Z\cap\Gamma_{P, 1}|=\{Q\}$ and 
$\mult_Q(C^Z\cap\Gamma_{P, 1})=2$. 
Thus $k_Z=2$ and $|\Delta_Z|=\{Q\}$ by Lemma \ref{ZS4}. 
Since $((C^M)^2)=-3$, $\deg(\Delta_X\setminus\{P\})=6$ and 
$\Delta_X\setminus\{P\}\subset C$. 
This case is nothing but the type \textbf{$[$3$]_{CB}$}.

\subsubsection{The case $E_X=C+l$ $(C:$ nonsingular conic and $l:$ line that meet 
two points$)$}

Set $\{P_1, P_2\}:=C\cap l$. 
By Lemmas \ref{ZS2} and \ref{XS2}, both $C^M$ and $l^M$ are $(-3)$-curves 
and $\mult_{P_i}\Delta_X\leqslant 1$. 
Thus $\deg(\Delta_X\cap C)=5$, $\deg(\Delta_Z\cap C^Z)=2$, 
$\deg(\Delta_X\cap l)=2$ and $\deg(\Delta_Z\cap l^Z)=2$. 
By the condition $(\sB 9)$, we can assume that $P_1\in\Delta_X$.
If $P_2\in\Delta_X$, then this induces the type \textbf{$[$3$]_{AA}$}. 
If $P_2\not\in\Delta_X$, 
then this induces the type \textbf{$[$3$]_{AB}$}.

\subsubsection{The case $E_X=C+l$ $(C:$ nonsingular conic and $l:$ line that contacts 
with each other$)$}

Set $P:=|C\cap l|$, $d_C^X:=\deg(\Delta_X\cap C)$, $d_C^Z:=\deg(\Delta_Z\cap C^Z)$, 
$d_l^X:=\deg(\Delta_X\cap l)$ and $d_l^Z:=\deg(\Delta_Z\cap l^Z)$. 
By Claim \ref{tetI_curve_claim}, 
$(d_C^X, d_C^Z, ((C^M)^2))=(6, 0, -2)$ or $(5, 2, -3)$, and 
$(d_l^X, d_l^Z, ((l^M)^2))=(3, 0, -2)$ or $(2, 2, -3)$.
By the condition $(\sB 9)$, $P\in\Delta_X$. 

Assume that 
$\mult_P(\Delta_X\cap l)>\mult_P(\Delta_X\cap C)$. Then 
$\mult_P(\Delta_X\cap l)=b$ and $\mult_P(\Delta_X\cap C)=2$ by Lemma \ref{XS4}. 
In this case, $\Delta_Z\cap\Gamma_{P, 2}=\emptyset$. Thus $((C^M)^2)=-2$. 
In particular, $\deg(\Delta_X\cap C\setminus\{P\})=4$. Since $b\geqslant 3$, 
we have $b=d_X^l=3$. In particular, $\Delta_X\cap l\setminus\{P\}=\emptyset$. 
This contradicts to the condition $(\sB 9)$. This implies that 
$b=\mult_P(\Delta_X\cap C)\geqslant\mult_P(\Delta_X\cap l)=2$ 
by Lemma \ref{XS4}.

We consider the case $((l^M)^2)=-3$.
Set $Q_l:=l^Z\cap\Gamma_{P, 2}$ and $Q_C:=C^Z\cap\Gamma_{P, b}$. 
Since $l^M\cap\Gamma_{P, 2}^M=\emptyset$, we have $b=2$. 
Moreover, $\mult_{Q_l}\Delta_Z=\mult_{Q_l}(\Delta_Z\cap l^Z)=2$ and 
$\mult_{Q_l}(\Delta_Z\cap\Gamma_{P, 2})=1$. 
Assume that $Q_C\not\in\Delta_Z$. Then $((C^M)^2)=-2$. 
In this case, $\deg(\Delta_X\cap C\setminus\{P\})=4$ and 
$\Delta_X\cap l\setminus\{P\}=\emptyset$. 
This contradicts to the condition $(\sB 9)$.
Thus $Q_C\in\Delta_Z$, $((C^M)^2)=-3$, 
$\mult_{Q_C}\Delta_Z=\mult_{Q_C}(\Delta_Z\cap C^Z)=2$ and 
$\mult_{Q_C}(\Delta_Z\cap\Gamma_{P, 2})=1$. 
This case induces the type \textbf{$[$3$]_{KA}$}.

We consider the case $((l^M)^2)=-2$. 
If $((C^M)^2)=-2$, then $2\leqslant b\leqslant 6$. 
Moreover, $\Delta_Z\subset\Gamma_{P, b}$. 
This case induces the type \textbf{$[$3$]_{KB}\langle${\bi b}$\rangle$}. 
If $((C^M)^2)=-3$, then $2\leqslant b\leqslant 5$. 
Moreover, $\mult_Q\Delta_Z=\mult_Q(\Delta_Z\cap C^Z)=2$ and 
$\mult_Q(\Delta_Z\cap\Gamma_{P, b})=1$, where $Q:=C^Z\cap\Gamma_{P, b}$. 
This case induces the type \textbf{$[$3$]_{KC}\langle${\bi b}$\rangle$}.

\subsubsection{The case $E_X=2l_1+l_2$ $(l_i:$ distinct lines$)$ and $P\in\Delta_X$, 
where $P=l_1\cap l_2$}

Set $d_i^X:=\deg(\Delta_X\cap l_i)$, 
$d_i^Z:=\deg(\Delta_Z\cap l_i^Z)$ and $b:=\mult_P\Delta_X$. 
Then $(d_1^X, d_1^Z, ((l_1^M)^2))=(3, 0, -2)$, $(2, 2, -3)$ or $(1, 4, -4)$, 
and $(d_2^X, d_2^Z, ((l_2^M)^2))=(3, 0, -2)$ or $(2, 2, -3)$.
By Lemma \ref{XS2}, we have $\mult_P(\Delta_X\cap l_1)=1$. 
Moreover, one of the following holds: 

\begin{enumerate}
\renewcommand{\theenumi}{\arabic{enumi}}
\renewcommand{\labelenumi}{(\theenumi)}
\item\label{3_2_-1}
$b=\mult_P(\Delta_X\cap l_2)\leqslant 3$, $((l_2^M)^2)=-2$ and 
$\Delta_Z\cap l_2^Z=\emptyset$.
\item\label{3_2_-2}
$b=\mult_P(\Delta_X\cap l_2)\leqslant 2$, $((l_2^M)^2)=-3$, 
$\mult_Q\Delta_Z=\mult_Q(\Delta_Z\cap l_2^Z)=2$, 
$\mult_Q(\Delta_Z\cap\Gamma_{P, b})=1$, $k_X\neq 4$ and 
$\deg(\Delta_Z\cap\Gamma_{P, b})=2$, where $Q:=l_2^Z\cap\Gamma_{P, b}$. 
\item\label{3_2_-3}
$b=\mult_P(\Delta_X\cap l_2)+2\leqslant 5$, $((l_2^M)^2)=-2$, 
$\Delta_Z\cap l_2^Z=\emptyset$ and $\Delta_X\cap l_1\setminus\{P\}=\emptyset$.
\end{enumerate}

{\noindent\textbf{The case $d_1^X=3$:}}

In this case, $|\Delta_X|\cap l_1=\{P, P_1\}$ with 
$(\mult_{P_1}\Delta_X, \mult_{P_1}(\Delta_X\cap l_1))=(2, 2)$.
Moreover, $b=\mult_P(\Delta_X\cap l_2)$ and $k_X=d_1^X+d_2^X-1=4$. Therefore 
only the case \eqref{3_2_-1} occurs. This case induces the type 
\textbf{$[$3$]_{2A}\langle${\bi b}$\rangle$}.

{\noindent\textbf{The case $d_1^X=2$:}}

In this case, $|\Delta_X|\cap l_1=\{P, P_1\}$ and 
$\mult_{P_1}(\Delta_X\cap l_1)=1$. 
Assume that $\mult_{P_1}\Delta_X=2$. 
Then $d_2^X=3$. However, in this case, 
we must have $\mult_{P_1}(\Delta_X\cap l_1)=2$ or $\deg(\Delta_X\cap l_1)=1$ 
by the condition $(\sB 11)$. This is a contradiction. 
Thus 
$\mult_{P_1}\Delta_X=1$.
In this case, $k_X=1+d_2^X$. 
Assume that $l_2$ satisfies the case $({\tt y})$ for ${\tt y}\in\{1, 2\}$. 
If $b\geqslant 2$, 
then this case corresponds to the type 
\textbf{$[$3$]_{2B{\tt y}}\langle${\bi b}$\rangle$}. 
Assume the case $b=1$. Set $Q:=l_1^Z\cap\Gamma_{P, 1}$, 
$c:=\mult_Q(\Delta_Z\cap l_1^Z)$ and 
$d:=\mult_Q(\Delta_Z\cap\Gamma_{P, 1})$. 
Then this case corresponds to the type 
\textbf{$[$3$]_{2B{\tt y}}\langle$1$\rangle$({\bi c,d})}.

{\noindent\textbf{The case $d_1^X=1$:}}

We can show that the case ({\tt y}) 
(${\tt y}\in\{$\ref{3_2_-1}, \ref{3_2_-2}, \ref{3_2_-3}$\}$) corresponds to the 
type \textbf{$[$3$]_{2C{\tt y}}\langle${\bi b}$\rangle$} 
unless ${\tt y}\in\{$\ref{3_2_-1}, \ref{3_2_-2}$\}$ and $b=1$. 
Assume that $b=1$. 
Set $Q:=l_1^Z\cap\Gamma_{P, 1}$, 
$c:=\mult_Q(\Delta_Z\cap l_1^Z)$ and $d:=\mult_Q(\Delta_Z\cap\Gamma_{P, 1})$.  
If ${\tt y}\in\{$\ref{3_2_-1}, \ref{3_2_-2}$\}$, then this corresponds to the 
type \textbf{$[$3$]_{2C{\tt y}}\langle$1$\rangle$({\bi c,d})}.

\subsubsection{The case $E_X=2l_1+l_2$ $(l_i:$ distinct lines$)$ and $P\not\in\Delta_X$, 
where $P=l_1\cap l_2$}

Let $Q\in Z$ be the inverse image of $P\in X$. 
In this case, $Q\in\Delta_Z$ if and only if $((l_2^M)^2)=-3$. 
We note that $d_1^X\leqslant 2$.

{\noindent\textbf{The case $d_1^X=2$:}}

In this case, 
$|\Delta_X|\cap l_1=\{P_1\}$ and 
$(\mult_{P_1}\Delta_X, \mult_{P_1}(\Delta_X\cap l_1))=(2, 2)$. 
Set $Q_1:=l_1^Z\cap\Gamma_{P_1, 2}$, $c:=\mult_{Q_1}(\Delta_Z\cap l_1^Z)$ and 
$d:=\mult_{Q_1}(\Delta_Z\cap\Gamma_{P_1, 2})$. 
Assume that $Q\in\Delta_Z$. Then $d_2^X=2$ and $k_X=4$. This is a contradiction. 
Thus $Q\not\in\Delta_Z$, This corresponds to the type 
\textbf{$[$3$]_{2D}$({\bi c,d})}.

{\noindent\textbf{The case $d_1^X=1$:}}

In this case, one of the following holds:

\begin{enumerate}
\setcounter{enumi}{0}
\renewcommand{\theenumi}{\Alph{enumi}}
\renewcommand{\labelenumi}{(\theenumi)}
\item\label{3_2_E-}
$|\Delta_X|\cap l_1=\{P_1\}$ with 
$(\mult_{P_1}\Delta_X, \mult_{P_1}(\Delta_X\cap l_1))=(2, 1)$.
\item\label{3_2_F-}
$|\Delta_X|\cap l_1=\{P_1\}$ with 
$(\mult_{P_1}\Delta_X, \mult_{P_1}(\Delta_X\cap l_1))=(1, 1)$.
\end{enumerate}

We consider the case \eqref{3_2_E-}. 
Assume that $Q\in\Delta_Z$. Then $d_2^X=2$ and $k_X=4$, a contradiction. 
Thus $Q\not\in\Delta_Z$. This corresponds to the type \textbf{$[$3$]_{2E}$}. 
We consider the case \eqref{3_2_F-}. If $Q\in\Delta_Z$, then this corresponds to 
the type \textbf{$[$3$]_{2F1}$}. 
If $Q\not\in\Delta_Z$, then this corresponds to the type \textbf{$[$3$]_{2F2}$}.

{\noindent\textbf{The case $d_1^X=0$:}}

If $Q\in\Delta_Z$, then this corresponds to 
the type \textbf{$[$3$]_{2G1}$}. 
If $Q\not\in\Delta_Z$, then this corresponds to the type \textbf{$[$3$]_{2G2}$}.

\subsubsection{The case $E_X=l_1+l_2+l_3$ $(l_i:$ distinct lines$)$}

Set $P_{ij}:=l_i\cap l_j$ for $1\leqslant i<j\leqslant 3$. 
By the condition $(\sB 10)$, 
$l_1\cap l_2\cap l_3=\emptyset$ and we can assume that 
$P_{12}$, $P_{13}\in\Delta_X$. 
By Lemma \ref{XS1}, 
$\mult_{P_{ij}}\Delta_X\leqslant 1$ and any component of $E_M$ is reduced. 
Thus $((l_i^M)^2)=-3$ for $i=1$, $2$, $3$. 
If $P_{23}\not\in\Delta_X$, 
then this corresponds to the type \textbf{$[$3$]_{3A}$}. 
If $P_{23}\in\Delta_X$, then this corresponds to the type 
\textbf{$[$3$]_{3B}$}.

\subsection{The case $X=\F_n$}

Let $(X=\F_n, E_X; \Delta_Z, \Delta_X)$ be a 
bottom tetrad such that $2K_X+L_X$ is trivial, 
where $L_X$ is the fundamental divisor. 
We note that $\Delta_X=\emptyset$ and $n=0$ or $2$. In particular, 
$Z=X$. 
Let $\phi\colon M\to Z$ be the elimination of $\Delta_Z$, 
$E_M:=(E_X)_M^{\Delta_Z, 2}$.
Since $2K_X+L_X$ is trivial, we have $L_X\sim 4\sigma+2(n+2)l$, 
$E_X\sim 2\sigma+(n+2)l$ 
and $\deg\Delta_Z=8$.

\subsubsection{The case $n=0$}\label{tetIII0_section}

Take an irreducible component $C\leqslant E_X$. 
Assume that $C$ is singular. 
Then $E_X=C$. In this case, $C$ has a unique singular point which is 
locally isomorphic to plane cubic singularity 
since $C$ is a rational curve. Thus $\deg\Delta_Z\leqslant 1$, 
a contradiction. 
Assume that $C\sim\sigma+2l$. Then $E_X=C+\sigma$ and 
$\deg\Delta_Z\leqslant 2$ 
by Lemmas \ref{ZS2} and \ref{ZS4}, a contradiction. 
Assume that $C\sim\sigma+l$. If $\coeff_CE_X=1$, then 
$\deg\Delta_Z\leqslant 3$ 
by Lemmas \ref{ZS2} and \ref{ZS3}, a contradiction. 
Thus $E_X=2C$. In this case, $\Delta_Z\subset C$. This is nothing but the 
type \textbf{$[$0;2,2$]_0$}. 

From now on, we can assume that any component of $E_X$ is either $\sigma$ or $l$. 
By Lemma \ref{ZS2}, we have $E_X=2\sigma+2l$. 
Set $c:=\mult_Q(\Delta_Z\cap\sigma)$ and $d:=\mult_Q(\Delta_Z\cap l)$. 
We may assume that $c\geqslant d$.
Then $\mult_Q\Delta_Z=c+d$ by Lemma \ref{ZS2}. 
Moreover, $\deg(\Delta_Z\cap\sigma)=\deg(\Delta_Z\cap l)=4$. This is nothing but the 
type \textbf{$[$0;2,2$]_1$({\bi c,d})}.

\subsubsection{The case $n=2$}\label{tetIII2_section}

By the argument in Section \ref{tetIII0_section}, we have 
$E_X=2\sigma_\infty$ or $2\sigma+2l_1+2l_2$. 
If $E_X=2\sigma_\infty$, then this corresponds to the type \textbf{$[$2;2,4$]_0$}. 
If $E_X=2\sigma+2l_1+2l_2$, then this corresponds to the type \textbf{$[$2;2,4$]_1$}. 

Consequently, we have completed the proof of Theorem \ref{tetIII_thm}.

\section{Structure properties}\label{struc_section}

In this section, we treat some structure properties of bottom tetrads, median 
triplets and $3$-basic pairs.

\begin{definition}\label{type_dfn}
For the type of the form \textbf{$[\bullet]_\bullet$($\bullet$)} 
(resp.\ \textbf{$[\bullet]_\bullet\langle\bullet\rangle$($\bullet$)}, 
\textbf{$[\bullet]_\bullet\langle\bullet\rangle$}) of
a bottom tetrad, the form \textbf{$[\bullet]_\bullet$} 
(resp.\ \textbf{$[\bullet]_\bullet\langle\bullet\rangle$}, 
\textbf{$[\bullet]_\bullet\langle\bullet\rangle$})
is said to be the \emph{median part} of the type. 
\end{definition}

The next proposition ensures that there is no overlapping in 
bottom tetrads and in median triplets. 
The proof is essentially same as \cite[Theorem 4.9]{N}.

\begin{table}[p]
\caption{The weighted dual graphs of $E_Z$ for the bottom tetrads 
$(X=\pr^2, E_X; \Delta_Z, \Delta_X)$ with $E_X\sim -K_X$.}\label{E_Z_graph}
\begin{tabular}{|c|c||c|c|}
\hline
{\tiny median} & {\scriptsize Graph} & 
{\tiny median} & {\scriptsize Graph} \\ 
{\tiny part of the type} & & {\tiny part of the type} & \\ \hline \hline 
 & \rule[0mm]{40mm}{0mm}
 & 
 & \rule[0mm]{40mm}{0mm}
\\
\textbf{$[$3$]_{NA}$} &
\begin{picture}(5,5)(0,0)
    \put(0, 2){\makebox(0, 0){\large$\oslash$}}
    \put(7, -6){\makebox(0, 0){\tiny $(1)$}}
    \put(3, 9){\makebox(0, 0)[b]{{\tiny (nodal)}}}
\end{picture}
& 
\textbf{$[$3$]_{NB}$} &
\begin{picture}(4,5)(10,0)
    \put(0, 0){\makebox(0, 0){\textcircled{\tiny $1$}}}
    \put(7, -8){\makebox(0, 0){\tiny $(1)$}}
    \qbezier(5,4)(11,8)(17,4)
    \qbezier(5,0)(11,-4)(17,0)
    \put(22, 0){\makebox(0, 0){\textcircled{\tiny $1$}}}
    \put(29, -8){\makebox(0, 0){\tiny $(1)$}}
\end{picture}
\\
 & & & 
 
\\ \hline

 & & & 
\\
\textbf{$[$3$]_{CA}$} & 
\begin{picture}(45,5)(0,0)
    \put(20, 2){\makebox(0, 0){\large$\oslash$}}
    \put(27, -6){\makebox(0, 0){\tiny $(1)$}}
    \put(23, 9){\makebox(0, 0)[b]{{\tiny (cuspidal)}}}
\end{picture}
 & \textbf{$[$3$]_{CB}$} & 
\begin{picture}(35,5)(10,0)
    \put(10, 2){\makebox(0, -3){\textcircled{\tiny $1$}}}
    \put(17, -6){\makebox(0, 0){\tiny $(1)$}}
    \put(15, 2){\line(1, 0){8}} 
    \put(23, -1){\framebox(6, 6){\tiny $2$}}
    \put(29, 2){\line(1, 0){8}} 
    \put(42, 2){\makebox(0, -3){\textcircled{\tiny $1$}}}
    \put(49, -6){\makebox(0, 0){\tiny $(1)$}} 
\end{picture}
\\
 & & & 
  
\\ \hline

 & & & 
\\
\textbf{$[$3$]_{AA}$} &
\begin{picture}(20,10)(0,0)
    \put(0, -5){\makebox(0, 0){\textcircled{\tiny $1$}}}
    \put(7, -13){\makebox(0, 0){\tiny $(1)$}}
    \put(0, 2){\line(0, 1){10}} 
    \put(0, 16){\makebox(0, 0){\textcircled{\tiny $1$}}}
    \put(7, 8){\makebox(0, 0){\tiny $(1)$}}
    \put(5, 18){\line(1, 0){10}} 
    \put(20, -5){\makebox(0, 0){\textcircled{\tiny $1$}}}
    \put(27, -13){\makebox(0, 0){\tiny $(1)$}}
    \put(5, -3){\line(1, 0){10}}
    \put(20, 2){\line(0, 1){10}}
    \put(20, 16){\makebox(0, 0){\textcircled{\tiny $1$}}}
    \put(27, 8){\makebox(0, 0){\tiny $(1)$}}
\end{picture}
 & \textbf{$[$3$]_{BB}$} & 
\begin{picture}(20,10)(0,0)
    \put(0, -5){\makebox(0, 0){\textcircled{\tiny $1$}}}
    \put(7, -13){\makebox(0, 0){\tiny $(1)$}}
    \put(0, 2){\line(0, 1){10}} 
    \put(0, 16){\makebox(0, 0){\textcircled{\tiny $1$}}}
    \put(7, 8){\makebox(0, 0){\tiny $(1)$}}
    \put(5, -5){\line(5, 2){15}} 
    \put(5, 16){\line(5, -2){15}} 
    \put(23, 4){\makebox(0, 0){\textcircled{\tiny $1$}}}
    \put(30, -3){\makebox(0, 0){\tiny $(1)$}}
\end{picture}
\\
 & & & 

\\ \hline

 & & & 
\\
\textbf{$[$3$]_{KA}$} &
\begin{picture}(40,10)(0,-15)
    \put(0, 0){\makebox(0, 0){\textcircled{\tiny $1$}}}
    \put(7, -8){\makebox(0, 0){\tiny $(1)$}}
    \put(5, 0){\line(1, 0){10}} 
    \put(20, 0){\makebox(0, 0){\textcircled{\tiny $1$}}}
    \put(27, -8){\makebox(0, 0){\tiny $(2)$}}
    \put(25, 0){\line(1, 0){10}}
    \put(40, 0){\makebox(0, 0){\textcircled{\tiny $2$}}}
    \put(47, -8){\makebox(0, 0){\tiny $(1)$}}
    \put(20, -4){\line(0, -1){8}}
    \put(20, -19){\makebox(0, 0){\textcircled{\tiny $1$}}}
    \put(27, -27){\makebox(0, 0){\tiny $(1)$}}
\end{picture}
 & \textbf{$[$3$]_{KB}\langle${\bi b}$\rangle$} & 
\begin{picture}(35,10)(26,-15)
    \put(0, 0){\makebox(0, 0){\textcircled{\tiny $2$}}}
    \put(7, -8){\makebox(0, 0){\tiny $(1)$}}
    \put(5, 1){\line(1, 0){8}} 
    \put(18, 0){\makebox(0, 0){\textcircled{\tiny $1$}}}
    \put(25, -8){\makebox(0, 0){\tiny $(2)$}}
    \put(23, 1){\line(1, 0){8}}
    \put(36, 0){\makebox(0, 0){\textcircled{\tiny $2$}}}
    \put(43, -8){\makebox(0, 0){\tiny $(2)$}}
    \put(41, 1){\line(1, 0){5}}
    \put(47, 1){\line(1, 0){2}}
    \put(50, 1){\line(1, 0){2}}
    \put(53, 1){\line(1, 0){5}}
    \put(63, 0){\makebox(0, 0){\textcircled{\tiny $2$}}}
    \put(70, -8){\makebox(0, 0){\tiny $(2)$}}
    \put(68, 1){\line(1, 0){8}}
    \put(81, 0){\makebox(0, 0){\textcircled{\tiny $2$}}}
    \put(88, -8){\makebox(0, 0){\tiny $(1)$}}
    \put(63, -4){\line(0, -1){8}} 
    \put(63, -19){\makebox(0, 0){\textcircled{\tiny $2$}}}
    \put(70, -27){\makebox(0, 0){\tiny $(1)$}}
    \put(20, -25){\makebox(0, 0){\tiny ($b+2$ vertices)}}
\end{picture}
\\
 & & & 
 
\\ \hline

 & & & 
\\
\textbf{$[$3$]_{KC}\langle${\bi b}$\rangle$}  &
\begin{picture}(35,10)(26,-15)
    \put(0, 0){\makebox(0, 0){\textcircled{\tiny $1$}}}
    \put(7, -8){\makebox(0, 0){\tiny $(1)$}}
    \put(5, 1){\line(1, 0){8}} 
    \put(18, 0){\makebox(0, 0){\textcircled{\tiny $1$}}}
    \put(25, -8){\makebox(0, 0){\tiny $(2)$}}
    \put(23, 1){\line(1, 0){8}}
    \put(36, 0){\makebox(0, 0){\textcircled{\tiny $2$}}}
    \put(43, -8){\makebox(0, 0){\tiny $(2)$}}
    \put(41, 1){\line(1, 0){5}}
    \put(47, 1){\line(1, 0){2}}
    \put(50, 1){\line(1, 0){2}}
    \put(53, 1){\line(1, 0){5}}
    \put(63, 0){\makebox(0, 0){\textcircled{\tiny $2$}}}
    \put(70, -8){\makebox(0, 0){\tiny $(2)$}}
    \put(68, 1){\line(1, 0){8}}
    \put(81, 0){\makebox(0, 0){\textcircled{\tiny $2$}}}
    \put(88, -8){\makebox(0, 0){\tiny $(1)$}}
    \put(63, -4){\line(0, -1){8}} 
    \put(63, -19){\makebox(0, 0){\textcircled{\tiny $2$}}}
    \put(70, -27){\makebox(0, 0){\tiny $(1)$}}
    \put(20, -25){\makebox(0, 0){\tiny ($b+2$ vertices)}}
\end{picture}
 & \textbf{$[$3$]_{2A}\langle${\bi b}$\rangle$} & 
\begin{picture}(35,10)(35,-10)
    \put(0, 0){\makebox(0, 0){\textcircled{\tiny $2$}}}
    \put(7, -8){\makebox(0, 0){\tiny $(1)$}}
    \put(5, 1){\line(1, 0){8}} 
    \put(18, 0){\makebox(0, 0){\textcircled{\tiny $1$}}}
    \put(25, -8){\makebox(0, 0){\tiny $(2)$}}
    \put(23, 1){\line(1, 0){8}}
    \put(36, 0){\makebox(0, 0){\textcircled{\tiny $2$}}}
    \put(43, -8){\makebox(0, 0){\tiny $(2)$}}
    \put(41, 1){\line(1, 0){5}} 
    \put(47, 1){\line(1, 0){2}} 
    \put(50, 1){\line(1, 0){2}} 
    \put(53, 1){\line(1, 0){5}} 
    \put(63, 0){\makebox(0, 0){\textcircled{\tiny $2$}}}
    \put(70, -8){\makebox(0, 0){\tiny $(2)$}}
    \put(68, 1){\line(1, 0){8}} 
    \put(81, 0){\makebox(0, 0){\textcircled{\tiny $1$}}}
    \put(88, -8){\makebox(0, 0){\tiny $(2)$}}
    \put(86, 1){\line(1, 0){8}} 
    \put(99, 0){\makebox(0, 0){\textcircled{\tiny $2$}}}
    \put(106, -8){\makebox(0, 0){\tiny $(1)$}}
    \put(30, -20){\makebox(0, 0){\tiny ($b+4$ vertices)}}
\end{picture}
\\
 & & & 

\\ \hline

 & & & 
\\ 
\textbf{$[$3$]_{2B1}\langle${\bi b}$\rangle$}  &
\begin{picture}(35,10)(35,-10)
    \put(0, 0){\makebox(0, 0){\textcircled{\tiny $2$}}}
    \put(7, -8){\makebox(0, 0){\tiny $(1)$}}
    \put(5, 1){\line(1, 0){8}} 
    \put(18, 0){\makebox(0, 0){\textcircled{\tiny $1$}}}
    \put(25, -8){\makebox(0, 0){\tiny $(2)$}}
    \put(23, 1){\line(1, 0){8}}
    \put(36, 0){\makebox(0, 0){\textcircled{\tiny $2$}}}
    \put(43, -8){\makebox(0, 0){\tiny $(2)$}}
    \put(41, 1){\line(1, 0){5}} 
    \put(47, 1){\line(1, 0){2}} 
    \put(50, 1){\line(1, 0){2}} 
    \put(53, 1){\line(1, 0){5}} 
    \put(63, 0){\makebox(0, 0){\textcircled{\tiny $2$}}}
    \put(70, -8){\makebox(0, 0){\tiny $(2)$}}
    \put(68, 1){\line(1, 0){8}} 
    \put(81, 0){\makebox(0, 0){\textcircled{\tiny $1$}}}
    \put(88, -8){\makebox(0, 0){\tiny $(2)$}}
    \put(86, 1){\line(1, 0){8}} 
    \put(99, 0){\makebox(0, 0){\textcircled{\tiny $1$}}}
    \put(106, -8){\makebox(0, 0){\tiny $(1)$}}
    \put(30, -20){\makebox(0, 0){\tiny ($b+3$ vertices)}}
\end{picture}
 & \textbf{$[$3$]_{2B2}\langle$1$\rangle$}  &
\begin{picture}(30,10)(15,0)
    \put(0, 0){\makebox(0, 0){\textcircled{\tiny $1$}}}
    \put(7, -8){\makebox(0, 0){\tiny $(1)$}}
    \put(5, 1){\line(1, 0){8}} 
    \put(18, 0){\makebox(0, 0){\textcircled{\tiny $1$}}}
    \put(25, -8){\makebox(0, 0){\tiny $(2)$}}
    \put(23, 1){\line(1, 0){8}}
    \put(36, 0){\makebox(0, 0){\textcircled{\tiny $1$}}}
    \put(43, -8){\makebox(0, 0){\tiny $(2)$}}
    \put(41, 1){\line(1, 0){8}} 
    \put(54, 0){\makebox(0, 0){\textcircled{\tiny $1$}}}
    \put(61, -8){\makebox(0, 0){\tiny $(1)$}}
\end{picture}
\\
 & & & 

\\ \hline

 & & & 
\\
\textbf{$[$3$]_{2B2}\langle$2$\rangle$}& 
\begin{picture}(30,10)(20,0)
    \put(0, 0){\makebox(0, 0){\textcircled{\tiny $1$}}}
    \put(7, -8){\makebox(0, 0){\tiny $(1)$}}
    \put(5, 1){\line(1, 0){8}} 
    \put(18, 0){\makebox(0, 0){\textcircled{\tiny $1$}}}
    \put(25, -8){\makebox(0, 0){\tiny $(2)$}}
    \put(23, 1){\line(1, 0){8}}
    \put(36, 0){\makebox(0, 0){\textcircled{\tiny $2$}}}
    \put(43, -8){\makebox(0, 0){\tiny $(2)$}}
    \put(41, 1){\line(1, 0){8}} 
    \put(54, 0){\makebox(0, 0){\textcircled{\tiny $1$}}}
    \put(61, -8){\makebox(0, 0){\tiny $(2)$}}
    \put(59, 1){\line(1, 0){8}} 
    \put(72, 0){\makebox(0, 0){\textcircled{\tiny $1$}}}
    \put(79, -8){\makebox(0, 0){\tiny $(1)$}}
\end{picture}
 & \textbf{$[$3$]_{2C1}\langle${\bi b}$\rangle$}& 
\begin{picture}(35,10)(25,-10)
    \put(0, 0){\makebox(0, 0){\textcircled{\tiny $2$}}}
    \put(7, -8){\makebox(0, 0){\tiny $(1)$}}
    \put(5, 1){\line(1, 0){8}} 
    \put(18, 0){\makebox(0, 0){\textcircled{\tiny $1$}}}
    \put(25, -8){\makebox(0, 0){\tiny $(2)$}}
    \put(23, 1){\line(1, 0){8}}
    \put(36, 0){\makebox(0, 0){\textcircled{\tiny $2$}}}
    \put(43, -8){\makebox(0, 0){\tiny $(2)$}}
    \put(41, 1){\line(1, 0){5}} 
    \put(47, 1){\line(1, 0){2}} 
    \put(50, 1){\line(1, 0){2}} 
    \put(53, 1){\line(1, 0){5}} 
    \put(63, 0){\makebox(0, 0){\textcircled{\tiny $2$}}}
    \put(70, -8){\makebox(0, 0){\tiny $(2)$}}
    \put(68, 1){\line(1, 0){8}} 
    \put(81, 0){\makebox(0, 0){\textcircled{\tiny $0$}}}
    \put(88, -8){\makebox(0, 0){\tiny $(2)$}}
    \put(30, -20){\makebox(0, 0){\tiny ($b+3$ vertices)}}
\end{picture}
\\
 & & & 
 
\\ \hline

 & & & 
\\
\textbf{$[$3$]_{2C2}\langle$1$\rangle$}   &
\begin{picture}(20,10)(10,0)
    \put(0, 0){\makebox(0, 0){\textcircled{\tiny $1$}}}
    \put(7, -8){\makebox(0, 0){\tiny $(1)$}}
    \put(5, 1){\line(1, 0){8}} 
    \put(18, 0){\makebox(0, 0){\textcircled{\tiny $1$}}}
    \put(25, -8){\makebox(0, 0){\tiny $(2)$}}
    \put(23, 1){\line(1, 0){8}}
    \put(36, 0){\makebox(0, 0){\textcircled{\tiny $0$}}}
    \put(43, -8){\makebox(0, 0){\tiny $(2)$}}
\end{picture} 
 & \textbf{$[$3$]_{2C2}\langle$2$\rangle$}& 
\begin{picture}(25,10)(15,0)
    \put(0, 0){\makebox(0, 0){\textcircled{\tiny $1$}}}
    \put(7,-8 ){\makebox(0, 0){\tiny $(1)$}}
    \put(5, 1){\line(1, 0){8}} 
    \put(18, 0){\makebox(0, 0){\textcircled{\tiny $1$}}}
    \put(25, -8){\makebox(0, 0){\tiny $(2)$}}
    \put(23, 1){\line(1, 0){8}}
    \put(36, 0){\makebox(0, 0){\textcircled{\tiny $2$}}}
    \put(43, -8){\makebox(0, 0){\tiny $(2)$}}
    \put(41, 1){\line(1, 0){8}}
    \put(54, 0){\makebox(0, 0){\textcircled{\tiny $0$}}}
    \put(61, -8){\makebox(0, 0){\tiny $(2)$}}
\end{picture}
\\
 & & & 
 
\\ \hline

 & & & 
\\
\textbf{$[$3$]_{2C3}\langle${\bi b}$\rangle$}   &
\begin{picture}(35,10)(33,-15)
    \put(18, 0){\makebox(0, 0){\textcircled{\tiny $0$}}}
    \put(25, -8){\makebox(0, 0){\tiny $(2)$}}
    \put(23, 1){\line(1, 0){8}}
    \put(36, 0){\makebox(0, 0){\textcircled{\tiny $2$}}}
    \put(43, -8){\makebox(0, 0){\tiny $(2)$}}
    \put(41, 1){\line(1, 0){5}}
    \put(47, 1){\line(1, 0){2}}
    \put(50, 1){\line(1, 0){2}}
    \put(53, 1){\line(1, 0){5}}
    \put(63, 0){\makebox(0, 0){\textcircled{\tiny $2$}}}
    \put(70, -8){\makebox(0, 0){\tiny $(2)$}}
    \put(68, 1){\line(1, 0){8}}
    \put(81, 0){\makebox(0, 0){\textcircled{\tiny $2$}}}
    \put(88, -8){\makebox(0, 0){\tiny $(1)$}}
    \put(63, -4){\line(0, -1){8}} 
    \put(63, -19){\makebox(0, 0){\textcircled{\tiny $2$}}}
    \put(70, -27){\makebox(0, 0){\tiny $(1)$}}
    \put(25, -25){\makebox(0, 0){\tiny ($b+1$ vertices)}}
\end{picture}
 & \textbf{$[$3$]_{2D}$}  & 
\begin{picture}(25,10)(15,0)
    \put(0, 0){\makebox(0, 0){\textcircled{\tiny $2$}}}
    \put(7, -8){\makebox(0, 0){\tiny $(1)$}}
    \put(5, 1){\line(1, 0){8}} 
    \put(18, 0){\makebox(0, 0){\textcircled{\tiny $1$}}}
    \put(25, -8){\makebox(0, 0){\tiny $(2)$}}
    \put(23, 1){\line(1, 0){8}}
    \put(36, 0){\makebox(0, 0){\textcircled{\tiny $1$}}}
    \put(43, -8){\makebox(0, 0){\tiny $(2)$}}
    \put(41, 1){\line(1, 0){8}}
    \put(54, 0){\makebox(0, 0){\textcircled{\tiny $2$}}}
    \put(61, -8){\makebox(0, 0){\tiny $(1)$}}
\end{picture}
\\
 & & & 

\\ \hline

 & & & 
\\
\textbf{$[$3$]_{2E}$}   &
\begin{picture}(20,10)(10,0)
    \put(0, 0){\makebox(0, 0){\textcircled{\tiny $2$}}}
    \put(7, -8){\makebox(0, 0){\tiny $(1)$}}
    \put(5, 1){\line(1, 0){8}} 
    \put(18, 0){\makebox(0, 0){\textcircled{\tiny $1$}}}
    \put(25, -8){\makebox(0, 0){\tiny $(2)$}}
    \put(23, 1){\line(1, 0){8}}
    \put(36, 0){\makebox(0, 0){\textcircled{\tiny $2$}}}
    \put(43, -8){\makebox(0, 0){\tiny $(1)$}}
\end{picture} 
 & \textbf{$[$3$]_{2F1}$}  & 
\begin{picture}(20,10)(10,0)
    \put(0, 0){\makebox(0, 0){\textcircled{\tiny $1$}}}
    \put(7, -8){\makebox(0, 0){\tiny $(1)$}}
    \put(5, 1){\line(1, 0){8}} 
    \put(18, 0){\makebox(0, 0){\textcircled{\tiny $0$}}}
    \put(25, -8){\makebox(0, 0){\tiny $(2)$}}
    \put(23, 1){\line(1, 0){8}}
    \put(36, 0){\makebox(0, 0){\textcircled{\tiny $1$}}}
    \put(43, -8){\makebox(0, 0){\tiny $(1)$}}
\end{picture} 
\\
 & & & 
\\ \hline

 & & & 
\\
\textbf{$[$3$]_{2F2}$}   &
\begin{picture}(25,10)(5,0)
    \put(0, 0){\makebox(0, 0){\textcircled{\tiny $2$}}}
    \put(7, -8){\makebox(0, 0){\tiny $(1)$}}
    \put(5, 1){\line(1, 0){8}} 
    \put(18, 0){\makebox(0, 0){\textcircled{\tiny $0$}}}
    \put(25, -8){\makebox(0, 0){\tiny $(2)$}}
    \put(23, 1){\line(1, 0){8}}
    \put(36, 0){\makebox(0, 0){\textcircled{\tiny $1$}}}
    \put(43, -8){\makebox(0, 0){\tiny $(1)$}}
\end{picture}
  & \textbf{$[$3$]_{2G1}$}  & 
\begin{picture}(15,10)(5,0)
    \put(0, 1){\makebox(0, 0){\tiny $1$}}
    \put(0, 1){\circle{14}}
    \put(10, -8){\makebox(0, 0){\tiny $(1)$}}
    \put(7, 1){\line(1, 0){10}} 
    \put(24, 1){\makebox(0, 0){{\tiny $-1$}}}
    \put(24, 1){\circle{14}}
    \put(34, -8){\makebox(0, 0){\tiny $(2)$}}
\end{picture}\\
 & & & 

\\ \hline

 & & & 
\\
\textbf{$[$3$]_{2G2}$}   &
\begin{picture}(15,10)(5,0)
    \put(0, 1){\makebox(0, 0){\tiny $2$}}
    \put(0, 1){\circle{14}}
    \put(10, -8){\makebox(0, 0){\tiny $(1)$}}
    \put(7, 1){\line(1, 0){10}} 
    \put(24, 1){\makebox(0, 0){{\tiny $-1$}}}
    \put(24, 1){\circle{14}}
    \put(34, -8){\makebox(0, 0){\tiny $(2)$}}
\end{picture}
   & \textbf{$[$3$]_{3A}$}  & 
\begin{picture}(30,15)(5,5)
    \put(0, 0){\makebox(0, 0){\textcircled{\tiny $1$}}}
    \put(7, -8){\makebox(0, 0){\tiny $(1)$}}
    \put(0, 7){\line(0, 1){10}} 
    \put(0, 21){\makebox(0, 0){\textcircled{\tiny $1$}}}
    \put(7, 13){\makebox(0, 0){\tiny $(1)$}}
    \put(5, 23){\line(1, 0){10}} 
    \put(5, 0){\line(1, 0){10}} 
    \put(20, 0){\makebox(0, 0){\textcircled{\tiny $1$}}}
    \put(27, -8){\makebox(0, 0){\tiny $(1)$}} 
    \put(20, 21){\makebox(0, 0){\textcircled{\tiny $1$}}}
    \put(27, 13){\makebox(0, 0){\tiny $(1)$}}
    \put(25, 0){\line(5, 2){15}} 
    \put(25, 21){\line(5, -2){15}} 
    \put(43, 9){\makebox(0, 0){\textcircled{\tiny $1$}}}
    \put(50, 2){\makebox(0, 0){\tiny $(1)$}}
\end{picture}
\\
 & & & 

\\ \hline

 & & & 
\\
\textbf{$[$3$]_{3B}$}   &
\begin{picture}(25,15)(10,5)
    \put(0, 0){\makebox(0, 0){\textcircled{\tiny $1$}}}
    \put(7, -8){\makebox(0, 0){\tiny $(1)$}}
    \put(0, 7){\line(0, 1){10}} 
    \put(0, 21){\makebox(0, 0){\textcircled{\tiny $1$}}}
    \put(7, 13){\makebox(0, 0){\tiny $(1)$}}
    \put(5, 23){\line(1, 0){10}} 
    \put(5, 0){\line(1, 0){10}} 
    \put(20, 0){\makebox(0, 0){\textcircled{\tiny $1$}}}
    \put(27, -8){\makebox(0, 0){\tiny $(1)$}}
    \put(20, 21){\makebox(0, 0){\textcircled{\tiny $1$}}}
    \put(27, 13){\makebox(0, 0){\tiny $(1)$}}
    \put(25, 23){\line(1, 0){10}} 
    \put(40, 0){\makebox(0, 0){\textcircled{\tiny $1$}}}
    \put(47, -8){\makebox(0, 0){\tiny $(1)$}}
    \put(25, 2){\line(1, 0){10}}
    \put(40, 7){\line(0, 1){10}}
    \put(40, 21){\makebox(0, 0){\textcircled{\tiny $1$}}}
    \put(47, 13){\makebox(0, 0){\tiny $(1)$}}
\end{picture}

& &
\\
& & &
 
\\ \hline
\end{tabular}
\end{table}

\begin{proposition}\label{kaburi_prop}
\begin{enumerate}
\renewcommand{\theenumi}{\arabic{enumi}}
\renewcommand{\labelenumi}{(\theenumi)}
\item\label{kaburi_prop1}
Let $(Z_i, E_{Z_i}; \Delta_{Z_i})$ $(i=1$, $2)$ 
be median triplets such that both give the same $3$-basic pair 
$(M, E_M)$. Then the type of each triplet is same. 
\item\label{kaburi_prop2}
Let $(X_i, E_{X_i}; \Delta_Z, \Delta_{X_i})$ $(i=1$, $2)$ 
be bottom tetrads such that both give the same pseudo-median triplet 
$(Z, E_Z; \Delta_Z)$. Then the median part of each tetrad is same. 
\item\label{kaburi_prop3}
Let $(X, E_X; \Delta_Z, \Delta_X)$ be a bottom tetrad, $(Z, E_Z; \Delta_Z)$ be the 
associated pseudo-median triplet and $(Z', E_{Z'}; \Delta_{Z'})$ be another 
pseudo-median triplet. If both $(Z, E_Z; \Delta_Z)$ and $(Z', E_{Z'}; \Delta_{Z'})$ 
give same $3$-basic pair, then the two triplets are isomorphic to each other. 
In particular, $(Z', E_{Z'}; \Delta_{Z'})$ is not a median triplet. 
\end{enumerate}
\end{proposition}

\begin{proof}
\eqref{kaburi_prop1}
Let $L_M$ be the fundamental divisor of a $3$-basic pair $(M, E_M)$. 
If $K_M+L_M$ is big, then the corresponding $3$-fundamental triplet is unique up to 
isomorphism. If $K_M+L_M$ is non-big, then the compositions $M\to Z_i\to\pr^1$ are 
same. Thus the assertion follows from the conditions $(\sF 6)$ and $(\sF 7)$. 

\eqref{kaburi_prop2}, \eqref{kaburi_prop3}
Let $L_Z$ be the fundamental divisor of a pseudo-median triplet $(Z, E_Z, \Delta_Z)$. 
If $2K_Z+L_Z$ is big, then the corresponding bottom tetrad is unique up to 
isomorphism. If $2K_Z+L_Z$ is non-big and nontrivial, 
then the compositions $Z\to X_i\to\pr^1$ are 
same. Thus the assertion follows from the conditions $(\sB 6)$, $(\sB 7)$ and $(\sB 8)$.
From now on, assume that $2K_Z+L_Z$ is trivial, that is, $E_Z\sim -K_X$. 
We can assume that $X=\pr^2$. 
In this case, the weighted dual graphs of $E_Z$ are different if the median part 
of the type of bottom tetrads are different by Table \ref{E_Z_graph}. 
Therefore the assertion follows. 
\end{proof}

Finally, as an immediate consequence, we can 
give the weighted dual graphs of all of the $3$-basic pairs.

\begin{proposition}\label{dual_graph_prop}
\begin{enumerate}
\renewcommand{\theenumi}{\arabic{enumi}}
\renewcommand{\labelenumi}{(\theenumi)}
\item\label{dual_graph_prop1}
Let $(Z, E_Z; \Delta_Z)$ be a median triplet and 
$(M, E_M)$ be the associated $3$-basic pair. 
Then the symbol of the weighted dual graph of $E_M$ 
is characterized by the type of the $3$-fundamental triplet 
and is listed in Table \ref{E_M_graph_F}. 
\item\label{dual_graph_prop2}
Let $(X, E_X; \Delta_Z, \Delta_X)$ be a bottom tetrad and $(M, E_M)$ be the associated 
$3$-basic pair. 
Then the symbol $($see Table \ref{graph_table}$)$ of the weighted dual graph of $E_M$ 
is characterized by the type of the bottom tetrad 
and is listed in Tables \ref{E_M_graph_B_k2}, \ref{E_M_graph_B_k1} and 
\ref{E_M_graph_B_k0}. 
\end{enumerate}
\end{proposition}

\begin{longtable}[p]{|c|c||c|c|}
\caption{The symbol of the weighted dual graph of $E_M$ 
for median triplets.}\label{E_M_graph_F}
\\
\hline
Type & Symbol & Type & Symbol  \\ \hline \hline
 \textbf{$[$4$]_0$} & $\gA_1(2)$ 
& \textbf{$[$4$]_{2}$({\bi c,d})} & $\gA_{s(c,d)+2}(2,2)$
\\ \hline
\textbf{$[$5$]_{K}$} & $\gD_4(2)+\gA_1(1)$ 
& \textbf{$[$5$]_{A}$} & $\gA_3(1,1)+\gA_1(1)$ 
\\ \hline
 \textbf{$[$5$]_{3}$({\bi c,d})} & $\gA_{s(c,d)+4}(1,1)+\gA_1(1)$
&\textbf{$[$5$]_{4}$} & $\gD_4(1)+3\gA_1(1)$ 
\\ \hline
 \textbf{$[$5$]_{5}$} & $5\gA_1(1)$ 
& \textbf{$[$0;3,3$]_{D}$} & $\gA_3(1,1)+2\gA_1(1)$
\\ \hline
\textbf{$[$0;3,3$]_{22}$({\bi c,d})} & $\gA_{s(c,d)+4}(1,1)+2\gA_1(1)$ 
& \textbf{$[$0;3,3$]_{23}$} & $\gD_4(1)+4\gA_1(1)$ 
\\ \hline
 \textbf{$[$0;3,3$]_{33}$} & $6\gA_1(1)$
&\textbf{$[$1;3,4$]_{0}$} & $\gA_2(1,2)+\gA_1(1)$ 
\\ \hline
 \textbf{$[$1;3,4$]_{1}$({\bi c,d})} & $\gA_{s(c,d)+3}(1,2)+\gA_1(1)$ 
& \textbf{$[$1;3,4$]_{2}$} & $\gA_3(1,1)+3\gA_1(1)$
\\ \hline
\textbf{$[$1;4,4$]$} & $\gA_1(2)$ 
& \textbf{$[$1;4,5$]_{K}$({\bi c})} & $\gD_{c+1}(2)$ 
\\ \hline
\textbf{$[$1;4,5$]_{A}$} & $\gA_3(1,1)$ 
& \textbf{$[$2;3,5$]_1$} & $\gA_2(1,2)+2\gA_1(1)$ 
\\ \hline
 \textbf{$[$2;3,6$]_{0}$} & $\gA_2(1,2)$
&\textbf{$[$2;3,6$]_{1}$({\bi c,d})} & $\gA_{s(c,d)+3}(1,2)$ 
\\ \hline
 \textbf{$[$3;3,6$]$} & $\gA_1(1)+\gA_1(2)$ 
& \textbf{$[$3;4,9$]_{A}$} & $\gA_4(1,1)$
\\ \hline
\textbf{$[$3;4,9$]_{B}$} & $4\gA_1(1)$ 
& \textbf{$[$3;4,9$]_{C}$({\bi c,d})} & $\gA_{s(c,d)+5}(1,1)$ 
\\ \hline
 \textbf{$[$3;4,9$]_{D}$} & $2\gD_4(1)$
&\textbf{$[$3;4,9$]_{E}$} & $\gD_5(1)+2\gA_1(1)$ 
\\ \hline
 \textbf{$[$3;4,9$]_{F}$} & $\gD_4(1)+2\gA_1(1)$ 
 &\textbf{$[$4;4,10$]_{0}$} & $\gA_2(2,2)$
 \\ \hline
\textbf{$[$4;4,10$]_{1}$({\bi c,d})} & $\gA_{s(c,d)+3}(2,2)$
& \textbf{$[$4;4,10$]_2$} & $2\gA_3(1,1)$ 
\\ \hline
 \textbf{$[$5;4,11$]_{1}$} & $2\gA_2(1,2)$
& \textbf{$[$6;4,12$]_{0}$} & $2\gA_1(2)$
\\ \hline
\end{longtable}

\begin{longtable}[p]{|c|c||c|c|}
\caption{The symbol of the weighted dual graph of $E_M$ 
for bottom tetrads with big $2K_X+L_X$.}\label{E_M_graph_B_k2}
\\ \hline
Type & Symbol & Type & Symbol  \\ \hline \hline
\textbf{$[$1$]_0$} & $\gA_1(1)$ 
& \textbf{$[$2$]_0$} & $\gA_1(1)$ 
\\ \hline
 \textbf{$[$2$]_{1A}$} & $\gD_4(1)$ 
&\textbf{$[$2$]_{1B}$} & $\gD_4(1)+\gA_1(1)$ 
\\ \hline
 \textbf{$[$2$]_{1C}$} & $\gD_5(1)$ 
& \textbf{$[$2$]_{1D}$} & $\gD_5(1)+\gA_1(1)$
\\ \hline
\textbf{$[$2$]_{1E}$({\bi c,d})} & $\gA_{s(c,d)+4}(1,1)$ 
& \textbf{$[$2$]_{1F}$} & $\gA_4(1,1)+\gA_1(1)$ 
\\ \hline
 \textbf{$[$2$]_{1G}$} & $\gA_3(1,1)+2\gA_1(1)$
&\textbf{$[$2$]_{1H}$} & $\gA_3(1,1)+\gA_1(1)$ 
\\ \hline
 \textbf{$[$2$]_{1I}$} & $\gA_3(1,1)$ 
& \textbf{$[$2$]_{1J}$({\bi c,d})} & $\gA_{s(c,d)+3}(1,2)$
\\ \hline
\textbf{$[$2$]_{1K}$} & $\gD_4(2)$ 
& \textbf{$[$2$]_{1L}$} & $\gA_2(1,2)$ 
\\ \hline
 \textbf{$[$2$]_{1M}$} & $\gA_2(1,2)+\gA_1(1)$
&\textbf{$[$2$]_{1N}$} & $\gA_1(2)$ 
\\ \hline
 \textbf{$[$2$]_{2A}$} & $3\gA_1(1)$ 
& \textbf{$[$2$]_{2B}$} & $2\gA_1(1)$
\\ \hline
\textbf{$[$0;1,0$]$} & $\gA_1(1)$ 
& \textbf{$[$0;1,1$]_0$} & $\gA_1(1)$ 
\\ \hline
 \textbf{$[$0;1,1$]_1\langle$0$\rangle$} & $2\gA_1(1)$
&\textbf{$[$0;1,1$]_1\langle$1$\rangle$} & $3\gA_1(1)$ 
\\ \hline
 \textbf{$[$1;1,0$]$} & $\gA_1(1)$ 
& \textbf{$[$1;1,1$]_0$} & $\gA_1(1)$
\\ \hline
\textbf{$[$1;1,1$]_1\langle${0}$\rangle$} & $2\gA_1(1)$ 
&\textbf{$[$1;1,1$]_1\langle${1}$\rangle$} & $3\gA_1(1)$ 
\\ \hline
 \textbf{$[$2;1,0$]$} & $\gA_1(1)$ 
& \textbf{$[$2;1,1$]$} & $2\gA_1(1)$
\\ \hline
\textbf{$[$2;1,2$]_0$} & $\gA_1(1)$ 
& \textbf{$[$2;1,2$]_{1A}$} & $\gD_4(1)$ 
\\ \hline
 \textbf{$[$2;1,2$]_{1B}$} & $\gD_4(1)+\gA_1(1)$
&\textbf{$[$2;1,2$]_{1C}$} & $\gD_5(1)$ 
\\ \hline
 \textbf{$[$2;1,2$]_{1D}$({\bi c,d})} & $\gA_{s(c,d)+4}(1,1)$ 
& \textbf{$[$2;1,2$]_{1E}$} & $\gA_3(1,1)$
\\ \hline
\textbf{$[$2;1,2$]_{1F}$} & $\gA_3(1,1)+\gA_1(1)$ 
& \textbf{$[$2;1,2$]_{1G}$} & $\gA_2(1,2)$ 
\\ \hline
 \textbf{$[$3;1,0$]_0$} & $\gA_1(1)$&&
 \\ \hline

\end{longtable}

\bigskip\bigskip

\begin{longtable}[p]{|c|c||c|c|}
\caption{The symbol of the weighted dual graph of $E_M$ 
for bottom tetrads with non-big, non-trivial $2K_X+L_X$.}\label{E_M_graph_B_k1}
\\ \hline
Type & Symbol & Type & Symbol 
 \\ \hline \hline
\textbf{$[$0;2,0$]$} & $\gA_1(2)$ 
&\textbf{$[$1;2,0$]$} & $\gA_1(2)$
\\ \hline
\textbf{$[$1;2,1$]_{1A}$} & $\gA_2(1,2)$ 
& \textbf{$[$1;2,1$]_{1B}$} & $\gA_2(1,2)+\gA_1(1)$
\\ \hline
 \textbf{$[$1;2,2$]_U$} & $\gA_1(1)$ 
& \textbf{$[$1;2,2$]_{0A}$} & $\gA_2(1,2)$ 
\\ \hline
 \textbf{$[$1;2,2$]_{0B}$} & $\gA_2(1,2)+\gA_1(1)$
&\textbf{$[$1;2,2$]_{0C}$} & $\gA_1(2)$ 
\\ \hline
 \textbf{$[$1;2,2$]_{1A}$} & $\gD_4(2)$ 
& \textbf{$[$1;2,2$]_{1B}$} & $\gD_5(2)$
\\ \hline
\textbf{$[$1;2,2$]_{1C}$} & $\gA_4(1,2)$ 
& \textbf{$[$1;2,2$]_{1D}$({\bi c,d})} & $\gA_{s(c,d)+3}(1,2)$ 
\\ \hline
 \textbf{$[$1;2,2$]_{1E}$({\bi c,d})} & $\gA_{s(c,d)+3}(1,2)+\gA_1(1)$
&\textbf{$[$1;2,2$]_{1F}$({\bi c,d})} & $\gA_{s(c,d)+2}(2,2)$ 
\\ \hline
 \textbf{$[$1;2,2$]_{2A}$} & $\gA_3(1,1)$ 
& \textbf{$[$1;2,2$]_{2B}$} & $\gA_3(1,1)+\gA_1(1)$
\\ \hline
\textbf{$[$1;2,2$]_{2C}$} & $\gA_3(1,1)+2\gA_1(1)$ 
& \textbf{$[$2;2,0$]$} & $\gA_1(2)$ 
\\ \hline
 \textbf{$[$2;2,1$]_{1A}$} & $\gA_2(1,2)$
&\textbf{$[$2;2,1$]_{1B}$} & $\gA_2(1,2)+\gA_1(1)$ 
\\ \hline
 \textbf{$[$2;2,2$]_{1A}$} & $\gD_4(2)$ 
& \textbf{$[$2;2,2$]_{1B}$} & $\gD_5(2)$
\\ \hline
\textbf{$[$2;2,2$]_{1C}$} & $\gA_4(1,2)$ 
& \textbf{$[$2;2,2$]_{1D}$({\bi c,d})} & $\gA_{s(c,d)+3}(1,2)$ 
\\ \hline
 \textbf{$[$2;2,2$]_{1E}$({\bi c,d})} & $\gA_{s(c,d)+3}(1,2)+\gA_1(1)$
&\textbf{$[$2;2,2$]_{1F}$({\bi c,d})} & $\gA_{s(c,d)+2}(2,2)$ 
\\ \hline
 \textbf{$[$2;2,2$]_{2A}$} & $\gA_3(1,1)$ 
&\textbf{$[$2;2,2$]_{2B}$} & $\gA_3(1,1)+\gA_1(1)$
\\ \hline
\textbf{$[$2;2,2$]_{2C}$} & $\gA_3(1,1)+2\gA_1(1)$ 
& \textbf{$[$2;2,3$]_{V}$} & $2\gA_1(1)$ 
\\ \hline
 \textbf{$[$2;2,3$]_{H}\langle${0}$\rangle$} & $3\gA_1(1)$
&\textbf{$[$2;2,3$]_{H}\langle${1}$\rangle$} & $4\gA_1(1)$
\\ \hline
\textbf{$[$2;2,3$]_{2A1}$} & $\gD_5(1)$ 
& \textbf{$[$2;2,3$]_{2A2}$} & $\gD_5(1)+\gA_1(1)$ 
\\ \hline
 \textbf{$[$2;2,3$]_{2B1}$} & $\gD_6(1)$
&\textbf{$[$2;2,3$]_{2B2}$} & $\gD_6(1)+\gA_1(1)$ 
\\ \hline 
\textbf{$[$2;2,3$]_{2C1}$} & $\gA_5(1,1)$ 
& \textbf{$[$2;2,3$]_{2C2}$} & $\gA_5(1,1)+\gA_1(1)$
\\ \hline
\textbf{$[$2;2,3$]_{2D1}$({\bi c,d})} & $\gA_{s(c,d)+4}(1,1)$ 
& \textbf{$[$2;2,3$]_{2D2}$} & $\gA_{s(c,d)+4}(1,1)+\gA_1(1)$ 
\\ \hline
 \textbf{$[$2;2,3$]_{2E1}$({\bi c,d})} & $\gA_{s(c,d)+4}(1,1)+\gA_1(1)$
&\textbf{$[$2;2,3$]_{2E2}$} & $\gA_4(1,1)+2\gA_1(1)$ 
\\ \hline
 \textbf{$[$2;2,3$]_{2F1}$({\bi c,d})} & $\gA_{s(c,d)+3}(1,2)$ 
& \textbf{$[$2;2,3$]_{2F2}$} & $\gA_3(1,2)+\gA_1(1)$
\\ \hline
\textbf{$[$2;2,3$]_{3A}$} & $\gD_4(1)$ 
&\textbf{$[$2;2,3$]_{3B}$} & $\gD_4(1)+\gA_1(1)$ 
\\ \hline
 \textbf{$[$3;2,0$]$} & $\gA_1(2)$
&\textbf{$[$3;2,1$]_{1A}$} & $\gA_2(1,2)$ 
\\ \hline
 \textbf{$[$3;2,1$]_{1B}$} & $\gA_2(1,2)+\gA_1(1)$ 
& \textbf{$[$3;2,2$]_{1A}$} & $\gD_4(2)$
\\ \hline
\textbf{$[$3;2,2$]_{1B}$} & $\gD_5(2)$ 
& \textbf{$[$3;2,2$]_{1C}$} & $\gA_4(1,2)$ 
\\ \hline
 \textbf{$[$3;2,2$]_{1D}$({\bi c,d})} & $\gA_{s(c,d)+3}(1,2)$
&\textbf{$[$3;2,2$]_{1E}$({\bi c,d})} & $\gA_{s(c,d)+3}(1,2)+\gA_1(1)$ 
\\ \hline
 \textbf{$[$3;2,2$]_{1F}$({\bi c,d})} & $\gA_{s(c,d)+2}(2,2)$ 
& \textbf{$[$3;2,2$]_{2A}$} & $\gA_3(1,1)$
\\ \hline
\textbf{$[$3;2,2$]_{2B}$} & $\gA_3(1,1)+\gA_1(1)$ 
& \textbf{$[$3;2,3$]_{0}$} & $2\gA_1(1)$ 
\\ \hline
 \textbf{$[$3;2,3$]_{2A}$} & $\gD_5(1)$
&\textbf{$[$3;2,3$]_{2B}$} & $\gD_6(1)$ 
\\ \hline
 \textbf{$[$3;2,3$]_{2C}$} & $\gA_5(1,1)$ 
& \textbf{$[$3;2,3$]_{2D}$} & $\gA_4(1,1)$
\\ \hline
\textbf{$[$3;2,3$]_{2E}$} & $\gA_4(1,1)+\gA_1(1)$ 
& \textbf{$[$3;2,3$]_{2F}$} & $\gA_3(1,2)$ 
\\ \hline
 \textbf{$[$3;2,3$]_{3}$} & $\gD_4(1)$
&\textbf{$[$4;2,0$]$} & $\gA_1(2)$ 
\\ \hline
 \textbf{$[$4;2,1$]_{1A}$} & $\gA_2(1,2)$ 
& \textbf{$[$4;2,1$]_{1B}$} & $\gA_2(1,2)+\gA_1(1)$
\\ \hline
\textbf{$[$4;2,2$]_{1A}$} & $\gD_4(2)$ 
& \textbf{$[$4;2,2$]_{1B}$} & $\gD_5(2)$ 
\\ \hline
 \textbf{$[$4;2,2$]_{1C}$} & $\gA_4(1,2)$
&\textbf{$[$4;2,2$]_{1D}$} & $\gA_3(1,2)$ 
\\ \hline
 \textbf{$[$4;2,2$]_{1E}$} & $\gA_3(1,2)+\gA_1(1)$ 
& \textbf{$[$4;2,2$]_{1F}$} & $\gA_2(2,2)$
\\ \hline
\textbf{$[$4;2,2$]_{2}$} & $\gA_3(1,1)$ 
& \textbf{$[$5;2,0$]$} & $\gA_1(2)$ 
\\ \hline
 \textbf{$[$5;2,1$]_{1}$} & $\gA_2(1,2)$
&\textbf{$[$6;2,0$]$} & $\gA_1(2)$ 
\\ \hline

\end{longtable}

\begin{longtable}[p]{|c|c||c|c|}
\caption{The symbol of the weighted dual graph of $E_M$ 
for bottom tetrads with $2K_X+L_X\sim 0$.}\label{E_M_graph_B_k0}
\\ \hline
Type & Symbol & Type & Symbol  \\ \hline \hline

 \textbf{$[$3$]_{NA}$} & $\gA_1(1)$ 
& \textbf{$[$3$]_{NB}$} & $2\gA_1(1)$
\\ \hline
\textbf{$[$3$]_{CA}$} & $\gA_1(1)$ 
& \textbf{$[$3$]_{CB}$} & $2\gA_1(1)$ 
\\ \hline
 \textbf{$[$3$]_{AA}$} & $4\gA_1(1)$
&\textbf{$[$3$]_{AB}$} & $3\gA_1(1)$ 
\\ \hline
 \textbf{$[$3$]_{KA}$} & $\gD_4(1)+2\gA_1(1)$ 
& \textbf{$[$3$]_{KB}\langle${\bi b}$\rangle$} & $\gD_{b+2}(1)$
\\ \hline
\textbf{$[$3$]_{KC}\langle${\bi b}$\rangle$} & $\gD_{b+2}(1)+\gA_1(1)$ 
& \textbf{$[$3$]_{2A}\langle${\bi b}$\rangle$} & $\gA_{b+4}(1,1)$ 
\\ \hline
 \textbf{$[$3$]_{2B1}\langle$1$\rangle$({\bi c,d})} & $\gA_{s(c,d)+4}(1,1)+\gA_1(1)$
&\textbf{$[$3$]_{2B1}\langle${\bi b}$\rangle$} & $\gA_{b+3}(1)+\gA_1(1)$ 
\\ \hline
 \textbf{$[$3$]_{2B2}\langle${1}$\rangle$({\bi c,d})} & $\gA_{s(c,d)+4}(1,1)+2\gA_1(1)$ 
& \textbf{$[$3$]_{2B2}\langle${2}$\rangle$} & $\gA_5(1,1)+2\gA_1(1)$
\\ \hline
\textbf{$[$3$]_{2C1}\langle${1}$\rangle$({\bi c,d})} & $\gA_{s(c,d)+3}(1,2)$ 
& \textbf{$[$3$]_{2C1}\langle${\bi b}$\rangle$} & $\gA_{b+2}(1,2)$ 
\\ \hline
 \textbf{$[$3$]_{2C2}\langle${1}$\rangle$({\bi c,d})} & $\gA_{s(c,d)+3}(1,2)+\gA_1(1)$
&\textbf{$[$3$]_{2C2}\langle${2}$\rangle$} & $\gA_4(1,2)+\gA_1(1)$ 
\\ \hline
 \textbf{$[$3$]_{2C3}\langle${\bi b}$\rangle$} & $\gD_{b+1}(2)$ 
& \textbf{$[$3$]_{2D}$({\bi c,d})} & $\gA_{s(c,d)+4}(1,1)$
\\ \hline
\textbf{$[$3$]_{2E}$} & $\gA_3(1,1)$ 
& \textbf{$[$3$]_{2F1}$} & $\gA_3(1,1)+2\gA_1(1)$ 
\\ \hline
 \textbf{$[$3$]_{2F2}$} & $\gA_3(1,1)+\gA_1(1)$
&\textbf{$[$3$]_{2G1}$} & $\gA_2(1,2)+\gA_1(1)$ 
\\ \hline
 \textbf{$[$3$]_{2G2}$} & $\gA_2(1,2)$ 
& \textbf{$[$3$]_{3A}$} & $5\gA_1(1)$
\\ \hline
\textbf{$[$3$]_{3B}$} & $6\gA_1(1)$ 
& \textbf{$[$0;2,2$]_{0}$} & $\gA_1(2)$ 
\\ \hline
 \textbf{$[$0;2,2$]_{1}$({\bi c,d})} & $\gA_{s(c,d)+2}(2,2)$
&\textbf{$[$2;2,4$]_{0}$} & $\gA_1(2)$ 
\\ \hline
 \textbf{$[$2;2,4$]_{1}$} & $\gA_3(2,2)$ &&
\\ \hline
\end{longtable}

\end{document}